\documentclass{article}
\usepackage[british]{babel}
\usepackage{amsmath,amssymb,wasysym,stmaryrd,enumerate,bbm,xcolor,latexsym,theorem,booktabs,longtable}

\usepackage{tikz}
\colorlet{symbols}{blue!90!black!}
\colorlet{testcolor}{green!60!black}

\usetikzlibrary{shapes.misc}
\usetikzlibrary{shapes.symbols}
\usetikzlibrary{shapes.geometric}
\usetikzlibrary{snakes}
\usetikzlibrary{decorations}
\usetikzlibrary{decorations.markings}

\usetikzlibrary{calc}
\usetikzlibrary{external}

\tikzset{
	dot/.style={circle,fill=symbols,draw=symbols,inner sep=0pt,minimum size=0.9mm},
	blackdot/.style={circle,fill=black,draw=black,inner sep=0pt,minimum size=0.9mm}
	}
\makeatletter
\def\DeclareSymbol#1#2#3{\expandafter\gdef\csname MH@symb@#1\endcsname{\tikz[baseline=#2,scale=0.15,draw=symbols]{#3}}}
\def\<#1>{\csname MH@symb@#1\endcsname}
\makeatother

\DeclareSymbol{0}{0}{\draw (0,.8) node[dot] {} ;}
\DeclareSymbol{1}{0}{\draw[white] (-.5,0) -- (.5,0); \draw (0,0)  -- (0,1.5) node[dot] {};}
\DeclareSymbol{11}{0}{\draw (-.5,1.5) node[dot] {} -- (.5,.2) node[dot]{};}

\DeclareSymbol{0b}{0}{\draw (0,.8) node[blackdot] {} ;}
\DeclareSymbol{1b}{0}{\draw[white] (-.5,0) -- (.5,0); \draw[black] (0,0)  -- (0,1.5) node[blackdot] {};}

\newcommand{\noise}{\<0>}
\newcommand{\lolli}{\<1>}
\newcommand{\dumb}{\<11>}
\newcommand{\X}{\textcolor{blue}{X}}
\newcommand{\Xnoise}{\textcolor{blue}{X}\<0>}
\newcommand{\btau}{\textcolor{blue}{\tau}}

\newcommand{\noiseb}{\<0b>}
\newcommand{\lollib}{\<1b>}

\newcommand{\Cnoise}[1]{C_{\textcolor{blue}{1} #1}}
\newcommand{\Cdumb}[1]{C_{\textcolor{blue}{2} #1}}
\newcommand{\tCnoise}[1]{\tilde{C}_{\textcolor{blue}{1} #1}}
\newcommand{\tCdumb}[1]{\tilde{C}_{\textcolor{blue}{2} #1}}

\newcommand{\betanoise}{\beta_{\textcolor{blue}{1}}}
\newcommand{\betadumb}{\beta_{\textcolor{blue}{2}}}

\newcommand{\Mnoise}{M_{\textcolor{blue}{1},2 \nu p}}
\newcommand{\Mdumb}{M_{\textcolor{blue}{2},\nu p}}
\newcommand{\Pnoise}{P_{\textcolor{blue}{1},2 \nu p}}
\newcommand{\Pdumb}{P_{\textcolor{blue}{2},\nu p}}

\newcommand{\blue}[1]{\textcolor{blue}{#1}}

\newcommand{\allsym}{\{\noise,\lolli, \X, \Xnoise,\dumb\}}
\newcommand{\negsym}{\{\noise, \Xnoise,\dumb\}}
\newcommand{\R}{\mathbb R}
\newcommand{\N}{\mathbb N}
\newcommand{\T}{\mathbb T}
\newcommand{\Z}{\mathbb Z}

\def \P{\mathbb P}
\def \E{\mathbb E}

\def \Ftest{\mathfrak F}
\def \heat{(\partial_t - \Delta)}
\def \si{\sigma}

\def \ueps{u^{(\varepsilon)}}
\def \Ceps{C^{(\varepsilon)}}

\def \veps{\varepsilon}
\def \vphi{\varphi}
\def \hxi{\hat{\xi}}
\newcommand{\refchange}[1]{{  #1}}
\newcommand{\seqeps}[1]{#1^{(\varepsilon)}}

\newcommand{\lb}{\langle}
\newcommand{\rb}{\rangle}

\newcommand{\Beta}{\mathrm{B}}
\newcommand{\assign}{:=}
\newcommand{\backassign}{=:}
\newcommand{\cdummy}{\cdot}
\newcommand{\infixand}{\text{ and }}
\newcommand{\mathLaplace}{\Delta}
\newcommand{\mathd}{\mathrm{d}}
\newcommand{\nin}{\not\in}
\newcommand{\nobracket}{}

\newcommand{\tmaffiliation}[1]{\\ #1}
\newcommand{\tmcolor}[2]{{\color{#1}{#2}}}
\newcommand{\tmmathbf}[1]{\ensuremath{\boldsymbol{#1}}}
\newcommand{\tmop}[1]{\ensuremath{\operatorname{#1}}}
\newcommand{\tmtextbf}[1]{\text{{\bfseries{#1}}}}
\newcommand{\tmtextit}[1]{\text{{\itshape{#1}}}}
\newcommand{\tmxspace}{\hspace{1em}}
\newenvironment{enumerateroman}{\begin{enumerate}[i.] }{\end{enumerate}}
\newenvironment{proof}{\noindent\textbf{Proof.\ }}{\hspace*{\fill}$\Box$\medskip}
\newenvironment{proof*}[1]{\noindent\textbf{#1\ }}{\hspace*{\fill}$\Box$\medskip}
\newtheorem{corollary}{Corollary}
\newtheorem{lemma}{Lemma}
\newtheorem{proposition}{Proposition}
{\theorembodyfont{\rmfamily}\newtheorem{remark}{Remark}}
\newtheorem{theorem}{Theorem}

\newtheorem{assumption}{Assumption}

\usepackage{color}
\usepackage[colorlinks=true, pdfstartview=FitV, linkcolor=blue, citecolor=blue, urlcolor=blue,pagebackref=false]{hyperref}

\begin{document}

\title{A priori bounds for the generalised Parabolic Anderson Model}

\author{
  Ajay Chandra$^{\star}$, Guilherme L. Feltes$^{\ast}$ and Hendrik
  Weber$^{\ast}$
  \tmaffiliation{$\star$Department of Mathematics\\
  Purdue University\\
  $\ast$Institut f{\"u}r Analysis und Numerik\\
  Universit{\"a}t M{\"u}nster}
}

\maketitle

\begin{abstract}
  
  We show a priori bounds for solutions to $(\partial_t - \Delta) u =
  \sigma (u) \xi$ in finite volume in the framework of Hairer's Regularity Structures
  [Invent Math 198:269--504, 2014]. We
  assume $\sigma \in C_b^2 (\mathbb{R})$ and that $\xi$ is of negative
  H{\"o}lder regularity of order $- 1 - \kappa$ where $\kappa < \bar{\kappa}$ for an
  explicit $\bar{\kappa}< 1/3$, 
  and that it can be lifted to a model
  in the sense of Regularity Structures. Our main results guarantee
  non-explosion of the solution in finite time and a growth which is at most
  polynomial in $t > 0$.  Our estimates imply global
  well-posedness for the 2-d generalised parabolic Anderson model on
  the torus, as well as for the parabolic quantisation of the
  Sine-Gordon Euclidean Quantum Field Theory (EQFT) on the torus in the regime $\beta^2 \in (4 \pi, (1 +
  \bar{\kappa}) 4 \pi)$. We also consider the parabolic quantisation of a massive Sine-Gordon 
 EQFT and derive estimates that imply the existence of the measure for the same range of $\beta$. 
  Finally, our estimates apply to It\^o SPDEs in the sense of 
  Da Prato-Zabczyk   [\textit{Stochastic Equations in Infinite Dimensions},
  Enc. Math. App., Cambridge Univ. Press, 1992] and imply existence of a stochastic flow beyond the trace-class regime.  
%
\end{abstract}

\section{Introduction}

The aim of this paper is to derive a priori bounds for periodic in space solutions $u : [0,
\infty) \times \R^d \rightarrow \mathbb{R}$ to the equation
\begin{equation}
    (\partial_t - \Delta) u = \sigma (u) \xi ,
    \label{mainequation}
\end{equation}
where $\sigma : \mathbb{R} \rightarrow \mathbb{R}$ is a smooth function and
$\xi \in C^{- 1 - \kappa}$\label{kappaintro}, for $\kappa \in (0, 1 / 3)$, is a (periodic in space) distribution of
negative H{\"o}lder regularity that can be lifted to a model in the sense of
Regularity Structures of {\cite{Hai14}}, see
 \eqref{homogeneityofsymbols}-\eqref{defsmoothdumbwithC}. Equation
\eqref{mainequation} is a prototypical example of a singular Stochastic
Partial Differential Equation (SPDE), which is a class of non-linear Partial
Differential Equations (PDEs), usually driven by a rough noise term, that requires
a renormalisation procedure to make sense. In the case of
\eqref{mainequation}, the (singular) product $\sigma (u)
\xi$ has to be interpreted 
in a renormalised sense if $\kappa > 0$. Since the
groundbreaking works on \tmtextit{Regularity Structures} by
{\cite{Hai14}} and on \tmtextit{Paracontrolled Calculus} by
{\cite{GIP15}}, the \tmtextit{pathwise} study
of the so-called \tmtextit{sub-critical} SPDEs has been developed into
a systematic machinery that yields local solution theory for this class of
equations, see \cite{BCCH20,BHZ19,CH16}.

%

This work is part of a programme whose goal is to
obtain a priori bounds 
to 
go beyond the short time solution theory. This has already been successful
for a number of interesting examples, including $\Phi^4$ theories \cite{MW17plane,
MW173d,AK17,GH19,Moinatandweber2020cpam,CMW23}, the 
Kardar--Parisi--Zhang equation \cite{GP17,PR19,ZZZ22},
and more recently the 2-$d$ Stochastic Navier-Stokes equations \cite{HR23}.

Equations with multiplicative noise, such as \eqref{mainequation}, have been a major focus
in the development of the theory of SPDEs from the beginning: 
 classically, the case of a noise which is white in time, i.e. $\xi = \mathd W$ for an
 infinite-dimensional Wiener process $W$,
has been widely studied in the It{\^o} sense - e.g. \cite{Par75,Wal86,DZ92}.
Equation \eqref{mainequation} was also among the first equations studied in a
pathwise sense \cite{GT10,Hai14,GIP15}.
Despite the prominence of \eqref{mainequation}, pathwise
global in time solutions have only been obtained in  very special cases, namely the linear case 
 $\sigma(u)=u$  \cite{Hai14,GIP15,HL15,HL18}, where global
 existence in finite volume follows immediately from the local theory, and when $\sigma$ has zeroes
 that give trivial super- and subsolutions, see \cite{CFG17}. Here we treat the case of general
 non-linear coefficient $\sigma$ for the first time. 
Our results apply to noise terms $\xi$ where the classical It\^o approach does not apply, including
a noise term $\xi$ which is white in space and constant in time (non-linear parabolic Anderson model)
and - after a transformation - to  the stochastic quantisation equation for the Sine-Gordon model. 
Our method also gives new estimates in the case of ``white-in-time" noise where we obtain the
existence of a stochastic flow in situations where the covariance of the infinite dimensional Wiener
process $W$ is not of trace class.

We introduce some notions from the theory of Regularity Structures to state our main result: 
we denote by  $T \assign \allsym$ a set of symbols with their associated \textit{homogeneities} 
\begin{equation}
  |\noise| = -1 - \kappa , \quad
  |\lolli| = 1 - \kappa , \quad
  |\X| = 1, \quad
  |\Xnoise| = - \kappa , \quad
  |\dumb| = - 2 \kappa ,
  \label{homogeneityofsymbols}
\end{equation}
corresponding to the assumed regularity $\xi \in C^{- 1 - \kappa}$. We furthermore require a base point dependent map $z \mapsto \Pi_z$.
The map
$\Pi$ is called a \textit{lift} of $\xi$ and it
maps $T$ into stochastic objects built from $\xi$ that satisfy certain properties.
\refchange{
We mention right away that the local solution theory using Regularity Structures applies in the so-called full subcritical regime which amounts to choosing $\kappa < 1$.  As soon as $\kappa \geqslant 1/3$ more symbols than 
$\allsym$ become necessary. 
Our global in time theory will only hold for  $\kappa < \bar \kappa < 1/3$ where the threshold $\bar{\kappa}$ seems unrelated to the regularity thresholds that appear in the local theory. 
For this reason, the set of symbols $\allsym$ is enough for our purposes.
} 

As concrete counterparts of the abstract symbols $\noise, \lolli, \X \in T$, we describe the
objects $\noiseb, \lollib, X$\footnote{This notation differs slightly from \cite{Hai14}, where the black symbols correspond to the boldface $\mathbf \Pi$ map that realises the abstract blue symbols into concrete objects.} as follows. $\noiseb$ represents a realisation of the noise $\xi$, while
$\lollib$ denotes a periodic in space solution to the linear equation
\begin{equation}
  (\partial_t - \Delta) \lollib = \noiseb .
  \label{definitionlollipop}
\end{equation}
We point out that the initial condition on the definition
of $\lollib$ in \eqref{definitionlollipop} does not appear directly in our analysis,
and we leave its prescription for the examples we present later. 
For a space-time point $z = (t,x) \in [0,\infty) \times \R^d$, we let $X$ denote its spatial component
$X(z) = x$.

To ease the presentation, we now make the qualitative assumption that $\xi$ is smooth
and present a lift under which we may state our main theorems. For this, we describe how the
map $z \mapsto \Pi_z$ relates the abstract objects $\allsym$ with the concrete objects $\noiseb, \lollib, X$
recentred at a space-time point $z$ as follows
\begin{eqnarray}
  \Pi_z \noise & = & \noiseb
  \label{Pinoise} \\
  \Pi_z \lolli & = & \lollib - \lollib(z)
  \label{Pilolli} \\
  \Pi_z \X & = & X - X(z) 
  \\
  \Pi_z \Xnoise & = & 
  (X - X (z))\noiseb
  \label{defsmoothXnoise} \\
  \Pi_z \dumb & = & 
  (\lollib - \lollib (z)) \noiseb - C .
  \label{defsmoothdumbwithC}
\end{eqnarray}
In applications, we think of $\xi$ as a regularisation of the rough noise 
in \eqref{mainequation},
so that $C \in \R$ in \eqref{defsmoothdumbwithC} is a renormalisation constant that
depends on the regularisation scale and typically diverges as the regularisation is removed.

The homogeneity of the symbols $\allsym$ in \eqref{homogeneityofsymbols} indicates
the targeted regularity of these objects,
allowing for rough noise in \eqref{mainequation}.
We measure regularity as follows.
For $z = (t,x),\bar{z} = (\bar{t},\bar{x})$, 
consider the parabolic distance
\begin{equation}
  d ((t, x), (\bar{t}, \bar{x})) = \max \left\{ \sqrt{| t - \bar{t} |}, | x -
  \bar{x} | \right\}, \label{parabolicdistance}
\end{equation}
where $|x - \bar{x}|$ denotes the Euclidean distance between $x,\bar{x} \in \R^d$.
The parabolic ball centred at $z$ of radius $L > 0$
only looking to
the past is defined as
\begin{equation}
  B (z, L) = \{ \bar{z} = (\bar{t}, \bar{x}) \in [0,\infty) \times \R^d, d (z, \bar{z})
  < L, \bar{t} < t \} . \label{balllookingtopast}
\end{equation}
Let $\mathfrak F$ be the set of all smooth, non-negative functions
$\varphi$, with support in $B (0, 1)$, such that 
$\varphi (t,x) = \varphi (t,-x)$ 
for every $z=(t,x)$, with all of its derivatives of up to order 2
 bounded and with $\int \varphi = 1$. For $\varphi \in \mathfrak F$,
 $z \in [0,\infty) \times \R^d$ and $L>0$, set
 \begin{equation}
  \varphi_z^L (\bar{z}) = \frac{1}{L^{d + 2}} \varphi \left( \frac{\bar{t} -
  t}{L^2}, \frac{\bar{x} - x}{L} \right) .
  \label{defofmollifierkernel}
\end{equation}
%
%
Then, we measure the size of $\Pi$ with the following semi-norms. For any
$B \subset [0,\infty) \times \R^d$, define for the symbols $\btau \in \negsym$
of negative homogeneity,
\begin{equation}
  [\Pi;\btau]_{|\btau|, B} \assign \sup_{\varphi \in \mathfrak F}
  \sup_{z \in B}
  \sup_{L < 1} | \langle \Pi_z \btau, \varphi_z^L \rangle | L^{-|\btau|} ,
  \label{orderbounds}
\end{equation}
and for $\lolli$ of positive homogeneity,
\begin{equation}
  [\Pi;\lolli]_{1 - \kappa, B} \assign \sup_{z \neq w \in B}
  \frac{| (\Pi_z \lolli) (w) |}{d (z, w)^{1 - \kappa}} .
  \label{orderboundlollipop}
\end{equation}
\refchange{
We consider $2\pi$-periodic boundary conditions in $\R^d$ 
and identify 
$(-\pi,\pi]^d$ with the $d$-dimensional torus
$\T^d := (\R/2\pi \Z)^d$.
Let $D := [0,\infty) \times \T^d$ and $D_a^b := [a,b] \times \T^d$ for $0 \leqslant a < b$. \label{spacetimedomain}
For any $n \in \mathbb{N}$, define 
\begin{equation}
    \Cnoise{,n} := \max_{\btau \in \{\noise, \lolli, \Xnoise\}} [\Pi;\btau]_{|\btau|,D_{n - 1}^n}
    \quad \tmop{and} \quad
    \Cdumb{,n} := [\Pi;\dumb]_{- 2 \kappa, D_{n-1}^n} .
    \label{assumptiononstochasticobjects}
\end{equation}
The constant $\Cnoise{,n} > 0$ controls the objects that are \textit{linear} in the noise
over time intervals from $t=n-1$ to $t=n$
and the constant $\Cdumb{,n} > 0$ controls the object that is \text{quadratic} in the noise
over time intervals from $t=n-1$ to $t=n$.}

In view of 
 the definition of $\Pi_z \dumb$ for smooth
noise in \eqref{defsmoothdumbwithC}, we consider $u : D \to \mathbb R$ to be
the (smooth) $2\pi$-periodic in space solution to the renormalised 
version of \eqref{mainequation}
\refchange{
\begin{equation}
  \left\{\begin{array}{rllll}
    \heat u & = & \si (u) \diamond \xi & , & t >  0 \\
    u & = & u_0 & , & t = 0
  \end{array}\right.,
  \label{mainequationrenormalised}
\end{equation}
where we introduce $\si(u) \diamond \xi$ as the
renormalised product of $\si(u)$ and $\xi$, 
\begin{equation*}
    \si (u) \diamond \xi 
    := \si (u) \xi - \si' (u) \si (u) C \; \; ,
\end{equation*}
with the same constant $C$ 
as in \eqref{defsmoothdumbwithC}.
$u_0 : \T^d \to \R$
denotes a bounded initial condition.   
}

In what follows, $\| \cdot \|$ denotes the supremum norm over $\R^d$ or $\R$. \label{supnorm0} 
Similarly, for a $B \subset \R^{d+1}$ we write $\| \cdot \|_{B}$ for the supremum norm over $B$. \label{supnorm1}
We work under the following assumption.

\begin{assumption}
  \label{nonlinearsigma}
  We assume 
  $\sigma \in C_b^2 (\mathbb{R})$ and that there exists a
  constant $C_{\sigma} > 0$, for which
  \begin{equation}
    \max \{ \| \sigma \|, \| \sigma' \|, \| \sigma'' \| \} \leqslant
    C_{\sigma} . \label{assumptionsigmabounded}
  \end{equation}
\refchange{For $\kappa>0$, define $\betadumb$ and fix $\betanoise$
according to}
\begin{equation}
  \betadumb := \frac{2 (1 + \kappa)}{3 - 2 \kappa} 
  \quad \tmop{and} \quad 
    1 > \betanoise > \betadumb + \frac{(1 + \kappa) \kappa}{1 - \kappa} 
   \label{betas} .
\end{equation}
\refchange{We assume} that we take $\kappa < \bar{\kappa}$, where $\bar{\kappa}$ is the solution of
\begin{equation}
 (1 + \kappa) (2 + \kappa - 2 \kappa^2) = (3 - 2 \kappa) (1 - \kappa)
 \label{equationforkappa}
\end{equation}
that belongs to the interval $(0,1/3)$.
\end{assumption}

\refchange{The choice above allows us to choose $\betanoise$ satisfying \eqref{betas}
and such that $\betadumb < \betanoise < 1$}. In particular, this is possible if we take $\kappa < 0.132$. See Section \ref{sectionstrategyofproof} for the role that $\betanoise, \betadumb$ and $\bar{\kappa}$ play in the analysis.

Our first main result is control on $u$ over the time interval $[0, 1]$.
\begin{theorem}
  \label{maintheorem}Let $u$ be the periodic in space solution to \eqref{mainequationrenormalised}. \refchange{Assume that
    $\kappa < \bar{\kappa}$, for $\bar{\kappa} \in (0,1/3)$ solution to
  \eqref{equationforkappa}, as in Assumption \ref{nonlinearsigma}. 
  Then, }
  \begin{equation}
    \| u \|_{D_0^1} \leqslant C (\kappa, d, C_{\sigma}) \max \left\{ \| u_0
    \|, \Cnoise{,1}^{\frac{2}{(1 - \kappa) (1 - \betanoise)}},
    C_{\textcolor{blue}{2},1}^{\frac{1}{(1 - \kappa) (1 - \betadumb)}} \right\},
    \label{mainresultbound}
  \end{equation}
  where $C (\kappa, d, C_{\sigma}) > 0$ is a constant that only depends on
  $\kappa, d$ and $C_{\sigma}$, and the estimate is independent of the renormalisation
  constant $C>0$ in \eqref{defsmoothdumbwithC}-\eqref{mainequationrenormalised}.
  The exponents $0 < \betadumb 
  < \betanoise < 1$ are as in \eqref{betas}.
\end{theorem}

Our second main result is a polynomial growth bound for $\| u \|_{D_0^t}$
in $t$.

\begin{theorem}
  \label{theorempolynomialgrowth}Under the same assumptions of Theorem
  \ref{maintheorem}, for every $n < t \leqslant n + 1$, it holds that
  \begin{equation}
    \| u \|_{D_0^t} \leqslant C' (\kappa, d, C_{\sigma}) \max \left\{ \| u_0
    \|, \max_{1 \leqslant i \leqslant n + 1} \left\{ \Cnoise{,i}^{\frac{2}{1 - \kappa}}, 
    \Cdumb{,i}^{\frac{1}{1 - \kappa}}
    \right\}^{\frac{1}{(1 - \betanoise)^2}} t^{\frac{1}{1 - \betanoise}}
    \right\} \label{polynomialgrowth}
  \end{equation}
  where $C' (\kappa, d, C_{\sigma}) > 0$ is a constant that only depends on
  $\kappa, d$ and $C_{\sigma}$, and the estimate is independent of the renormalisation
  constant $C>0$ in \eqref{defsmoothdumbwithC}-\eqref{mainequationrenormalised}.
  $\betanoise$ is the same as in Theorem
  \ref{maintheorem}.
\end{theorem}


As a direct consequence of Theorem \ref{theorempolynomialgrowth}, we have the
following pathwise growth estimate result.

\refchange{
\begin{corollary}[pathwise growth]
  \label{pathwisegrowth}
  Let the assumptions of Theorem \ref{maintheorem} be in force. Assume that
  \begin{equation}
    \max_{1 \leqslant i \leqslant n} \{ \Cnoise{,i}^2, \Cdumb{,i} \}
    \leqslant C (\xi) n^\delta ,
    \label{weightsinstochasticobjects}
  \end{equation}
  where $\delta \geqslant 0$ and $C (\xi) > 0$ is a constant
  that depends on the realisation of $\xi$. 
  Then, for $\bar \delta
  := \delta (1 - \kappa)^{-1} (1 - \betanoise)^{-2}$,
  \eqref{polynomialgrowth}
  translates to
  \[ \| u \|_{D_0^t} \leqslant C' (\kappa, d, C_{\sigma}) \max \left\{ \| u_0
     \|, C (\xi) (t + 1)^{\frac{1}{1 - \betanoise} + \bar \delta} \right\}, \]
  where $C' (\kappa, d, C_{\sigma}) > 0$ is the same constant as in Theorem
  \ref{theorempolynomialgrowth}
  and the estimate is independent of the renormalisation
  constant $C>0$ in \eqref{defsmoothdumbwithC}-\eqref{mainequationrenormalised}.
\end{corollary}
}

\refchange{
\begin{remark} \label{second.paper.remark.p.moments.imply.controlled.growth.in.n.remark}
	Note that if the constants $\{\Cnoise{,n}, \Cdumb{,n}\}_{n \in \N}$ satisfy
	    \begin{equation}
 		   \sup_{n \geqslant 1} \E[\Cnoise{, n}^{2p}] < \infty
		    \quad \text{and} \quad 
		    \sup_{n \geqslant 1} \E[\Cdumb{, n}^p] < \infty \; \; ,
		    \label{second.paper.remark.p.moments.imply.controlled.growth.in.n}
	  \end{equation}
	  for some $p \geqslant 1$,
	then by setting for any $\delta > 0$
	\begin{equation*}
		C_\delta(\xi) := \sup_{n \in \N} 
		\frac{\max \{ \Cnoise{,n}^2, \Cdumb{,n} \} }{n^\delta} \; \; ,
	\end{equation*}
	we see that if $\delta > p^{-1}$, then
	\begin{equation*}
		\E [|C_\delta(\xi)|^p] 
		\leqslant \E\left[ \sum_{n=1}^\infty \left| \frac{\max \{ \Cnoise{,n}^2, \Cdumb{,n} \} }{n^\delta} \right|^p \right]
		\leqslant \sum_{n=1}^\infty \frac{1}{n^{\delta p}} < \infty \; \; .
	\end{equation*}
	In particular, if \eqref{second.paper.remark.p.moments.imply.controlled.growth.in.n} holds for every $p \geqslant 1$,
	then \eqref{weightsinstochasticobjects} holds for any $\delta>0$ and the random variable $C(\xi) = C_\delta(\xi)$
	has all $p$-moments.
\end{remark}
}

\begin{remark} \label{remarkmultinoise}
  Our main theorems extend to equations as \eqref{mainequation} with $\sigma(u) \xi$
  replaced by $\sum_{i=1}^m \sigma_i(u) \xi_i$, for $m \in \N$, at the expense of adding decorations
  to the symbols $\allsym$. In fact, let $\noiseb_i = \xi_i$ and define 
  $\Pi_z \lolli_i$ and $\Pi_z \Xnoise_i$ according to \eqref{Pinoise}-\eqref{defsmoothXnoise}.
  Then, the second order terms $\Pi_z \dumb_{i,j}$, for $i,j \in \{1,\dots,m\}$, are given by
  \begin{equation}
    z \mapsto \Pi_z \dumb_{i,j} = 
    (\lollib_j - \lollib_j (z)) \noiseb_i - C_{i,j} ,
    \label{decorateddumb}
  \end{equation}
  for $m^2$ renormalisation constants $C_{i,j} \in \R$. 
  In this case, equation \eqref{mainequation} is renormalised as
  \begin{equation}
  (\partial_t - \Delta) u = \sum_{i=1}^m \sigma_i (u) \xi_i - \sum_{i,j=1}^m \sigma_j' (u) \sigma_i (u)C_{i,j} .
  \label{multicompequationrenormalised}
\end{equation}
  The order bounds \eqref{orderbounds} and \eqref{orderboundlollipop} are defined in the same way.
 \end{remark}
 
 \begin{remark} \label{remarkgeneraltheorem}
   The qualitative smoothness assumption on the noise $\xi$ is made for ease of presentation:
   Theorems \ref{maintheorem} and \ref{theorempolynomialgrowth} hold under more general
   assumptions cf. the theory of Regularity Structures. Our definition of $\lollib$ in \eqref{definitionlollipop}
   corresponds to \cite[Def. 5.9]{Hai14}. The formulas for $\Pi_z \Xnoise$ in \eqref{defsmoothXnoise}
   and for $\Pi_z \dumb$ in \eqref{defsmoothdumbwithC} can be relaxed into the change of base
   point (CBP) formulas
   \begin{eqnarray}
     \Pi_z\Xnoise & = & \Pi_w\Xnoise - (\Pi_w \X) (z) \Pi_w \noise
     \label{changeofbasepointXnoise}, \\
     \Pi_z \dumb & = & \Pi_w\dumb - (\Pi_w \lolli) (z) \Pi_w \noise .
     \label{changeofbasepointdumbbell}
    \end{eqnarray}
    In view of the order bounds (OB) in \eqref{orderbounds} and \eqref{orderboundlollipop}, the constants
    $\Cnoise{,n}$ and $\Cdumb{,n}$ in 
  \eqref{assumptiononstochasticobjects} are assumed to be finite for every $n \in \N$. Conditions CBP and OB
    correspond to \cite[eq. (2.15) in Def. 2.17]{Hai14}.
    These conditions are flexible enough to allow for the renormalisation present in \eqref{defsmoothdumbwithC},
    but also are independent of the smoothness of $\xi$.
    
    Finally, the smoothness assumption on the noise can be replaced by the assumption that
    the function $\mathcal U : D \to \operatorname{span} \{ \blue{\mathbf 1}, \lolli , \X\}$ given by 
    \begin{equation*}
       \mathcal U(z) = u(z) \blue{\mathbf 1} + \si(u(z)) \lolli + u_X(z) \cdot \X, 
    \end{equation*}
    is a \textit{modelled distribution} in $\mathcal D^{2-2\kappa}$ 
    w.r.t. $\Pi$, cf. \cite[Def. 3.1]{Hai14}. 
    In our notation
    this corresponds to the assumption that 
    the semi-norm $[U]_{2-2\kappa}$ in \eqref{finitegammanormofU} is finite. 
    Due to Assumption
    \ref{nonlinearsigma} on the regularity of $\si$, the function
    \begin{equation*}
      \si( \mathcal U)(z) = \si(u(z)) \blue{\mathbf 1} + \si'\si(u(z)) \lolli + \si'(u(z))u_X(z) \cdot \X
    \end{equation*}
    is also in $\mathcal D^{2-2\kappa}$ - see \cite[Thm 4.16]{Hai14}. 
    This allows to define the singular product $\si(u) \xi$ in
    \eqref{mainequation} as the reconstruction of the modelled distribution $\si(\mathcal U) \noise$,
    denoted by $\mathcal R(\si(\mathcal U) \noise)$,
    cf. Theorem \ref{reconstructiontheorem}, which is a variant of \cite[Thm 3.10]{Hai14}.   
   
   Then, equation \eqref{mainequation} can be interpreted as
    \begin{equation*}
      \heat u = \mathcal R(\si(\mathcal U) \noise) , 
    \end{equation*}
  regardless of the smoothness of $\xi$. Our analysis goes through in this case without changes.
  In fact, conditions CBP, OB and
  $[U]_{2 - 2\kappa} < \infty$ in \eqref{finitegammanormofU} are the only assumptions we use in our proofs.
 \end{remark}

We shall now discuss some concrete examples that fit into the general form of \eqref{mainequation}.
Since our applications deal with random noise $\xi$, we \refchange{write 
$(\Omega,\P)$ to denote the probability space that supports a specific choice of random distribution $\xi$ in \eqref{mainequation}
  and $\E$ to denote  expectation w.r.t. $\P$. }
The case of the 2-$d$ gPAM corresponds to the choice of $\xi$ as a $2\pi$-periodic spatial white noise over
$\R^2$, which satisfies $\xi \in \bigcap_{\delta>0} C^{-1-\delta}$. $\xi$ can be realised over $\T^2$ via the random
Fourier series (denote $e_k(x) := e^{i k\cdot x}$, for $x \in \T^2$) 
\begin{equation}
  \xi = \sum_{k \in \Z^2} \hat{\xi}_k e_k ,
  \label{fourierserieswhitenoisegpam}
\end{equation}
\refchange{where the}
random coefficients $\hat{\xi}_k$ are independent complex Gaussians (up to $\bar{\hxi}_k = \hxi_{-k}$) with
mean zero and covariance $\E[\hxi_k \hxi_{-l}] = (2\pi)^{-2} \delta_{k l}$, and where $\delta_{kl}$ is the \textit{Kronecker} delta.
For any $z=(t,x)\in D$, we define
\begin{equation}
   \lollib(z) := \hxi_0 t + \tilde{\lollib}(x) ,
   \quad \text{where} \quad
   \tilde{\lollib} := \sum_{k \in \Z^2, k \neq 0} \frac{\hat{\xi}_k}{|k|^2} e_k ,
   \label{fourierserieslolligpam}
\end{equation}
which satisfies \eqref{definitionlollipop} starting from $\lollib_{|t=0} = \tilde{\lollib}$. Even though $\lollib$ is not stationary in time,
its increments are. 
Indeed, $\tilde{\lollib}$ is time-independent, and for every
$p\geqslant1$, we have $\E[|\tilde{\lollib}(z) - \tilde{\lollib}(\bar{z})|^p] \lesssim C_{p,\xi} |x-\bar{x}|^{(1-\delta)p}$ for any $\delta>0$.
Therefore, we obtain that $[\Pi ; \lolli]_{1-\delta,D_{n-1}^n} < \infty$ in \eqref{orderboundlollipop} a.s. for any 
$n \in \N$ and $\delta>0$. If we denote by $\noiseb^{(\veps)}$ and $\lollib^{(\veps)}$ the regularisation at scale
$\veps$ of the objects in \eqref{fourierserieswhitenoisegpam} and \eqref{fourierserieslolligpam} via
ultraviolet cut-off, 
$\Pi_z \dumb$ can be constructed as
\begin{equation*}
  \Pi_z \dumb =
  \lim_{\veps \downarrow 0} \left(\lollib^{(\veps)} - \lollib^{(\veps)}(z)\right) \noiseb^{(\veps)} - C^{(\veps)} ,
\end{equation*}
where
\begin{equation}
  \Ceps = 
  \E[\tilde\lollib^{(\veps)}(0) \bar \noiseb^{(\veps)}(0)] 
  = \frac{1}{(2\pi)^2} \sum_{0<|k| \leqslant 1/\veps} \frac{1}{|k|^2}
  \sim \frac{1}{(2\pi)^2} \log (\veps^{-1}).
  \label{choiceofCepsgpam}
\end{equation}
The arguments in \cite[Sec. 10.4]{Hai14} can be followed to show that 
the OBs
in \eqref{orderbounds}-\eqref{orderboundlollipop} satisfy
  \begin{equation}
    \E[\Cnoise{}^p], \E[\Cdumb{}^p] < \infty
    \quad \text{for} \quad
    \Cnoise{} := \sup_{n \in \N} \Cnoise{,n}
    \quad \text{and} \quad
    \Cdumb{} := \sup_{n \in \N} \Cdumb{,n}
    \label{momentsforgPAM}
  \end{equation}
  uniformly in $\veps>0$, for every $p \geqslant 1$, where
 and $\Cnoise{,n}, \Cdumb{,n} < \infty$ are as in 
\eqref{assumptiononstochasticobjects}.
\refchange{We refer to \cite[Thm 3.3.4]{dLF24} for the precise statement and detailed proof of the estimates for the objects defined above.}

Furthermore, \cite[Sec. 1.5.1]{Hai14} states that the solutions $\ueps$
to \eqref{mainequationrenormalised} w.r.t. $\noiseb^{(\veps)}$ and $\Ceps$
in \eqref{choiceofCepsgpam} converge in a suitable sense as $\varepsilon \downarrow 0$ to
a limiting object $u$, which is interpreted as the solution to \eqref{mainequation}.
See \cite[Sec. 5]{GIP15} for a similar discussion.
Since our estimates
in Theorems \ref{maintheorem} and \ref{theorempolynomialgrowth} are independent of
$C^{(\varepsilon)}$, they pass to the limit $\varepsilon \downarrow 0$. 

\begin{corollary}[2-d gPAM]
  Let $\xi$ be a $2\pi$-periodic spatial white noise over $\R^2$, $\si$ satisfying Assumption
  \ref{nonlinearsigma} and $\nu = (1 - \kappa)^{-1}(1 - \betanoise)^{-2}$.
  Then 
  \refchange{ 
  \eqref{mainequationrenormalised},}
  is globally well-posed for any $t > 0$ 
  and \refchange{the solution $u$} satisfies the moment estimate
  \begin{equation*}
    \E[\| u \|_{D_0^t}^p] \leqslant C' (\kappa, d, C_{\sigma}) \max \left\{ \| u_0\|,
    \E[\Cnoise{}^{2\nu p}] t^{\frac{p}{1 - \betanoise}},
    \E[\Cdumb{}^{\nu p}] t^{\frac{p}{1 - \betanoise}}
    \right\} ,
  \end{equation*}
  where 
  $\E[\Cnoise{}^{2\nu p}], \E[\Cdumb{}^{\nu p}]$ in \eqref{momentsforgPAM} are finite
  for any $p \geqslant 1$ and $\betanoise$ is the same as in Theorem
  \ref{maintheorem}.
\end{corollary}

Equation \eqref{mainequation} also has connections with additive SPDEs. 
The parabolic
quantisation of the $\Phi_d^4$ Euclidean Quantum Field Theory (EQFT),
which is the singular SPDE formally given by
\[ (\partial_t - \Delta) \Phi = - \Phi^3 + \zeta, \]
where $\zeta$ is d-dimensional space-time white noise for $d \geqslant 2$. 
If we set $u = \Phi - Z$, where
$Z$ solves $(\partial_t - \Delta) Z = \zeta$, then $u$ itself solves
\[ (\partial_t - \Delta) u = - u^3 - u^2 Z - u Z^2 - Z^3. \]
This in turn can be seen as a gPAM in \eqref{mainequation} with extra damping
$- u^3$, where $\sigma (u) \xi$ in \eqref{mainequation} is replaced by
$\sum_{i = 1}^3 \sigma_i (u) \xi_i$ with $\sigma_i (u) = u^{3 - i}$, $i = 1,
2, 3$ and $\xi_i = Z^i$, cf. Remark \ref{remarkmultinoise}.
The global solution analysis in 
\cite{MW17plane,MW173d,AK17,GH19,Moinatandweber2020cpam,CMW23}
 relies heavily on this transformation and the strong damping term.
Our analysis applies to EQFTs without such strong damping, such as
the parabolic quantisation of the 2-dimensional
Sine-Gordon model (SG). Consider $\Phi$ the
solution to the formal equation
\begin{equation}
  (\partial_t - \Delta) \Phi = \sin (\beta \Phi) + \zeta ,
  \label{equationforSG}
\end{equation}
where $\beta > 0$ is a parameter and
$\zeta$ is a 2-$d$ space-time white noise.
Again we set $u = \Phi - Z$, where 
\refchange{here we take $Z$ to be the solution to
$(\partial_t - \Delta) Z = \zeta$
starting from an average zero $2d$ Gaussian Free Field, i.e., the initial datum $Z_0$ can be realised as
\begin{equation} \label{initialconditionZ}
  Z_0 = 
  \sum_{k \in \Z^2, k \neq 0} \frac{\hat{\xi}_k}{|k|^2} e_k \; ,
  \quad 
  \text{cf. \eqref{fourierserieswhitenoisegpam}-\eqref{fourierserieslolligpam} .} 
\end{equation}
}Therefore, $u = \Phi - Z$ formally solves
\begin{equation}
  (\partial_t - \Delta) u
  = \sin (\beta u)\cos (\beta Z) + \cos (\beta u)\sin (\beta Z) ,
  \label{parabolicSG}
\end{equation}
which can be seen as \eqref{mainequation} with $\sigma (u) \xi$
replaced by $\sum_{i = 1}^2 \sigma_i (u) \xi_i$, where
$\sigma_1(u) = \sin(\beta u)$, $\sigma_2(u) = \cos(\beta u)$,
$\xi_1 = \cos(\beta Z)$ and $\xi_2 = \sin(\beta Z)$, 
cf. Remark \ref{remarkmultinoise}. In {\cite{HS16}}-{\cite{CHS18}}, the local theory for equation
\eqref{parabolicSG} in the full subcritical regime $\beta^2 \in (0,8\pi)$ is developed.

We now make this more precise at the level of smooth approximations to state our result.
Consider a $2\pi$-periodic in space
space-time white noise $\zeta$ 
and let $\seqeps{\zeta}$
be a regularisation of $\zeta$ at scale $\veps$. 
In \cite[Sec. 2]{HS16}, the authors
define $\seqeps{\tilde{Z}} := K \ast \seqeps{\zeta}$, where $\ast$ denotes space-time convolution and
$K$ is a post-processed heat kernel, in a way that
$\heat \seqeps{\tilde{Z}} - \seqeps{\zeta} = \seqeps{R}_\zeta$,
where $\seqeps{R}_\zeta \to R_\zeta$ in $C^\gamma$ for any $\gamma \geqslant 0$
in the limit $\veps \downarrow 0$.
For $\beta^2 \in (0,16 \pi /3)$, define
\begin{equation}
  \seqeps{\noiseb} := \veps^{-\frac{\beta^2}{4\pi}} e^{i \beta \seqeps{\tilde{Z}}} ,
  \quad
  \seqeps{\noiseb}_c := \veps^{-\frac{\beta^2}{4\pi}} \cos(\beta \seqeps{\tilde{Z}}) ,
  \quad
  \seqeps{\noiseb}_s := \veps^{-\frac{\beta^2}{4\pi}} \sin(\beta \seqeps{\tilde{Z}}) ,
  \label{defofnoisesSG2}
\end{equation}
so that $\seqeps{\noiseb} = \seqeps{\noiseb}_c + i \seqeps{\noiseb}_s$. With this definition, \cite[Thm 2.1]{HS16}
shows convergence $\seqeps{\noiseb} \to \noiseb$ in
$C^{\alpha}$ for any $\alpha < -\beta^2/4\pi$. 
We provide in Appendix \ref{AppendixstochasticobjectsSG} a discussion on how the objects 
$\seqeps{\lollib}_a, \Pi_z \dumb_{a,b}$, for $a,b \in \{c,s\}$,
can be constructed from $\seqeps{\noiseb}$ in \eqref{defofnoisesSG2} by following the results in
\cite{HS16}. Moreover, the lift $\Pi$
satisfies the OBs
in \eqref{orderbounds}-\eqref{orderboundlollipop} with
$\E[\Cnoise{,n}^p], \E[\Cdumb{,n}^p] < \infty$
uniformly in $\veps>0$, for any $\kappa > \beta^2/4\pi - 1$, $n \in \N$ and $p \geqslant 1$, where
\begin{equation}
  \Cnoise{,n} := 
  \max_{\btau \in \{\noise_a, \lolli_a, \Xnoise_a \; : \; a=c,s\}} [\Pi;\btau]_{|\btau|,D_{n - 1}^n} ,
  \; \; 
  \Cdumb{,n} := \max_{a,b \in \{c,s\}} [\Pi;\dumb_{a,b}]_{- 2 \kappa, D_{n-1}^n}
  \label{defofconstantsOBforSG}
\end{equation}
are as in \eqref{assumptiononstochasticobjects} in view of Remark \ref{remarkmultinoise}.

In \cite[Thm 2.5]{HS16} they show that $\ueps = \seqeps{\Phi} - \seqeps{\tilde{Z}}$, solution to
the renormalised equation
\begin{equation}
  (\partial_t - \Delta) u^{(\varepsilon)} =
  \sin (\beta u^{(\varepsilon)}) \seqeps{\noiseb}_c 
   + \cos (\beta u^{(\varepsilon)}) \seqeps{\noiseb}_s 
   + \seqeps{R} ,
   \label{renormalisedequationforSG}
\end{equation}
converges to a limiting equation $u$, which is interpreted as the solution to \eqref{parabolicSG}.
Moreover, since $\sin(\beta \cdot)$ and $\cos(\beta \cdot)$ satisfy Assumption \ref{nonlinearsigma} and
the term $\seqeps R$ can be treated as a third (smooth) noise with coefficient $\si_3(u) = 1$,
our results apply to this equation. 

\begin{corollary}[2-d Sine-Gordon] \label{corollaryglobalSG}
\refchange{
Fix $\bar{\kappa} > 0$
  as in Theorem \ref{maintheorem}.
  Let $\zeta$ be a ($2\pi$-periodic in space) space-time white noise over $\R^2$ and
  $\beta^2 \in (0, (1 + \bar{\kappa}) 4 \pi)$. Fix the $2\pi$-periodic initial condition $\Phi_0 := Z_0 + u_0$, where 
  $Z_0$ is as in \eqref{initialconditionZ} 
  and $u_0$ is bounded.

Then the equation \eqref{equationforSG} is globally well-posed. 
The solution can be decomposed as $\Phi = Z + u$, where $Z$ is a Gaussian process that satisfies
  $Z \in C^{\alpha}$ for $-1/3 < \alpha < 0$ almost surely w.r.t. $\zeta$,
  and $u$ satisfies }
    \begin{equation*}
      \E[\| u \|_{D_0^t}^p] \leqslant C' (\kappa, d, C_{\sigma}) \max \left\{ \| u_0\|,
      \Pnoise(t) t^{\frac{p}{1 - \betanoise}},
      \Pdumb(t) t^{\frac{p}{1 - \betanoise}}
      \right\} ,
  \end{equation*}
  where for $\nu = (1 - \kappa)^{-1}(1 - \betanoise)^{-2}$, $\Cnoise{,n},\Cdumb{,n}$ in \eqref{defofconstantsOBforSG}
  and $n < t \leqslant n + 1$,
  \begin{equation*}
    \Pnoise(t) := \E\left[ \left( \max_{1 \leqslant i \leqslant n + 1} \Cnoise{,i} \right)^{2\nu p} \right]
    \; \text{and} \; \;
    \Pdumb(t) := \E\left[ \left( \max_{1 \leqslant i \leqslant n + 1} \Cdumb{,i} \right)^{\nu p} \right]
  \end{equation*}
  are finite for any $p \geqslant 1$ and $\betanoise$ is the same as in Theorem
  \ref{maintheorem}.

\end{corollary}

Consider the parabolic quantisation of the massive 2-d SG, which reads as
\begin{equation}
  (\partial_t - \Delta + m^2) \Phi = \sin (\beta \Phi) + \zeta ,
  \label{themassiveSG}
\end{equation}
where $m^2 > 0$ is the mass parameter. It can be reduced to
\eqref{parabolicSG} with an extra damping term $- m^2 u$ in the r.h.s. 
if we consider $u = \Phi - Z$ where $Z$ solves
$(\partial_t - \Delta + m^2) Z = \zeta$. 
The additional
mass adds confinement to the periodic potential $\sin (\beta u)$ in \eqref{themassiveSG},
hence the massive SG on the torus is expected to have an invariant
measure.

Next is a theorem for equation \eqref{mainequation} with the additional mass
term.

\begin{theorem}[the massive equation]
  \label{corollarythemassiveequation}
  Under the same hypothesis of Theorem
  \ref{theorempolynomialgrowth}, consider equation \eqref{mainequationrenormalised} with
  an extra mass term $- m^2u$ in the r.h.s. 
  Then, for $n < t \leqslant n + 1$, we obtain a
  pathwise growth analogous to \eqref{polynomialgrowth} as
  \begin{equation}
    \| u \|_{D_0^t} \leqslant C (m, \kappa, d, C_{\sigma}) \max \left\{ \| u_0
    \|, \max_{1 \leqslant i \leqslant n + 1} \left\{ \Cnoise{,i}^{\frac{2}{1 - \kappa}}, 
    \Cdumb{,i}^{\frac{1}{1 - \kappa}}
    \right\}^{\frac{1}{(1 - \betanoise)^2}} \right\}
    \label{corollarymassive},
  \end{equation}
  where $C (m, \kappa, d, C_{\sigma}) > 0$ is a constant that only depends on
  $m, \kappa, d$ and $C_{\sigma}$ and the estimate is independent of the renormalisation
  constant $C>0$ in \eqref{defsmoothdumbwithC}-\eqref{mainequationrenormalised}.
  Furthermore, 
   for $\nu = (1 - \kappa)^{-1}(1 - \betanoise)^{-2}$, consider
    \begin{equation}
    \Mnoise := \sup_{n \geqslant 1} \E[\Cnoise{, n}^{2 \nu p}]^{\frac{1}{2 \nu p}}
    \quad \text{and} \quad 
    \Mdumb := \sup_{n \geqslant 1} \E[\Cdumb{, n}^{\nu p}]^{\frac{1}{\nu p}}
    , \label{assumptiononmoments}
  \end{equation}
  and then
  \begin{equation}
    \sup_{n \geqslant 1} \mathbb{E} [\| u \|^p_{D_{n - 1}^n}] \leqslant C (m,
    \kappa, d, p, C_{\sigma}) \max \{ \| u_0 \|, \Mnoise^{2 \nu p},
     \Mdumb^{\nu p} \} \label{growthofmoments},
  \end{equation}
  where 
  $C (m, \kappa, d, p, C_{\sigma}) > 0$ is a constant that only depends
  on $m, \kappa, d, p$ and $C_{\sigma}$ and the estimate is independent of
  $C>0$ in \eqref{defsmoothdumbwithC}-\eqref{mainequationrenormalised}.
\end{theorem}

The main application of this theorem we consider here is to equation \eqref{themassiveSG}, for which
the post-processed heat kernel $K$ in the definition of $\seqeps{\tilde Z}$ is replaced by the massive
heat kernel post-processed in the same way.

\begin{corollary}[massive 2-d Sine-Gordon]
  Consider the same assumptions of Corollary \ref{corollaryglobalSG} but with respect to
  $2\pi$-periodic in space solutions to \eqref{themassiveSG} instead. Then, we may write its solution
  as $\Phi = Z + u$, where $Z$ is a Gaussian process that satisfies
  $Z \in C^{\alpha}$ for $-1/3 < \alpha < 0$ almost surely w.r.t. $\zeta$,
  and 
%
  \begin{equation}
    \sup_{t \geqslant 0} \mathbb E [ \| u (t, \cdot) \|^p ] \leqslant C (m,
    \kappa, d, p, C_{\sigma}) \max \{ \| u_0 \|, \Mnoise^{2 \nu p},
     \Mdumb^{\nu p} \} ,
    \label{controlpmomentsofu}
  \end{equation}
  where $\Mnoise, \Mdumb$ are as in \eqref{assumptiononmoments} w.r.t. $\Cnoise{,n}, \Cdumb{,n}$
  in \eqref{defofconstantsOBforSG} and are finite for every $p \geqslant 1$. In particular,
  there exists an invariant measure for \eqref{themassiveSG}.
\end{corollary}

The construction of an invariant measure for \eqref{themassiveSG} can be done via the Krylov--Bogoliubov existence theorem by using the uniform in time bound on the $p$-moments
of $u(t,\cdot)$ in \eqref{controlpmomentsofu}, as in e.g. {\cite[Sec. 4]{TW18}}.
While the construction of this
measure was known before (see e.g. {\cite{DH00}} and the very recent work \cite{GM24}),
the construction via
parabolic quantisation is novel, at least for the rough regime of $\beta^2 \in
[4 \pi, (1 + \bar{\kappa}) 4 \pi)$.



As a third application, we address the construction of stochastic
flows for SPDEs with multiplicative noise.
Consider
\begin{equation}
   \mathd u = \Delta u \mathd t + \sigma (u) \mathd W ,
  \label{randomdynamicalsystem}
\end{equation}
where $W$ is a spatially correlated Wiener process. 
We may take for instance $W$ with 
covariance formally given by $\mathbb{E} [W (t, x) W (s,
y)] = (t \wedge s) (1 - \Delta)^{- (d / 2 - \delta)} (x, y)$. 
%
%
%
More precisely, $W$ can be realised over $[0,\infty)\times \T^d$ as a 
random Fourier series 
via ($e_k(x) = e^{i k \cdot x}$)
\begin{equation}
  W = \sum_{k \in \Z^d} e_k c_k B_k ,
  \quad \text{where} \quad 
  c_k := \frac{1}{(2\pi)^\frac{d}{2}(1 + |k|^2)^{\frac{d}{4} - \frac{\delta}{2}}} ,
  \label{cylindricalwienerRDS}
\end{equation}
and $B_k$ are independent (up to $\bar{B}_k = B_{-k}$) complex-valued Brownian motions.
Note that the operator $(1 - \Delta)^{- (d / 2 - \delta)}$ is not trace-class for any $\delta > 0$ and
therefore for any $t$ the process $W_t$ does not take values in $L^2(\T^d)$. 
The solution theory \refchange{based on the mild formulation and It\^o integration} for this equation is classical and is developed in \cite[Sec. 6.4]{DZ92} for the whole range
$\delta < 1$. 
However,
\refchange{this classical solution theory does not directly allow the construction of stochastic flows
- {\cite[Sec. 9.1.2]{DZ92}}.
Indeed, the set of full probability on which the solution
is constructed depends on the initial condition, and
extracting a single set that works for all initial conditions is
not generally possible - see \cite[Ex. 9.12]{DZ92} for a counter example.
Existence of stochastic flows is non-trivial even in the finite dimensional setting, see e.g. \cite{LS11}.
On the other hand, one can construct stochastic flows
by solving the equation pathwise, as we do here, see
the discussion in \cite[Sec. 9.1.2]{DZ92}
and \cite[Sec. 8.10]{frizandhairer2020}.
}

Equation \eqref{randomdynamicalsystem} can be treated within the framework of Regularity Structures, see e.g. \cite{HP15}.
In this case, \eqref{randomdynamicalsystem}
corresponds to \eqref{mainequation} with 
$\xi = \dot{W} \in C^{- 1 - \kappa}$
\refchange{for $\kappa > \delta$,}
and we refer to Appendix \ref{AppendixstochasticobjectsDZSPDE} where we discuss
smooth approximations to \eqref{randomdynamicalsystem} and how to construct the objects $\lollib, \Pi_z \dumb$
for this case.
\refchange{See also the thesis
\cite[Sec. 3.3.2]{dLF24} for more details
on these stochastic estimates.
We emphasise that the specific choice of noise in
\eqref{cylindricalwienerRDS} was given for concreteness.
In fact, any Wiener process $W$ with the right singularity at the diagonal of the Green's function
of the covariance operator, i.e., 
\begin{equation*}
    |\E[W_t(x)W_t(y)]|  \lesssim t |y-x|^{-2\delta} 
    \quad \text{for} \quad |x-y| \ll 1 \; ,
\end{equation*}
can be treated within our framework, provided that $\delta < \bar \kappa$.
}
\begin{corollary}
For any $0 \leqslant s \leqslant t < \infty$, denote by $u (t ; s, v)$ the
It\^o solution to \eqref{randomdynamicalsystem} with initial condition $u (s ; s, v)
= v$,
where $W$ is the Wiener process in \eqref{cylindricalwienerRDS}
for $\delta < \bar{\kappa}$ with $\bar{\kappa}$ as in Theorem \ref{maintheorem}.
Then, 
$\mathbb{P}$-almost surely $u (t ; r, u (r ; s, v)) = u (t ; s, v)$, and
therefore the map
\[ \begin{array}{rll}
     \phi_{t, s} : L^{\infty} (\mathbb{T}^d) & \longrightarrow & L^{\infty}
     (\mathbb{T}^d)\\
     v & \longmapsto & \phi_{t, s} (v) = u (t ; s, v)
   \end{array} \]
defines a stochastic flow for \eqref{randomdynamicalsystem}, in the sense of
{\cite[Sec. 9.1.2]{DZ92}}. Furthermore, for any $p \geqslant 1$
\begin{equation*}
  \sup_{\| u_0 \| \leqslant 1} \E[\|u\|_{D_0^1}^p] \lesssim C_{p,W} .
\end{equation*}
\end{corollary}

\refchange{
\begin{remark}
    One choice of Gaussian noise $\xi$ of interest is space-time white noise, that is $\xi$ has covariance formally given by
    \begin{equation*}
         \E[\xi(t,x)\xi(s,y)] = \delta(t-s)
    \delta(x-y) \;,
    \end{equation*}
   where above $\delta$ represents a Dirac delta either in space or time. 
   In this case $\xi \in \mathcal C^{-\frac{d+2}{2} - \theta}$ for any $\theta>0$, where $d$ is the spatial
    dimension.
    
Equation \eqref{mainequation}
    is then subcritical only when $d=1$. 
    However, for $d=1$, we have  $\kappa > \frac{1}{2} + \theta > 1/3$ which means a priori bounds for this equation are beyond the reach of our methods. 
\end{remark}
}

\subsection{Organisation of the article}

The rest of this article will be organised as follows. In Section
\ref{sectionstrategyofproof} we give an overview of the proof of the main
theorems, highlighting the main ideas of this work. In Section \ref{sectionsetting}
we introduce the setting in which we are proving our theorems. In Section
\ref{sectioninteriorestimate} we prove an interior estimate Theorem
\ref{theointeriorestimate} and state a control on the $L^\infty$ norm of $u$ for 
short times in Theorem \ref{theoinfinityschauderpostprocessing}. In Section
\ref{sectionmaximumprinciple} we prove the Maximum Principle Theorem 
\ref{theoremtopostprocessintopolynimial}, and as a consequence, we prove the
main Theorem \ref{maintheorem}. This is the most important part of the paper,
and an instrumental result is the decomposition in Proposition
\ref{propositiontransportdecomposition}, whose proof is found in Subsection
\ref{proofofmainproposition}. The proof of Theorem \ref{theorempolynomialgrowth}
is carried in Subsection \ref{sectionproofoftheorempolynomial}. In Section
\ref{sectionproofofshorttimescontrol} we prove Theorem
\ref{theoinfinityschauderpostprocessing} by following \cite[Prop.
6.9]{Hai14}. In Section \ref{sectionthemassiveequation}
the analysis of the massive equation is developed and the proof
of Corollary \ref{corollarythemassiveequation} is obtained. In Appendix
\ref{appendixreconstructionandintegration} we provide the Reconstruction
Theorem \ref{reconstructiontheorem} and the Schauder estimate Lemma
\ref{lemmaschauderestimate}, both following closely \cite{Moinatandweber2020cpam}.
Appendix \ref{AppendixB} contains auxiliary calculations for some proofs.
\section{Strategy of the proof} \label{sectionstrategyofproof}
 
In this section we briefly discuss the strategy of the proof of the main theorems.
We compare the cases of $\sigma$ linear and non-linear, motivating why
the non-linear case is harder and how we overcome 
this. 

First, we review the local theory for \eqref{mainequation} and highlight the
main difference between the linear case $\si(u) = u$
and the case of non-linear $\si$. 
This is done by formulating the method for local
solutions in \cite{Hai14} into our setting, and amounts to
\begin{itemize}
\item[1.] obtain a good ansatz for the local description of $u$ (and of $\si(u)$);
\item[2.] make sense of the ill-defined product in the r.h.s. of \eqref{mainequation};
\item[3.] use Schauder theory to translate r.h.s. back to local description of $u$.
\end{itemize}
Let us think of equation \eqref{mainequation}
(interpreted as \eqref{mainequationrenormalised})
with the rough noise
$\xi \in C^{- 1 - \kappa}$ and with $\si(u) = u$. By Schauder theory we expect $u$ to have at most
$(1 - \kappa)$-H{\"o}lder regularity, which is not enough to define the product
$u \xi$ classically. 
A main idea in the theory of Regularity Structures is to look for local approximation
in terms of perturbative expansions.
One possible approach is the ``freezing of coefficients".
For a fixed point $z \in D$, we subtract $u(z) \Pi_z \noise$ from both sides of 
\eqref{mainequationrenormalised},
obtaining
\begin{equation}
  (\partial_t - \Delta)[u - u(z) \Pi_z \lolli] 
  = u \diamond \noiseb - u (z)
  \Pi_z \noise ,
  \label{frozenequationstrategy}
\end{equation}
which is equivalent to \eqref{mainequationrenormalised}
as $(\partial_t - \Delta)$ only acts on the implicit variable and not on the frozen $z$.
Note that the increment $u - u(z)$ satisfies $|u - u(z)| \lesssim d(\cdot, z)^{1-\kappa}$.
This suggests that $u$ can be \textit{modelled} after $\lollib$, cf. \cite[Def. 3.1]{Hai14},
in the sense that a good local description of $u$ satisfies
\begin{equation}
  |u (w) - u (z) - u (z) (\Pi_z \lolli) (w) - u_X(z)\cdot(\Pi_z \X)(w)| \lesssim
  d(w,z)^\gamma
  \label{localdescriptionustrategy}
\end{equation}
for $1 < \gamma \leqslant 2-2\kappa$, where $u_X(z) \in \R^d$ is a generalised gradient \label{page_generalised_grad1} of $u$. I
n the context of \textit{rough paths}, this corresponds to the
notion of $u$ being \textit{controlled} by $\lollib$, see \cite{gubinelli2004}.
We let $[U]_\gamma$ denote the best constant in \eqref{localdescriptionustrategy}.
%
%

Since we have postulated that the objects $\dumb$ and $\Xnoise$ play the role of multiplication
of the objects $\lolli, \X$ in \eqref{localdescriptionustrategy} with the noise $\noise$ and satisfy the correct
CBP and OB, we may write the r.h.s. of \eqref{frozenequationstrategy} as
\begin{equation}
  u \diamond \noiseb - u (z) \Pi_z \noise = u (z) \Pi_z \dumb +
  u_X (z) \cdot \Pi_z \Xnoise + R_z ,
  \label{errorreconstructionstrategy}
\end{equation}
\refchange{
where the first two terms in the r.h.s. above satisfy
\begin{eqnarray}
  |\langle u (z) \Pi_z \dumb, \varphi_z^L \rangle|
  & \lesssim & L^{-2\kappa} |u(z)| [\Pi;\dumb]_{-2\kappa} \label{reconstructionstrategy.term1}\\
  |\langle u_X (z) \cdot \Pi_z \Xnoise, \varphi_z^L \rangle|
  & \lesssim &
  L^{-\kappa} |u_X(z)| [\Pi;\Xnoise]_{-\kappa} \; . \label{reconstructionstrategy.term2}
\end{eqnarray}
}The expansion \eqref{errorreconstructionstrategy} allows us to \refchange{use a variant of the 
Reconstruction Theorem \cite[Thm 3.10]{Hai14} - 
see Lemma \ref{reconstructiontheorem} -}
to make sense of $u \diamond \noiseb$.
%
%
%
%
This gives, to leading order, 
\begin{equation}
  |\langle R_z, \varphi_z^L \rangle| \lesssim L^{\gamma - 1 - \kappa}
  ([U]_\gamma [\Pi ; \noise]_{-1-\kappa} + \cdots ) , \label{reconstructionstrategy}
\end{equation}
as long as $\gamma > 1 + \kappa$
- see Lemma \ref{lemmacontrolonrhsmainequation} to see what ``$\cdots$" includes
and the precise estimate for general $\si$.

The final step in the local theory is integration
\refchange{
for which we use a variant of the Schauder estimate in
\cite[Thm 5.12, Prop. 6.16]{Hai14} - see 
Lemma \ref{lemmaschauderestimate}.
}
%
%
%
%
Since we consider initial conditions $u_0$ which are only in $L^\infty(\T^d)$,
we need to deal with some blow up of higher regularity norms near $t=0$. 
The semi-norm
$[U]_{\gamma,0,[a,b]}$ 
\refchange{adds} a weight at the initial time $t = a$ to account
for such blow up, defined in \eqref{normofUblowup}, c.f. \cite[Def. 6.2]{Hai14}.
For $0 < T \leqslant t$, 
Lemma \ref{lemmaschauderestimate}
\refchange{
combined with \eqref{reconstructionstrategy.term1}, \eqref{reconstructionstrategy.term2}
and \eqref{reconstructionstrategy}
}
gives,\refchange{to leading order,}
\begin{align}
  [U]_{\gamma,0,[t-T,t]} & \lesssim \|u\|_{D_{t-T}^t} [\Pi;\dumb]_{-2\kappa}
  + T^{\frac{\kappa}{2}} \| u_X\|_{D_{t-T}^t} [\Pi ; \Xnoise]_{-\kappa} \nonumber \\
  & + T^{\frac{1-\kappa}{2}} ([U]_{\gamma,0,[t-T,t]} [\Pi ; \noise]_{-1-\kappa} + \cdots ) ,
  \label{reducesigmatoUstrategy}
\end{align}
so that if $T^{\frac{1-\kappa}{2}}[\Pi ; \noise]_{-1-\kappa} \ll 1$, 
then
\begin{equation}
  [U]_{\gamma,0,[t-T,t]} \lesssim \|u\|_{D_{t-T}^t} [\Pi;\dumb]_{-2\kappa}
  + T^{\frac{\kappa}{2}} \|u_X\|_{D_{t-T}^t} [\Pi ; \Xnoise]_{-\kappa} + \cdots ,
  \label{interiorestimateforlinearcase}
\end{equation}
where only lower regularity norms of $u$ appear in the r.h.s.


Since the 
absorption condition $T^{\frac{1-\kappa}{2}}[\Pi ; \noise]_{-1-\kappa} \ll 1$ is independent of $u$,
this can be iterated and it is enough to obtain non-explosion in $t \gg 1$ for $u$.
This is why the global theory in finite volume for the linear case $\si(u)=u$ follows from the
local theory in \cite{Hai14}-\cite{GIP15}.

However, this argument breaks down completely in the case of non-linear $\si$.
In fact, even if $\si$ and all of its derivatives are bounded, we still pick up
super-linear terms in $u$ in the error estimate from the local description of $\si(u)$
to the local description of $u$. 
Since this point is central
to our analysis, we go through it in some detail.

\refchange
{
First, we consider the standard chain rule from calculus
to illustrate the problem:
for smooth $f,\sigma : \R \to \R$, we let for
$1 < \gamma < 2$,
$$[f]_\gamma := \sup_{x \neq y}
\frac{|f(y) - f(x) - f'(x)(y-x)|}{|y-x|^\gamma} \; ,$$
so the goal is to reduce $[\si(f)]_\gamma$ in terms
of $[f]_\gamma$. For $\eta := f'(x)(y-x)$, we have
\begin{eqnarray*}
  | \si(f(y)) - \si(f(x)) - \si'(f(x))\eta |
  & \leqslant & | \sigma (f (y)) - \sigma (f (x) + \eta) | \\
  & + & | \sigma (f (x) + \eta) - \sigma (f(x))
  - \sigma' (f (x))\eta | \; ,
\end{eqnarray*}
where the first term in the r.h.s. can be bounded by
$\|\sigma'\|[f]_\gamma |y-x|^\gamma$, while the second by
\begin{equation*}
    \|\sigma''\|\eta^2 = \|\sigma''\||f'(x)|^2|y-x|^2 \; ,
\end{equation*}
i.e., it contains the gradient of $f$ squared, which is reminiscent of
the fact that $(\sigma(f))'' = \sigma''(f')^2 + \sigma'f''$.
See \eqref{gradientsquaredinerrorterm} below for the equivalent
statement when the usual gradient is replaced by the generalised
gradient of Regularity Structures.
}

In view of \eqref{frozenequationstrategy}, we now subtract $\si(u(z)) \Pi_z \noise$ from both
sides of \eqref{mainequationrenormalised},
from which the ansatz for the local description of $u$ and $\si(u)$
become
\begin{equation*}
   |u (w) - u (z) - \sigma (u (z)) (\Pi_z \lolli) (w) - u_X(z)\cdot(\Pi_z \X)(w)| \lesssim d(w,z)^\gamma ,
\end{equation*}
\begin{equation*}
  |\sigma(u (w)) - \sigma(u (z)) - \sigma^{\prime}\sigma (u (z)) (\Pi_z \lolli) (w) 
  - \sigma^{\prime}(u(z)) u_X(z)\cdot(\Pi_z \X)(w)| \lesssim d(w,z)^\gamma ,
\end{equation*}
where 
$[U]_\gamma, [\sigma(U)]_\gamma$ denote the best constants in these
estimates. This way, the estimates in \eqref{reconstructionstrategy}-\eqref{reducesigmatoUstrategy}
hold with $[U]_\gamma$ replaced by $[\sigma(U)]_\gamma$ in the r.h.s., and therefore a
chain rule argument is necessary to reduce $[\sigma(U)]_\gamma$ back to 
$[U]_\gamma$ and allow its absorption into the l.h.s..
This can be done via Taylor expansions, as in \cite[Thm. 4.16]{Hai14} and \cite[Lem. 2.7, Lem. C.1]{GIP15}.
Letting $\eta_z(w) := \sigma(u(z))(\Pi_z \lolli)(w) + u_X(z) (\Pi_z \X)(w)$, 
\begin{align}
  \sigma(u(w)) - \sigma(u(z)) & = \sigma(u(w)) - \sigma(u(z) + \eta_z(w)) \label{sigmatoUstrat1}\\
  & + \sigma(u(z) + \eta_z(w)) - \sigma(u(z)) \label{sigmatoUstrat2} .
\end{align}
While the term in \eqref{sigmatoUstrat1} is bounded by $\|\sigma'\|[U]_\gamma d(w,z)^\gamma$,
the one in \eqref{sigmatoUstrat2} can be decomposed as
\begin{equation}
  \sigma'(u(z))\eta_z(w) + E_z(w)
  \quad \tmop{where} \quad
  |E_z(w)| \lesssim \|\sigma''\|\eta_z(w)^2 .
  \label{gradientsquaredinerrorterm}
\end{equation}
This makes 
the quadratic term $u_X(z)^2$ appear in the error estimate of $[\sigma(U)]_\gamma$
- see Lemma \ref{lemmachainrule}
for a refined version of this estimate, where the error is of total homogeneity $\gamma$ in $u$.
Thus, as opposed to the linear case, for general $\si$ the choice of $T$ in \eqref{reducesigmatoUstrategy}
depends on $u$. 

Despite this fact, in the first step
of our analysis, we show that the local theory still provides an interior estimate in intervals
that are $u$-dependent, as
\begin{equation}
  [U]_{2-2\kappa,0,[0,T]} \lesssim \|u\|_{D_0^T} \quad \tmop{if} \quad
  T^\frac{1-\kappa}{2} \ll [\Pi ; \noise]_{-1-\kappa}^{-1} \|u\|_{D_0^T}^{-\kappa - \delta} ,
  \label{strategyinteriorestimate}
\end{equation}
where $\delta > 0$ -
see Theorem \ref{theointeriorestimate} and Corollary \ref{corollaryimprovedinteriorestimate}
for the precise statement.
\refchange{However, in contrast with the linear case (cf. \eqref{interiorestimateforlinearcase}), we cannot iterate
the interior estimate in \eqref{strategyinteriorestimate} to obtain global well-posedness.
}
%
%

This concludes the small scales analysis, and we move to a different strategy to study large scales.
To go beyond short time estimates for \eqref{mainequationrenormalised} in the case of non-linear $\si$, we
look to a regularisation of equation \eqref{mainequationrenormalised}
at scale $L>0$ obtained by testing both sides of the equation against $\varphi_z^L$ for a special choice of mollifier
$\vphi \in \mathfrak F$ in \eqref{choiceofmollifiersemigroup} which is described in
Appendix \ref{appendixreconstructionandintegration} and giving an equation for $u_L(z) := \langle u , \varphi_z^L \rangle$.

This is inspired by
the works 
\cite{Moinatandweber2020ejp,Moinatandweber2020cpam,CMW23}, 
where the authors use the damping term $-u_L^3$ in the r.h.s.
to make use of a strong Maximum Principle, which allows them to reduce any term of order up to
$3$ in the r.h.s. into sub-linear terms - see \cite[Thm 4.4]{Moinatandweber2020ejp}. 
However, we do not
have such a strong damping for our equation, so we must devise a new method. 
Our approach is summarised
in the following steps.

\begin{itemize}
\item[1.] regularise the equation \eqref{mainequationrenormalised}  at space-time scale $L>0$ as described above;
\item[2.] relate the generalised (spatial) gradient $u_X$ with the actual (spatial) gradient $\nabla u_L$;
\item[3.] use the explicit error in the Reconstruction to isolate gradients into transport terms on r.h.s.;
\item[4.] apply a Maximum Principle for $u_L$ that allows us to disregard transport terms. Choose $L$ to balance terms of positive and
negative order in $L$ into a sub-linear in $u$ control, while making use of the interior estimate
in \eqref{strategyinteriorestimate}.
\end{itemize}

\refchange{Our requirement $\kappa \ll 1$ arises in Step 4 to guarantee the estimate is indeed
sub-linearity in $u$. 
In particular, everything
up to \eqref{strategyinteriorestimate} above, i.e.
the interior estimate (or control on small scales), makes
sense for all $\kappa \in (0,1/3)$, as expected.}

Let $0 < L < \sqrt{T}/2$ be a fixed regularisation scale and for every point $z \in D_T^1$,
consider \eqref{mainequation} tested against $\vphi_z^L$, i.e., $u_L(z) := \langle u , \varphi_z^L \rangle$ \label{firstdefinitionofuL}
satisfies
\begin{align*}
  (\partial_t - \Delta)u_L(z)
   & = \langle \sigma(u) \diamond \noiseb , \varphi_z^L \rangle \\
  & = \sigma(u(z)) \langle \Pi_z \noise , \varphi_z^L \rangle
  + \sigma'\sigma(u(z)) \langle \Pi_z \dumb , \varphi_z^L \rangle \\
  & + \sigma'(u(z))u_X(z)\cdot \langle \Pi_z \Xnoise , \varphi_z^L \rangle
  + \langle R_z , \varphi_z^L \rangle .
\end{align*}
Using the OBs for $\noise, \dumb, \Xnoise$ and the Reconstruction for $R_z$ we obtain
\begin{eqnarray}
  | \sigma(u(z)) \langle \Pi_z \noise , \varphi_z^L \rangle| & \leqslant & 
  \|\sigma\|[\Pi ; \noise]_{-1-\kappa} L^{-1-\kappa} \label{strategynoisetermregularised} \\
  |\sigma'\sigma(u(z)) \langle \Pi_z \dumb , \varphi_z^L \rangle | & \leqslant &
  \|\sigma'\sigma\| [\Pi;\dumb]_{-2\kappa} L^{-2\kappa} \nonumber \\
  | \sigma'(u(z))u_X(z)\cdot \langle \Pi_z \Xnoise , \varphi_z^L \rangle | & \leqslant & 
  \|\sigma'\| |u_X(z)|[\Pi ; \Xnoise]_{-\kappa} L^{-\kappa} 
  \label{strategygeneralisedgradientregularised} \\
  | \langle R_z , \varphi_z^L \rangle | & \lesssim &
  L^{1 - 3\kappa} ([\sigma(U)]_{2-2\kappa} [\Pi ; \noise]_{-1-\kappa} + \cdots ) .
   \label{strategyerrorreconstructionregularised}
\end{eqnarray}
On one hand, since the error estimate for $[\si(U)]_\gamma$ is super-linear in $u$ and
$1 - 3\kappa > 0$ in \eqref{strategyerrorreconstructionregularised}, the scale $L$ has 
to be chosen as a negative power of $\|u\|$ in order
to result in a sub-linear in $u$ control of the r.h.s.. On the other hand, as long as $\kappa>0$,
plugging in such a choice for $L$ in \eqref{strategygeneralisedgradientregularised}
can only result in a super-linear in $u$ estimate, since $u_X$ is of homogeneity $1$ in $u$.

Overcoming the challenge above is where Steps 2 and 3 from our summary come in. 
Step 2 entails relating the generalised (spatial) gradient $u_X$ to the
actual (spatial) gradient $\nabla u_L$ using the formula
\begin{equation}
  \nabla u_L (z) = u_X (z) + \sigma (u (z)) \langle \Pi_z \lolli, \nabla_x
  \tilde{\varphi}^L_z \rangle + E_z^L,
  \label{formulaforgradientsstrategy}
\end{equation}
where $\tilde{\varphi}(z) := \varphi(-z)$ and $E_z^L$ satisfies
$| E_z^L | \leqslant [U]_{\gamma, B (z, L)} L^{\gamma - 1}$, for 
any $1 + \kappa < \gamma \leqslant 2-2\kappa$ - see Lemma 
\ref{lemmarelatinggradients}.
\refchange{Again, the advantage of
\eqref{formulaforgradientsstrategy} is that it allows
us to extract part of $u_X$ into a transport term (which
does not influence the control over the $L^\infty$ of $u$ we obtain from our Maximum Principle)
in the next observation. Nevertheless, \eqref{formulaforgradientsstrategy} alone is not sufficient, see
Remark \ref{remarkwhynaivebalancingdoesnotwork}.
}
%

Turning to Step 3, the source of the super-linearity in the
error estimate of $[\si(U)]_\gamma$ comes from the term $u_X^2$. But since the error
$R_z$ in \eqref{strategyerrorreconstructionregularised}
is explicit and can be further decomposed, we may also use formula 
\eqref{formulaforgradientsstrategy} to split $u_X^2$ into transport plus 
error terms. In fact, 
writing
$R_z = \sigma(u) \diamond \noiseb - G_z$, where
\begin{equation*}
  G_z := \sigma (u (z)) \Pi_z \noise + \sigma' \sigma (u (z)) \Pi_z \dumb +
  \sigma' (u (z)) u_X (z) \cdot \Pi_z \Xnoise ,
\end{equation*}
and using the properties of $\varphi$ in \eqref{choiceofmollifiersemigroup} in
the proof of Theorem \ref{reconstructiontheorem}, we obtain
\begin{equation}
  \langle R_z, \varphi_z^L \rangle = 
  \sum_{n = 0}^{\infty} \int \int \left\langle G_{z_2} - G_{z_1},
  \varphi_{z_2}^{\frac{L}{2^{n + 1}}} \right\rangle
  \psi^{\frac{L}{2^{n + 1}}}_{z_1}
  (z_2) \varphi^{L, n}_z (z_1) \mathd z_2 \mathd z_1 .
  \label{multiscalereconstructionstrategy}
\end{equation}
This multi-scale decomposition was obtained in the versions of the Reconstruction
Theorem developed in \cite[Thm 2.8]{Moinatandweber2020cpam}-\cite[Prop. 1]{OSSW18}.
The CBP formulas \eqref{changeofbasepointXnoise} and
\eqref{changeofbasepointdumbbell} allow us to decompose $G_{z_2} - G_{z_1}$ into
several terms, isolating the ones that contain generalised gradients $u_X$. By using
the OBs
\eqref{orderbounds} and \eqref{orderboundlollipop}, we show that each of the terms
gives rise to a smooth and uniformly bounded function with
$L^\infty$ norm proportional to
a power of the scale $L$. Applying formula \eqref{formulaforgradientsstrategy}, we
rewrite the regularised equation as
\begin{equation}
  (\partial_t - \Delta)u_L = a_L + b_L \cdot \nabla u_L + c_L ,
  \label{decompositionabcstrategy}
\end{equation}
where $a_L$ is bounded by positive powers of $L$, $c_L$ by negative
powers of $L$ and $b_L$ contains the worst super-linear terms
coming from $u_X^2$ - see Proposition \ref{regularisedequationwithtransport}.
An adequate choice of $L = \tilde{L}$ (see \eqref{choiceofL}) provides a balanced
control on $a_{\tilde{L}} + c_{\tilde{L}}$, which is sub-linear in $u$
(with exponents $\betanoise$ and $\betadumb$ in \eqref{betas},
\refchange{
which correspond to the elements that
are linear in the noise and of second
order, respectively),
}
provided that $\kappa < \bar{\kappa}$, where $\bar{\kappa} \in (0,1/3)$ is the solution to
\eqref{equationforkappa}. Since the term
$b_{\tilde{L}} \cdot \nabla u_{\tilde{L}}$ is a transport term (this was inspired by \cite{ZZZ22}), 
c.f. Proposition
\ref{propositiontransportdecomposition}, it does not contribute to the
$L^\infty$ norm of the solution in the Maximum Principle we obtain with
Theorem \ref{theoremtopostprocessintopolynimial}. The control on $u_{\tilde{L}}$
can be translated into control on $u$ 
by using \eqref{ineqhL-h}.

A technical limitation of the method just explained is that for any choice of scale $L>0$,
the regularised equation at scale $L$ only makes sense for $z = (t,x)$ such that
$t \geqslant 4 L^2$. Therefore, we need a control for the $L^\infty$ norm of $u$ for short times.
This can be obtained by post-processing the interior estimate obtained from the local theory
 into Theorem \ref{theoinfinityschauderpostprocessing}, which states that
\begin{equation}
  \|u\|_{D_0^T} \lesssim \|u_0\| \vee 1 ,
  \label{shorttimesestimatestrategy}
\end{equation}
where $T$ is as in the interior estimate \eqref{strategyinteriorestimate}.
For this we follow \cite[Prop. 6.9]{Hai14}. This completes the strategy to prove Theorem \ref{maintheorem}.

Finally, we explain how to extend the result in Theorem \ref{maintheorem}
to large times in Theorem \ref{theorempolynomialgrowth} by induction.
Set $Y_n \assign \| u \|_{D_{n - 1}^n}$ for $n \in \mathbb{N}$.
By refining the combined use of Theorem \ref{mainequation} with the Maximum Principle
Theorem \ref{theoremtopostprocessintopolynimial},
we arrive at
\[ \left\{\begin{array}{lll}
     Y_{n + 1} & \leqslant & Y_n + q_n {Y_n^{\betanoise}} \\
     Y_1 & = & \| u \|_{D_0^1}
   \end{array},\right. \]
where $q_n$ are constants that depend on $\Cnoise{, n}, \Cnoise{, n
+ 1}, \Cdumb{, n}$ and $\Cdumb{, n + 1}$. We pick up influences of the noise in
both intervals with $q_n$ since $u_L$ at $t = n$ invades the previous domain
$D_{n - 1}^n$. The difference inequality above 
leads to \eqref{polynomialgrowth} - see Subsection
\ref{Appendixdifferenceequations}.


\section{Setup} \label{sectionsetting}

In this section we provide notation and definitions that
will be used throughout. We will first focus on the first time interval, denoted by
$D_0^1$, corresponding to $n = 1$, and will omit the subindex $n$ from the
elements in \eqref{assumptiononstochasticobjects} to lighten the notation.
Thus we denote $\Cnoise{}  := \Cnoise{,1}$ and $\Cdumb{}  := \Cdumb{,1}$.

We will make the dependence on the constants $\Cnoise{}$ and
$\Cdumb{}$ in \eqref{assumptiononstochasticobjects} explicit in the
bounds throughout this work, but not constants depending on $d, \kappa$ or
$C_{\sigma}$. Henceforth, we denote with $\lesssim$ a bound that holds up to a
multiplicative constant which only depends on $d, \kappa$ or $C_{\sigma}$. 
In this
process, we make repeated use of \eqref{assumptionsigmabounded}.

\refchange{ Given $u: D \rightarrow \R$, we define $U:D\times D \to \R$, written $D \times D \ni (w,z) \mapsto U_{z}(w) = U(w,z)$, as follows: }
\begin{equation}
  U_z (w) \assign u (w) - (u (z) + \sigma (u (z)) (\Pi_z \lolli) (w))\;.
  \label{defofU}
\end{equation}
We also write $\|U\|_B := \sup_{z,w \in B} |U_z(w)|$ and $\|U\| \assign \|U\|_{D}$ \label{supnorm2}.
For $\gamma \in (1, 2)$ and $B \subset D$ we define the semi-norm
\begin{equation}
  [U]_{\gamma, B} \assign \sup_{z \in B} \inf_{\nu (z) \in \mathbb{R}^d}
  \sup_{w \neq z \in B} \frac{| U_z (w) - \nu (z) \cdot (\Pi_z \X) (w) |}{d (z,
  w)^{\gamma}} . \label{finitegammanormofU}
\end{equation}
In addition, if \eqref{finitegammanormofU} is finite, 
the vector field $u_X = (u_{X_1}, \ldots, u_{X_d})$ \label{page_generalised_grad2} given by
\begin{equation}
  z \mapsto u_{X_i} (z) \assign \lim_{h \downarrow 0} \frac{U (z + h e_i, z)}{h} =
  \frac{\partial}{{\partial w_i} } U (w, z) |_{w = z}
  \label{generalisedgradient}
\end{equation}
generalises the gradient of $u$ as the unique
vector field which attains the minimum in \eqref{finitegammanormofU} - see
Lemma \ref{lemmacontrolongradient}.
%
Also, let
$\sigma (U) : D \times D \rightarrow \mathbb{R}$ be given by
\begin{equation}
  \sigma (U)_z (w) \assign \sigma (u (w)) - (\sigma (u (z)) + \sigma' \sigma
  (u (z)) (\Pi_z \lolli) (w)) .
  \label{defsigmaofU}
\end{equation}
The norm $[\sigma (U)]_{\gamma,B}$ is defined according to
\eqref{finitegammanormofU}, where the $\nu (z)$ that attains the infimum in this case is
$\sigma' (u (z)) u_X (z)$, according to \eqref{generalisedgradient}. 

In the same way as \eqref{orderboundlollipop}, for $B \subset D$,
we define the semi-norm of $u$ as
\begin{equation}
  [u]_{1 - \kappa, B} \assign \sup_{z \neq w \in B}
  \frac{| u(w) - u(z) |}{d (z, w)^{1 - \kappa}} .
  \label{semi-normofu}
\end{equation}
The next lemma establishes simple relationships between the norms and semi-norms
of $U$ and $u$.
\begin{lemma}
  Recall the parabolic ball $B(\cdot,\cdot)$ in \eqref{balllookingtopast}.
  Then, for every $z \in D_0^1$ and $L>0$ such that $B(z,L) \subset D_0^1$, 
  \begin{equation}
    [u]_{1 - \kappa, B (z, L)} \lesssim [U]_{\gamma, B (z, L)} L^{\gamma - 1 +
    \kappa} + \Cnoise{} + \| u_X \|_{B (z, L)} L^{\kappa} . \label{unormbyUnorm}
  \end{equation}
  Furthermore, for any $B \subset D_0^1$ and $r>0$,
  \begin{equation}
    \| U \|_{B, r}  := \sup_{\underset{d(z,w) < r}{w,
    z \in B}} | U_z (w) | 
    \lesssim \| u \|_{B}
    + \Cnoise{} r^{1 - \kappa}  \label{linfinityUbyusetting}.
  \end{equation}
\end{lemma}
\begin{proof}
 In fact, for every $w \in B(z,L)$,
  \begin{eqnarray*}
    | u (w) - u (z) | & \leqslant & | U_z (w) - u_X (z) \cdot (\Pi_z \X) (w) |\\
    & + & | \sigma (u (z)) | | (\Pi_z \lolli) (w)  | + | u_X (z) | | (\Pi_z \X) (w) |\\
    & \leqslant & [U]_{\gamma} d(z,w)^{\gamma} + \| \sigma \| [\Pi ; \lolli]_{1 -
    \kappa} d(z,w)^{1 - \kappa} + | u_X (z) | d(z,w) ,
  \end{eqnarray*}
  so \eqref{unormbyUnorm} follows from \eqref{assumptionsigmabounded}
  and \eqref{assumptiononstochasticobjects}. By definition \eqref{defofU},
  \begin{equation*}
    | U_z (w) | \leqslant |u(w)| + |u(z)| + \| \sigma \| [\Pi ; \lolli]_{1 - \kappa}
    d(z,w)^{1-\kappa} ,
  \end{equation*}
  and \eqref{linfinityUbyusetting} follows from taking the supremum over the set
  $\{z,w \in B : d(z,w) < r\}$.
\end{proof}

Furthermore, we recall some scaling estimates from
\cite[eqs. (2.10), (2.11)]{Moinatandweber2020cpam}.
For $\varphi^L$ as in \eqref{defofmollifierkernel} with $L>0$ and any $\vphi \in \Ftest$,
\begin{eqnarray}
  \int_D | \varphi^L (z - \bar{z}) | d (z, \bar{z})^{1-\kappa} \mathd \bar{z} &
  \leqslant & L^{1-\kappa},\nonumber \\
  \int_D | \nabla_x \varphi^L (z - \bar{z}) | d (z, \bar{z})^{1-\kappa} \mathd
  \bar{z} & \lesssim & L^{-\kappa}, \label{boundongradphiL}
\end{eqnarray}
where 
$\nabla_x$ acts on $x$ in $z=(t,x)$. 
The estimates above imply that for any $B \subset D_0^1$ and $L>0$
such that $B(z,L) \subset D_0^1$ for every $z \in B$,
\begin{equation}
  \| u_L - u \|_B \leqslant L^{1-\kappa} \sup_{z \in B}  [u]_{1-\kappa, B (z, L)}
  . \label{ineqhL-h}
\end{equation}

\section{Interior estimate}\label{sectioninteriorestimate}

The goal of this section is to bound 
the semi-norm $[U]_{\gamma}$ in \eqref{finitegammanormofU} in terms of lower
order regularity norms of the lift $\Pi$ in \eqref{orderbounds}-\eqref{orderboundlollipop}
and $\|u\|$. We divide the argument into three steps. First, we employ a local
expansion on the r.h.s. of
\begin{equation}
  (\partial_t - \mathLaplace) U_z = \sigma (u) \diamond \noiseb - \sigma (u (z)) \Pi_z \noise . 
  \label{equationforU}
\end{equation}
to match the local description of the product
$\sigma(u) \diamond \noiseb$ around a base
point $z \in D_0^{T_\star}$. 
Then, the Reconstruction Theorem \ref{reconstructiontheorem}
gives precise meaning to $\sigma (u) \diamond \noiseb$ and provides control on the error
with respect to the local description. Secondly,
we reduce the semi-norms appearing in the output of the reconstruction to
semi-norms and norms of the function $U$. Finally, we use the Schauder estimate in
Lemma \ref{lemmaschauderestimate}
and control the $\gamma$-H{\"o}lder semi-norm of $U$ by 
$\|u\|$. This procedure is divided into three subsections.

\refchange{
More specifically, the Schauder estimate in Lemma
\ref{lemmaschauderestimate}, states that if for
scales $0 < \ell \leqslant L$ and $1 < \gamma < 2$, if
\begin{equation*}
    \| (\partial_t - \Delta) (U_z)_{\ell} \|_{B (z, L)} \leqslant
    M^{(1)} \sum_{\beta \in \Beta} \ell^{\beta - 2} L^{\gamma
    - \beta} \; ,
  \end{equation*}
  and the ``three-point
  continuity'' condition
  \begin{equation*}
    | U_{z_0} (z_2) - U_{z_0} (z_1) - U_{z_1} (z_2) | \leqslant
    M^{(2)} \sum_{\beta \in \Beta} | z_1 - z_0
    |^{\beta} d(z_1,z_2)^{\gamma - \beta} 
  \end{equation*}
  hold true, then
there exists a constant $C^{\Beta, \gamma}_{\text{Schauder}} > 0$ which only depends on
  $\gamma$ and $\Beta$ such that
  \begin{equation*}
    [U]_{\gamma, 0, [a, b]} \leqslant 
    C^{\Beta, \gamma}_{\text{Schauder}}
    \left[ \sup_{\tau \in (a, b]}
    d_{\tau}^{\gamma} \left( M^{(1)} +
    M^{(2)} \right) +
    \sup_{\tau \in (a, b]} \| U \|_{D_{\tau}^b, d_{\tau}} \right],
  \end{equation*}
  where $\| U \|_{B, d}$ denotes the $L^{\infty}$ norm of $U$ restricted to
  the set $\{ z, w \in B : d(w,z) \leqslant d \}$. That is, 
  we must verify the two conditions above with a good control
  on the constants $M^{(1)}$ and $M^{(2)}$, which are estimated
  in Lemmas \ref{lemmacheckingorderboundschauder} and
  \ref{lemmachecking3ptcontinuityschauder}, and then
  reduce all quantities in the r.h.s. to
  $[U]_\gamma$ and $\|u\|$ in Lemma
  \ref{lemmachainrule}.
}

The proof of the main Theorem \ref{maintheorem} will follow
a bootstrapping argument. For a given initial condition $u_0 \in L^\infty$,
consider a constant $C_{\star} > \| u_0 \| \vee 1$, where we introduce the
notation $a \vee b := \max\{a,b\}$. Define
the following stopping time
\begin{equation}
  T_{\star} \assign \inf \{ t \in [0, 1] : \| u \|_{D_0^t} \geqslant C_{\star}
  \}, \label{definitionstoppingtime}
\end{equation}
with the convention that $T_{\star} = 1$ if $\| u \|_{D_0^t} \leqslant
C_{\star}$ for every $t \in [0, 1]$. The local solution theory provides that
there exists a $T_{\star} = T_{\star} (C_{\star}) > 0$, given by
\eqref{definitionstoppingtime} for which
\begin{equation}
  \| u \|_{D_0^{T_{\star}}} \leqslant C_{\star} .
  \label{trickwithlinfinity} 
\end{equation}
The ultimate goal is then to show that there is a way to choose $C_{\star} >
0$ large enough so that a better bound than \eqref{trickwithlinfinity} is
obtained for $\| u \|_{D_0^{T_{\star}}}$. This way, $T_{\star} = T_{\star}
(C_{\star})$ for this choice has to be equal to 1, and the value of
$C_{\star}$ immediately prescribes the final bound for $\| u \|_{D_0^1}$ in
\eqref{mainresultbound}.

The main result of this section is as follows.

\begin{theorem}[Interior estimate]
  \label{theointeriorestimate}
  For any $\kappa \in \left( 0, \frac{1}{3}
  \right)$, consider $\gamma$ satisfying $1 + \kappa < \gamma \leqslant 2 - 2
  \kappa$ and $T:=T(\gamma)$ satisfying 
  \begin{equation}
    T(\gamma) \leqslant \varepsilon \Big( \Cdumb{}^{- \frac{1}{1 - \kappa}} \wedge \Cnoise{}^{-
    \frac{2}{1 - \kappa}} C_{\star}^{- e(\gamma)} \Big),
    \quad \tmop{where} \quad
    e(\gamma) \assign \frac{2 (\gamma - 1)}{1 -
    \kappa} \label{conditionforT}\; ,
  \end{equation}
  and where $\varepsilon>0$ is small enough.
  Then, for every $t \in [T, T_{\star}]$, it holds that
  \begin{equation}
    [U]_{\gamma, 0, [t - T, t]} \lesssim 1 \vee \| u \|_{D_{t - T}^t}
    \leqslant C_{\star}, \label{theinteriorestimate}
  \end{equation}
  where we recall that $T_{\star}$ and $C_{\star}$ are defined
  in \eqref{definitionstoppingtime}-\eqref{trickwithlinfinity} and $\lesssim$
  holds up to a multiplicative constant that only depends on $\kappa, d$ and
  $C_{\sigma}$ in \eqref{assumptionsigmabounded}.
\end{theorem}

\refchange{The seminorm $[U]_{\gamma, 0, [t - T, t]}$ appearing in \eqref{theinteriorestimate} is defined in \eqref{normofUblowup}.}
By post-processing Theorem~\ref{theointeriorestimate} we obtain the following.

\begin{corollary}
  \label{corollaryimprovedinteriorestimate}
  For any $\gamma_1$ which satisfies
  the conditions of Theorem \ref{theointeriorestimate}, consider $T_1 := T(\gamma_1)$ 
  in \eqref{conditionforT}.
  Then, for any $\gamma_1 \leqslant \gamma_2 \leqslant 2 - 2 \kappa$ and $t
  \in [T_1, T_{\star}]$, it holds that
  \begin{equation}
    [U]_{\gamma_2, 0, [t - T_1, t]} \lesssim 1 \vee \| u \|_{D_{t - T_1}^t}
    \leqslant C_{\star} . \label{interiorestimateimproved}
  \end{equation}
\end{corollary}

These are the main results we prove in this section. An application of them
provides control on $\|u\|$ for short times which is independent of $C_{\star}$. 
This is the content of the next theorem, whose proof is
postponed to Section \ref{sectionproofofshorttimescontrol}.

\begin{theorem}
  \label{theoinfinityschauderpostprocessing}Under the assumptions of Theorem
  \ref{theointeriorestimate},
  \begin{equation}
    \| u \|_{D_0^{T}} \leqslant 2 (\| u_0 \| \vee 1) ,
    \label{infinityschauderpostprocessing}
  \end{equation}
  where $T := T(\gamma)$ satisfies \eqref{conditionforT}.
\end{theorem}

We proceed with the strategy to prove Theorem \ref{theointeriorestimate}
and Corollary \ref{corollaryimprovedinteriorestimate}.

\subsection{Bound on the right hand side (Reconstruction)}\label{sectionReconstruction}

In this subsection we control the r.h.s. of \eqref{equationforU} and we assume throughout the
hypothesis of Theorem \ref{theointeriorestimate}. Recall from Section 
\ref{sectionstrategyofproof} that the local description of the product $\sigma(u) \diamond \noiseb$
is given by
\begin{equation}
  G_z \assign \sigma (u (z)) \Pi_z \noise + \sigma' \sigma (u (z)) \Pi_z \dumb +
  \sigma' (u (z)) u_X (z) \cdot \Pi_z \Xnoise, 
  \label{localapproxG}
\end{equation}
where 
$\Pi$ satisfies \eqref{orderbounds}-\eqref{orderboundlollipop} and
\eqref{changeofbasepointXnoise}-\eqref{changeofbasepointdumbbell}.
In addition, the definition of $\sigma(U)$ in \eqref{defsigmaofU}
and of $[\sigma (U)]_{\gamma}$ according to \eqref{finitegammanormofU}
in Section \ref{sectionsetting}. 
\refchange{
For a set $B \subset D_0^{T_\star}$,
the semi-norms 
$[\sigma' \sigma (u)]_{\gamma - 1 + \kappa,B} :=
[f_1]_{\alpha_1, B}$
and $[\sigma' (u) u_X]_{\gamma - 1,B} := [f_2]_{\alpha_2, B}$
denote the usual H\"older norms as in \eqref{orderboundlollipop},
i.e.
}
\begin{equation*}
  [f_i]_{\alpha_i, B} := \sup_{z \neq w \in B}
  \frac{|f_i(w) - f_i(z)|}{d(z,w)^{\alpha_i}} 
  \quad i=1,2 ,
\end{equation*}
where 
$f_2$ is vector-valued and hence for $i=2$ $|\cdot|$ denotes the
Euclidean norm over $\R^d$. 
The main result of this subsection is the following.

\begin{lemma} \label{lemmacontrolonrhsmainequation}
  Recall that $1 + \kappa < \gamma \leqslant 2 - 2\kappa$. Assume further that
  $[\sigma (U)]_{\gamma,B}$, $[\sigma' \sigma (u)]_{\gamma - 1 + \kappa,B}$ and 
  $[\sigma' (u) u_X]_{\gamma - 1,B}$ are finite for any $B \subset D_0^{T_\star}$.
  Then, we may reconstruct $\sigma (u) \diamond \noiseb$
  and obtain, for any $z \in D_0^{T_\star}$
  and $L > 0$ such that $B(z,L) \subset D_0^{T_\star}$,
  that the error in the local description satisfies
  \begin{equation}
    | \langle \sigma (u) \diamond \noiseb - G_z, \varphi_z^L \rangle | \leqslant
    C_{\mathcal{R}} L^{\gamma - 1 - \kappa} 
    (C^{\sigma,\noise}_{z,L} +
    C^{\sigma,\dumb}_{z,L} + 
    C^{\sigma,\Xnoise}_{z,L}),
    \label{outputreconstruction}
  \end{equation}
  where $C_{\mathcal R}>0$ comes from Theorem \ref{reconstructiontheorem} and
\label{page_modelled_coeff_bounds}
  \begin{eqnarray}
    C^{\sigma,\noise}_{z,L} & : = & [\sigma (U)]_{\gamma, B (z, L)} [\Pi ; \noise]_{- 1
    - \kappa} 
    \label{defCsigmanoise}\\
    C^{\sigma,\dumb}_{z,L} & := & [\sigma' \sigma (u)]_{\gamma - 1 +
    \kappa, B (z, L)} [\Pi ; \dumb]_{- 2 \kappa}
    \label{defCsigmadumb}\\
    C^{\sigma,\Xnoise}_{z,L} & := & [\sigma' (u) u_X]_{\gamma - 1, B (z, L)}
    [\Pi ; \Xnoise]_{- \kappa} .
    \label{defCsigmaXnoise}
  \end{eqnarray}
  Moreover, if $F_z := \sigma (u) \diamond \noiseb - \sigma (u (z)) \Pi_z \noise$ denotes the
  r.h.s. of \eqref{equationforU}, then
  \begin{equation}
    | \langle F_z, \varphi_z^L \rangle | \lesssim L^{\gamma - 1 - \kappa}
    (C^{\sigma,\noise}_{z,L} + C^{\sigma,\dumb}_{z,L} +
    C^{\sigma,\Xnoise}_{z,L}) + \Cdumb{} L^{- 2 \kappa} +
    \Cnoise{} | u_X (z) | L^{- \kappa}
    \label{finalbounddiagonal} .
  \end{equation}

\end{lemma}

\begin{proof}
  Using \eqref{localapproxG} we 
  write
  \begin{eqnarray*}
       F_z = \sigma (u) \diamond \noiseb - \sigma (u (z)) \Pi_z \noise & = & \sigma (u) \diamond \noiseb - G_z\\
       & + & \sigma' \sigma (u (z)) \Pi_z \dumb + \sigma' (u (z)) u_X
       (z) \Pi_z \Xnoise ,
  \end{eqnarray*}
  so that for any $z \in D_0^{T_\star}$ and $L > 0$ such that $B(z,L) \subset D_0^{T_\star}$,
  \begin{eqnarray}
    | \langle F_z, \varphi_z^L \rangle | & \leqslant & | \langle \sigma (u) \diamond \noiseb
    - G_z, \varphi_z^L \rangle | \nonumber\\
    & + & \| \sigma' \sigma \| [\Pi ; \dumb]_{- 2 \kappa} L^{- 2 \kappa} + \|
    \sigma' \| | u_X (z) | [\Pi ; \Xnoise]_{- \kappa} L^{- \kappa} . 
    \label{righthandsideFz}
  \end{eqnarray}
  To control $| \langle \sigma (u) \diamond \noiseb - G_z, \varphi_z^L \rangle |$, we check
  that $\{ G_z \}_z$ satisfies the assumptions of Theorem
  \ref{reconstructiontheorem}. 
  By using \eqref{changeofbasepointXnoise} and
  \eqref{changeofbasepointdumbbell}, we see that
  \begin{eqnarray*}
    G_{z_1} - G_{z_2} & = & (\sigma (U)_{z_2} (z_1) - \sigma' (u (z_2)) u_X
    (z_2) (\Pi_{z_2} \X)(z_1)) \Pi_{z_1} \noise\\
    & + & (\sigma' \sigma (u (z_1)) - \sigma' \sigma (u (z_2))) \Pi_{z_1} \dumb\\
    & + & (\sigma' (u (z_1)) u_X (z_1) - \sigma' (u (z_2)) u_X (z_2)) \cdot
    \Pi_{z_1} \Xnoise,
  \end{eqnarray*}
  so that for every $1 + \kappa < \gamma \leqslant 2 - 2\kappa$,
  $\ell \in (0, L)$ and $z_1, z_2 \in B (z,
  L - \ell)$, we have
  \begin{eqnarray*}
    | \langle G_{z_1} - G_{z_2}, \varphi_{z_1}^{\ell} \rangle | & \leqslant &
    [\sigma (U)]_{\gamma, B (z, L)} d(z_2,z_1)^{\gamma} [\Pi ; \noise]_{- 1 - \kappa}
    \ell^{- 1 - \kappa}\\
    & + & [\sigma' \sigma (u)]_{\gamma - 1 + \kappa, B (z, L)} d(z_2,z_1)
    ^{\gamma - 1 + \kappa} [\Pi ; \dumb]_{- 2 \kappa} \ell^{- 2
    \kappa}\\
    & + & [\sigma' (u) u_X]_{\gamma - 1, B (z, L)} d(z_2,z_1)^{\gamma - 1}
    [\Pi ; \Xnoise]_{- \kappa} \ell^{- \kappa} .
  \end{eqnarray*}
  Thus, condition \eqref{conditionforreconstruction} is 
  satisfied for every $\theta \in \{ - 1 - \kappa, - 2 \kappa, - \kappa \}$ and
  $\gamma_{\theta} = \tilde{\theta} = \gamma - 1 - \kappa > 0$. Therefore,
  a direct application of Theorem \ref{reconstructiontheorem} gives
  \eqref{outputreconstruction}.
  Finally, \eqref{righthandsideFz} together with \eqref{orderbounds}, 
  \eqref{assumptionsigmabounded} and \eqref{outputreconstruction} gives
  \eqref{finalbounddiagonal} and concludes the proof.
\end{proof}


\subsection{Reducing $\sigma (U)$ in terms of $U$ (Chain
rule)}\label{reduction}

In this subsection we control the quantities $C^{\sigma,\noise}_{z,L}$,
$C^{\sigma,\dumb}_{z,L}$ and $C^{\sigma,\Xnoise}_{z,L}$ in
\eqref{finalbounddiagonal} in terms of $\| U \|$, $[U]_{\gamma}$, $\| u_X \|$
and $[u_X]_{\gamma - 1}$. By Lemma \ref{lemmacontrolongradient}, the control
on $\| u_X \|_{B (z, L)}$ and $[u_X]_{\gamma - 1, B (z, L)}$ will again be
reduced to $\| U \|$ and $[U]_{\gamma}$ in Section \ref{sectionchoiceofT0}.

The next lemma is the main result of this section.

\begin{lemma} \label{lemmachainrule}
  Assume the hypothesis of Theorem \ref{theointeriorestimate}. For any $z \in D_0^{T_\star}$
  and $L>0$ such that $B(z,L) \subset D_0^{T_\star}$, we have the following estimates
  \begin{eqnarray}
    &  & [\sigma (U)]_{\gamma, B (z, L)} \nonumber\\
    & \lesssim & [U]_{\gamma, B (z, L)} + \| u_X \|^{\gamma}_{B (z, L)} +
    \Cnoise{}^2 L^{2 - 2 \kappa - \gamma} + \Cnoise{} \| u_X \|_{B (z, L)} L^{2 -
    \kappa - \gamma} ;  \label{mainreductionI}\\
    & & \nonumber \\
    &  & [\sigma' \sigma (u)]_{\gamma - 1 + \kappa, B (z, L)} \nonumber\\
    & \lesssim & [U]_{\gamma, B (z, L)} L^{1 - \kappa} + \Cnoise{} L^{2 - 2
    \kappa - \gamma} + \| u_X \|_{B (z, L)} L^{2 - \kappa - \gamma} ; 
    \label{mainreductionIV}\\
    & & \nonumber \\
    &  & [\sigma' (u) u_X]_{\gamma - 1, B (z, L)} \nonumber\\
    & \lesssim & \| u_X \|_{B (z, L)} [U]_{\gamma, B (z, L)}^{\frac{\gamma -
    1}{\gamma}} + \| u_X \|^{\gamma}_{B (z, L)} + \Cnoise{} \| u_X \|_{B (z, L)}
    L^{2 - \kappa - \gamma}  \label{mainreductionII}\\
    & + & [u_X]_{\gamma - 1, B (z, L)} .  \label{mainreductionIII}
  \end{eqnarray}
\end{lemma}

\begin{proof} 
We start with $[\sigma (U)]_{\gamma, B (z, L)}$. For every $w \in B (z, L)$, let
\begin{equation}
  \eta_1 : = \sigma (u (z)) (\Pi_z \lolli) (w), \quad \eta_2 \assign u_X (z)
  \cdot (\Pi_z \X) (w) \quad \tmop{and} \quad \eta \assign \eta_1 + \eta_2
  \label{auxiliareta}
\end{equation}
then
\begin{eqnarray}
  &  & | \sigma (U)_z (w) - \sigma' (u (z)) u_X (z) \cdot (\Pi_z \X) (w) |
  \nonumber\\
  & = & | \sigma (u (w)) - (\sigma (u (z)) + \sigma' (u (z)) \eta) |
  \nonumber\\
  & \leqslant & | \sigma (u (w)) - \sigma (u (z) + \eta) | 
  \label{sigmaUpartI}\\
  & + & | \sigma (u (z) + \eta) - \sigma (u (z) + \eta_2) - \sigma' (u (z))
  \eta_1 |  \label{sigmaUpartII}\\
  & + & | \sigma (u (z) + \eta_2) - \sigma (u (z)) - \sigma' (u (z)) \eta_2 |
  \label{sigmaUpartIII}
\end{eqnarray}
and we control each term of the sum separately. For the first, we have
\begin{eqnarray*}
  \eqref{sigmaUpartI} & \leqslant & \| \sigma' \| | u (w) - u (z) - \sigma (u
  (z)) (\Pi_z \lolli) (w) - u_X (z) \cdot (\Pi_z \X) (w) |\\
  & \lesssim & [U]_{\gamma, \refchange{B(z, L)}} d(z,w)^{\gamma} .
\end{eqnarray*}
For the second, we get
\begin{eqnarray*}
  \eqref{sigmaUpartII} & = & \left| \int_0^1 \frac{\mathd}{\mathd \lambda}
  (\sigma (u (z) + \eta_2 + \lambda \eta_1)) \mathd \lambda - \sigma' (u (z))
  \eta_1 \right|\\
  & = & \left| \eta_1 \int_0^1 (\sigma' (u (z) + \eta_2 + \lambda \eta_1) -
  \sigma' (u (z))) \mathd \lambda \right|\\
  & \leqslant & \| \sigma'' \| (| \eta_1 |^2 + | \eta_1 | | \eta_2 |)\\
  & \leqslant & \| \sigma'' \| (\| \sigma \|^2 [\Pi ; \lolli]^2_{1 - \kappa} d(z,w)^{2 - 2 \kappa}
  + \| \sigma \|  [\Pi ; \lolli]_{1 - \kappa} | u_X (z) | d(z,w)^{2 - \kappa})\\
  & \lesssim & (\Cnoise{}^2 L^{2 - 2 \kappa - \gamma} + \Cnoise{} | u_X (z) |
  L^{2 - \kappa - \gamma}) d(z,w)^{\gamma}
\end{eqnarray*}
And for the third, it holds that
\begin{eqnarray*}
  \eqref{sigmaUpartIII} & = & \left| \int_0^1 \frac{\mathd}{\mathd \lambda}
  (\sigma (u (z) + \lambda \eta_2)) \mathd \lambda - \sigma' (u (z)) \eta_2
  \right|\\
  & = & \left| \eta_2 \int_0^1 (\sigma' (u (z) + \lambda \eta_2) - \sigma' (u
  (z))) \mathd \lambda \right|\\
  & \leqslant & \left\{\begin{array}{l}
    2 \| \sigma' \| | \eta_2 |\\
    \| \sigma'' \| | \eta_2 |^2
  \end{array}\right. .
\end{eqnarray*}
Therefore, by interpolating between the two, we obtain
\[ \eqref{sigmaUpartIII} \lesssim | u_X (z) |^{\gamma} d(z,w)^{\gamma} , \]
which leads to \eqref{mainreductionI}. 

Next, we treat $[\sigma' \sigma (u)]_{\gamma - 1 + \kappa, B (z, L)}$. It
follows that
\[ | \sigma' \sigma (u (w)) - \sigma' \sigma (u (z)) | \leqslant \| \sigma''
   \sigma + (\sigma')^2 \| [u]_{1 - \kappa, B (z, L)} d(z,w)^{1 - \kappa},
\]
and thus since $\gamma \leqslant 2 - 2 \kappa \Longrightarrow \gamma - 1 +
\kappa \leqslant 1 - \kappa$,
\begin{eqnarray}
  &  & [\sigma' \sigma (u)]_{\gamma - 1 + \kappa, B (z, L)} \nonumber\\
  & \leqslant & 2 C_{\sigma}^2 L^{2 - 2 \kappa - \gamma} [u]_{1 - \kappa, B
  (z, L)} \nonumber\\
  & \lesssim & [U]_{\gamma, B (z, L)} L^{1 - \kappa} + \Cnoise{} L^{2 - 2
  \kappa - \gamma} + \| u_X \|_{B (z, L)} L^{2 - \kappa - \gamma} ,
  \label{normsigmasigmaprime}
\end{eqnarray}
where we have used \eqref{unormbyUnorm}. This is \eqref{mainreductionIV}
Note in particular that assuming that
$\| u_X \|_{B (z, L)} \geqslant 1$, the estimate in
\eqref{normsigmasigmaprime} is dominated by \eqref{mainreductionI},
since $\gamma > 1$. Finally, we move to $[\sigma' (u) u_X]_{\gamma - 1, B (z,
L)}$. Using \eqref{auxiliareta},
\begin{eqnarray}
  &  & | \sigma' (u (w)) u_X (w) - \sigma' (u (z \nobracket) u_X (z) |
  \nonumber\\
  & \leqslant & | \sigma' (u (z)) | | u_X (w) - u_X (z) | 
  \label{sigmaprimegraduII}\\
  & + & | \sigma' (u (w)) - \sigma' (u (z) + \eta) | | u_X (w) | 
  \label{sigmaprimegraduI}\\
  & + & \left| \sigma' (u (z) + \eta) - \sigma' \left( u (z) {+ \eta_1} 
  \right) \right| | u_X (w) |  \label{sigmaprimegradu3}\\
  & + & | \sigma' (u (z) + \eta_1) - \sigma' (u (z)) | | u_X (w) | 
  \label{sigmaprimegradu4}
\end{eqnarray}
and we proceed to treat each term individually. The first term is simpler:
\[ \eqref{sigmaprimegraduII} \leqslant \| \sigma' \| [u_X]_{\gamma - 1, B (z,
   L)} d(z,w)^{\gamma - 1} \lesssim [u_X]_{\gamma - 1, B (z, L)} d(z,w)^{\gamma - 1} . \]
For \eqref{sigmaprimegraduI}, it holds that
\[ \eqref{sigmaprimegraduI} \leqslant \left\{\begin{array}{l}
     2 \| \sigma' \| \| u_X \|_{B (z, L)}\\
     \| \sigma'' \| \| u_X \|_{B (z, L)} [U]_{\gamma, B (z, L)} d(z,w)^{\gamma}
   \end{array}\right., \]
so by interpolating between the two, we obtain
\[ \eqref{sigmaprimegraduI} \lesssim \| u_X \|_{B (z, L)} [U]_{\gamma, B (z,
   L)}^{\frac{\gamma - 1}{\gamma}} d(z,w)^{\gamma - 1} . \]
Moving to \eqref{sigmaprimegradu3}, we get
\[ \eqref{sigmaprimegradu3} \leqslant \left\{\begin{array}{l}
     2 \| \sigma' \| \| u_X \|_{B (z, L)}\\
     \| \sigma'' \| \| u_X \|^2_{B (z, L)} d(z,w)
   \end{array}\right., \]
so by interpolating between the two, we obtain
\[ \eqref{sigmaprimegradu3} \lesssim \| u_X \|^{\gamma}_{B (z, L)} d(z,w)^{\gamma - 1} . \]
Lastly, we treat \eqref{sigmaprimegradu4}, for which we obtain
\[ \eqref{sigmaprimegradu4} \leqslant \| \sigma'' \| \| \sigma \| [\Pi ; \lolli]_{1
   - \kappa} | u_X (w) | d(z,w)^{1 - \kappa} \lesssim \Cnoise{} \| u_X \|_{B
   (z, L)} L^{2 - \kappa - \gamma} d(z,w)^{\gamma - 1} . \]
Thus, collecting the estimates obtained for
\eqref{sigmaprimegraduII}-\eqref{sigmaprimegradu4}
gives \eqref{mainreductionII}-\eqref{mainreductionIII} and concludes the proof.
\end{proof}
\begin{remark}
  We keep all of the terms in Lemma \ref{lemmachainrule} to make sure we
  carry on the correct dependence of
  $\Cnoise{}$ and $\Cdumb{}$ in the final estimate. Nevertheless, we point out
  that for the heart of the
  argument the dominant term is $\| u_X \|^{\gamma}_{B (z, L)}$,
  \refchange{which comes from the estimate of the term
  \eqref{sigmaprimegradu3}, corresponding to the error $E_z$
  in \eqref{gradientsquaredinerrorterm}.
  }.
\end{remark}

\subsection{Integration}\label{sectionchoiceofT0}

In this step, we control $[U]_{\gamma}$ by inverting the heat operator
$(\partial_t - \Delta)$ and using the estimate on the r.h.s. obtained by the
previous step. 
Since we measure a higher regularity for positive times than what we
assume for the initial data,
the semi-norm $[U]_{\gamma, B}$ is expected to blow up as we allow for the set $B$ to contain
space-time points with time close to $t = 0$. This motivates the following.
Recall that for $[a, b] \subset [0, \infty)$, $D_a^b = [a, b] \times
\mathbb{T}^d$. For $1 + \kappa < \gamma \leqslant 2 - 2 \kappa$ and $\eta
\leqslant \gamma$, let
\begin{equation}
  [U]_{\gamma, \eta, [a, b]} \assign \sup_{\tau \in (a, b]} d_{\tau}^{\gamma -
  \eta} [U]_{\gamma, D_{\tau}^b, d_{\tau}},
  \label{normofUblowup}
\end{equation}
where $d_{\tau} \assign \sqrt{\tau - a}$\label{dist_to_bdry} denotes the minimum distance from the
parabolic boundary $t = a$ of $D_a^b$ and $[U]_{\gamma, D_{\tau}^b, d_{\tau}}$
denotes the $\gamma$-H{\"o}lder norm of $U$ in \eqref{finitegammanormofU}
restricted to the set $\{ z, w \in D_{\tau}^b : d(w,z) \leqslant d_{\tau}
\}$. The notation $[\cdummy]_{\gamma, \eta, [a, b]}$ indicates the regularity
$\gamma > 0$, and the blow up of order $\eta - \gamma$ at $t \downarrow a$ and
the time interval $[a, b]$; it is consistent with the notation in
{\cite[Def. 6.2]{Hai14}}.

To apply Lemma \ref{lemmaschauderestimate} to \eqref{equationforU}, we
need to verify conditions \eqref{orderboundcondition} and
\eqref{3ptcontinuitycondition}. Recall from Lemma
\ref{lemmacontrolonrhsmainequation} that
$F_z := \sigma (u) \diamond \noiseb - \sigma (u (z)) \Pi_z \noise$ denotes the r.h.s. of
\eqref{equationforU}. The conditions are verified in the next two lemmas.

\begin{lemma} \label{lemmacheckingorderboundschauder}
  Condition \eqref{orderboundcondition} in Lemma \ref{lemmaschauderestimate}
  is satisfied for every $1 + \kappa < \gamma \leqslant 2 - 2\kappa$ and any
  $D_a^b \subset D_0^{T_{\star}}$. 
  Moreover,
    \refchange{if we consider two exponents
    $1 + \kappa < \gamma_1 \leqslant \gamma_2 \leqslant 2 - 2\kappa$,
  the estimate \eqref{outputofschauder} holds with $M^{(1)}_{D_{\tau}^b, \frac{d_{\tau}}{2}}
  = M^{(1)}_{D_{\tau}^b, \frac{d_{\tau}}{2}}(\gamma_1,\gamma_2)$
  such that
  \begin{equation}
    \renewcommand{\arraystretch}{1.5}
    \begin{array}{cl}
      &
      d_{\tau}^{\gamma_2}M^{(1)}_{D_{\tau}^b, \frac{d_{\tau}}{2}} \\
      \lesssim &
      \Cnoise{} \left([U]_{\gamma_1, D_{\tau}^b, \frac{d_{\tau}}{2}}
      + \| u_X \|^{\gamma_1}_{D_{\tau}^b}
      + \| u_X \|_{D_{\tau}^b} 
      [U]_{\gamma_1, D_{\tau}^b, \frac{d_{\tau}}{2}}^{\frac{\gamma_1 - 1}{\gamma_1}}
      + [u_X]_{\gamma_1 - 1, D_{\tau}^b, \frac{d_{\tau}}{2}}
      \right) d_{\tau}^{\gamma_1 + 1 - \kappa}
      \\
      + &
      \left(
      \Cnoise{}^2 + \Cdumb{} 
      \right)
      \| u_X \|_{D_{\tau}^b}d_{\tau}^{3 - 2\kappa}
      +
      \left(\Cnoise{}^3 + \Cdumb{}\Cnoise{}\right)
      d_{\tau}^{3 - 3\kappa}
      +
      \Cdumb{} [U]_{\gamma_1, D_{\tau}^b, \frac{d_{\tau}}{2}}
      d_{\tau}^{\gamma_1 +2 - 2\kappa}
      \\
      + &
      (\Cnoise{}^2 + \Cdumb{}) d_{\tau}^{2 - 2\kappa}
      + \Cnoise{} \| u_X \|_{D_{\tau}^b} d_{\tau}^{2 - \kappa} \; .
    \end{array}
    \label{M1schauder}
  \end{equation}  
i.e., it holds w.r.t. $\gamma_2$ while containing only
semi-norms of order $\gamma_1$ in the r.h.s.. 
}
\end{lemma}
\refchange{
\begin{proof}
Recall the definition of $U : D_a^b \times D_a^b \rightarrow
\mathbb{R}$ in \eqref{defofU}. Clearly $U_z (z) = 0$ for any $z \in
D_a^b$. For every $\tau \in (a, b]$, let $L \leqslant d_{\tau} / 4$. For all
base points $z \in D_{\tau + L}^b$, $w \in B (z, L)$ and scales $\ell
\leqslant L$, it holds that
\begin{equation}
  | (\partial_t - \mathLaplace) (U_z)_{\ell} (w) | = | \langle F_z,
  \varphi_w^{\ell} \rangle | \leqslant | \langle F_w, \varphi_w^{\ell} \rangle
  | + | \langle F_z - F_w, \varphi_w^{\ell} \rangle |\;. \label{whatisUzlw}
\end{equation}
The off-diagonal 
term $| \langle F_z - F_w, \varphi_w^{\ell} \rangle |$ satisfies 
(recall $\Pi_z \noise = \Pi_w \noise$)
\begin{equation*}
  F_z - F_w = (\sigma (u (w)) - \sigma (u (z))) \Pi_z \noise \; ,
\end{equation*}
and thus by \eqref{orderbounds}
and \eqref{unormbyUnorm}
\begin{eqnarray}
  | \langle F_z - F_w, \varphi_w^{\ell} \rangle | 
  & \leqslant &
  \|\si'\| |u (w) - u (z)|
  \Cnoise{} \ell^{-1-\kappa} \nonumber\\
  & \lesssim &
  \Cnoise{} \ell^{- 1 - \kappa}
  \left([U]_{\gamma_1, B (z, L)}
  L^{\gamma_1} + \Cnoise{} L^{1 - \kappa}
  + | u_X (z) | L\right).
  \label{boundforchangeofbasepoint}
\end{eqnarray}
Moving to the diagonal term 
$| \langle F_w, \varphi_w^{\ell} \rangle |$,
since $\gamma_1 > 1 + \kappa$,  \eqref{finalbounddiagonal} gives
\begin{equation}
    | \langle F_w, \varphi_w^\ell \rangle |
    \lesssim \ell^{\gamma - 1 - \kappa}
    (C^{\sigma,\noise}_{w,\ell} + C^{\sigma,\dumb}_{w,\ell} +
    C^{\sigma,\Xnoise}_{w,\ell}) 
    + \Cdumb{} \ell^{- 2 \kappa} +
    \Cnoise{} | u_X (w) | \ell^{- \kappa}
    \label{controlondiagonalpretreatment}
\end{equation}
and we proceed to treat each term individually, making use of
Lemma \ref{lemmachainrule}. By \eqref{defCsigmanoise} and
\eqref{mainreductionI}, we learn that
\begin{eqnarray}
    & & \ell^{\gamma_1 - 1 - \kappa}C^{\sigma,\noise}_{w,\ell} 
    \nonumber\\
    & \lesssim &
    \Cnoise{}
    \left([U]_{\gamma_1, B (w, \ell)} +
    \| u_X \|^{\gamma_1}_{B (w, \ell)}\right) \ell^{\gamma_1 - 1 - \kappa}
    + \Cnoise{}^3 \ell^{1 - 3\kappa} + 
    \Cnoise{}^2 \| u_X \|_{B (w, \ell)} \ell^{1 - 2\kappa}
    \nonumber
\end{eqnarray}
while by \eqref{defCsigmadumb} and
\eqref{mainreductionIV}, we get
\begin{equation*}
    \ell^{\gamma_1 - 1 - \kappa}C^{\sigma,\dumb}_{w,\ell} 
    \lesssim
    \Cdumb{}\left(
    [U]_{\gamma_1, B (w, \ell)} \ell^{\gamma_1 - 2\kappa}
    + \Cnoise{} \ell^{1 - 3 \kappa}
    + \| u_X \|_{B (w, \ell)} \ell^{1 - 2\kappa}
    \right)
\end{equation*}
and finally by \eqref{defCsigmaXnoise} and
\eqref{mainreductionII}-\eqref{mainreductionIII} that
\begin{eqnarray}
    & & \ell^{\gamma_1 - 1 - \kappa}C^{\sigma,\Xnoise}_{w,\ell} 
    \nonumber\\
    & \lesssim &
    \Cnoise{} \left(
    \| u_X \|_{B (w, \ell)} 
    [U]_{\gamma_1, B (w, \ell)}^{\frac{\gamma_1 - 1}{\gamma_1}}
    + \| u_X \|^{\gamma_1}_{B (w, \ell)}
    + [u_X]_{\gamma_1 - 1, B (w, \ell)}
    \right) \ell^{\gamma_1 - 1 - \kappa}
    \nonumber\\
    & + &
    \Cnoise{}^2 \| u_X \|_{B (w, \ell)} \ell^{1 - 2\kappa} 
    \nonumber.
\end{eqnarray}
Putting \eqref{whatisUzlw}, \eqref{boundforchangeofbasepoint} and
\eqref{controlondiagonalpretreatment} together, grouping the terms
by powers of $\ell, L$, and taking the supremum
over $w \in B (z, L)$ we get
 \begin{equation}
    \renewcommand{\arraystretch}{1.5}
    \begin{array}{cl}
      &
      \| (\partial_t - \mathLaplace) (U_z)_{\ell} \|_{B (z, L)} \\
      \lesssim &
      \Cnoise{} \left([U]_{\gamma_1, D_{\tau}^b, \frac{d_{\tau}}{2}}
      + \| u_X \|^{\gamma_1}_{D_{\tau}^b}
      + \| u_X \|_{D_{\tau}^b} 
      [U]_{\gamma_1, D_{\tau}^b, \frac{d_{\tau}}{2}}^{\frac{\gamma_1 - 1}{\gamma_1}}
      + [u_X]_{\gamma_1 - 1, D_{\tau}^b, \frac{d_{\tau}}{2}}
      \right) \ell^{\gamma_1 - 1 - \kappa}
      \\
      + &
      \left(
      \Cnoise{}^2
      + \Cdumb{} 
      \right)
      \| u_X \|_{D_{\tau}^b}\ell^{1 - 2\kappa}
      +
      \left(\Cnoise{}^3 + \Cdumb{}\Cnoise{}\right)
      \ell^{1 - 3\kappa}
      +
      \Cdumb{} [U]_{\gamma_1, D_{\tau}^b, \frac{d_{\tau}}{2}}
      \ell^{\gamma_1 - 2\kappa}
      \\
      + &
      \Cdumb{} \ell^{- 2 \kappa} + \Cnoise{} \| u_X \|_{D_{\tau}^b} \ell^{- \kappa}
      \\
      + &
      \Cnoise{}
      \left([U]_{\gamma_1, D_{\tau}^b, \frac{d_{\tau}}{2}}
      L^{\gamma_1} + \Cnoise{} L^{1 - \kappa} + \| u_X \|_{D_{\tau}^b} L
      \right) \ell^{- 1 - \kappa}
      .
    \end{array}
    \label{controlonUzlw}
  \end{equation}  
  Since $0 < \ell \leqslant L \leqslant d_\tau/4$ and $\gamma \leqslant 2 - 2\kappa$,
  we show that \eqref{controlonUzlw} satisfies \eqref{orderboundcondition}
  with some powers of $d_\tau$ to spare. \refchange{Below} 
  we indicate to which
  $\beta$ each term above corresponds to,
  \refchange{
  which allows us to identify the set $\Beta$ as 
  $\Beta = \{0,1 - \kappa\}$:
  }
  \begin{equation}
    \begin{array}{rclcl}
      \displaystyle
      \ell^{\gamma_1 - 1 - \kappa} & \leqslant & 
      d_\tau^{\gamma_1-\gamma_2 + 1-\kappa} \ell^{-2} L^{\gamma_2}
      & \tmop{and} &\beta = 0 \\
      \displaystyle
      \ell^{1 - 2\kappa} & \leqslant & 
      d_\tau^{3-2\kappa -\gamma_2} \ell^{-2} L^{\gamma_2}
      & \tmop{and} &\beta = 0 \\
      \displaystyle
      \ell^{1 - 3\kappa} & \leqslant & 
      d_\tau^{3-3\kappa -\gamma_2} \ell^{-2} L^{\gamma_2}
      & \tmop{and} &\beta = 0 \\
      \displaystyle
      \ell^{\gamma_1 - 2\kappa} & \leqslant & 
      d_\tau^{\gamma_1-\gamma_2 + 2-2\kappa} \ell^{-2} L^{\gamma_2}
      & \tmop{and} &\beta = 0 \\
      \displaystyle
      \ell^{-2\kappa} & \leqslant & d_\tau^{2 - 2\kappa - \gamma_2} \ell^{-2} L^{\gamma_2}
      & \tmop{and} &\beta = 0 \\
      \displaystyle
      \ell^{-\kappa} & \leqslant & d_\tau^{2 - \kappa
      -\gamma_2} \ell^{-2} L^{\gamma_2}
      & \tmop{and} &\beta = 0 \\
      \displaystyle
      \ell^{-1 - \kappa}L^{\gamma_1} & \leqslant & 
      d_\tau^{\gamma_1-\gamma_2 + 1 - \kappa} \ell^{-1 - \kappa} L^{\gamma_2 - 1 + \kappa}
      & \tmop{and} &\beta = 1 - \kappa \\
      \displaystyle
      \ell^{-1 - \kappa}L^{1 - \kappa} & \leqslant & d_\tau^{2 - 2\kappa -\gamma_2}
      \ell^{-1 - \kappa} L^{\gamma_2 - 1 + \kappa}
      & \tmop{and} &\beta = 1 - \kappa \\
      \displaystyle
      \ell^{-1 - \kappa}L^{1} & \leqslant & d_\tau^{2 - \kappa - \gamma_2}
      \ell^{-1 - \kappa} L^{\gamma_2 - 1 + \kappa}
      & \tmop{and} &\beta = 1 - \kappa .
    \end{array}
    \label{checkingorderboundschauder}
  \end{equation}
  To identify the constant $M^{(1)}_{D_{\tau}^b, \frac{d_{\tau}}{2}}$ in \eqref{outputofschauder} 
  satisfying \eqref{M1schauder}, put \eqref{controlonUzlw} together with the powers of $d_\tau$
  in \eqref{checkingorderboundschauder}
  and note that $\gamma_1-\gamma_2 + 1 - \kappa > 0$. This concludes the proof.
%
%
%
%
\end{proof}
}

\refchange{
\begin{lemma} \label{lemmachecking3ptcontinuityschauder}
  Condition \eqref{3ptcontinuitycondition} in Lemma \ref{lemmaschauderestimate}
  is satisfied for every $1 + \kappa < \gamma \leqslant 2 - 2\kappa$ and any
  $D_a^b \subset D_0^{T_{\star}}$. Furthermore,
  if we consider two exponents
    $1 + \kappa < \gamma_1 \leqslant \gamma_2 \leqslant 2 - 2\kappa$,
  estimate \eqref{outputofschauder} holds with 
  $M^{(2)}_{D_{\tau}^b, \frac{d_{\tau}}{2}, \frac{d_{\tau}}{4}}
  = M^{(2)}_{D_{\tau}^b, \frac{d_{\tau}}{2}, \frac{d_{\tau}}{4}}(\gamma_1,\gamma_2)$ such that
  \begin{equation}
    d_{\tau}^{\gamma_2}M^{(2)}_{D_{\tau}^b, \frac{d_{\tau}}{2}, \frac{d_{\tau}}{4}} \lesssim
    \Cnoise{} \left( [U]_{\gamma_1, D_{\tau}^b, \frac{d_{\tau}}{2}} d_{\tau}^{\gamma_1 + 1 -
    \kappa} + \Cnoise{} d_{\tau}^{2 - 2 \kappa} + \| u_X \|_{D_{\tau}^b}
    d_{\tau}^{2 - \kappa} \right) .
  \label{controlonM2}
  \end{equation}
\end{lemma}

\begin{proof}
For every $z_0 \in D_{\tau + \frac{d_{\tau}}{2}}^b,
z_1 \in B \left( z_0, \frac{d_{\tau}}{2} \right)$ and $z_2 \in B \left( z_1,
\frac{d_{\tau}}{4} \right)$, by \eqref{unormbyUnorm} we have
\begin{eqnarray*}
  &  & | U_{z_0} (z_2) - U_{z_0} (z_1) - U_{z_1} (z_2) |\\
  & = & | (\sigma (u (z_1)) - \sigma (u (z_0))) (\lollib (z_2) - \lollib (z_1))
  |\\
  & \leqslant & \| \sigma' \| [u]_{1 - \kappa, B (z_0, d_{\tau})} d(z_0,z_1)^{1 - \kappa}
  [\Pi ; \lolli]_{1 - \kappa} d(z_1,z_2)^{1 - \kappa}\\
  & \lesssim & \Cnoise{} ([U]_{\gamma_1, B (z_0, d_{\tau})} d_{\tau}^{\gamma_1 - 1
  + \kappa} + \Cnoise{} + \| u_X \|_{B (z_0, d_{\tau})} d_{\tau}^{\kappa}) d(z_0,z_1)^{1 - \kappa}
  d(z_1,z_2)^{1 - \kappa}\\
  & \lesssim &
    \Cnoise{} ([U]_{\gamma_1, B (z_0, d_{\tau})} d_{\tau}^{\gamma_1 - 1 + \kappa} +
    \Cnoise{} + \| u_X \|_{B (z_0, d_{\tau})} d_{\tau}^{\kappa}) d_{\tau}^{2 - 2
    \kappa - \gamma_2} \\
    & & \times d(z_0,z_1)^{1 - \kappa} d(z_1,z_2)^{\gamma_2 - 1 + \kappa} ,
\end{eqnarray*}
since $\gamma_2 \leqslant 2 - 2\kappa$. Therefore condition \eqref{3ptcontinuitycondition}
is satisfied with $\beta = 1 - \kappa$. From the estimate above, we readily identify
$M^{(2)}_{D_{\tau}^b, \frac{d_{\tau}}{2}, \frac{d_{\tau}}{4}}$ in \eqref{outputofschauder}
satisfying \eqref{controlonM2}.
\end{proof}
}

We are ready to prove the main theorem of this section.

\subsection{Proof of Theorem \ref{theointeriorestimate}
and Corollary \ref{corollaryimprovedinteriorestimate}} \label{sectionproofofinteriorestimate}

In this subsection we put together all estimates obtained in Section \ref{sectioninteriorestimate}
and prove the interior estimate in Theorem \ref{theointeriorestimate} 
\refchange{
and its post-processed estimate in Corollary \ref{corollaryimprovedinteriorestimate}.
We start with the main
analytical step, which is a combination of integration and reconstruction:

\begin{proposition} \label{gamma1gamma2proposition}
    For any $\kappa \in (0,\frac{1}{3})$, consider 
    $\gamma_1$ and $\gamma_2$ such that
    $$1 + \kappa < \gamma_1 \leqslant \gamma_2 \leqslant 2 - 2\kappa \; .$$
    Then, for any $[a,b] \subset [0,T_{\star}]$,
    \begin{equation}
     \begin{array}{lll}
    [U]_{\gamma_2, 0, [a, b]} & \lesssim & \displaystyle T^{\frac{1 - \kappa}{2}} \Cnoise{} \|
    u \|_{D_a^b}^{\gamma_1 - 1}  [U]_{\gamma_1, 0, [a, b]}\\
    & + & \displaystyle \left[ T^{\frac{1 - \kappa}{2}} \Cnoise{} + T^{1 - \kappa} (\Cnoise{}^2
    + \Cdumb{}) \right] [U]_{\gamma_1, 0, [a, b]}\\
    & + & \displaystyle \left[ T^{\frac{1 - \kappa}{2}} \Cnoise{} + T^{1 - \kappa} (\Cnoise{}^2
    + \Cdumb{}) \right] \| u \|_{D_a^b}\\
    & + & \displaystyle T^{\frac{(1 - \kappa) \gamma_1}{2}} \Cnoise{}^{\gamma_1} + T^{1 - \kappa}
    (\Cnoise{}^2 + \Cdumb{}) + T^{\frac{3 - 3 \kappa}{2}} (\Cdumb{}
    \Cnoise{} + \Cnoise{}^3)\\
    & + & \displaystyle T^{\frac{1 - \kappa}{2}} \Cnoise{} + \| u \|_{D_a^b} \; ,
  \end{array} \label{finalboundbeforeabsorption}
\end{equation}
    where $T := b-a$ denotes the length of the time interval $[a,b]$.
\end{proposition}

\begin{proof}
    Recall that for $\tau
\in [a, b]$, $d_{\tau} = \sqrt{\tau - a}$.
Lemma \ref{lemmacheckingorderboundschauder}
and Lemma \ref{lemmachecking3ptcontinuityschauder} allow us to apply
Lemma \ref{lemmaschauderestimate} with
\eqref{M1schauder} and \eqref{controlonM2} to obtain
\begin{eqnarray}
  & & [U]_{\gamma_2, 0, [a, b]} \nonumber \\
  & \lesssim & \hspace{-2mm} 
  \begin{array}{rl}
    \displaystyle \sup_{\tau \in (a, b]} d_{\tau}^{\gamma_1 + 1 - \kappa}
    \Cnoise{} 
    & \left([U]_{\gamma_1, D_{\tau}^b, \frac{d_{\tau}}{2}}
      + \| u_X \|^{\gamma_1}_{D_{\tau}^b}
    \right.\\
    & \displaystyle \left.
    + \| u_X \|_{D_{\tau}^b} 
      [U]_{\gamma_1, D_{\tau}^b, \frac{d_{\tau}}{2}}^{\frac{\gamma_1 - 1}{\gamma_1}}
      + [u_X]_{\gamma_1 - 1, D_{\tau}^b, \frac{d_{\tau}}{2}}
    \right)
  \end{array}  \label{schauderRHSII}
  \\
  & + & \sup_{\tau \in (a, b]} 
  [d_{\tau}^{3 - 3 \kappa} (\Cnoise{}^3 + \Cdumb{}\Cnoise{})
  + d_{\tau}^{3 - 2\kappa} (\Cnoise{}^2 + \Cdumb{}) \| u_X \|_{D_{\tau}^b}]
  \label{schauderRHSIV}
  \\
  & + & \sup_{\tau \in (a, b]} [d_{\tau}^{2 - 2 \kappa} (\Cnoise{}^2 + \Cdumb{})
  + d_{\tau}^{2 - \kappa} \Cnoise{} \| u_X \|_{D_{\tau}^b}]
  \label{schauderRHSI}
  \\
  & + & \sup_{\tau \in (a, b]} \left[ d_{\tau}^{\gamma_1 + 2 - 2\kappa} 
  \Cdumb{} [U]_{\gamma_1, D_{\tau}^b, \frac{d_{\tau}}{2}} 
  + \| U \|_{D_{\tau}^b, d_{\tau}} \right] ,
  \label{schauderRHSIII}
\end{eqnarray}
We proceed to treat each term separately. 
Starting with the first term in \eqref{schauderRHSIV} and the
first term in \eqref{schauderRHSI}, we see that
\begin{eqnarray}
  & & \sup_{\tau \in (a, b]} \left[
  d_{\tau}^{2 - 2 \kappa} (\Cnoise{}^2 + \Cdumb{})
  + d_{\tau}^{3 - 3 \kappa} (\Cnoise{}^3 + \Cdumb{}\Cnoise{})
  \right] \nonumber\\
  & \leqslant & 
  T^{1 - \kappa} (\Cnoise{}^2 + \Cdumb{})
  + T^{\frac{3 - 3\kappa}{2}} (\Cnoise{}^3 + \Cdumb{}\Cnoise{}).
  \label{newsecondcontrolafterschauder}
\end{eqnarray}
Next, we treat the
second term in \eqref{schauderRHSIV} and the
second term in \eqref{schauderRHSI} ($\|u_X\|$ terms): we apply
\eqref{boundonlinfinitygradient} for every $\tau \in (a, b]$ and $r =
d_{\tau}$ to obtain the control
\begin{eqnarray}
  & & \sup_{\tau \in (a, b]} 
  \left[
  \left(
  \Cnoise{} d_{\tau}^{1 - \kappa}
  + (\Cnoise{}^2 + \Cdumb{}) d_{\tau}^{2 - 2\kappa}
  \right)
  \left( d_{\tau}^{\gamma_1}
  [U]_{\gamma_1, D_{\tau}^b, \frac{d_{\tau}}{2}} + \| U \|_{D_{\tau}^b, d_{\tau}} \right) 
  \right] \nonumber\\
  & \leqslant &
  \left(
  T^{\frac{1 - \kappa}{2}} \Cnoise{}
  + T^{1 - \kappa} (\Cnoise{}^2 + \Cdumb{})
  \right)
  \left(
  [U]_{\gamma_1, 0, [a, b]}
  + \sup_{\tau \in (a, b]} \| U \|_{D_{\tau}^b, d_{\tau}} \right) \; .
  \label{firstcontrolafterschauder}
\end{eqnarray}
The first term in \eqref{schauderRHSII} and the first term in
\eqref{schauderRHSIII} 
produce terms involving $[U]_{\gamma_1, 0, [a, b]}$ which are equal to the
ones contained in
\eqref{firstcontrolafterschauder}.
Moving to the second term in \eqref{schauderRHSII} - the first super-linear
in $u$ term - we
note that for every $\tau \in (a, b]$ and $r
\in [0, d_{\tau}/2]$, by \eqref{boundonlinfinitygradient} we have that
\[ \| u_X \|_{D_{\tau}^b} \leqslant r^{\gamma_1 - 1} [U]_{\gamma_1, D_{\tau}^b,
   \frac{d_{\tau}}{2}} + r^{- 1} \| U \|_{D_{\tau}^b, r} \leqslant r^{\gamma_1 -
   1} [U]_{\gamma_1, D_{\tau}^b, \frac{d_{\tau}}{2}} + r^{- 1} \| U
   \|_{D_{\tau}^b, d_{\tau}}. \]
If we choose $r$ as
\begin{equation}
  r = \left( \frac{\| U \|_{D_{\tau}^b, \frac{d_{\tau}}{2}}}{[U]_{\gamma_1,
  D_{\tau}^b, \frac{d_{\tau}}{2}}} \right)^{\frac{1}{\gamma_1}}
  , \label{choiceofr}
\end{equation}
there are two cases: either $r \in \left[0, \frac{d_{\tau}}{2}\right]$,
for which we arrive at
\begin{equation}
  \| u_X \|_{D_{\tau}^b} \leqslant 2 [U]^{\frac{1}{\gamma_1}}_{\gamma_1,
  D_{\tau}^b, \frac{d_{\tau}}{2}} \| U \|^{1 - \frac{1}{\gamma_1}}_{D_{\tau}^b,
  d_{\tau}} , \label{interpolationboundforgradient}
\end{equation}
which in turn yields
\begin{equation}
  \Cnoise{} \sup_{\tau \in (a, b]} \tmcolor{red}{} d_{\tau}^{\gamma_1 + 1 -
  \kappa} \| u_X \|^{\gamma_1}_{D_{\tau}^b} \lesssim T^{\frac{1 -
  \kappa}{2}} \Cnoise{} [U]_{\gamma_1, 0, [a, b]} \sup_{\tau \in (a, b]} \| U
  \|^{\gamma_1 - 1}_{D_{\tau}^b, \frac{d_{\tau}}{2}} . \label{mainterm}
\end{equation}
Or 
$r$ in \eqref{choiceofr} satisfies $r > d_{\tau}/2$, which by
\eqref{linfinityUbyusetting}  implies that
\[ d_{\tau}^{\gamma_1} [U]_{\gamma_1, D_{\tau}^b, \frac{d_{\tau}}{2}} \lesssim
   \| U \|_{D_{\tau}^b, \frac{d_{\tau}}{2}} \Longrightarrow [U]_{\gamma_1, 0,
   [a, b]} \lesssim \| u \|_{D_a^b}. \]
For the third term in \eqref{schauderRHSII} - the second and last super-linear
in $u$ term - we apply
\eqref{interpolationboundforgradient} to get the bound
\begin{equation}
  \Cnoise{} \sup_{\tau \in (a, b]} \tmcolor{red}{} d_{\tau}^{\gamma_1 + 1 -
  \kappa} [U]^{\frac{\gamma_1 - 1}{\gamma_1} + \frac{1}{\gamma_1}}_{\gamma_1,
  D_{\tau}^b, \frac{d_{\tau}}{2}} \| U \|^{1 - \frac{1}{\gamma_1}}_{D_{\tau}^b,
  \frac{d_{\tau}}{2}} = T^{\frac{1 - \kappa}{2}} \Cnoise{} [U]_{\gamma_1, 0, [a,
  b]} \sup_{\tau \in (a, b]} \| U \|^{1 - \frac{1}{\gamma_1}}_{D_{\tau}^b,
  \frac{d_{\tau}}{2}},
\end{equation}
which is controlled by the one in \eqref{mainterm},
since $1-1/\gamma < \gamma - 1$. 
Finally, we
treat the fourth term in \eqref{schauderRHSII}:
from
\eqref{boundongammanormgradient} we get that it is bounded by
\begin{equation*}
    \Cnoise{} \sup_{\tau \in (a, b]} \tmcolor{red}{}
  d_{\tau}^{\gamma_1 + 1 -
  \kappa}\left([U]_{\gamma_1, D_{\tau}^b, d_{\tau}} + M_{D_{\tau}^b,
  \frac{d_{\tau}}{4}, \frac{d_{\tau}}{4}}^{(2)} \right) \; .
\end{equation*}
The first term above produces a term as in \eqref{firstcontrolafterschauder},
while for the second we use \eqref{controlonM2}
to get that it is bounded by
\begin{equation}
    \Cnoise{} \sup_{\tau \in (a, b]} [\tmcolor{red}{} d_{\tau}^{\gamma_1 + 2
  - 2 \kappa} [U]_{\gamma_1, D_{\tau}^b, d_{\tau}} + d_{\tau}^{3 - 3 \kappa}
  \Cnoise{} + \tmcolor{red}{} d_{\tau}^{3 - 2 \kappa} \| u_X \|_{D_{\tau}^b}] \; ,
\end{equation}
which in turn produces terms contained in \eqref{newsecondcontrolafterschauder}
and also in \eqref{firstcontrolafterschauder} - which come from
\eqref{schauderRHSIV}-\eqref{schauderRHSI}.

Putting everything together 
(i.e.
\eqref{newsecondcontrolafterschauder}, \eqref{firstcontrolafterschauder},
\eqref{mainterm} and the last term in \eqref{schauderRHSIII}), we arrive to
\[ \begin{array}{lll}
     {}[U]_{\gamma_2, 0, [a, b]} & \lesssim & \displaystyle T^{\frac{1 - \kappa}{2}} \Cnoise{}
     \displaystyle \sup_{\tau \in (a, b]} \| U \|^{\gamma_1 - 1}_{D_{\tau}^b,
     \frac{d_{\tau}}{2}}  [U]_{\gamma_1, 0, [a, b]}\\
     & + & \displaystyle \left[ T^{\frac{1 - \kappa}{2}} \Cnoise{} + T^{1 - \kappa}
     (\Cnoise{}^2 + \Cdumb{}) \right] [U]_{\gamma_1, 0, [a, b]}\\
     & + & \displaystyle \left[ T^{\frac{1 - \kappa}{2}} \Cnoise{} + T^{1 - \kappa}
     (\Cnoise{}^2 + \Cdumb{}) \right] \sup_{\tau \in (a, b]} \| U
     \|_{D_{\tau}^b, d_{\tau}}\\
     & + & \displaystyle T^{1 - \kappa} (\Cnoise{}^2 + \Cdumb{}) + T^{\frac{3 - 3
     \kappa}{2}} (\Cdumb{} \Cnoise{} + \Cnoise{}^3)\\
     & + & \displaystyle \sup_{\tau \in (a, b]} {\| U \|_{D_{\tau}^b, d_{\tau}}}  .
   \end{array} \]
One important observation here is that, if we write $\alpha = \frac{1 -
\kappa}{2}$, then, in the previous estimate,  
$\Cnoise{}$ comes with $T^\alpha$, 
both $\Cnoise{}^2$ and $\Cdumb{}$ come with $T^{2\alpha}$, and  
both $\Cnoise{}^3$ and $\Cdumb{} \Cnoise{}$ come with $T^{3 \alpha}$.
The final step is to reduce $\| U \|$ to $\| u \|$, 
which is obtained with the use of \eqref{linfinityUbyusetting}
\begin{equation}
  {\| U \|_{D_{\tau}^b, d_{\tau}}}  
  \lesssim \| u \|_{D_{\tau}^b}
  + \Cnoise{} d_{\tau}^{1 - \kappa}  \label{linfinityUbyu},
\end{equation}
so the control on $[U]_{\gamma_1, 0, [a, b]}$ becomes the desired
\eqref{finalboundbeforeabsorption},
where we used that $\Cnoise{}$ in \eqref{linfinityUbyu} comes with the correct
power of $d_{\tau}$ and that $\gamma_1 - 1 \in (0, 1)$ implies $\| U \|^{\gamma_1
- 1}_{D_{\tau}^b, \frac{d_{\tau}}{2}} \lesssim \| u \|^{\gamma_1 -
1}_{D_{\tau}^b} + (\Cnoise{} d_{\tau}^{1 - \kappa})^{\gamma_1 - 1}$, so
$\Cnoise{}^{\gamma_1}$ also comes with the correct power of $T$ in
\eqref{finalboundbeforeabsorption}.
\end{proof}

}

\vspace{2mm}
\begin{proof*}{Proof of Theorem \ref{theointeriorestimate}.} 
\refchange{
Fix $\gamma_1 = \gamma_2 = \gamma$ satisfying the assumptions of
Proposition \ref{gamma1gamma2proposition}.
Recall that 
$[a, b] \subset [0, T_{\star}]$, so \eqref{trickwithlinfinity} gives
\[ T^{\frac{1 - \kappa}{2}} \Cnoise{} \| u \|_{D_a^b}^{\gamma - 1} \leqslant
   T^{\frac{1 - \kappa}{2}} \Cnoise{} C_{\star}^{\gamma - 1} , \]
and the choice of $T$ in \eqref{conditionforT} guarantees the absorption condition
\begin{equation}
  T^{\frac{1 - \kappa}{2}} \Cnoise{} C_{\star}^{\gamma - 1} \leqslant
  \frac{1}{4}
  \quad \tmop{and} \quad
  T^{1 - \kappa} \Cdumb{} \leqslant \frac{1}{16}
  \label{absorbtioncondition}
\end{equation}
Since $C_{\star} > 1$, we have in particular that $T^{\frac{1 -
\kappa}{2}} \Cnoise{} \leqslant \frac{1}{4}$ and $T^{\frac{1 - \kappa}{2}}
\Cnoise{} + T^{1 - \kappa} \Cnoise{}^2 + T^{1 - \kappa} \Cdumb{} \leqslant
\frac{1}{2}$, so that we may absorb $[U]_{\gamma, 0, [a, b]}$ into the l.h.s.
of \eqref{finalboundbeforeabsorption} to obtain finally that
\[ [U]_{\gamma, 0, [a, b]} \lesssim 1 \vee \| u \|_{D_a^b} \leqslant
   C_{\star}, \]
since the sum of the powers of $T^{\frac{1 - \kappa}{2}} \Cnoise{}$ and $T^{1 -
\kappa} \Cdumb{}$ are bounded by $1$. This gives \eqref{theinteriorestimate}
and concludes the proof.
}
\end{proof*}

We conclude this section with the proof of Corollary \ref{corollaryimprovedinteriorestimate}.

\vspace{2mm}
\begin{proof*}{Proof of Corollary \ref{corollaryimprovedinteriorestimate}.}
    \refchange{
    Use Proposition \ref{gamma1gamma2proposition} with
    $\gamma_1 \leqslant \gamma_2$ and use Theorem \ref{theointeriorestimate} with $\gamma = \gamma_1$.
    }
\end{proof*}

\section{Maximum principle}\label{sectionmaximumprinciple}

In this section we obtain a control on the $L^{\infty}$ norm
of $u$ by dealing with a regularisation of equation \eqref{mainequationrenormalised} at a
specially chosen scale $\tilde{L} > 0$. 
Recall
that for $L > 0$ and a set $B \subset D_0^1$ such that $B (z, L) \subset
D_0^1$ for every $z \in B$, the regularisation of $u$ at scale $L$ is denoted
by $u_L : B \rightarrow \mathbb{R}$ and given by
\begin{equation}
  u_L (z) = \langle u, \varphi^L_z \rangle = \int_D u (\bar{z}) \varphi^L
  (\bar{z} - z) \mathd \bar{z} = \int_D u (\bar{z}) \varphi^L (z - \bar{z})
  \mathd \bar{z} = u \ast \varphi^L (z) .
  \label{definitionofuL}
\end{equation}
Here, $\varphi^L_z$ is as in 
\eqref{defofmollifierkernel} for the choice of mollifier
$\vphi \in \mathfrak F$ in \eqref{choiceofmollifiersemigroup}, described in detail in
Appendix \ref{appendixreconstructionandintegration}.
Since $u_L$ is
smooth and $(\partial_t - \Delta)$ commutes with regularisation $u \mapsto
u_L$, it satisfies the following equation
\begin{equation}
  (\partial_t - \Delta) u_L (z) 
  = \langle \sigma (u) \diamond
  \noiseb, \varphi_z^L \rangle \label{regularisedequationbeginning} .
\end{equation}
However, for any $z = (t, x) \in
D_0^1$, we may only look to regularisations at scales $L < \sqrt{t} / 2$, so
in order for the estimates to hold uniformly over $z \in B$ for some choice
of $\tilde{L} > 0$, the domain $B$ must stay away from $t = 0$. Recall that Theorem
\ref{theoinfinityschauderpostprocessing} gives control on
the solution $u$ to \eqref{mainequationrenormalised} over $D_0^{T_1}$, where $T_1$ can be taken as
\begin{equation}
  T_1 \sim \Cdumb{}^{- \frac{1}{1 - \kappa}} \wedge \Cnoise{}^{-
     \frac{2}{1 - \kappa}} C_{\star}^{- \frac{2 (\kappa + \delta)}{1 -
     \kappa}} ,
     \label{choiceofT1maxprinc}
\end{equation}
since it satisfies \eqref{conditionforT} for $T_1 = T(\gamma_1)$ with
 $\gamma = 1 + \kappa + \delta$ in Theorem
\ref{theointeriorestimate}. So if we estimate the solution $u_L$ to
\eqref{regularisedequationbeginning} over $D^{T_{\star}}_{T_1}$ for some scale
$L > 0$ that satisfies $L < \sqrt{T_1}/2$, we may use \eqref{ineqhL-h}
to translate it back to a control of the solution $u$ to \eqref{mainequationrenormalised} over
$D^{T_{\star}}_{T_1}$.

The following is the main theorem of this section and it is 
the last ingredient to prove
our main result 
Theorem \ref{maintheorem}.

\begin{theorem}[Maximum Principle]
  \label{theoremtopostprocessintopolynimial}
  Take $\gamma_1 = 1 + \kappa +
  \delta$, for $\delta > 0$, $\gamma_2 = 2 - 2 \kappa$ and
  $T_1$ as in \eqref{choiceofT1maxprinc}, which satisfy the
  assumptions of Corollary \ref{corollaryimprovedinteriorestimate}. 
  Then, we have that
  \begin{equation}
    \| u \|_{D_{T_1}^{T_{\star}}} \leqslant \| u \|_{D_0^{T_1}} + C (\kappa,
    d, C_{\sigma}) \left( \Cdumb{}^{\frac{1}{1 - \kappa}}
    C^{\betadumb}_{\star} \vee \Cnoise{}^{\frac{2}{1 - \kappa}}
    {C^{\betanoise}_{\star}}  \right) , \label{controlawayfrom0}
  \end{equation}
  where 
  \begin{equation}
    \betanoise \assign \frac{2 (1 + \kappa)}{3 - 2 \kappa} + \frac{(1 + \kappa)
    (\kappa + \delta)}{1 - \kappa}
    \label{betanoiseofdelta}
  \end{equation}
  and $\betadumb$ satisfy \eqref{betas},
  $C (\kappa, d, C_{\sigma}) > 0$ denotes a constant
  which only depends on
  $\kappa, d$ and $C_{\sigma}$, and $T_\star, C_{\star}$ are defined in
  \eqref{definitionstoppingtime}-\eqref{trickwithlinfinity}.
\end{theorem}
\refchange{Equation~\eqref{betanoiseofdelta} above and the requirement that $\betanoise < 1$ determine the regime of roughness (that is, how large we can choose $\kappa$) for which we can prove our main result Theorem~\ref{maintheorem}. 
The choice \eqref{betanoiseofdelta} arises from careful power counting and balancing of the terms appearing in the estimates given in Proposition~\ref{regularisedequationwithtransport} - this balancing is explained in the proof of Proposition~\ref{propositiontransportdecomposition} on p.~\pageref{pageref:proof_of_prop2}.}
We now provide the proof of Theorem~\ref{maintheorem}.

\vspace{2mm}
\begin{proof*}{Proof of Theorem \ref{maintheorem}.}
  Putting \eqref{controlawayfrom0} together with
  \eqref{infinityschauderpostprocessing}, we note that
  \begin{eqnarray*}
    \| u \|_{D_0^{T_{\star}}} & \leqslant & \max \left\{ \| u \|_{D_0^{T_1}},
    \| u \|_{D_{T_1}^{T_{\star}}} \right\}\\
    & \leqslant & \| u \|_{D_0^{T_1}} + C (\kappa, d, C_{\sigma}) \left(
    \Cdumb{}^{\frac{1}{1 - \kappa}} C^{\betadumb}_{\star} \vee
    \Cnoise{}^{\frac{2}{1 - \kappa}} {C^{\betanoise}_{\star}}  \right)\\
    & \leqslant & 2 \| u_0 \| + C (\kappa, d, C_{\sigma}) \left( \Cdumb{}
    ^{\frac{1}{1 - \kappa}} C^{\betadumb}_{\star} \vee
    \Cnoise{}^{\frac{2}{1 - \kappa}} {C^{\betanoise}_{\star}}  \right) ,
  \end{eqnarray*}
  which allows us to conclude the bootstrapping argument introduced at the 
  beginning of Section \ref{sectioninteriorestimate}.
  In fact, comparing the bound
  above with \eqref{trickwithlinfinity}, we conclude that
  \begin{eqnarray*}
    2 \| u_0 \| + C (\kappa, d, C_{\sigma}) \left( \Cdumb{}
    ^{\frac{1}{1 - \kappa}} C^{\betadumb}_{\star} \vee \Cnoise{}
    ^{\frac{2}{1 - \kappa}} {C^{\betanoise}_{\star}}  \right) & \leqslant &
    C_{\star}\\
    & \Updownarrow & \\
    \frac{2 \| u_0 \|}{C_{\star}} + C (\kappa, d, C_{\sigma}) \left(
    \frac{\Cdumb{}^{\frac{1}{1 - \kappa}}}{C^{1 - \betadumb
    }_{\star}} \vee \frac{\Cnoise{}^{\frac{2}{1 - \kappa}}}{{C^{1 -
    \betanoise}_{\star}} } \right) & \leqslant & 1.
  \end{eqnarray*}
  Since we assume that $\kappa < \bar{\kappa}$, for $\bar{\kappa} \in (0,1/3)$
  solution to \eqref{equationforkappa}, this implies that for a sufficiently small
  $\delta > 0$ the exponent $\betanoise$ in \eqref{betanoiseofdelta}
  satisfies $\betadumb<\betanoise < 1$. This can be seen from the fact that
  \[ \delta \mapsto \frac{(1 + \kappa) (\kappa + \delta)}{1 - \kappa} \]
  is continuous at $\delta = 0$ and increasing for $\delta \geqslant 0$, and
  $\betanoise < 1$ is equivalent to
  \[ 2 (1 + \kappa) (1 - \kappa) + (1 + \kappa) (\kappa + \delta) (3 - 2 \kappa)
     < (3 - 2 \kappa) (1 - \kappa), \]
  from which we deduce that $\bar{\kappa}$ is
  the relevant root of \eqref{equationforkappa}.
  Therefore, as $C (\kappa, d, C_{\sigma})$ and $\betadumb <
  \betanoise < 1$ are independent of $C_{\star}$, we may choose $C_{\star}$
  to satisfy
  \begin{equation}
    \begin{array}{lll}
      C_{\star} & \geqslant & \max \left\{ 4 \| u_0 \|, (2 C (\kappa, d,
      C_{\sigma}))^{\frac{1}{1 - \betadumb}} \Cdumb{}
      ^{\frac{1}{(1 - \kappa) (1 - \betadumb)}} \right.,\\
      &  & \left. (2 C (\kappa, d, C_{\sigma}))^{\frac{1}{1 - \betanoise}}
      \Cnoise{}^{\frac{2}{(1 - \kappa) (1 - \betanoise)}} \right\}
    \end{array} \label{choiceofCstar}
  \end{equation}
  which completes the proof. The constant $C (\kappa, d, C_{\sigma})$ in
  \eqref{mainresultbound} differs from the one above, but only depends on
  $\kappa, d$ and $C_{\sigma}$.
\end{proof*}

In order to prove Theorem \ref{theoremtopostprocessintopolynimial}, 
one important step is to use the interior estimate 
obtained in the last section. The following lemma is a consequence of Corollary
\ref{corollaryimprovedinteriorestimate} and will be used repeatedly in this section.

\begin{lemma}
  \label{lemmaofblowups}Under the assumptions of Theorem
  \ref{theoremtopostprocessintopolynimial}, the following hold true for any
  $z \in D_{T_1}^{T_{\star}}$ and $0 < L < \sqrt{T_1} / 2$
  \begin{enumerateroman}
    \item
    \begin{equation}
      [U]_{\gamma_2, B (z, L)} \lesssim C_{\star} (\Cdumb{} \vee
      \Cnoise{}^2 C_{\star}^{2 \kappa + 2 \delta}) ; \label{comparisonofnorms}
    \end{equation}
    \item
    \begin{equation}
      \| u_X \|_{B (z, L)} \lesssim \left( C_{\star} + \Cnoise{}^{\frac{1 - 2
      \kappa}{2 (1 - \kappa)}} L^{\frac{1}{2} - \kappa} C^{\frac{1}{2 (1 -
      \kappa)}}_{\star} \right) (\Cdumb{} \vee \Cnoise{}^2 C_{\star}^{2
      \kappa + 2 \delta})^{\frac{1}{2 (1 - \kappa)}} ;
      \label{reductionofLinfinityofgradient}
    \end{equation}
    \item
    \begin{equation}
      \begin{array}{lll}
        &  & [u_X]_{\gamma_2 - 1, B (z, L)}\\
        & \lesssim & (1 + \Cnoise{} L^{1 - \kappa}) C_{\star} (\Cdumb{}
        \vee \Cnoise{}^2 C_{\star}^{2 \kappa + 2 \delta}) + \Cnoise{}^2\\
        & + & \left( \Cnoise{} L^{\kappa} C_{\star} + \Cnoise{}^{1 + \frac{1 - 2
        \kappa}{2 (1 - \kappa)}} L^{\frac{1}{2}} C^{\frac{1}{2 (1 -
        \kappa)}}_{\star} \right) (\Cdumb{} \vee \Cnoise{}^2 C_{\star}^{2
        \kappa + 2 \delta})^{\frac{1}{2 (1 - \kappa)}} .
      \end{array} \label{reductionofgammaofgradient}
    \end{equation}
  \end{enumerateroman}
\end{lemma}

\begin{proof}
  We start with \eqref{comparisonofnorms}. Recall that $d_{\tau} \assign
  \sqrt{\tau - t + T_1}$ and that $L \leqslant \sqrt{T_1}/2$. By
  definition \eqref{normofUblowup}, for every $z = (t, x) \in
  D_{T_1}^{T_{\star}}$, we have
  \begin{eqnarray}
    {}[U]_{\gamma_2, 0, [t - T_1, t]} & = & \sup_{\tau \in (t - T_1, t]} 
    {(\tau - t + T_1)^{\frac{\gamma_2}{2}}}  [U]_{\gamma_2, D_{\tau}^t,
    d_{\tau}} \nonumber\\
    & \geqslant & \sup_{\tau \in (t - T_1, t - T_1 / 2]}  {(\tau - t +
    T_1)^{\frac{\gamma_2}{2}}}  [U]_{\gamma_2, D_{\tau}^t, d_{\tau}}
    \nonumber\\
    & \geqslant & \sup_{\tau \in (t - T_1, t - T_1 / 2]}  {(\tau - t +
    T_1)^{\frac{\gamma_2}{2}}}  [U]_{\gamma_2, B (z, d_{\tau})} \nonumber\\
    & = & 2^{- \frac{\gamma_2}{2}} {T_1^{\frac{\gamma_2}{2}}}  [U]_{\gamma_2,
    B \left( z, \frac{\sqrt{T_1}}{2} \right)} 
    \gtrsim {T_1^{\frac{\gamma_2}{2}}}  [U]_{\gamma_2, B (z, L)} 
    \label{noblowupthenblowup}
  \end{eqnarray}
  since $B (z, d_{\tau}) \subset \{ w_1, w_2 \in D_{\tau}^t : d(w_2,w_1) <
  d_{\tau} \}$ for every $\tau \in (t - T_1, t - T_1 / 2]$. This in turn
  implies that by choosing $T_1 \sim \Cdumb{}^{- \frac{1}{1 - \kappa}}
  \wedge \Cnoise{}^{- \frac{2}{1 - \kappa}} C_{\star}^{- e (\gamma_1)}$ in
  \eqref{choiceofT1maxprinc}, and using \eqref{interiorestimateimproved} in
  Corollary \ref{corollaryimprovedinteriorestimate}, we obtain
  \begin{eqnarray*}
    {}[U]_{\gamma_2, B (z, L)} & \lesssim & {T_1^{- \frac{\gamma_2}{2}}} 
    [U]_{\gamma_2, 0, [t - T_1, t]}\\
    & \lesssim & \left( \Cdumb{}^{\frac{\gamma_2}{2 (1 - \kappa)}} \vee
    \Cnoise{}^{\frac{\gamma_2}{1 - \kappa}} C_{\star}^{\frac{\gamma_2}{2} e
    (\gamma_1)} \right) C_{\star}\\
    & = & \Cdumb{}^{\frac{\gamma_2}{2 (1 - \kappa)}} C_{\star} \vee
    \Cnoise{}^{\frac{\gamma_2}{1 - \kappa}} C_{\star}^{1 + \frac{\gamma_2}{2} e
    (\gamma_1)} .
  \end{eqnarray*}
  We move to \eqref{reductionofLinfinityofgradient}, for which we use
  \eqref{interpolationboundforgradient}, \eqref{linfinityUbyu},
  \eqref{trickwithlinfinity} and \eqref{comparisonofnorms} to get
  \begin{eqnarray*}
    \| u_X \|_{B (z, L)} & \leqslant & 2 [U]^{\frac{1}{\gamma_2}}_{\gamma_2, B
    (z, L)} \| U \|^{1 - \frac{1}{\gamma_2}}_{B (z, L)}\\
    & \lesssim & \left( \Cdumb{}^{\frac{1}{2 (1 - \kappa)}}
    C^{\frac{1}{\gamma_2}}_{\star} \vee \Cnoise{}^{\frac{1}{1 - \kappa}}
    C_{\star}^{\frac{1}{\gamma_2} + \frac{1}{2} e (\gamma_1)} \right)
    (C_{\star} + \Cnoise{} L^{1 - \kappa})^{1 - \frac{1}{\gamma_2}}\\
    & \leqslant & \Cdumb{}^{\frac{1}{2 (1 - \kappa)}} C_{\star} \vee
    \Cnoise{}^{\frac{1}{1 - \kappa}} C_{\star}^{1 + \frac{1}{2} e (\gamma_1)}\\
    & + & \left( \Cdumb{}^{\frac{1}{2 (1 - \kappa)}}
    C^{\frac{1}{\gamma_2}}_{\star} \vee \Cnoise{}^{\frac{1}{1 - \kappa}}
    C_{\star}^{\frac{1}{\gamma_2} + \frac{1}{2} e (\gamma_1)} \right)
    \Cnoise{}^{\frac{\gamma_2 - 1}{\gamma_2}} L^{\frac{(1 - \kappa) (\gamma_2 -
    1)}{\gamma_2}} .
  \end{eqnarray*}
  For \eqref{reductionofgammaofgradient}, we use
  \eqref{boundongammanormgradient}, \eqref{controlonM2},
  \eqref{comparisonofnorms} and \eqref{reductionofLinfinityofgradient} to
  derive
  \begin{eqnarray*}
    &  & [u_X]_{\gamma_2 - 1, B (z, L)}\\
    & \leqslant & [U]_{\gamma_2, B (z, L)} + M_{D_{\tau}^b,
    \frac{d_{\tau}}{4}, \frac{d_{\tau}}{4}}^{(2)}\\
    & \leqslant & [U]_{\gamma_2, B (z, L)} + \Cnoise{} ([U]_{\gamma_2, B (z,
    L)} L^{1 - \kappa} + \Cnoise{} L^{2 - 2 \kappa - \gamma_2} + \| u_X \|_{B
    (z, L)} L^{2 - \kappa - \gamma_2})\\
    & \lesssim & (1 + \Cnoise{} L^{1 - \kappa}) \left( \Cdumb{}
    ^{\frac{\gamma_2}{2 (1 - \kappa)}} C_{\star} \vee
    \Cnoise{}^{\frac{\gamma_2}{1 - \kappa}} C_{\star}^{1 + \frac{\gamma_2}{2} e
    (\gamma_1)} \right) + \Cnoise{}^2 L^{2 - 2 \kappa - \gamma_2}\\
    & + & \left( \Cdumb{}^{\frac{1}{2 (1 - \kappa)}} C_{\star} \vee
    \Cnoise{}^{\frac{1}{1 - \kappa}} C_{\star}^{1 + \frac{1}{2} e (\gamma_1)}
    \right) \Cnoise{} L^{2 - \kappa - \gamma_2}\\
    & + & \left( \Cdumb{}^{\frac{1}{2 (1 - \kappa)}}
    C^{\frac{1}{\gamma_2}}_{\star} \vee \Cnoise{}^{\frac{1}{1 - \kappa}}
    C_{\star}^{\frac{1}{\gamma_2} + \frac{1}{2} e (\gamma_1)} \right)
    \Cnoise{}^{1 + \frac{\gamma_2 - 1}{\gamma_2}} L^{\frac{(1 - \kappa)
    (\gamma_2 - 1)}{\gamma_2} + 2 - \kappa - \gamma_2} .
  \end{eqnarray*}
  The proof is concluded by taking $\gamma_1 = 1 + \kappa + \delta$ and
  $\gamma_2 = 2 - 2 \kappa$.
\end{proof}

We now try
to obtain a good control on the r.h.s. of
\eqref{regularisedequationbeginning}. 
However, the same treatment as in
Section \ref{sectioninteriorestimate}
is not enough to accomplish this, 
see the following remark.
\refchange{
This makes rigorous the point made in Section 
\ref{sectionstrategyofproof} right below display
\eqref{formulaforgradientsstrategy},
which shows the need of a further treatment
of the error terms, accomplished here by Propositions
\ref{propositiontransportdecomposition} and
\ref{regularisedequationwithtransport}.
}

\begin{remark} \label{remarkwhynaivebalancingdoesnotwork}
Consider any $1 + \kappa < \gamma_1 \leqslant \gamma_2 \leqslant 2 - 2 \kappa$
and take $T_1:=T(\gamma_1)$ in \eqref{conditionforT}.
Recall the definition of $[U]_{\gamma, 0, [a, b]}$ in \eqref{normofUblowup}.
Then, for every $L < \sqrt{T_1}/2$ and $z = (t, x) {\in
D^{T_{\star}}_{T_1}} $, \eqref{noblowupthenblowup} implies that
\begin{equation}
  [U]_{\gamma_2, B (z, L)} \lesssim T_1^{- \frac{\gamma_2}{2}} [U]_{\gamma_2,
  0, [t - T_1, t]} \leqslant \left( \Cdumb{}^{\frac{\gamma_2}{2 (1 -
  \kappa)}} \vee \Cnoise{}^{\frac{\gamma_2}{1 - \kappa}}
  C_{\star}^{\frac{\gamma_2 e (\gamma_1)}{2}} \right) [U]_{\gamma_2, 0, [t -
  T_1, t]}. \label{noblowupthenblowupmotivation}
\end{equation}
If we perform the same estimates
as in Section \ref{sectioninteriorestimate} to try to
 control the r.h.s. of
\eqref{regularisedequationbeginning} we obtain, in view of
\eqref{finalbounddiagonal}, \eqref{mainreductionI},
\eqref{interpolationboundforgradient} and \eqref{linfinityUbyu},
\[ | (\sigma (u) \diamond \noiseb)_L (z) | \lesssim L^{\gamma_2 - 1 - \kappa} [\Pi ; \noise]_{- 1 -
   \kappa} [U]_{\gamma_2, B (z, L)} \| u \|^{\gamma_2 - 1}_{B (z, L)} +
   [\Pi ; \noise]_{- 1 - \kappa} L^{- 1 - \kappa} + l.o.t., \]
where $l.o.t.$ corresponds to lower order terms present in the expansion
\eqref{regularisedequationbeginning}. Now, use \eqref{trickwithlinfinity},
\eqref{interiorestimateimproved} and \eqref{noblowupthenblowupmotivation} to
get
\begin{equation}
  | (\sigma (u) \diamond \noiseb)_L (z) | \lesssim L^{\gamma_2 - 1 - \kappa} \Cnoise{} \left(
  \Cdumb{}^{\frac{\gamma_2}{2 (1 - \kappa)}} \vee
  \Cnoise{}^{\frac{\gamma_2}{1 - \kappa}} C_{\star}^{\frac{\gamma_2 e
  (\gamma_1)}{2}} \right) C^{\gamma_2}_{\star} + \Cnoise{} L^{- 1 - \kappa} +
  l.o.t., \label{leadingorderbeforetransport}
\end{equation}
where the choice of $L > 0$ which balances both leading order terms above
corresponds to
\begin{eqnarray*}
  \Cnoise{} L^{- 1 - \kappa} & = & L^{\gamma_2 - 1 - \kappa} \Cnoise{} \left(
  \Cdumb{}^{\frac{\gamma_2}{2 (1 - \kappa)}} \vee
  \Cnoise{}^{\frac{\gamma_2}{1 - \kappa}} C_{\star}^{\frac{\gamma_2 e
  (\gamma_1)}{2}} \right) C^{\gamma_2}_{\star}\\
  & \Updownarrow & \\
  L = \tilde{L} & \assign & \left( \Cdumb{}^{\frac{1}{2 (1 - \kappa)}}
  \vee \Cnoise{}^{\frac{1}{1 - \kappa}} C_{\star}^{\frac{e (\gamma_1)}{2}}
  \right)^{- 1} C^{- 1}_{\star} .
\end{eqnarray*}
This in turn leads to the final estimate 
\begin{eqnarray}
  | (\sigma (u) \diamond \noiseb)_L (z) | & \lesssim & \Cnoise{} \left( \left( \Cdumb{}
  ^{\frac{1}{2 (1 - \kappa)}} \vee \Cnoise{}^{\frac{1}{1 - \kappa}}
  C_{\star}^{\frac{e (\gamma_1)}{2}} \right)^{- 1} C^{- 1}_{\star} \right)^{-
  1 - \kappa} \nonumber\\
  & = & \Cnoise{} C^{1 + \kappa}_{\star} \left( \Cdumb{}^{\frac{1 +
  \kappa}{2 (1 - \kappa)}} \vee \Cnoise{}^{\frac{1 + \kappa}{1 - \kappa}}
  C_{\star}^{\frac{(\gamma_1 - 1) (1 + \kappa)}{1 - \kappa}} \right) , 
  \label{whytheneedoftransport1}
\end{eqnarray}
which is indeed not good enough, since
\begin{equation}
  \left( 1 + \frac{\gamma_1 - 1}{1 - \kappa} \right) (1 + \kappa) > 1 + \kappa
  > 1 \label{whytheneedoftransport2}
\end{equation}
for any $\gamma_1 > 1 + \kappa$ and $\kappa > 0$. Note in particular that this
is independent of $\gamma_2$. In fact, everything so far holds identically for
any $1 + \kappa < \gamma = \gamma_1 \leqslant 2 - 2 \kappa$ and $T = T
(\gamma)$ satisfying the assumptions of Theorem \ref{theointeriorestimate}.
\end{remark}

The need for the treatment with two different $\gamma_1 < \gamma_2$ in
Corollary \ref{corollaryimprovedinteriorestimate} and the choice in
\eqref{choiceofT1maxprinc} arises in the proof of the
next proposition, which is how we circumvent
\eqref{whytheneedoftransport1} and \eqref{whytheneedoftransport2}, at least
for $0 < \kappa < \bar{\kappa}$ - see also Remark
\ref{remarkofwhyneedimprovedinteriorestimate}.

Write $\nabla u_L$ for the spatial gradient of the (smooth) solution $u_L$ to
\eqref{regularisedequationbeginning}. The following proposition is enough to prove
Theorem \ref{theoremtopostprocessintopolynimial}.

\begin{proposition} \label{propositiontransportdecomposition}
  Let the assumptions of Theorem \ref{theoremtopostprocessintopolynimial}
  be in force. Consider the solution $v:=u_{\tilde{L}}$ to
  \eqref{regularisedequationbeginning} at scale $L=\tilde{L}$ starting from
  $u_{\tilde{L} | t = T_1}$, where
  \begin{equation}
    \tilde{L} := C^{- \frac{2}{3 - 2 \kappa}}_{\star} (\Cdumb{}
     \vee \Cnoise{}^2 C_{\star}^{2 \kappa + 2 \delta})^{- \frac{1}{2 (1 -
    \kappa)}} .  \label{choiceofL}
  \end{equation}
  Then, $\tilde{L}$ satisfies $\tilde{L} < \sqrt{T_1}/2$ and $v:D_{T_1}^{T_\star} \to \R$ solves
  \begin{equation}
    \left\{\begin{array}{rll}
      (\partial_t - \Delta) v & = & b \cdot \nabla v + f\\
      v_{| t = T_1} & = & u_{\tilde{L} | t = T_1}
    \end{array}\right.,
    \label{equationforv}
  \end{equation}
  where $b, f \in C^{\infty} (D_{T_1}^{T_{\star}})$ are bounded and in
  particular $f$ satisfies
  \begin{equation}
  \|f\|_{D_{T_1}^{T_\star}} \lesssim \Cdumb{}^{\frac{1}{1 - \kappa}} C_{\star}^{\betadumb}
   \vee \Cnoise{}^{\frac{2}{1 - \kappa}}
  C_{\star}^{\betanoise}   \label{controlonfRHSofv} .
  \end{equation}
\end{proposition}

\vspace{2mm}
\begin{proof*}{Proof of Theorem \ref{theoremtopostprocessintopolynimial}.}
  For every $t \in [T_1, T_{\star}], x \in \mathbb{T}^d$, let $(X^{t, x}_s)_{s
  \in [T_1, t]}$ be the solution to the following It{\^o}'s SDE
  \[ \left\{\begin{array}{rll}
       \mathd X^{t, x}_s & = & b (t - s, X^{t, x}_s) \mathd s + \sqrt{2} \mathd
       B_s\\
       X^{t, x}_{T_1} & = & x
     \end{array}\right., \]
  where $(B_t)_{t \geqslant 0}$ is a Brownian motion under the probability
  measure $\mathbb{P}$. Then, denoting by $\mathbb{E}_x$ the expectation
  w.r.t. $\mathbb{P}$ conditioned to $X^{t, x}_{T_1} = x$, we get that
  $D_{T_1}^{T_{\star}} \ni (t, x) \mapsto v_t (x)$ given by
  \begin{equation}
    v_t (x) =\mathbb{E}_x [v_{T_1} (X^{t, x}_t)] +  \mathbb{E}_x \left[
    \int_{T_1}^t f (t - s, X^{t, x}_s) \tmop{ds} \right]
    \label{representationformula}
  \end{equation}
  solves \eqref{equationforv}. Therefore, by \eqref{controlonfRHSofv} we
  obtain 
  \begin{eqnarray}
    \| v \|_{D_{T_1}^{T_{\star}}} & = & \sup_{T_1 \leqslant t \leqslant
    T_{\star}} \| v_t \| \nonumber\\
    & \leqslant & \| v_{T_1} \| + (T_{\star} - T_1) \| f
    \|_{D_{T_1}^{T_{\star}}} \nonumber\\
    & \leqslant & \| v_{T_1} \| + C (\kappa, d, C_{\sigma}) \left( \Cdumb{}
    ^{\frac{1}{1 - \kappa}} C^{\betadumb}_{\star} \vee
    \Cnoise{}^{\frac{2}{1 - \kappa}} {C^{\betanoise}_{\star}}  \right) 
    \label{maximumprinciple}
  \end{eqnarray}
  where $\betanoise, \betadumb$ are given by \eqref{betanoiseofdelta}
  and \eqref{betas}. Now, recall that $v = u_{\tilde{L}}$ and that
  $\gamma_2 = 2 - 2 \kappa$, so by \eqref{ineqhL-h},
  \eqref{unormbyUnorm}, \eqref{comparisonofnorms},
  \eqref{reductionofLinfinityofgradient}, 
  we obtain for $u-v$
  \begin{eqnarray*}
    &  & \| u - v \|_{D_{T_1}^{T_{\star}}}\\
    & \leqslant & \sup_{z \in D_{T_1}^{T_{\star}}} [u]_{1 - \kappa, B (z,
    \tilde{L})} \tilde{L}^{1 - \kappa}\\
    & \leqslant & \sup_{z \in D_{T_1}^{T_{\star}}} ([U]_{\gamma_2, B (z,
    \tilde{L})} \tilde{L}^{\gamma_2 - 1 + \kappa} + \| \sigma \| [\Pi ; \lolli]_{1 -
    \kappa} + \| u_X \|_{B (z, \tilde{L})} \tilde{L}^{\kappa}) \tilde{L}^{1 -
    \kappa}\\
    & \lesssim & \sup_{z \in D_{T_1}^{T_{\star}}} ([U]_{\gamma_2, B (z,
    \tilde{L})} \tilde{L}^{\gamma_2} + \Cnoise{} \tilde{L}^{1 - \kappa} + \| u_X
    \|_{B (z, \tilde{L})} \tilde{L})\\
    & \lesssim & C_{\star} (\Cdumb{} \vee \Cnoise{}^2 C_{\star}^{2 \kappa
    + 2 \delta}) \tilde{L}^{2 - 2 \kappa} + \Cnoise{} \tilde{L}^{1 - \kappa}\\
    & + & \left( C_{\star} \tilde{L} + \Cnoise{}^{\frac{1 - 2 \kappa}{2 (1 -
    \kappa)}} \tilde{L}^{\frac{3}{2} - \kappa} C^{\frac{1}{2 (1 -
    \kappa)}}_{\star} \right) (\Cdumb{} \vee \Cnoise{}^2 C_{\star}^{2
    \kappa + 2 \delta})^{\frac{1}{2 (1 - \kappa)}}\\
    & \lesssim & C_{\star}^{- \frac{1 - 2 \kappa}{3 - 2 \kappa}} + \Cnoise{}
    C^{- \frac{2 (1 - \kappa)}{3 - 2 \kappa}}_{\star} \left(
    \Cdumb{}^{\frac{1}{2}} \vee \Cnoise{} C_{\star}^{\kappa + \delta}
    \right)^{- 1}\\
    & + & C^{- \frac{1 - 2 \kappa}{2 (1 - \kappa)}}_{\star} \left( 1 +
    \Cnoise{}^{\frac{1 - 2 \kappa}{2 (1 - \kappa)}} \left. (\Cdumb{} \vee
    \Cnoise{}^2 C_{\star}^{2 \kappa + 2 \delta})^{- \frac{1 - 2 \kappa}{4 (1 -
    \kappa)}} \right) \right.\\
    & \lesssim & 1,
  \end{eqnarray*}
  where the second to last inequality comes from taking $\tilde{L}$ as in
  \eqref{choiceofL}. 
  Thus, we may conclude that
  \begin{equation}
    \| u \|_{D_{T_1}^{T_{\star}}} \leqslant \| v \|_{D_{T_1}^{T_{\star}}} + \|
    u - v \|_{D_{T_1}^{T_{\star}}} \leqslant \| v_{T_1} \| + C (\kappa, d,
    C_{\sigma}) \left( \Cdumb{}^{\frac{1}{1 - \kappa}} C^{\betadumb}
    _{\star} \vee \Cnoise{}^{\frac{2}{1 - \kappa}}
    {C^{\betanoise}_{\star}}  \right) \label{torepeatinthepolynomialgrowth}
  \end{equation}
  from where \eqref{controlawayfrom0} follows by noting that $\| v_{T_1} \| =
  \| u_{\tilde{L}} (T_1, \cdot) \| \leqslant {\| u \|_{D_0^{T_1}}} $.
\end{proof*}

The proof of Proposition \ref{propositiontransportdecomposition} is the main
part of this work and is given in Subsection \ref{proofofmainproposition}.
It is the inspection of a careful (and long) decomposition of the local
description of the singular product $\sigma (u) \xi$ that singles out the terms that contain
generalised gradients $u_X$. As mentioned in Section \ref{sectionstrategyofproof},
there are two main ideas. One of which is the use of the multi-scale formula
\eqref{multiscalereconstructionstrategy} in the proof of Theorem \ref{reconstructiontheorem}.
The other is the relationship between $u_X$ and $\nabla u_L$ in
\eqref{formulaforgradientsstrategy}, which is the content of the next lemma.

\begin{lemma} \label{lemmarelatinggradients}
  For every $z \in D_0^1$ and $L>0$ such that $B(z,L) \subset D_0^1$,
  \refchange{
  there exists an error function $E_z^L$ such that
  }
\label{page_gradientdifference}
  \begin{equation}
    \nabla u_L (z) = u_X (z) + \sigma (u (z)) \langle \Pi_z \lolli, \nabla_x
    \tilde{\varphi}^L_z \rangle + E_z^L, \label{formulaforgradient}
  \end{equation}
  where $\tilde{\varphi}(z) := \varphi(-z)$. Furthermore, we have the 
  estimates
  \begin{equation}
    |\sigma (u (z)) \langle \Pi_z \lolli, \nabla_x \tilde{\varphi}^L_z \rangle|
    \lesssim \Cnoise{} L^{- \kappa} ,
    \label{controllolligradphi}
  \end{equation}
  and
  \begin{equation}
    | E_z^L | 
    \leqslant [U]_{\gamma, B (z, L)} L^{\gamma - 1} ,
    \label{controlonremaindergradient}
  \end{equation}
  for any $1 + \kappa < \gamma \leqslant 2 - 2\kappa$.
\end{lemma}
\begin{proof}
  Recall
  that by \eqref{finitegammanormofU} and \eqref{generalisedgradient}, we may
  write for $\bar{z} \in B(z,L)$,
  \[ u (\bar{z}) = u (z) + \sigma (u (z)) (\Pi_z \lolli) (\bar{z}) + u_X (z) \cdot
     (\Pi_z \X) (\bar{z}) + R_z (\bar{z}), \]
  where the remainder $R_z (\bar{z}) \assign U_z (\bar{z}) - u_X (z) (\Pi_z \X) (\bar{z})$
  satisfies
  \[ | R_z (\bar{z}) | \leqslant [U]_{\gamma, B(z,L)} d(z,\bar{z})^{\gamma} . \]
  Therefore, by \eqref{definitionofuL},
  for $t=(t,x)$
  \begin{eqnarray}
    &  & \nabla u_L (z) \nonumber\\
    & = & \int_D \nabla_x \varphi^L (z - \bar{z}) u (\bar{z}) \mathd \bar{z}
    \nonumber\\
    & = & \int_D \nabla_x \varphi^L (z - \bar{z}) (u (z) + \sigma (u (z)) (\Pi_z
    \lolli) (\bar{z}) + u_X (z) \cdot (\Pi_z \X) (\bar{z}) + R_z (\bar{z})) \mathd
    \bar{z} \nonumber\\
    & = & \sigma (u (z)) \int_D \nabla_x \varphi^L (z - \bar{z}) (\Pi_z \lolli)
    (\bar{z}) \mathd \bar{z}  \label{termwithlollipop}\\
    & + & \int_D \nabla_x \varphi^L (z - \bar{z}) u_X (z) \cdot (\Pi_z \X)
    (\bar{z}) \mathd \bar{z}  \label{termwithgradient}\\
    & + & \int_D \nabla_x \varphi^L (z - \bar{z}) R_z (\bar{z}) \mathd \bar{z},
    \label{termwitherror}
  \end{eqnarray}
  as the first term disappears since 
  $\varphi (t,-x) = \varphi (t,x)$ implies
  $\int_D \nabla_x \varphi^L (z - \bar{z}) \mathd \bar{z} = - \int_D \nabla_x
  \varphi^L (z - \bar{z}) \mathd \bar{z} = 0$. Here $\nabla_x$ is the spatial
  gradient with respect to the variable $x \in \mathbb{T}^d$ in $z = (t, x)$.
  The term in \eqref{termwithlollipop} is equal to $\sigma (u (z)) \langle \Pi_z
  \lolli, \nabla_x \tilde{\varphi}^L_z \rangle$, 
  and it satisfies \eqref{controllolligradphi}, since by \eqref{boundongradphiL}
  \[ | \sigma (u (z)) \langle \Pi_z \lolli, \nabla_x \tilde{\varphi}^L_z \rangle
     | \leqslant \| \sigma \| [\Pi ; \lolli]_{1 - \kappa} L^{- \kappa} \lesssim
     \Cnoise{} L^{- \kappa} . \]
  In coordinates $i = 1, \ldots, d$, \eqref{termwithgradient}
  simplifies to (for $z = (t, x), \bar{z} = (\bar{t}, \bar{x})$)
  \begin{eqnarray*}
    &  & \sum_{j = 1}^d \int_D \partial_{x_i} \varphi^L (z - \bar{z}) u_{X_j}
    (z) (\bar{x}_j - x_j) \mathd \bar{z}\\
    & = & - \sum_{j = 1}^d \int_D \partial_{\bar{x}_i} \varphi^L (z - \bar{z})
    u_{X_j} (z) (\bar{x}_j - x_j) \mathd \bar{z}\\
    & = & \sum_{j = 1}^d \int_D \varphi^L (z - \bar{z}) u_{X_j} (z)
    \partial_{\bar{x}_i} (\bar{x}_j - x_j) \mathd \bar{z}\\
    & = & u_{X_i} (z) \int_D \varphi^L (z - \bar{z}) \mathd \bar{z}\\
    & = & u_{X_i} (z)
  \end{eqnarray*}
  by integration by parts, so $\eqref{termwithgradient} = u_X (z)$.
  By \eqref{boundongradphiL} again, $E_z^L \assign \eqref{termwitherror}$ satisfies
  \begin{equation}
    | E_z^L | = \left| \int_D \nabla_x \varphi^L (z - \bar{z}) R_z (\bar{z})
    \mathd \bar{z} \right| \leqslant [U]_{\gamma, B (z, L)} L^{\gamma - 1} .
  \end{equation}
\end{proof}

\subsection{Proof of Proposition \ref{propositiontransportdecomposition}} \label{proofofmainproposition}

The goal is to use formula \eqref{formulaforgradient} together with
\eqref{onescaletotheother}, \eqref{definitionofreconstruction} and
\eqref{proofofreconstruction} in the proof of the Reconstruction Theorem 
\ref{reconstructiontheorem} to
write equation \eqref{regularisedequationbeginning} as 
\begin{equation}
  (\partial_t - \mathLaplace) u_L (z) = a_L (z) + b_L (z) \cdot \nabla u_L (z)
  + c_L (z), \label{regularisedequation}
\end{equation}
for $z \in D_{T_1}^{T_\star}$. The functions $a_L, b_L, c_L$ above are all smooth and bounded
over $D$. In this decomposition, $b_L$ is the coefficient of the transport
term, $a_L$ contains non-linear terms in $u$ paired with positive powers of
$L$ and $c_L$ contains negative powers of $L$ but all independent from $u$.
For the specially chosen scale $L = \tilde{L}$ in \eqref{choiceofL},
$f := a_{\tilde{L}} + c_{\tilde{L}}$ and $b := b_{\tilde{L}}$ in \eqref{equationforv}.


The counter-effect of decomposition \eqref{regularisedequation} is that the
dominant term in $a_L$ is no longer the one with power $L^{\gamma_2 - 1 -
\kappa}$ in \eqref{leadingorderbeforetransport}, but instead something else
that leads to a choice of $L$ which does not cancel out the exponent
$\gamma_2$ - see \eqref{gamma2doesnotcancelout} in Remark
\ref{remarkofwhyneedimprovedinteriorestimate}. The full potential of
\eqref{interiorestimateimproved} in Corollary
\ref{corollaryimprovedinteriorestimate} is therefore required, and it is
used by taking $\gamma_1$ close to $1 + \kappa$ 
while taking $\gamma_2$ equal to $2 - 2 \kappa$.

The decomposition \eqref{regularisedequation} is obtained in the following
proposition.

\begin{proposition} \label{regularisedequationwithtransport}
  Let the assumptions of Theorem \ref{theoremtopostprocessintopolynimial}
  be in force.
  Then, for
  every $L > 0$ such that $L < \sqrt{T_1} / 2$ 
  and every $z \in D_{T_1}^{T_{\star}}$, the
  regularised equation \eqref{regularisedequationbeginning}
  can be
  rewritten as \eqref{regularisedequation}. Furthermore, the coefficients
  $a_L, b_L, c_L$ are smooth and $a_L$ satisfies
  \begin{eqnarray}
    &  & | a_L (z) | \nonumber\\
    & \lesssim & 
    (\Cnoise{} L^{1 - 3 \kappa} + (\Cnoise{}^2 + \Cdumb{})
    L^{2 - 4 \kappa} + (\Cnoise{}^2 + \Cdumb{}) \Cnoise{} L^{3 - 5 \kappa}) \nonumber \\
    & & \times C_{\star} (\Cdumb{} \vee \Cnoise{}^2 C_{\star}^{2 \kappa + 2 \delta}) 
    \label{controlALdominant1} \\
    & + & (\Cnoise{} L^{3 - 5 \kappa} + \Cnoise{}^2 L^{4 - 6 \kappa}) C^2_{\star}
    (\Cdumb{} \vee \Cnoise{}^2 C_{\star}^{2 \kappa + 2 \delta})^2 
    \label{controlALdominant2}\\
    & + & (1 + \Cnoise{} L^{1 - \kappa}) \Cnoise{} L^{2 - 3 \kappa} C^2_{\star}
    (\Cdumb{} \vee \Cnoise{}^2 C_{\star}^{2 \kappa + 2 \delta})^{1 +
    \frac{1}{2 (1 - \kappa)}}  \label{controlALdominant3}\\
    & + & (\Cnoise{} + (\Cnoise{}^2 + \Cdumb{}) L^{1 - \kappa}) \Cnoise{}
    L^{1 - 2 \kappa} C_{\star} (\Cdumb{} \vee \Cnoise{}^2 C_{\star}^{2
    \kappa + 2 \delta})^{\frac{1}{2 (1 - \kappa)}}  \label{controlALmixed1}\\
    & + & \hspace{-1.7mm} \begin{array}{l}
    (\Cnoise{} + (\Cnoise{}^2 + \Cdumb{}) L^{1 - \kappa}) \Cnoise{}^{1
    + \frac{1 - 2 \kappa}{2 (1 - \kappa)}} L^{\frac{3}{2} - 3 \kappa}\\
    \times C^{\frac{1}{2 (1 - \kappa)}}_{\star} (\Cdumb{} \vee \Cnoise{}^2
    C_{\star}^{2 \kappa + 2 \delta})^{\frac{1}{2 (1 - \kappa)}}
    \end{array}
    \label{controlALmixed2}\\
    & + & \hspace{-1.7mm} \begin{array}{l}
    \displaystyle (1 + \Cnoise{} L^{1 - \kappa}) \Cnoise{}^{1 + \frac{1 - 2 \kappa}{2 (1
    - \kappa)}} L^{\frac{5}{2} - 4 \kappa} C^{1 + \frac{1}{2 (1 -
    \kappa)}}_{\star} (\Cdumb{} \vee \Cnoise{}^2 C_{\star}^{2 \kappa + 2
    \delta})^{1 + \frac{1}{2 (1 - \kappa)}} 
    \end{array}
    \label{controlALmixed4}\\
    & + &
    \left( \Cnoise{}^2 L^{2 - 2 \kappa} C^2_{\star} + \Cnoise{}^{2 +
    \frac{1 - 2 \kappa}{2 (1 - \kappa)}} L^{\frac{5}{2} - 3 \kappa} C^{1 +
    \frac{1}{2 (1 - \kappa)}}_{\star} \right) \nonumber \\
    & & \times (\Cdumb{} \vee \Cnoise{}^2
    C_{\star}^{2 \kappa + 2 \delta})^{\frac{1}{1 - \kappa}} 
    \label{controlALmixed5} \\
    & + & \Cnoise{}^{2 + \frac{1 - 2 \kappa}{1 - \kappa}} L^{3 - 4 \kappa}
    C^{\frac{1}{1 - \kappa}}_{\star} (\Cdumb{} \vee \Cnoise{}^2
    C_{\star}^{2 \kappa + 2 \delta})^{\frac{1}{1 - \kappa}} 
    \label{controlALmixed3}\\
    & + & ((\Cnoise{}^2 + \Cdumb{}) + (\Cnoise{}^2 + \Cdumb{})
    \Cnoise{} L^{1 - \kappa}) \Cnoise{} L^{1 - 3 \kappa}  \label{controlALnoise} ,
  \end{eqnarray}
  while $c_L$ satisfies
  \begin{equation}
    | c_L (z) |
    \lesssim [\Cnoise{} + (\Cdumb{} + \Cnoise{}^2) L^{1 - \kappa}] L^{-
    1 - \kappa} ,
    \label{controlonCL}
  \end{equation}
  both uniformly over $z \in D_{T_1}^{T_{\star}}$. The inequalities with $\lesssim$
  hold up to a constant which only depends on $d, \kappa$ and $C_{\sigma}$.
  The constants $\Cnoise{} = \Cnoise{,1}, \Cdumb{} = \Cdumb{,1} > 0$
  are given by \eqref{assumptiononstochasticobjects} and $C_{\star}$ by
  \eqref{definitionstoppingtime}-\eqref{trickwithlinfinity}.
\end{proposition}

\begin{remark}
  The control on $b_L$ is not necessary for our analysis. Nevertheless, it holds that
  \begin{eqnarray}
    &  & | b_L (z) | \nonumber\\
    & \lesssim & \Cnoise{} L^{2 - 3 \kappa} C_{\star} (\Cdumb{} \vee
    \Cnoise{}^2 C_{\star}^{2 \kappa + 2 \delta})  \label{controlonBL3}\\
    & + & \left( \Cnoise{} L^{1 - \kappa} C_{\star} + \Cnoise{}^{1 + \frac{1 - 2
    \kappa}{2 (1 - \kappa)}} L^{\frac{3}{2} - 2 \kappa} C^{\frac{1}{2 (1 -
    \kappa)}}_{\star} \right) (\Cdumb{} \vee \Cnoise{}^2 C_{\star}^{2
    \kappa + 2 \delta})^{\frac{1}{2 (1 - \kappa)}}  \label{controlonBL2}\\
    & + & (\Cnoise{} + (\Cdumb{} + \Cnoise{}^2) L^{1 - \kappa}) L^{-
    \kappa} . \label{controlonBL}
  \end{eqnarray}
\end{remark}


Before proving Proposition \ref{regularisedequationwithtransport},
we first show how it gives Proposition
\ref{propositiontransportdecomposition}.

\vspace{2mm} \label{pageref:proof_of_prop2}
\begin{proof*}{Proof of Proposition \ref{propositiontransportdecomposition}.}
  We choose $L$ to balance
  expressions \eqref{controlALdominant1}-\eqref{controlALnoise} and
  \eqref{controlonCL}. For that, 
  we consider
  this choice for each term in \eqref{controlALdominant1}-\eqref{controlALnoise}
  separately and pick the worst at the end. We illustrate this here for the
  first three terms in \eqref{controlALdominant1}-\eqref{controlALdominant3},
  which are the dominant ones, and refer to Appendix \ref{AppendixB} for a
  detailed discussion for the terms in
  \eqref{controlALmixed1}-\eqref{controlALnoise}. The choice for
  \eqref{controlALdominant1} is
  \begin{eqnarray*}
    \Cnoise{} L^{- 1 - \kappa} & = & \Cnoise{} L^{1 - 3 \kappa} C_{\star} (\Cdumb{}
     \vee \Cnoise{}^2 C_{\star}^{2 \kappa + 2 \delta})\\
    & \Updownarrow & \\
    L = \tilde{L}_1 & \assign & (C_{\star} (\Cdumb{} \vee \Cnoise{}^2
    C_{\star}^{2 \kappa + 2 \delta}))^{- \frac{1}{2 (1 - \kappa)}},
  \end{eqnarray*}
  while the choice for the term in \eqref{controlALdominant2} is
  \begin{eqnarray*}
    \Cnoise{} L^{- 1 - \kappa} & = & \Cnoise{} L^{3 - 5 \kappa} C^2_{\star}
    (\Cdumb{} \vee \Cnoise{}^2 C_{\star}^{2 \kappa + 2 \delta})^2\\
    & \Updownarrow & \\
    L = \tilde{L}_2 & \assign & (C_{\star} (\Cdumb{} \vee \Cnoise{}^2
    C_{\star}^{2 \kappa + 2 \delta}))^{- \frac{1}{2 (1 - \kappa)}},
  \end{eqnarray*}
  which satisfies $\tilde{L}_1 = \tilde{L}_2$. The choice for the term in
  \eqref{controlALdominant3} is however
  \begin{eqnarray*}
    \Cnoise{} L^{- 1 - \kappa} & = & \Cnoise{} L^{2 - 3 \kappa} C^2_{\star}
    (\Cdumb{} \vee \Cnoise{}^2 C_{\star}^{2 \kappa + 2 \delta})^{1 +
    \frac{1}{2 (1 - \kappa)}} \nonumber\\
    & \Updownarrow &  \nonumber\\
    L = \tilde{L}_3 & \assign & C^{- \frac{2}{3 - 2 \kappa}}_{\star} (\Cdumb{}
    \vee \Cnoise{}^2 C_{\star}^{2 \kappa + 2 \delta})^{- \frac{1}{2 (1 -
    \kappa)}} .  
  \end{eqnarray*}
  Since the final bound for both $a_L$ and $c_L$ is a negative power of this
  choice, we conclude that the dominant term corresponds to the choice of
  $\tilde{L} = \tilde{L}_3$, which is \eqref{choiceofL}.
  In fact, comparing the exponents of $C_{\star}$ in
  $\tilde{L}_1 = \tilde{L}_2$ and $\tilde{L}_3$, this is equivalent to
  \[ \frac{1}{2 - 2 \kappa} \leqslant \frac{2}{3 - 2 \kappa} \Longleftrightarrow
     3 - 2 \kappa \leqslant 4 - 4 \kappa \Longleftrightarrow 2 \kappa \leqslant
     1,
   \]
  which is true for our regime $0 < \kappa < 1 / 3$. 
  Condition $\tilde{L} < \sqrt{T_1} / 2$ is respected for this choice, since
  $C_{\star} > 1$ implies that (recall $T_1$ in \eqref{choiceofT1maxprinc})
  \[ C^{- \frac{2}{3 - 2 \kappa}}_{\star}
    (\Cdumb{} \vee \Cnoise{}^2 C_{\star}^{2 \kappa + 2 \delta})^{-
     \frac{1}{2 (1 - \kappa)}} \leqslant \Cdumb{}^{- \frac{1}{2 (1 - \kappa)}}
     \vee \Cnoise{}^{- \frac{1}{1 - \kappa}} C_{\star}^{- \frac{\kappa + \delta}{1
     - \kappa}} \sim \sqrt{T_1} .
  \]

  As a consequence of Proposition \ref{regularisedequationwithtransport} 
  we obtain that $v \assign
  u_{\tilde{L}}$ solves \eqref{equationforv},
  where we identify 
  $b := b_{\tilde{L}}$ and  $f := a_{\tilde{L}} + c_{\tilde{L}}$. Moreover,
  \begin{eqnarray}
    \| f \|_{D_{T_1}^{T_{\star}}} 
    & \lesssim & \Cnoise{} {\tilde L}^{-(1+\kappa)}
    \nonumber \\ & = & 
    \Cnoise{} \left( C^{- \frac{2}{3 -
    2 \kappa}}_{\star} (\Cdumb{} \vee \Cnoise{}^2 C_{\star}^{2 \kappa + 2
    \delta})^{- \frac{1}{2 (1 - \kappa)}} \right)^{- (1 + \kappa)} \nonumber\\
    & = & \Cnoise{} C^{\frac{2 (1 + \kappa)}{3 - 2 \kappa}}_{\star} (\Cdumb{}
     \vee \Cnoise{}^2 C_{\star}^{2 \kappa + 2 \delta})^{\frac{1 + \kappa}{2 (1
    - \kappa)}}  \label{controlofRHSofvbeforeyoung}\\
    & = & \Cnoise{} \Cdumb{}^{\frac{1 + \kappa}{2 (1 - \kappa)}} C^{\frac{2
    (1 + \kappa)}{3 - 2 \kappa}}_{\star} \vee \Cnoise{}^{\frac{2}{1 - \kappa}}
    {C^{\frac{2 (1 + \kappa)}{3 - 2 \kappa} + \frac{(1 + \kappa) (\kappa +
    \delta)}{1 - \kappa}}_{\star}}  \nonumber\\
    & \lesssim & \Cdumb{}^{\frac{1}{1 - \kappa}} C^{\frac{2 (1 +
    \kappa)}{3 - 2 \kappa}}_{\star} \vee \Cnoise{}^{\frac{2}{1 - \kappa}}
    {C^{\frac{2 (1 + \kappa)}{3 - 2 \kappa} + \frac{(1 + \kappa) (\kappa +
    \delta)}{1 - \kappa}}_{\star}} ,   
  \end{eqnarray}
  where the last inequality is a consequence of Young's inequality applied as
  \[ \Cnoise{} \Cdumb{}^{\frac{1 + \kappa}{2 (1 - \kappa)}} \leqslant
     \frac{1 - \kappa}{2} \Cnoise{}^{\frac{2}{1 - \kappa}} + \frac{1 + \kappa}{2}
     \Cdumb{}^{\frac{1}{1 - \kappa}} . \]
  This gives \eqref{controlonfRHSofv}.
  
  Since we treat $b \cdot \nabla v$
  in \eqref{equationforv} as a transport term, the precise control
  on $\| b \|_{D_{T_1}^{T_{\star}}}$ is irrelevant. Nonetheless, it corresponds
  to \eqref{controlonBL3}-\eqref{controlonBL} with the choice of $\tilde{L}$
  in \eqref{choiceofL}. 
\end{proof*}

\begin{remark}
  \label{remarkofwhyneedimprovedinteriorestimate}
  We highlight the need
  of Corollary \ref{corollaryimprovedinteriorestimate} in view of
  \eqref{equationforkappa} and the condition $\betanoise < 1$. For simplicity,
  consider the case where $\Cdumb{} \leqslant \Cnoise{}^2 C_{\star}^{2 \kappa + 2 \delta}$
  and ignore the constant $\Cnoise{}$.  
  If we write the
  equivalent of the control in \eqref{controlALdominant3} in terms of $1 +
  \kappa < \gamma_1 \leqslant \gamma_2 \leqslant 2 - 2 \kappa$ and $T = T
  (\gamma_1)$ satisfying the assumptions of Corollary
  \ref{corollaryimprovedinteriorestimate}, we obtain
  \[ \eqref{controlALdominant3} \lesssim L^{\gamma_2 - \kappa} C_{\star}^{2 +
     \frac{\gamma_2 + 1}{2} e (\gamma_1)} = L^{\gamma_2 - \kappa} C_{\star}^{2
     + \frac{(\gamma_2 + 1) (\gamma_1 - 1)}{1 - \kappa}} . \]
  Then, the choice of $L$ corresponds to
  \begin{eqnarray}
    L^{\gamma_2 - \kappa} C_{\star}^{2 + \frac{(\gamma_2 + 1) (\gamma_1 -
    1)}{1 - \kappa}} & = & L^{- 1 - \kappa} \nonumber\\
    & \Updownarrow &  \nonumber\\
    L = \tilde{L}_{\gamma} & \assign & C_{\star}^{- \left( \frac{2}{\gamma_2 +
    1} + \frac{(\gamma_1 - 1)}{1 - \kappa} \right)}, 
    \label{gamma2doesnotcancelout}
  \end{eqnarray}
  which leads to (in view of \eqref{controlonfRHSofv})
  \[ \| f \|_{D_{T_1}^{T_{\star}}} \lesssim \left( C_{\star}^{- \left(
     \frac{2}{\gamma + 1} + \frac{(\gamma - 1)}{1 - \kappa} \right)}
     \right)^{- (1 + \kappa)} = C_{\star}^{\frac{2 + 2 \kappa}{\gamma_2 + 1} +
     \frac{(\gamma_1 - 1) (1 + \kappa)}{1 - \kappa}} . \]
  We then want to say that
  \begin{equation}
    \frac{2 + 2 \kappa}{\gamma_2 + 1} + \frac{(\gamma_1 - 1) (1 + \kappa)}{1 -
    \kappa} < 1, \label{exponentoffinalcontrolforg1g2}
  \end{equation}
  for which we need at the same time $\gamma_2$ to be as large as possible to
  control the first term in the l.h.s. above and $\gamma_1$ to be as close to
  $1$ as possible to control the second term. If $\gamma_1 = \gamma_2 =
  \gamma$, then condition \eqref{exponentoffinalcontrolforg1g2} is never
  satisfied for any $\kappa > 0$ and $1 + \kappa < \gamma \leqslant 2 - 2
  \kappa$, since for $\kappa = 0$ we see that
  \[ \frac{2}{\gamma + 1} + \gamma - 1 < 1 \Longleftrightarrow 2 + \gamma^2 <
     \gamma + 2 \Longleftrightarrow \gamma < 1. \]
\end{remark}

Moving to Proposition \ref{regularisedequationwithtransport}, we focus on the r.h.s. of \eqref{regularisedequationbeginning}.
Recall from
\eqref{localapproxG} in Section \ref{sectionReconstruction} that
\[ G_z = \sigma (u (z)) \Pi_z \noise + \sigma' \sigma (u (z)) \Pi_z \dumb +
   \sigma' (u (z)) u_X (z) \cdot \Pi_z \Xnoise \]
and expand the r.h.s. of \eqref{regularisedequationbeginning} as $\langle
\sigma (u) \diamond \noiseb, \varphi_z^L \rangle \pm \langle G_z, \varphi_z^L \rangle$. By setting
\label{page_reconstructionterms}
\begin{eqnarray*}
  \mathcal{R}_L (z) & \assign & \langle \sigma (u) \diamond \noiseb - G_z, \varphi_z^L
  \rangle\\
  \tilde{B}_L (z) & \assign & \sigma' (u (z)) u_X (z) \cdot \langle \Pi_z \Xnoise
  , \varphi_z^L \rangle,\\
  \Xi_L (z) & \assign & \sigma' \sigma (u (z)) \langle \Pi_z \dumb,
  \varphi_z^L \rangle + \sigma (u (z)) \langle \Pi_z \noise, \varphi_z^L \rangle,
\end{eqnarray*}
we get
\[ (\partial_t - \Delta) u_L (z) =\mathcal{R}_L (z) + \tilde{B}_L (z) + \Xi_L
   (z) . \] 
\refchange{
At this point, we need to extract from the terms $\mathcal{R}_L$ and $\tilde{B}_L$ the worst super-linear error terms into transport terms. While for
$\tilde{B}_L$ this is achieved directly with the
use of \eqref{formulaforgradient}, the 
reconstruction error $\mathcal{R}_L$
requires further treatment:
we track down the gradient terms,
localise them at the base point and convert
$u_X(z)$ into $\nabla u_L(z)$. The term
$\Xi_L$ is controlled by
model norms and negative powers of $L$ and no
further treatment is required. 
}

$\mathcal{R}_L$ 
can be written as $\mathcal{R}_L (z) = \langle \mathcal{R}G - G_z, \varphi_z^L \rangle$,
where $\mathcal{R}G$ denotes the reconstruction of 
$G = \sigma (U) \noise$, defined
as the limit \eqref{definitionofreconstruction} in Theorem
\ref{reconstructiontheorem}. 
Therefore, 
$\mathcal{R}_L (z) = \lim_{N \to \infty} \Lambda_{N, L} [G ] (z)$, where 
\label{page_LambdaNL}
\begin{equation*}
  \Lambda_{N, L} [G ] (z) \assign
  \left\langle \left\langle
   G_{\cdot}, \varphi_{\cdot}^{\frac{L}{2^N}} \right\rangle, \varphi^{L, N}_z
   \right\rangle - \nobracket \langle G_z \nobracket, \varphi^L_z \rangle ,
\end{equation*}
where by \eqref{proofofreconstruction}, for every $N \in \mathbb{N}$,
\begin{equation}
  \Lambda_{N, L} [G ] (z) = 
  \sum_{n = 0}^{N-1} \int \int \left\langle G_{z_2} - G_{z_1},
  \varphi_{z_2}^{\frac{L}{2^{n + 1}}} \right\rangle
  \psi^{\frac{L}{2^{n + 1}}}_{z_1}
  (z_2) \varphi^{L, n}_z (z_1) \mathd z_2 \mathd z_1 ,
  \label{sumofdifferentscales}
\end{equation}
cf. \eqref{multiscalereconstructionstrategy} in Section \ref{sectionstrategyofproof}.
We recall that
\begin{equation}
  G_{z_2} - G_{z_1} = A_{z_1}^{\noise} (z_2) \Pi_{z_2} \noise + A_{z_1}^{\dumb} (z_2)
  \Pi_{z_2} \dumb + A_{z_1}^{\Xnoise} (z_2) \cdot \Pi_{z_2} \Xnoise,
  \label{reductionofG}
\end{equation}
where
\begin{eqnarray*}
  A_{z_1}^{\noise} (z_2) & \assign & \sigma (u (z_2)) - \sigma (u (z_1)) -
  \sigma' \sigma (u (z_1)) (\Pi_{z_1} \lolli) (z_2)\\
  & - & \sigma' (u (z_1)) u_X (z_1) \cdot (\Pi_{z_1} \X) (z_2)\\
  A_{z_1}^{\dumb} (z_2) & \assign & \sigma' \sigma (u (z_2)) - \sigma'
  \sigma (u (z_1))\\
  A_{z_1}^{\Xnoise} (z_2) & \assign & \sigma' (u (z_2)) u_X (z_2) - \sigma' (u
  (z_1)) u_X (z_1) .
\end{eqnarray*}
\refchange{The goal is then to treat each of
the terms in}
\[ \mathcal{R}_L (z) = \lim_{N \rightarrow \infty} \Lambda_{N, L} [G ] (z) =
   \lim_{N \rightarrow \infty} [\Lambda_{N, L} [A^{\noise}] (z) + \Lambda_{N, L}
   [A^{\dumb}] (z) + \Lambda_{N, L} [A^{\Xnoise}] (z)] \]
separately, similarly as in Section \ref{reduction}. The difference
is that now we only write equalities, 
making sure that each term has the correct order bound to
make the sum in \eqref{sumofdifferentscales} well defined as $N \rightarrow
\infty$. We separate the analysis of \eqref{reductionofG}
into 
Lemmas \ref{splittinglemma1}, \ref{splittinglemma2} and \ref{splittinglemma3}.
Their proofs are very similar and the proof
of the first contains all the key ingredients, so we prove the first and postpone the proofs of the other two to
Appendix \ref{proofofdecompositionlemmas}. We will make repeated use of
Lemmas \ref{lemmaofblowups} and \ref{lemmarelatinggradients} throughout this
section.

\begin{lemma}
  \label{splittinglemma1}Let the assumptions of Proposition
  \ref{regularisedequationwithtransport} be in force. Then, it holds that
  \[ \lim_{N \rightarrow \infty} \Lambda_{N, L} [A^{\noise} \Pi \noise] (z) = a^{\noise}_L
     (z) + b^{\noise}_L (z) \cdot \nabla u_L (z), \]
  where the functions $a^{\noise}_L, b^{\noise}_L$ are smooth and the convergence is
  uniform over $z \in D_{T_1}^{T_{\star}}$. Moreover, $a^{\noise}_L$
  satisfies \eqref{controlALdominant1}-\eqref{controlALnoise} and $b^{\noise}_L$
  satisfies \eqref{controlonBL2}-\eqref{controlonBL}.
\end{lemma}

\begin{proof}
  Recall that for any function $H_{z_1} (z_2)$ such that $| H_{z_1} (z_2) |
  \leqslant [H]_{\lambda} d(z_1,z_2)^{\lambda}$, we have by
  \refchange{Theorem \ref{reconstructiontheorem}
  (cf. \eqref{sumofdifferentscales})}
  that the limit $\Lambda_L [H  \Pi \noise] (z) \assign
  \lim_{N \rightarrow \infty} \Lambda_{N, L} [H \Pi \noise] (z) $ is well defined and
  it satisfies
  \[ | \Lambda_L [H \Pi \noise] (z) | \leqslant [\Pi ; \noise]_{- 1 - \kappa} [H]_{\lambda, B
     (z, L)} L^{\lambda - 1 - \kappa} \]
  as long as $\lambda - 1 - \kappa > 0$. We start by expanding
  $A_{z_1}^{\noise} (z_2) = A^{\noise, 1}_{z_1} (z_2) + Q^{\noise}_{z_1} (z_2)$, where,
  for
  \[ \eta_{z_1} (z_2) : = \sigma (u (z_1)) (\Pi_{z_1} \lolli) (z_2) + u_X (z_1)
     \cdot (\Pi_{z_1} \X) (z_2), \]
  we let
  \begin{eqnarray*}
    A_{z_1}^{\noise, 1} (z_2) & = & \sigma (u (z_2)) - \sigma (u (z_1) +
    \eta_{z_1} (z_2))\\
    Q^{\noise}_{z_1} (z_2) & = & \sigma (u (z_1) + \eta_{z_1} (z_2)) - \sigma (u
    (z_1)) - \sigma' (u (z_1)) \eta_{z_1} (z_2) .
  \end{eqnarray*}
  For $A_{z_1}^{\noise, 1} (z_2)$, we obtain
  \begin{eqnarray*}
    A_{z_1}^{\noise, 1} (z_2) & = & \int_0^1 \frac{\mathd}{\mathd \lambda} \sigma
    (u (z_2) \lambda + (1 - \lambda) \sigma (u (z_1) + \eta_{z_1} (z_2)))
    \mathd \lambda\\
    & = & 
     (U_{z_1} (z_2) - u_X (z_1) \cdot (\Pi_{z_1} \X) (z_2))\\
     & & \times \int_0^1 \sigma' (u (z_2) \lambda + (1 - \lambda) (u (z_1) +
      \eta_{z_1} (z_2))) \mathd \lambda \\
    & \backassign & I_1 (\sigma', z_1, z_2) (U_{z_1} (z_2) - u_X (z_1) \cdot
    (\Pi_{z_1} \X) (z_2)),
  \end{eqnarray*}
  which satisfies
  \[ | A_{z_1}^{\noise, 1} (z_2) | \leqslant \| \sigma' \| [U]_{\gamma_2, B (z,
     L)} L^{\gamma_2} . \]
  In turn, $a^{\noise, 1}_L (z) \assign \Lambda_L [A_{z_1}^{\noise, 1} (z_2) \Pi \noise]
  (z)$ is well defined and it satisfies \eqref{controlALdominant1}, since by
  \eqref{comparisonofnorms}
  \[ | a^{\noise, 1}_L (z) | \leqslant \| \sigma' \| [\Pi ; \noise]_{- 1 - \kappa}
     [U]_{\gamma_2, B (z, L)} L^{\gamma_2 - 1 - \kappa} \lesssim \Cnoise{} L^{1
     - 3 \kappa} C_{\star} (\Cdumb{} \vee \Cnoise{}^2 C_{\star}^{2 \kappa
     + 2 \delta}) . \]
  To treat $Q^{\noise}_{z_1} (z_2)$, we proceed as follows
  \begin{eqnarray*}
    Q^{\noise}_{z_1} (z_2) & = & \int_0^1 \frac{\mathd}{\mathd \lambda} \sigma (u
    (z_1) + \lambda \eta_{z_1} (z_2)) \mathd \lambda - \sigma' (u (z_1))
    \eta_{z_1} (z_2)\\
    & = & \eta_{z_1} (z_2) \int_0^1 \sigma' (u (z_1) + \lambda \eta_{z_1}
    (z_2)) - \sigma' (u (z_1)) \mathd \lambda\\
    & = & \eta_{z_1} (z_2) \int_0^1 \int_0^1 \frac{\mathd}{\mathd
    \bar{\lambda}} \sigma' (u (z_1) + \bar{\lambda} \lambda \eta_{z_1} (z_2))
    \mathd \bar{\lambda} \mathd \lambda\\
    & = & \eta_{z_1} (z_2)^2 \int_0^1 \int_0^1 \lambda \sigma'' (u (z_1) +
    \bar{\lambda} \lambda \eta_{z_1} (z_2)) \mathd \bar{\lambda} \mathd
    \lambda\\
    & = : & \eta_{z_1} (z_2)^2 I_2 (\sigma'', z_1, z_2) .
  \end{eqnarray*}
  Adding a dependence on the point $z$ to $A^{\noise, 3}$, $A^{\noise, 4},
  \tilde{B}^{\noise, 1}$ and $\tilde{B}^{\noise, 2}$, we write $\eta_{z_1} (z_2)^2
  I_2 (\sigma'', z_1, z_2) = A^{\noise, 2}_{z_1} (z_2) + \tilde{B}^{\noise, 1}_{z_1}
  (z ; z_2) + A^{\noise, 3}_{z_1} (z ; z_2) + \tilde{B}^{\noise, 2}_{z_1} (z ; z_2)
  + A^{\noise, 4}_{z_1} (z ; z_2)$, where
  \begin{eqnarray*}
    A^{\noise, 2}_{z_1} (z_2) & = & I_2 (\sigma'', z_1, z_2) (\sigma (u (z_1))
    (\Pi_{z_1} \lolli) (z_2))^2\\
    \tilde{B}^{\noise, 1}_{z_1} (z ; z_2) & = & I_2 (\sigma'', z_1, z_2) [\sigma
    (u (z_1)) (\Pi_{z_1} \lolli) (z_2) (\Pi_{z_1} \X) (z_2)] \cdot u_X (z )\\
    A^{\noise, 3}_{z_1} (z ; z_2) & = & I_2 (\sigma'', z_1, z_2) \sigma (u (z_1))
    (\Pi_{z_1} \lolli) (z_2) (\Pi_{z_1} \X) (z_2) \cdot \left( u_X \left( {z_1} 
    \right) - u_X (z ) \right)\\
    \tilde{B}^{\noise, 2}_{z_1} (z ; z_2) & = & I_2 (\sigma'', z_1, z_2) [u_X
    (z_1) \cdot (\Pi_{z_1} \X) (z_2) (\Pi_{z_1} \X) (z_2)] \cdot u_X (z)\\
    A^{\noise, 4}_{z_1} (z ; z_2) & = & I_2 (\sigma'', z_1, z_2) u_X (z_1) \cdot
    (\Pi_{z_1} \X) (z_2) (\Pi_{z_1} \X) (z_2) \cdot (u_X (z_1) - u_X (z)) .
  \end{eqnarray*}
  This is to make the gradient $\nabla u_L (z)$ at the correct base point $z$
  appear later. We have that
  \[ | A^{\noise, 2}_{z_1} (z_2) | \leqslant \| \sigma'' \| \| \sigma \|^2
     [\Pi ; \lolli]_{1 - \kappa}^2 L^{2 - 2 \kappa}, \]
  which gives rise to $A^{\noise, 2}_L (z) \assign \Lambda_L [A_{z_1}^{\noise, 2}
  (z_2) \Pi \noise] (z)$ satisfying \eqref{controlALnoise}, since
  \[ | A^{\noise, 2}_L (z) | \leqslant | A^{\noise, 2}_{z_1} (z_2) | [\Pi ; \noise]_{- 1 -
     \kappa} L^{- 1 - \kappa} \lesssim \Cnoise{}^3 L^{1 - 3 \kappa} . \]
  The terms $A^{\noise, 3}_{z_1} (z ; z_2)$ and $A^{\noise, 4}_{z_1} (z ; z_2)$ need
  to be treated slightly differently, since they depend explicitly on $z$.
  This treatment is going to be recurrent in this section: note that for every
  $n \in \mathbb{N}$,
  \begin{eqnarray*}
    &  & \int \int | A^{\noise, 3}_{z_1} (z ; z_2) | \psi^{\frac{L}{2^{n + 1}}}
    (z_2 - z_1) \varphi^{L, n} (z_1 - z) \mathd z_2 \mathd z_1\\
    & \leqslant &
      \| \sigma'' \| \interleave \sigma \interleave [\Pi ; \lolli]_{1 - \kappa}
      [u_X]_{\gamma_2 - 1, B (z, L)} \\
      & & \times \int \int d(z_1,z_2)^{2 - \kappa} d(z,z_1)^{\gamma_2 - 1}
      \psi^{\frac{L}{2^{n + 1}}} (z_2 - z_1) \varphi^{L, n} (z_1 - z) \mathd
      z_2 \mathd z_1 \\
    & \lesssim & \Cnoise{} [u_X]_{\gamma_2 - 1, B (z, L)} L^{\gamma_2 - 1}
    \left( \frac{L}{2^{n + 1}} \right)^{2 - \kappa},
  \end{eqnarray*}
  so that for every $N \in \mathbb{N}$,
  \begin{eqnarray*}
    | \Lambda_{N, L} [A^{\noise, 3} \Pi \noise] (z) | & \lesssim & \Cnoise{}
    [u_X]_{\gamma_2 - 1, B (z, L)} [\Pi ; \noise]_{- 1 - \kappa} L^{\gamma_{2 - 1}}
    \sum_{n = 0}^{N - 1} \left( \frac{L}{2^{n + 1}} \right)^{1 - 2 \kappa}\\
    & \lesssim & \Cnoise{}^2 [u_X]_{\gamma_2 - 1, B (z, L)} L^{\gamma_{2 - 2
    \kappa}} .
  \end{eqnarray*}
  This implies that $A^{\noise, 3}_L (z) = \Lambda_L [A^{\noise, 3} \Pi \noise] (z)$
  exists. Moreover, we see that $A^{\noise, 3}_L$ satisfies
  \eqref{controlALdominant1}, \eqref{controlALmixed1}, \eqref{controlALmixed2}
  and \eqref{controlALnoise}, since by \eqref{reductionofgammaofgradient},
  \begin{eqnarray*}
    | A^{\noise, 3}_L (z) | & \lesssim & (\Cnoise{}^2 L^{2 - 4 \kappa} + \Cnoise{}^3
    L^{3 - 5 \kappa}) C_{\star} (\Cdumb{} \vee \Cnoise{}^2 C_{\star}^{2
    \kappa + 2 \delta}) + \Cnoise{}^4 L^{2 - 4 \kappa}\\
    & + & \Cnoise{}^2 L^{2 - 3 \kappa} C_{\star} (\Cdumb{} \vee \Cnoise{}^2
    C_{\star}^{2 \kappa + 2 \delta})^{\frac{1}{2 (1 - \kappa)}}\\
    & + & \Cnoise{}^{3 + \frac{1 - 2 \kappa}{2 (1 - \kappa)}} L^{\frac{5}{2} -
    4 \kappa} C^{\frac{1}{2 (1 - \kappa)}}_{\star} (\Cdumb{} \vee
    \Cnoise{}^2 C_{\star}^{2 \kappa + 2 \delta})^{\frac{1}{2 (1 - \kappa)}} .
  \end{eqnarray*}
  We treat $A^{\noise, 4}_{z_1} (z ; z_2)$ in the exact same fashion, and the
  same calculations give rise to the well defined limit $A^{\noise, 4}_L (z) =
  \Lambda_L [A^{\noise, 4} \Pi \noise] (z)$, where it satisfies
  \eqref{controlALdominant3}, \eqref{controlALmixed1},
  \eqref{controlALmixed2}, \eqref{controlALmixed4}, \eqref{controlALmixed5}
  and \eqref{controlALmixed3}, since by \eqref{reductionofLinfinityofgradient}
  and \eqref{reductionofgammaofgradient}
  \begin{eqnarray*}
    &  & | A^{\noise, 4}_L (z) |\\
    & \leqslant & \| \sigma'' \| \| u_X \|_{B (z, L)} [u_X]_{\gamma_2 - 1, B
    (z, L)} [\Pi ; \noise]_{- 1 - \kappa} L^{\gamma_2 - \kappa}\\
    & \lesssim &
      (1 + \Cnoise{} L^{1 - \kappa}) \left( \Cnoise{} L^{2 - 3 \kappa} C^2_{\star}
      + \Cnoise{}^{1 + \frac{1 - 2 \kappa}{2 (1 - \kappa)}} L^{\frac{5}{2} - 4
      \kappa} C^{1 + \frac{1}{2 (1 - \kappa)}}_{\star} \right) \\
      & & \times (\Cdumb{} \vee \Cnoise{}^2 C_{\star}^{2 \kappa + 2
      \delta})^{1 + \frac{1}{2 (1 - \kappa)}} \\
    & + & \left( \Cnoise{}^3 L^{2 - 3 \kappa} C_{\star} + \Cnoise{}^{3 + \frac{1
    - 2 \kappa}{2 (1 - \kappa)}} L^{\frac{5}{2} - 4 \kappa} C^{\frac{1}{2 (1 -
    \kappa)}}_{\star} \right) (\Cdumb{} \vee \Cnoise{}^2 C_{\star}^{2
    \kappa + 2 \delta})^{\frac{1}{2 (1 - \kappa)}}\\
    & + & \left( \Cnoise{}^2 L^{2 - 2 \kappa} C^2_{\star} + \Cnoise{}^{2 +
    \frac{1 - 2 \kappa}{2 (1 - \kappa)}} L^{\frac{5}{2} - 3 \kappa} C^{1 +
    \frac{1}{2 (1 - \kappa)}}_{\star} \right) (\Cdumb{} \vee \Cnoise{}^2
    C_{\star}^{2 \kappa + 2 \delta})^{\frac{1}{1 - \kappa}}\\
    & + &
      \left( \Cnoise{}^{2 + \frac{1 - 2 \kappa}{2 (1 - \kappa)}} L^{\frac{5}{2}
      - 3 \kappa} C^{1 + \frac{1}{2 (1 - \kappa)}}_{\star} + \Cnoise{}^{2 +
      \frac{1 - 2 \kappa}{1 - \kappa}} L^{3 - 4 \kappa} C^{\frac{1}{1 -
      \kappa}}_{\star} \right) \\
      & & \times (\Cdumb{} \vee \Cnoise{}^2 C_{\star}^{2 \kappa + 2
      \delta})^{\frac{1}{1 - \kappa}} .
  \end{eqnarray*}
  Again with the same procedure, we obtain $b^{\noise, 1}_L (z) \cdot u_X (z )
  \assign \Lambda_L [\tilde{B}^{\noise, 1} \Pi \noise] (z)$ and $b^{\noise, 2}_L (z) \cdot
  u_X (z ) \assign \Lambda_L [\tilde{B}^{\noise, 2} \Pi \noise] (z)$, satisfying
  \eqref{controlonBL2} and \eqref{controlonBL}, since
  \[ | b^{\noise, 1}_L (z) | \leqslant \| \sigma'' \| \| \sigma \| [\Pi ; \lolli]_{1 -
     \kappa} [\Pi ; \noise]_{- 1 - \kappa} L^{1 - 2 \kappa} \lesssim \Cnoise{}^2 L^{1 - 2
     \kappa} \]
  and by \eqref{reductionofLinfinityofgradient}
  \begin{eqnarray*}
    | b^{\noise, 2}_L (z) | & \leqslant & \| \sigma'' \| \| u_X \|_{B (z, L)}
    [\Pi ; \noise]_{- 1 - \kappa} L^{1 - \kappa}\\
    & \lesssim & \Cnoise{} L^{1 - \kappa} \left( C_{\star} + \Cnoise{}^{\frac{1 -
    2 \kappa}{2 (1 - \kappa)}} L^{\frac{1}{2} - \kappa} C^{\frac{1}{2 (1 -
    \kappa)}}_{\star} \right) (\Cdumb{} \vee \Cnoise{}^2 C_{\star}^{2
    \kappa + 2 \delta})^{\frac{1}{2 (1 - \kappa)}} .
  \end{eqnarray*}
  To conclude, we set $a_L^{\noise, 5} \assign b^{\noise, 1}_L \cdot (u_X - \nabla
  u_L)$ and $a_L^{\noise, 6} \assign b^{\noise, 2}_L \cdot (u_X - \nabla u_L)$,
  which satisfy \eqref{controlALdominant1}, \eqref{controlALdominant3},
  \eqref{controlALmixed1}, \eqref{controlALmixed2}, \eqref{controlALmixed4},
  \eqref{controlALnoise}, since by \eqref{formulaforgradient},
  \eqref{controlonremaindergradient} and \eqref{comparisonofnorms}
  \begin{eqnarray*}
    | a_L^{\noise, 5} (z) | & \leqslant & | b^{\noise, 1}_L (z) | | (\sigma (u (z))
    \langle \Pi_z \lolli, \nabla_x \tilde{\varphi}^L_z \rangle + E_z^L) |\\
    & \leqslant & | b^{\noise, 1}_L (z) | \| \sigma \| ([\Pi ; \lolli]_{1 - \kappa}
    L^{- \kappa} + [U]_{\gamma_2, B (z, L)} L^{\gamma_2 - 1})\\
    & \lesssim & \Cnoise{}^2 L^{2 - 4 \kappa} C_{\star} (\Cdumb{} \vee
    \Cnoise{}^2 C_{\star}^{2 \kappa + 2 \delta}) + \Cnoise{}^3 L^{1 - 3 \kappa}
  \end{eqnarray*}
  and
  \begin{eqnarray*}
    &  & | a_L^{\noise, 6} (z) |\\
    & \leqslant & | b^{\noise, 2}_L (z) | \| \sigma \| ([\Pi ; \lolli]_{1 - \kappa}
    L^{- \kappa} + [U]_{\gamma_2, B (z, L)} L^{\gamma_2 - 1})\\
    & \lesssim & \left( \Cnoise{} L^{2 - 3 \kappa} C^2_{\star} + \Cnoise{}^{1 +
    \frac{1 - 2 \kappa}{2 (1 - \kappa)}} L^{\frac{5}{2} - 4 \kappa} C^{1 +
    \frac{1}{2 (1 - \kappa)}}_{\star} \right) (\Cdumb{} \vee \Cnoise{}^2
    C_{\star}^{2 \kappa + 2 \delta})^{1 + \frac{1}{2 (1 - \kappa)}}\\
    & + & \left( \Cnoise{}^2 L^{1 - 2 \kappa} C_{\star} + \Cnoise{}^{2 + \frac{1
    - 2 \kappa}{2 (1 - \kappa)}} L^{\frac{3}{2} - 3 \kappa} C^{\frac{1}{2 (1 -
    \kappa)}}_{\star} \right) (\Cdumb{} \vee \Cnoise{}^2 C_{\star}^{2
    \kappa + 2 \delta})^{\frac{1}{2 (1 - \kappa)}} .
  \end{eqnarray*}
  We conclude the proof by setting
  \[ a^{\noise}_L (z) \assign \sum_{i = 1}^6 a_L^{\noise, i} (z) \infixand b^{\noise}_L
     (z) \assign \sum_{i = 1}^2 b_L^{\noise, i} (z) . \]
\end{proof}

The next two lemmas are
the analogues to Lemma
\ref{splittinglemma1} for the objects $A_{z_1}^{\dumb} (z_2)$ and
$A_{z_1}^{\Xnoise} (z_2)$ in \eqref{reductionofG}. Their proofs are in Appendix
\ref{proofofdecompositionlemmas}. 

\begin{lemma}
  \label{splittinglemma2}Let the assumptions of Proposition
  \ref{regularisedequationwithtransport} be in force. Then
  \[ \lim_{N \rightarrow \infty} \Lambda_{N, L} [A^{\dumb} \Pi \dumb
     ] (z) = a_L^{\dumb} (z) + b_L^{\dumb} (z) \cdot \nabla u_L
     (z), \]
  where the functions $a^{\dumb}_L, b^{\dumb}_L$ are smooth and the
  convergence is uniform over $z \in D_{T_1}^{T_{\star}}$. Moreover,
  $a_L^{\dumb}$ satisfies
  \eqref{controlALdominant1}-\eqref{controlALnoise} and $b_L^{\dumb}$
  satisfies \eqref{controlonBL2}-\eqref{controlonBL}.
\end{lemma}

\begin{lemma}
  \label{splittinglemma3}Let the assumptions of Proposition
  \ref{regularisedequationwithtransport} be in force. Then
  \begin{equation}
    \lim_{N \rightarrow \infty} \Lambda_{N, L} [A^{\Xnoise} \Pi \Xnoise] (z) =
    a^{\Xnoise}_L (z) + b^{\Xnoise}_L (z) \cdot \nabla u_L (z),
  \end{equation}
  where the convergence is uniform over $z \in D_{T_1}^{T_{\star}}$. Furthermore,
  $a^{\Xnoise}_L$ satisfies
  \eqref{controlALdominant1}-\eqref{controlALnoise} and $b^{\Xnoise}_L$ satisfies
  \eqref{controlonBL2}-\eqref{controlonBL}.
\end{lemma}

We are ready to prove Proposition \ref{regularisedequationwithtransport}.

\vspace{2mm}
\begin{proof*}{Proof of Proposition \ref{regularisedequationwithtransport}.}
We start by
noting that, for $L < 1$, $\Xi_L$ satisfies
\[ | \Xi_L (z) | \leqslant \| \sigma' \| \| \sigma \| [\Pi ; \dumb]_{- 2
   \kappa} L^{- 2 \kappa} + \| \sigma \| [\Pi ; \noise]_{- 1 - \kappa} L^{- 1 - \kappa}
   \lesssim \Cdumb{} L^{- 2 \kappa} + \Cnoise{} L^{- 1 - \kappa}, \]
which 
allows us to include
$\Xi_L$ as part of $c_L$ and satisfy \eqref{controlonCL}. By using
\eqref{formulaforgradient}, we may write
\begin{eqnarray}
  \tilde{B}_L (z) & = & \sigma' (u (z)) \langle \Pi_z \Xnoise, \varphi_z^L
  \rangle \cdot \nabla u_L (z)  \label{negativetransportterm}\\
  & - & \sigma' \sigma (u (z)) \langle \Pi_z \Xnoise, \varphi_z^L \rangle
  \cdot \langle \Pi_z \lolli, \nabla_x \tilde{\varphi}^L_z \rangle 
  \label{extranegativeterm}\\
  & - & \sigma' (u (z)) \langle \Pi_z \Xnoise, \varphi_z^L \rangle \cdot
  E_z^L,  \label{remainderoftrasport}
\end{eqnarray}
from which we see that the term in \eqref{negativetransportterm} can contribute
to $b_L$ in \eqref{regularisedequation} with the negative power of $L$ in
\eqref{controlonBL}, since
\[ | \sigma' (u (z)) \langle \Pi_z \Xnoise, \varphi_z^L \rangle | \leqslant \|
   \sigma' \| [\Pi ; \Xnoise]_{- \kappa} L^{- \kappa} \lesssim \Cnoise{} L^{- \kappa} .
\]
The term in \eqref{extranegativeterm} can be part of $c_L$ in
\eqref{regularisedequation}, as it satisfies
\[ | \sigma' \sigma (u (z)) \langle \Pi_z \Xnoise, \varphi_z^L \rangle \cdot
   \langle \Pi_z \lolli, \nabla_x \tilde{\varphi}^L_z \rangle | \leqslant \|
   \sigma' \| \| \sigma \| [\Pi ; \Xnoise]_{- \kappa} [\Pi ; \lolli]_{1 - \kappa} L^{- 2
   \kappa} \lesssim \Cnoise{}^2 L^{- 2 \kappa}, \]
which respects \eqref{controlonCL}. For the term in
\eqref{remainderoftrasport}, we use \eqref{controlonremaindergradient} with
$\gamma_2$ and \eqref{comparisonofnorms} to obtain
\begin{eqnarray*}
  | \sigma' (u (z)) \langle \Pi_z \Xnoise, \varphi_z^L \rangle \cdot E_z^L | &
  \leqslant & \| \sigma' \| [\Pi ; \Xnoise]_{- \kappa} [U]_{\gamma_2, B (z, L)}
  L^{\gamma_2 - 1 - \kappa}\\
  & \lesssim & \Cnoise{} L^{1 - 3 \kappa} C_{\star} (\Cdumb{} \vee
  \Cnoise{}^2 C_{\star}^{2 \kappa + 2 \delta})
\end{eqnarray*}
which makes it part of $a_L$ in \eqref{regularisedequation} satisfying
\eqref{controlALdominant1}. The proof 
is concluded
by grouping together as $a_L$ all the terms that satisfy
\eqref{controlALdominant1}-\eqref{controlALnoise}, as $b_L$ all that satisfy
\eqref{controlonBL3}-\eqref{controlonBL} and as $c_L$ all that satisfy
\eqref{controlonCL}, including those coming from 
Lemmas \ref{splittinglemma1}, \ref{splittinglemma2} and \ref{splittinglemma3}.

\end{proof*}

\subsection{Proof of Theorem \ref{theorempolynomialgrowth}} \label{sectionproofoftheorempolynomial}

To prove this theorem, we make explicit the dependence on the time interval
  of the constants $\Cnoise{,n}$ and $\Cdumb{,n}$, for every $n \in
  \mathbb{N}$, as in \eqref{assumptiononstochasticobjects}. We point out that
  everything that was done so far can be carried out analogously for each of
  the time intervals $D_{n - 1}^n$, and we will make use of it. In particular,
  Theorem \ref{maintheorem} holds identically for every $n \in \mathbb{N}$ and
  $D_{n - 1}^n$, with $\| u_0 \|$ replaced by $\| u_{| t = n - 1 \nobracket}
  \|$ and $\Cnoise{,1}, \Cdumb{,1}$ replaced by $\Cnoise{,n}, \Cdumb{,n}$,
   accordingly.
  
  Again the constant $C (\kappa, d, C_{\sigma}) > 0$ may change from line to
  line. For every $n \in \mathbb{N}$, let $T_{\star, n} = T_{\star, n}
  (C_{\star, n})$, for a constant $C_{\star, n} \geqslant \left\| {u_{| t = n
  - 1 \nobracket}}  \right\|$, be defined according to
  \eqref{definitionstoppingtime}, i.e.,
  \begin{equation}
    T_{\star, n} \assign \inf \{ t \in [n - 1, n] : \| u \|_{D_{n - 1}^t}
    \geqslant C_{\star, n} \}, \label{definitionofTstarN}
  \end{equation}
  which corresponds to \eqref{definitionstoppingtime} when $n = 1$. Our goal
  is then to look at $Y_n \assign \| u \|_{D_{n - 1}^n}$ and obtain a
  difference inequality for $Y_n$, which can be estimated recursively.

  \vspace{2mm}   
  \begin{proof*}{Proof of Theorem \ref{theorempolynomialgrowth}.}
  We start
  with Theorem \ref{maintheorem}, which provides us with a control in
  \eqref{mainresultbound} that corresponds to $n = 1$, meaning
  \[ Y_1 = \| u \|_{D_0^1} \leqslant C (\kappa, d, C_{\sigma}) \max \left\{ \|
     u_0 \|, \Cdumb{,1}^{\frac{1}{(1 - \kappa) (1 - \betadumb)}},
      \Cnoise{,1}^{\frac{2}{(1 - \kappa) (1 - \betanoise)}} \right\} .
  \]
  Further, assume control on $\| u \|_{D_{n - 1}^n}$ and proceed by induction.
  We may again obtain a control on $\| u \|_{D_n^{n + 1}}$ using Theorem
  \ref{maintheorem}, \ by making the choice
  \[ C_{\star, n + 1} \gtrsim \max \left\{ \| u_{| t = n \nobracket} \|,
     \Cdumb{,n + 1}^{\frac{1}{(1 - \kappa) (1 - \betadumb)}},
     \Cnoise{,n + 1}^{\frac{2}{(1 - \kappa) (1 - \betanoise)}} \right\}, \]
  which implies $T_{\star, n + 1} = n + 1$. Furthermore, this corresponds to
  the control
  \begin{equation}
    \| u \|_{D_n^{n + 1}} \lesssim C_{\star, n + 1} \lesssim Y_n + \Cdumb{, n + 1}^{
    \frac{1}{(1 - \kappa) (1 - \betadumb)}} \vee \Cnoise{,n + 1}^{\frac{2}{(1 - \kappa)
    (1 - \betanoise)}},
    \label{controlonNtoNplus1}
  \end{equation}
  since $\| u_{| t = n \nobracket} \| \leqslant \| u \|_{D_{n - 1}^n} = Y_n$.
  Then, we may repeat the analysis
  \refchange{of the step in the proof of Theorem 
  \ref{theoremtopostprocessintopolynimial}
  which makes use of Proposition
  \ref{propositiontransportdecomposition}
  and leads to
  \eqref{torepeatinthepolynomialgrowth},}
  but with one important
  difference. Recall
  \refchange{from \eqref{definitionofuL},
  \eqref{defofmollifierkernel}
  and \eqref{balllookingtopast}}
  that when looking to the regularised equation centred at
  a point $z = (t, x)$
  \refchange{and at scale $L < \sqrt{T_1/2}$}
  points that are at most $T_1 / 2$ into the past of $t$
  are also considered for the original equation. In the first time around, we
  separate the domain $D_0^1$ into $D_0^{T_1}$ and $D_{T_1}^1$ in order to be
  able to attain uniform bounds on $a_L$ and $c_L$ over $D_{T_1}^1$ in
  Proposition \ref{regularisedequationwithtransport}, as we stay uniformly
  away from $t = 0$ in this domain when looking at most $T_1 / 2$ into the
  past. The splitting was also done for $D_n^{n + 1}$ to obtain
  \eqref{controlonNtoNplus1}, but since $n \geqslant 1$, every point in
  $D_n^{n + 1}$ is away from $t = 0$ when looking at most $T_1 / 2$ into the
  past. 
  This observation allows us to perform step three directly over the
  whole domain $D_n^{n + 1}$, while making use of the estimate in
  \eqref{controlonNtoNplus1}. By carrying this on, the analog to
  \eqref{torepeatinthepolynomialgrowth} becomes
  \begin{equation}
    \| u \|_{D_n^{n + 1}} \leqslant \| u_{\tilde{L} | t = n \nobracket} \| + C
    (\kappa, d, C_{\sigma}) \left( \tCdumb{,n + 1}^{\frac{1}{1 -
    \kappa}} C^{\betadumb}_{\star, n + 1} \vee \tCnoise{, n +
    1}^{\frac{2}{1 - \kappa}} {C^{\betanoise}_{\star, n + 1}}  \right),
    \label{intermediatepolynomialgrowth}
  \end{equation}
  where $\tilde{L} \leqslant \sqrt{T_1}/2$ is exactly as in
  \eqref{choiceofL} - with the correct dependence on $n \in \mathbb{N}$ - and
  $\tCdumb{,n + 1}$, $\tCnoise{,n + 1}$ are defined as
  \begin{equation}
    \tCdumb{,n + 1} \assign \Cdumb{,n} \vee \Cdumb{,n+1} \quad \tmop{and} \quad \tCnoise{,n + 1} \assign \Cnoise{, n} \vee
     \Cnoise{,n + 1} . \label{definitionofCtildenoise}
  \end{equation}
  The new constants $\tCdumb{, n + 1}$ and $\tCnoise{,n + 1}$ are necessary
  because when we look to $u_{\tilde{L}}$
  over the whole domain $D_n^{n + 1}$, we invade the domain $D_{n -
  1}^n$ of the original $u$, since we look to the past of $t = n$. This means
  that we pick up influences of the noise in the previous domain $D_{n - 1}^n$
  - but not earlier since $\tilde{L} < 1$. For the same reason, we see that
  $\| u_{\tilde{L} | t = n \nobracket} \| \leqslant \| u \|_{D_{n - 1}^n} =
  Y_n$, and therefore by plugging \eqref{controlonNtoNplus1} into
  \eqref{intermediatepolynomialgrowth} we arrive at
  \begin{eqnarray*}
    &  & Y_{n + 1} - Y_n\\
    & \lesssim & \tCdumb{,n + 1}^{\frac{1}{1 - \kappa}}
    C^{\betadumb}_{\star, n + 1} \vee \tCnoise{,n + 1}
    ^{\frac{2}{1 - \kappa}} {C^{\betanoise}_{\star, n + 1}} \\
    & \lesssim & \tCdumb{,n + 1}^{\frac{1}{1 - \kappa}}
    Y_n^{\betadumb} \vee \tCnoise{,n + 1}^{\frac{2}{1 -
    \kappa}} {Y_n^{\betanoise}} \\
    & + & \tCdumb{,n + 1}^{\frac{1}{1 - \kappa}} \left(
    \Cdumb{,n + 1}^{\frac{\betadumb}{(1 - \kappa) (1 -
    \betadumb)}} \vee \Cnoise{,n + 1}^{\frac{2 \betadumb}{(1
    - \kappa) (1 - \betanoise)}} \right)\\
    & + & \tCnoise{,n + 1}^{\frac{2}{1 - \kappa}} \left( \Cdumb{,n + 1}
    ^{\frac{\betanoise}{(1 - \kappa) (1 - \betadumb)}} \vee
    \Cnoise{,n + 1}^{\frac{2 \betanoise}{(1 - \kappa) (1 - \betanoise)}}
    \right)\\
    & \lesssim & \tCdumb{,n + 1}^{\frac{1}{1 - \kappa}}
    Y_n^{\betadumb} \vee \tCnoise{,n + 1}^{\frac{2}{1 -
    \kappa}} {Y_n^{\betanoise}} \\
    & + & \tCdumb{,n + 1}^{\frac{1}{(1 - \kappa) (1 -
    \betadumb)}} \vee \tCdumb{,n + 1}^{\frac{1}{1 -
    \kappa}} \tCnoise{,n + 1}^{\frac{2 \betadumb}{(1 - \kappa)
    (1 - \betanoise)}} \vee \tCdumb{,n + 1}
    ^{\frac{\betanoise}{(1 - \kappa) (1 - \betadumb)}}
    \tCnoise{,n + 1}^{\frac{2}{1 - \kappa}} \vee \tCnoise{, n + 1}
    ^{\frac{2}{(1 - \kappa) (1 - \betanoise)}}\\
    & \lesssim & \tCdumb{,n + 1}^{\frac{1}{1 - \kappa}}
    Y_n^{\betadumb} \vee \tCnoise{,n + 1}^{\frac{2}{1 -
    \kappa}} {Y_n^{\betanoise}}  + \tCdumb{,n + 1}^{\frac{1}{(1
    - \kappa) (1 - \betadumb)}} \vee \tCnoise{, n + 1}
    ^{\frac{2}{(1 - \kappa) (1 - \betanoise)}},
  \end{eqnarray*}
  since $\Cnoise{,n + 1} \leqslant \tCnoise{,n + 1}$ and $\Cdumb{,n + 1}
   \leqslant \tCdumb{,n + 1}$. Here as usual $\lesssim$
  hides a constant that only depends on $\kappa, d$ and $C_{\sigma}$. Let us
  rewrite this equation as
  \begin{equation}
    Y_{n + 1} - Y_n \leqslant C (\kappa, d, C_{\sigma}) \left(
    \tCdumb{,n + 1}^{\frac{1}{1 - \kappa}} \vee \tCnoise{,n + 1}
    ^{\frac{2}{1 - \kappa}} \right)^{\frac{1}{1 - \betanoise}}
    {Y_n^{\betanoise}}  \label{differenceequationpolynomial},
  \end{equation}
  for which we refer to Appendix \ref{Appendixdifferenceequations} on how it
  can be estimated recursively to obtain, for every $n < t \leqslant n + 1$,
  \refchange{
  \begin{equation}
    \begin{array}{lll}
      &  & \displaystyle \| u \|_{D_0^t}\\
      & \leqslant & \displaystyle \max_{1 \leqslant i \leqslant n + 1} Y_i\\
      & \leqslant & C (\kappa, d, C_{\sigma}) \max \left\{ \| u \|_{D_0^1},
      \max_{1 \leqslant i \leqslant n + 1} \left\{ \Cdumb{,i}
      ^{\frac{1}{1 - \kappa}}, \Cnoise{,i}^{\frac{2}{1 - \kappa}}
      \right\}^{\frac{1}{(1 - \betanoise)^2}} n^{\frac{1}{1 - \betanoise}}
      \right\}\\
      & \leqslant & C (\kappa, d, C_{\sigma}) \max \left\{ \| u_0 \|, \max_{1
      \leqslant i \leqslant n + 1} \left\{ \Cdumb{,i}^{\frac{1}{1 -
      \kappa}}, \Cnoise{,i}^{\frac{2}{1 - \kappa}} \right\}^{\frac{1}{(1 -
      \betanoise)^2}} t^{\frac{1}{1 - \betanoise}} \right\}
    \end{array} \label{finalpolynomialgrowth}
  \end{equation}
  }which gives \eqref{polynomialgrowth} and concludes the proof of the theorem.
\end{proof*}

\section{Proof of Theorem \ref{theoinfinityschauderpostprocessing}
}\label{sectionproofofshorttimescontrol}

The strategy to prove this theorem is to split the original equation
\eqref{mainequationrenormalised} into two coupled equations \eqref{thesplittingu1} and
\eqref{thesplittingu2}. While \eqref{thesplittingu1} is straightforward to
deal with, \eqref{thesplittingu2} contains the singular product $\sigma (u) \diamond
\xi$. This decomposition \eqref{thesplittingu1}-\eqref{thesplittingu2} helps
us overcome the blow-up near the boundary $t = 0$ in the interior estimate of
Section \ref{sectionchoiceofT0}. Indeed, the zero initial condition in
\eqref{thesplittingu2} allows us to extend the inhomogeneous equation to
negative times and therefore get rid of the boundary.

The extension of $\sigma (u)\diamond \xi$ follows {\cite[Prop. 6.9]{Hai14}}, which requires
control of the blow-ups near $t = 0$
of each of the components in the local description of $\sigma (u) \diamond \xi$. Here,
we will make use of Theorem \ref{theointeriorestimate}, 
so henceforth its assumptions are in force. 
Recall that $D_0^{T} =
[0, T] \times \mathbb{T}^d$, where $T$ is given by \eqref{conditionforT},
and let $D_0 \assign \{ t = 0 \} \times \mathbb{T}^d$ and $D^{T}_+ \assign
D_0^{T} \setminus D_0 = (0, T] \times \mathbb{T}^d$. Consider $u_1 : D
\rightarrow \mathbb{R}$ the solution to
\begin{equation}
  \left\{\begin{array}{rllll}
    (\partial_t - \Delta) u_1 & = & 0 & \tmop{in} & D^{T}_+\\
    u_1 & = & u_0 &
    \tmop{on} & D_0
  \end{array}\right.  \label{thesplittingu1}
\end{equation}
and $u_2 : D_0^{T} \rightarrow \mathbb{R}$ the solution to
\begin{equation}
  \left\{\begin{array}{rllll}
    (\partial_t - \Delta) u_2 & = & \sigma (u_1 + u_2) \diamond \xi & \tmop{in} &
    D^{T}_+\\
    u_2 & = & 0 & \tmop{on}
    & D_0
  \end{array}\right., \label{thesplittingu2}
\end{equation}
so that $u : D_0^{T} \rightarrow \mathbb{R}$ solution to
\eqref{mainequation} can be rewritten as
\begin{equation}
  u = u_1 + u_2 . \label{splittingofu}
\end{equation}
For \eqref{thesplittingu1}, denote by $(P^t)_{t \geqslant 0}$ the heat
semigroup with periodic (spatial) boundary conditions, then
\begin{equation}
  D \ni z = (t, x) \mapsto u_1 (t, x) = P^t u_0 (x) \Longrightarrow \| u_1
  \|_{D_0^{T}} \leqslant \| u_0 \| . \label{controlonu1}
\end{equation}
We move to \eqref{thesplittingu2}, for which we seek to control $\| u_2
\|_{D_0^{T}}$. The goal is to extend the domain $D_0^{T}$ to negative
times, which we denote by
\[ D_{- \infty}^{T} \assign (- \infty, T] \times \mathbb{T}^d = D_- \cup
   D_0 \cup D^{T}_+  \quad \tmop{where} \quad D_- \assign (- \infty, 0)
   \times \mathbb{T}^d . \]
We recall how to define $\sigma (u) \diamond \xi$ over $D^{T}_+$ and to study the
blow up near $t = 0$ of each of the components of its local description. For
every $\tau \in (0, 1]$ let $d_{\tau} = \sqrt{\tau}$ and consider $L <
d_{\tau} / 2$, so that for every $z \in D^{T}_{\tau + L}$, $\varphi_z^L$ is
supported in $D_{\tau}^{T}$. By using \eqref{localapproxG} of Section
\ref{sectionReconstruction}, we may write for $z \in D^{T}_{\tau + L}$
\begin{eqnarray}
  \sigma (u) \diamond \noiseb & = & \sigma (u) \diamond \noiseb - G_z 
  \label{expansionblowupreconstruction}\\
  & + & \sigma (u (z)) \Pi_z \noise + \sigma' \sigma (u (z)) \Pi_z \dumb +
  \sigma' (u (z)) u_X (z) \Pi_z \Xnoise . 
  \label{expansionblowuplocaldescription}
\end{eqnarray}
We introduce norms with blow ups following {\cite[Def.
6.2]{Hai14}}, which generalise \eqref{normofUblowup}. For $\gamma >
0$, $\eta \in \mathbb{R}$, and a base point dependent family of distributions
$(f_z)_{z \in D_+^{T}}$, we let
\[ \| f_{\cdot} \|_{\gamma, \eta, [0, T]} \assign \sup_{\tau \in (0, T]}
   \sup_{L < \frac{d_{\tau}}{2}} \sup_{z \in D_{\tau + L}^{T}}
   d_{\tau}^{\gamma - \eta} | \langle f_z, \varphi_z^L \rangle | L^{- \gamma}
   . \]
We aim to control the norm above with the correct blow up for each of the
quantities in \eqref{expansionblowupreconstruction} and
\eqref{expansionblowuplocaldescription}. Here $1 + \kappa < \gamma \leqslant
2 - 2 \kappa$ as in Theorem \ref{theointeriorestimate}. 
We start with $\sigma (u) \diamond \noiseb - G_z$, which by \eqref{outputreconstruction}
\begin{eqnarray}
  &  & \| \sigma (u) \diamond \noiseb - G_{\cdot} \|_{\gamma - 1 - \kappa, - 1 - \kappa,
  [0, T]} \nonumber\\
  & = & \sup_{\tau \in (0, T]} \sup_{L < \frac{d_{\tau}}{2}} \sup_{z \in
  D_{\tau + L}^{T}} d_{\tau}^{\gamma} | \langle \sigma (u) \diamond \noiseb - G_z,
  \varphi_z^L \rangle | L^{- \gamma + 1 + \kappa} \nonumber\\
  & \lesssim & 
    \sup_{\tau \in (0, T]} \sup_{L < \frac{d_{\tau}}{2}} \sup_{z \in D_{\tau
    + L}^{T}} d_{\tau}^{\gamma} \nonumber \\
    & & \times (\Cnoise{} [\sigma (U)]_{\gamma, B (z, L)} + \Cdumb{} [\sigma'
    \sigma (u)]_{\gamma - 1 + \kappa, B (z, L)} + \Cnoise{} [\sigma' (u)
    u_X]_{\gamma - 1, B (z, L)}) \nonumber \\
    & \leqslant &
    \sup_{\tau \in (0, T]} d_{\tau}^{\gamma} 
    \left( \Cnoise{} [\sigma (U)]_{\gamma, D_{\tau}^{T}, \frac{d_{\tau}}{2}} \right. \nonumber \\
    & & \left. +
    \Cdumb{} [\sigma' \sigma (u)]_{\gamma - 1 + \kappa,
    D_{\tau}^{T}, \frac{d_{\tau}}{2}} + \Cnoise{} [\sigma' (u) u_X]_{\gamma
    - 1, D_{\tau}^{T}, \frac{d_{\tau}}{2}} \right)
  \label{controlbeforelastschauder}
\end{eqnarray}
Note that the blow up of order $\gamma - 1 - \kappa - (- 1 - \kappa) =
\gamma$ corresponds to the gain of regularity from $0$ to $\gamma$ of the
enhancement of $\sigma (u)$, denoted by $\sigma (U)$. The term above
corresponds to the one in \eqref{schauderRHSII}, except for an extra
$d_{\tau}^{1 - \kappa}$ factor. Repeating the arguments of Section
\ref{sectioninteriorestimate} we arrive at
\begin{eqnarray*}
  &  & \| \sigma (u) \diamond \noiseb - G_{\cdot} \|_{\gamma - 1 - \kappa, - 1 - \kappa,
  [0, T]}\\
  & \lesssim & \Cnoise{} \| u \|^{\gamma - 1}_{D_0^{T}}  [U]_{\gamma, 0,
  [0, T]}\\
  & + & \left[ \Cnoise{} + T^{\frac{1 - \kappa}{2}} (\Cnoise{}^2 + \Cdumb{}
  ) \right] [U]_{\gamma, 0, [0, T]}\\
  & + & \left[ \Cnoise{} + T^{\frac{1 - \kappa}{2}} (\Cnoise{}^2 + \Cdumb{}
  ) \right] \| u \|_{D_0^{T}}\\
  & + & T^{\frac{(1 - \kappa) (\gamma - 1)}{2}} \Cnoise{}^{\gamma} +
  T^{\frac{1 - \kappa}{2}} (\Cnoise{}^2 + \Cdumb{}) + T^{1 - \kappa}
  (\Cdumb{} \Cnoise{} + \Cnoise{}^3),
\end{eqnarray*}
which is the same control obtained in \eqref{finalboundbeforeabsorption} but
with a factor $T^{\frac{1 - \kappa}{2}}$ missing and without the last line
$T^{\frac{1 - \kappa}{2}} \Cnoise{} + \| u \|_{D_a^b}$ that comes from the
Schauder estimate Lemma \ref{lemmaschauderestimate}. We move to $\sigma (u
(\cdot)) \Pi \noise$
\[ \| \sigma (u (\cdot)) \Pi \noise \|_{- 1 - \kappa, - 1 - \kappa, [0, T]} =
   \sup_{\tau \in (0, T]} \sup_{L < \frac{d_{\tau}}{2}} \sup_{z \in D_{\tau
   + L}^{T}} d_{\tau}^0 | \sigma (u (z)) \langle \xi, \varphi_z^L \rangle |
   L^{1 + \kappa} \lesssim \Cnoise{} . \]
Next is $\sigma' \sigma (u (z)) \Pi_z \dumb$, for which we obtain
\begin{eqnarray*}
  &  & \| \sigma' \sigma (u (\cdot)) \Pi \dumb \|_{- 2 \kappa,
  - 1 - \kappa, [0, T]}\\
  & = & \sup_{\tau \in (0, T]} \sup_{L < \frac{d_{\tau}}{2}} \sup_{z \in
  D_{\tau + L}^{T}} d_{\tau}^{1 - \kappa} | \sigma' \sigma (u (z)) \langle
  \Pi_z \dumb, \varphi_z^L \rangle | L^{2 \kappa} \lesssim T^{\frac{1
  - \kappa}{2}} \Cdumb{} .
\end{eqnarray*}
At last, we look to $\sigma' (u (\cdot)) u_X (\cdot) \Pi \Xnoise$,
which produces a term like the second term in \eqref{schauderRHSI}, except for
a factor $T^{\frac{1 - \kappa}{2}}$ missing, and thus
\begin{eqnarray*}
  &  & \| \sigma' (u (\cdot)) u_X (\cdot) \Pi \Xnoise \|_{- \kappa, -
  1 - \kappa, [0, T]}\\
  & = & \sup_{\tau \in (0, T]} \sup_{L < \frac{d_{\tau}}{2}} \sup_{z \in
  D_{\tau + L}^{T}} d_{\tau} | \sigma' (u (z)) u_X (z) \langle \Pi_z \Xnoise
  , \varphi_z^L \rangle | L^{\kappa}\\
  & \lesssim & \Cnoise{} \sup_{\tau \in (0, T]} d_{\tau} \| u_X
  \|_{D_{\tau}^{T}}\\
  & \lesssim & \Cnoise{} [U]_{\gamma, 0, [0, T]} + \Cnoise{} \sup_{\tau \in
  (0, T]} \| U \|_{D_{\tau}^{T}, d_{\tau}}\\
  & \lesssim & \Cnoise{} [U]_{\gamma, 0, [0, T]} + \Cnoise{} \| u
  \|_{D_0^{T}} + T^{\frac{1 - \kappa}{2}} \Cnoise{}^2 .
\end{eqnarray*}
With the control on the norms with blow up at hand, we want to define the
distribution $\tilde{f}$ over $D_{- \infty}^{T} \setminus D_0$ by setting, at
least in a formal level, $\tilde{f} = \sigma (u) \diamond \xi$ over $D_+^{T}$ and $\tilde{f}
\equiv 0$ over $D_-$. This can be done rigorously by setting
\begin{equation}
  \left\{\begin{array}{lll}
    \langle \tilde{f}, \varphi_z^L \rangle = \langle \sigma (u) \diamond \noiseb,
    \varphi_z^L \rangle & \tmop{if} & z \in D_+^{T} \quad \tmop{and} \quad
    \tmop{supp} (\varphi_z^L) \subseteq D_+^{T}\\
    \langle \tilde{f}, \varphi_z^L \rangle = 0 & \tmop{if} & z \in D_-  \quad
    \tmop{and} \quad \tmop{supp} (\varphi_z^L) \subseteq D_-
  \end{array}\right. . \label{extendeddistribution}
\end{equation}
In view of
\eqref{expansionblowupreconstruction}-\eqref{expansionblowuplocaldescription},
we let $A \assign \{ \gamma - 1 - \kappa, - 1 - \kappa, - 2 \kappa, - \kappa
\}$ and further split
\begin{equation}
  \langle \tilde{f}, \varphi_z^L \rangle \assign \sum_{\alpha \in A} \langle
  \tilde{f}_{\alpha}, \varphi_z^L \rangle \tmmathbf{1}_{\{ \tmop{supp}
  (\varphi_z^L) \subseteq D_+^{T} \}} (\varphi_z^L),
  \label{furthersplittingofrhs}
\end{equation}
\refchange{
where we identify for $a \in A$,
\begin{eqnarray*}
    f_{\gamma - 1 - \kappa}
    & := & \| \sigma (u) \diamond \noiseb - G_{\cdot} \|_{\gamma - 1 - \kappa, - 1 - \kappa, [0, T]} \\
    f_{1 - \kappa}
    & := & \| \sigma (u (\cdot)) \Pi \noise \|_{- 1 - \kappa, - 1 - \kappa, [0,T]}\\
    f_{- 2\kappa}
    & := & \| \sigma' \sigma (u (\cdot)) \Pi \dumb \|_{- 2 \kappa, - 1 - \kappa, [0, T]}  \\
    f_{- \kappa}
    & := & \| \sigma' (u (\cdot)) u_X (\cdot) \Pi \Xnoise \|_{- \kappa, - 1 - \kappa, [0, T]}
\end{eqnarray*}
so that it is easy to see that by this construction we obtain (recall that
$\varphi$ is supported in a ball that only looks to the past)
\begin{equation}
  \begin{array}{ll}
    \displaystyle
    \interleave \tilde{f} \interleave_{\gamma - 1 - \kappa, - 1 - \kappa,
    (- \infty, 0) \cup (0, T]}
    \assign & \displaystyle \sup_{\tau \in (0, T]} \sup_{L < \frac{d_{\tau}}{2}} \sup_{z
    \in D_{\tau + L}^{T} \cup D_-} \\
    & \displaystyle \sum_{\alpha \in A} d_{\tau}^{\alpha + 1
    + \kappa} | \langle \tilde{f}_{\alpha}, \varphi_z^L \rangle | L^{-
    \alpha} \tmmathbf{1}_{\{ z \in D_+ \}}
    \; .
  \end{array} \label{boundonextendeddistribution}
\end{equation}
}

\refchange{
The following lemma is imported from \cite[Prop. 6.9]{Hai14},
which holds true as $- 1 - \kappa > - 2$. Its proof is left to the end of this
section.
\begin{lemma} \label{importedfromhairer}
There exists a unique
distribution $f$ over $D_{- \infty}^{T}$, such that it satisfies $\langle f,
\varphi_z^L \rangle = \langle \tilde{f}, \varphi_z^L \rangle$ whenever
$\tmop{supp} (\varphi_z^L) \subseteq D_{- \infty}^{T} \setminus D_0$, and in
particular \eqref{extendeddistribution}-\eqref{boundonextendeddistribution}.
Furthermore, $f \in C^{- 1 - \kappa} (D_{- \infty}^{T})$, in the sense that
\begin{equation}
  [f]_{- 1 - \kappa} \assign \sup_{L \leqslant 1} \sup_{z \in D_{-
  \infty}^{T}} | \langle f, \varphi_z^L \rangle | L^{1 + \kappa} \leqslant
  \interleave \tilde{f} \interleave_{\gamma - 1 - \kappa, - 1 - \kappa, (-
  \infty, 0) \cup (0, T]} . \label{controlonextendeddistribution}
\end{equation}
\end{lemma}}
\begin{remark}
  The condition $- 1 - \kappa > - 2$ in particular excludes delta
  distributions in time at $t = 0$, for example to add an initial conditions
  $\tilde{u}_0 : \mathbb{T}^d \rightarrow \mathbb{R}$ as space-time
  distribution $\delta_{\{ t = 0 \}} \tilde{u}_0$ which would have regularity
  $C^{- 2} (D_{- \infty}^{T})$.
\end{remark}

\refchange{
\begin{proof*}{Proof of Theorem \ref{theoinfinityschauderpostprocessing}.}
If we now extend $u_2$ solution to \eqref{thesplittingu2} to negative times by
setting $u_2 \equiv 0$ in $D_-$, then we obtain that $u_2 : D_{- \infty}^{T}
\rightarrow \mathbb{R}$ now solves
\[ (\partial_t - \Delta) u_2 = f \quad \tmop{in} \quad D_{- \infty}^{T} . \]
By applying {\cite[Lem. B.1]{Moinatandweber2020ejp}}, we obtain that for a
constant $C_S > 0$ which only depends on $d$ and $\kappa$,
\begin{equation}
  [u_2]_{1 - \kappa, D_{- \infty}^{T}} \leqslant C_S [f]_{- 1 - \kappa}
  \lesssim \interleave \tilde{f} \interleave_{\gamma - 1 - \kappa, - 1 -
  \kappa, (- \infty, 0) \cup (0, T]}, \label{useofschauderwithoutblowups}
\end{equation}
which readily implies that for every $z = (t, x) \in D_0^{T}$, since $u_2
(0, x) = 0$,
\[ | u_2 (z) | \lesssim T^{\frac{1 - \kappa}{2}} \interleave \tilde{f}
   \interleave_{\gamma - 1 - \kappa, - 1 - \kappa, (- \infty, 0) \cup (0,
   T]}, \]
and by \eqref{boundonextendeddistribution},
\begin{eqnarray*}
  \| u_2 \|_{D_0^{T}} & \lesssim & T^{\frac{1 - \kappa}{2}} \Cnoise{} \| u
  \|^{\gamma - 1}_{D_0^{T}}  [U]_{\gamma, 0, [0, T]}\\
  & + & \left[ T^{\frac{1 - \kappa}{2}} \Cnoise{} + T^{1 - \kappa}
  (\Cnoise{}^2 + \Cdumb{}) \right] [U]_{\gamma, 0, [0, T]}\\
  & + & \left[ T^{\frac{1 - \kappa}{2}} \Cnoise{} + T^{1 - \kappa}
  (\Cnoise{}^2 + \Cdumb{}) \right] \| u \|_{D_0^{T}}\\
  & + & T^{\frac{(1 - \kappa) \gamma}{2}} \Cnoise{}^{\gamma} + T^{1 -
  \kappa} (\Cnoise{}^2 + \Cdumb{}) + T^{\frac{3 - 3 \kappa}{2}}
  (\Cdumb{} \Cnoise{} + \Cnoise{}^3),
\end{eqnarray*}
which is precisely \eqref{finalboundbeforeabsorption}, but without the last
line $T^{\frac{1 - \kappa}{2}} \Cnoise{} + \| u \|_{D_a^b}$ that comes from the
Schauder estimate Lemma \ref{lemmaschauderestimate}. Since we have chosen $T
= T (\gamma)$ and $\gamma$ according to Theorem
\ref{theointeriorestimate}, we have at the same time that $[U]_{\gamma, 0,
[0, T]} \lesssim 1 \vee \| u \|_{D_0^{T}}$ and that we can make a possibly
smaller choice of $T$ in \eqref{conditionforT} such that the sum of
all the prefactors is less than $\frac{1}{2}$. Hence, by \eqref{splittingofu}
and \eqref{controlonu1}
\[ \| u \|_{D_0^{T}} \leqslant \| u_1 \|_{D_0^{T}} + \| u_2 \|_{D_0^{T}}
   \leqslant \| u_0 \| + \frac{1}{2} \| u \|_{D_0^{T}} + \frac{1}{2}, \]
which gives \eqref{infinityschauderpostprocessing} and concludes the proof of
Theorem \ref{theoinfinityschauderpostprocessing}.
\end{proof*}

The proof of Lemma \ref{importedfromhairer} follows the argument of {\cite[Prop.
6.9]{Hai14}} exactly, but we reproduce it here for completeness.
}

\vspace{2mm}
\begin{proof*}{Proof of Lemma \ref{importedfromhairer}}
  The control obtained in \eqref{boundonextendeddistribution} as a result of
  the Reconstruction Theorem~\ref{reconstructiontheorem} only holds for test
  functions satisfying \eqref{propertyofkernel}. However, by {\cite[Lem.
  A.3]{Ottoandweber2019}}, it actually holds for any choice of smooth,
  non-negative function $\Psi$, with support in $B (0, 1)$, and such that
  $\Psi (z) = \Psi (- z)$ for every $z \in D$ and $\int \Psi (z) \mathd z = 1$
  - see also {\cite{caravennaandzambotti2021}}. That said, we aim to construct
  a suitable partition of unity in $D_{- \infty}^{T} \setminus D_0$.
  Consider a smooth cut-off function $\rho : \mathbb{R}_+ \rightarrow [0, 1]$
  with $\rho (r) = 0$ if $r \nin [1 / 2, 2]$ and
  \[ \sum_{k \in \mathbb{Z}} \rho (2^k r) = 1, \quad \tmop{for} \tmop{every}
     \quad r > 0 . \]
  Compare with e.g. the construction of Littlewood-Paley blocks in
  {\cite[Sec. 2]{bahouri2011fourier}}. Let $\varrho : \mathbb{R}
  \rightarrow [0, 1]$ be a smooth function with $\tmop{supp} (\varrho) \subset
  [- 1, 1]$ and with
  \[ \sum_{k \in \mathbb{Z}} \varrho (x + k) = 1, \quad \tmop{for}
     \tmop{every} \quad x \in \mathbb{R}. \]
  For every $n \in \mathbb{Z}$, let $\Xi_0^n \subset D_0$ be defined as
  \[ \Xi_0^n = \left\{ z = (t, x) \in D_{- \infty}^{T} \quad : \quad t = 0
     \infixand x_i \in 2^{- n} \mathbb{Z}, \quad i = 1, 2, \ldots, d \right\}
     . \]
  Now for $n \in \mathbb{N}$ and $z = (0, x) \in \Xi_0^n$, define another
  cut-off function $\rho_{n, z} : D_{- \infty}^{T} \setminus D_0 \rightarrow
  [0, 1]$, such that for every $\bar{z} = (\bar{t}, \bar{x})$,
  \[ \rho_{n, z} (\bar{z}) = \rho \left( 2^n | \bar{t} |^{\frac{1}{2}} \right)
     \varrho (2^n (\bar{x}_1 - x_1)) \cdots \varrho (2^n (\bar{x}_d - x_d)) .
  \]
  Observe that it satisfies
  \[ \sum_{n \in \mathbb{Z}} \sum_{z \in \Xi_0^n} \rho_{n, z} (\bar{z}) = 1,
     \quad \tmop{for} \tmop{every} \quad \bar{z} \in D_{- \infty}^{T}
     \setminus D_0 . \]
  Finally, for any $N \in \mathbb{N}$, let $\rho_N = \sum_{n \leqslant N}
  \sum_{z \in \Xi_0^n} \rho_{n, z}$. Then, for every distribution $\zeta \in
  C^{- \alpha} (D_{- \infty}^{T})$ with $\alpha < 2$ and every smooth
  compactly supported test function $\Psi$, it holds that
  \[ \lim_{N \rightarrow \infty} \langle \zeta, \Psi (1 - \rho_N) \rangle = 0
     . \]
  Indeed, for every $N \in \mathbb{N}$, one has
  \begin{eqnarray*}
    | \langle \zeta, \Psi (1 - \rho_N) \rangle | & \leqslant & \sum_{n > N}
    \sum_{z \in \Xi_0^n} | \langle \zeta, \Psi \rho_{n, z} \rangle |\\
    & = & \sum_{n > N} 2^{- (d + 2) n} \sum_{z \in \Xi_0^n} 2^{(d + 2) n} |
    \langle \zeta, \Psi \rho_{n, z} \rangle |\\
    & \lesssim & \sum_{n > N} 2^{- (d + 2) n} 2^{n d} 2^{n a} = \sum_{n > N}
    2^{- n (2 - \alpha)} \lesssim 2^{- N},
  \end{eqnarray*}
  where the second inequality comes from the fact that there are approximately
  $2^{n d}$ terms that contribute as $2^{n \alpha}$ in the sum over $z \in
  \Xi_0^n$ and $\lesssim$ hides a constant that depends on $\Psi$. This means
  we may define $f$ over $D_{- \infty}^{T}$ by setting $\langle f, \Psi
  \rangle \assign \lim_{N \rightarrow \infty} \langle \tilde{f}, \Psi \rho_N
  \rangle$, provided that we show that the limit $\langle \tilde{f}, \Psi
  \rho_N \rangle$ exists for any smooth compactly supported test function
  $\Psi$.
  
  Recall that for any fixed $\bar{z} = (\bar{t}, \bar{x}) \in D_+^{T}$, by
  the definition of $(f_{\alpha})_{\alpha \in A}$ in
  \eqref{boundonextendeddistribution}, the following holds true whenever
  $\bar{L} < \sqrt{\bar{t}}/2$ for any $\alpha \in A = \{ \gamma - 1
  - \kappa, - 1 - \kappa, - 2 \kappa, - \kappa \}$
  \begin{eqnarray*}
    \| \tilde{f}_{\alpha} \|_{\alpha, - 1 - \kappa, [0, T]} & = & \sup_{\tau
    \in (0, T]} \sup_{L < \frac{d_{\tau}}{2}} \sup_{z \in D_{\tau +
    L}^{T}} d_{\tau}^{- 1 - \kappa - \alpha} | \langle
    \tilde{f}_{\alpha}, \varphi_z^L \rangle | L^{- \alpha}\\
    & \geqslant & \bar{t}^{\frac{\alpha + 1 + \kappa}{2}} | \langle
    \tilde{f}_{\alpha}, \varphi_{\bar{z}}^{\bar{L}} \rangle | \bar{L}^{-
    \alpha}
  \end{eqnarray*}
  which implies
  \begin{equation}
    | \langle \tilde{f}_{\alpha}, \Psi_{\bar{z}}^{\bar{L}} \rangle |
    \lesssim \bar{t}^{- \frac{\alpha + 1 + \kappa}{2}} \| \tilde{f}_{\alpha}
    \|_{\alpha, - 1 - \kappa, [0, T]} \bar{L}^{\alpha} 
    \label{insertingblowupdistributions}
  \end{equation}
  where $\lesssim$ denotes the constant coming from $\Psi$, a compactly
  supported in $B (0, 1)$ test function, and $\Psi_{\bar{z}}^{\bar{L}}$ is its
  translated and rescaled version as in \eqref{defofmollifierkernel}. We want
  to show that
  \begin{equation}
    | \langle \tilde{f}, \Psi_z^L \rangle | \leqslant \interleave \tilde{f}
    \interleave_{\gamma - 1 - \kappa, - 1 - \kappa, (- \infty, 0) \cup (0,
    T]} L^{- 1 - \kappa} \label{finalcontrolonextendeddistribution}
  \end{equation}
  for any $z \in D_{- \infty}^{T}$ and $L > 0$. For that, fix any $z = (t,
  x) \in D_{- \infty}^{T}$ and $L > 0$. Consider two cases.
  
  \tmtextbf{Case 1:} $\sqrt{t} \geqslant 2 L$. Then, by
  \eqref{furthersplittingofrhs} and \eqref{insertingblowupdistributions} we
  arrive at
  \begin{equation}
    | \langle \tilde{f}, \Psi_z^L \rangle | \lesssim \sum_{\alpha \in A} t^{-
    \frac{\alpha + 1 + \kappa}{2}} \| \tilde{f}_{\alpha} \|_{\alpha, - 1 -
    \kappa, [0, T]} L^{\alpha} \lesssim L^{- 1 - \kappa} \sum_{\alpha \in A}
    \| \tilde{f}_{\alpha} \|_{\alpha, - 1 - \kappa, [0, T]},
    \label{controlonextendeddistributioneasycase}
  \end{equation}
  which by \eqref{boundonextendeddistribution} gives
  \eqref{finalcontrolonextendeddistribution}.
  
  \tmtextbf{Case 2:} $\sqrt{t} < 2 L$. We make use of the partition of unity
  constructed before. Rewrite $\Psi_z^L$ as
  \[ \Psi_z^L = \sum_{n \geqslant n_0} \sum_{w \in \Xi_0^n} \Psi_z^L \rho_{n,
     w}, \]
  where $n_0$ is the greatest integer such that $2^{- n_0} \geqslant 3 L$, and
  also set
  \[ \chi_{n, z w} \assign L^{d + 2} 2^{n (d + 2)} \Psi_z^L \rho_{n, w}, \]
  so that we may write
  \begin{equation}
    \langle \tilde{f}, \Psi_z^L \rho_N \rangle = \sum^N_{n \geqslant n_0} L^{-
    (d + 2)} 2^{- n (d + 2)} \sum_{w \in \Xi_0^n} \langle \tilde{f}, \chi_{n,
    z w} \rangle . \label{approximatingnearboundary}
  \end{equation}
  Note that each point in the support of $\chi_{n, z w}$ is at an equal
  (parabolic) distance to $t = 0$, which is of order $2^{- n}$ and also
  $\chi_{n, z w} \equiv 0$ whenever $d(w,z) \gtrsim L$, so that the number
  of terms that contribute to the sum over $w \in \Xi_0^n$ in
  \eqref{approximatingnearboundary} is bounded by $\lesssim (2^n L)^d$, so
  that a similar treatment as in \eqref{controlonextendeddistributioneasycase}
  yields
  \begin{eqnarray*}
    &  & | \langle \tilde{f}, \Psi_z^L \rho_N \rangle |\\
    & \lesssim & \sum^N_{n \geqslant n_0} L^{- (d + 2)} 2^{- n (d + 2)}
    \sum_{w \in \Xi_0^n} \langle \tilde{f}, \chi_{n, z w} \rangle\\
    & \lesssim & \sum^N_{n \geqslant n_0} L^{- (d + 2)} 2^{- n (d + 2)} (2^n
    L)^d \sum_{\alpha \in A} 2^{n (\alpha + 1 + \kappa)} 2^{- n \alpha} \|
    \tilde{f}_{\alpha} \|_{\alpha, - 1 - \kappa, [0, T]}\\
    & = & L^{- 2} \sum_{\alpha \in A} \| \tilde{f}_{\alpha} \|_{\alpha, - 1 -
    \kappa, [0, T]} \sum^N_{n \geqslant n_0} 2^{- n (2 - 1 - \kappa)}\\
    & \lesssim & \interleave \tilde{f} \interleave_{\gamma - 1 - \kappa, -
    1 - \kappa, (- \infty, 0) \cup (0, T]} L^{- 2} L^{2 - 1 - \kappa},
  \end{eqnarray*}
  since $2^{- n_0} \geqslant 3 L$ and $1 + \kappa < 2$. Since this control is
  independent of $N$, the proof is concluded.
\end{proof*}

\section{The massive equation}\label{sectionthemassiveequation}

In this section we look at \eqref{mainequation} with an extra mass term $m >
0$, which reads as
\begin{equation}
  \left\{\begin{array}{rll}
    (\partial_t + m^2 - \Delta) u & = & \sigma (u) \diamond \xi\\
    u_{| \nobracket t = 0} & = & u_0
  \end{array}\right. \label{massiveequation} .
\end{equation}
The goal of this section is to show that the extra $- m^2 u$ term in the
r.h.s. does not influence the small scale analysis, i.e. the interior
estimates of Section \ref{sectioninteriorestimate}, and helps in the large
scale estimates, introducing an exponential damping to the Maximum principle
in Section \ref{sectionmaximumprinciple}.

We discuss how the five steps in the introduction can be handled for
\eqref{massiveequation}. We first work on the domain $D_0^1$ and for that
assume $\Cnoise{,1} \backassign \Cnoise{}$ and $\Cdumb{,1} \backassign
\Cdumb{}$ as in \eqref{assumptiononstochasticobjects}. Also, consider
$T_{\star} = T_{\star} (C_{\star})$ is defined according to
\eqref{definitionstoppingtime}. In view of \eqref{defofU}, for any $z \in
D_0^1$, an equation equivalent to \eqref{equationforU} in the context of
\eqref{massiveequation} is
\[ (\partial_t - \mathLaplace) U_z = \sigma (u) \diamond \noiseb - \sigma (u (z)) \Pi_z \noise - m^2 u . \]
We show Theorem \ref{theointeriorestimate} and Corollary
\ref{corollaryimprovedinteriorestimate} still hold true in the context of
\eqref{massiveequation}. Recall that in Subsection \ref{sectionchoiceofT0} we
showed that for every $\tau \in (a, b]$, $L \leqslant d_{\tau}/4$, all
base points $z \in D_{\tau + L}^b$, all $w \in B (z, L)$ and scales $\ell
\leqslant L$, we have good control on $| (\partial_t - \mathLaplace)
(U_z)_{\ell} (w) |$. In view of \eqref{whatisUzlw}, we now have
\[ | (\partial_t - \mathLaplace) (U_z)_{\ell} (w) | \leqslant | \langle F_w,
   \varphi_w^{\ell} \rangle | + | \langle F_z - F_w, \varphi_w^{\ell} \rangle
   | + m^2 | \langle u, \varphi_w^{\ell} \rangle | \]
where the first two terms are controlled by \eqref{controlonUzlw}. For the new
term, we get
\[ m^2 | \langle u, \varphi_w^{\ell} \rangle | \leqslant m^2 \| u \|_{B (w,
   \ell)}, \]
which then implies
\begin{eqnarray*}
  &  & \| (\partial_t - \mathLaplace) (U_z)_{\ell} \|_{B (z, L)}\\
  & \lesssim & \Cnoise{} [\sigma (U)]_{\gamma, D_{\tau}^b, \frac{d_{\tau}}{2}}
  \ell^{\gamma - 1 - \kappa} + \Cdumb{} [\sigma' \sigma (u)]_{\gamma - 1
  + \kappa, D_{\tau}^b, \frac{d_{\tau}}{2}} \ell^{\gamma - 1 - \kappa}\\
  & + & \Cnoise{} [\sigma' (u) u_X]_{\gamma - 1, D_{\tau}^b,
  \frac{d_{\tau}}{2}} \ell^{\gamma - 1 - \kappa} + \Cdumb{} \ell^{- 2
  \kappa} + \Cnoise{} \| u_X \|_{D_{\tau}^b} \ell^{- \kappa}\\
  & + & \Cnoise{} ([\sigma (U)]_{\gamma, B (z, L)} \ell^{- 1 - \kappa}
  L^{\gamma} + \Cnoise{} \ell^{- 1 - \kappa} L^{1 - \kappa} + \| u_X \|_{B (z,
  L)} \ell^{- 1 - \kappa} L)\\
  & + & m^2 \| u \|_{D_{\tau}^b} .
\end{eqnarray*}
After applying the Schauder estimate Lemma \ref{lemmaschauderestimate} and
arguing as before, we arrive at
\begin{eqnarray*}
  {}[U]_{\gamma, 0, [a, b]} & \lesssim & T^{\frac{1 - \kappa}{2}} \Cnoise{} \| u
  \|_{D_a^b}^{\gamma - 1}  [U]_{\gamma, 0, [a, b]}\\
  & + & \left[ T^{\frac{1 - \kappa}{2}} \Cnoise{} + T^{1 - \kappa} (\Cnoise{}^2 +
  \Cdumb{}) \right] [U]_{\gamma, 0, [a, b]}\\
  & + & \left[ T^{\frac{1 - \kappa}{2}} \Cnoise{} + T^{1 - \kappa} (\Cnoise{}^2 +
  \Cdumb{}) \right] \| u \|_{D_a^b}\\
  & + & T^{\frac{(1 - \kappa) \gamma}{2}} \Cnoise{}^{\gamma} + T^{1 - \kappa}
  (\Cnoise{}^2 + \Cdumb{}) + T^{\frac{3 - 3 \kappa}{2}} (\Cdumb{}
  \Cnoise{} + \Cnoise{}^3)\\
  & + & T^{\frac{1 - \kappa}{2}} \Cnoise{} + [1 + T m^2] \| u \|_{D_a^b},
\end{eqnarray*}
since we have gained a factor $\ell^2$ after Schauder, which turns into $T$.
This is the same control as \eqref{finalboundbeforeabsorption} \refchange{plus 
$T m^2 \| u \|_{D_a^b} \lesssim
T^{\frac{1 - \kappa}{2}}m^2 \| u \|_{D_a^b}$
in the last line}.
If we now make a choice of $T$ that, in addition to
\eqref{absorbtioncondition}, 
also satisfies
\[ T^{\frac{1 - \kappa}{2}} m^2 \leqslant \frac{1}{16}, 
   \]
then the choice analogous to \eqref{conditionforT} will be
\begin{equation}
  T \lesssim m^{- \frac{1}{1 - \kappa}} \wedge \Cdumb{}^{- \frac{1}{1 -
  \kappa}} \wedge \Cnoise{}^{- \frac{2}{1 - \kappa}} C_{\star}^{- e (\gamma)},
  \quad \tmxspace \tmop{where} \quad e (\gamma) \assign \frac{2 (\gamma -
  1)}{1 - \kappa} . \label{choiceofTtoabsorbmass}
\end{equation}
This way, we conclude that whenever $t \in [T, T_{\star}]$, it holds that
\[ [U]_{\gamma, 0, [t - T, t]} \lesssim \| u \|_{D_{t - T}^t} \vee 1 \lesssim
   C_{\star}, \]
as in Theorem \ref{theointeriorestimate}. The same post-processing can be done
to obtain
\begin{equation}
  [U]_{\gamma_2, 0, [t - T_1, t]} \lesssim \| u \|_{D_{t - T}^t} \vee 1
  \lesssim C_{\star},
\end{equation}
as in Corollary \ref{corollaryimprovedinteriorestimate}, where $T_1 := T(\gamma_1)$
in \eqref{conditionforT}
for any $1 + \kappa < \gamma_1 \leqslant \gamma_2 \leqslant 2 - 2\kappa$.

The second step is to prove an analog of Theorem
\ref{theoinfinityschauderpostprocessing} for \eqref{massiveequation}. The
splitting $u = u_{1 } + u_2$ in Section \ref{sectionproofofshorttimescontrol}
for the case of \eqref{massiveequation} is given by $u_1$ solving
\eqref{thesplittingu1} and replacing \eqref{thesplittingu2} for $u_2$ by
\[ \left\{\begin{array}{rllll}
     (\partial_t - \Delta) u_2 & = & \sigma (u) \diamond \xi - m^2 u & \tmop{in} &
     D^{T}_+\\
     u_2 & = & 0 & \tmop{on}
     & D_0
   \end{array}\right. . \]
Extending $- m^2 u$ to negative times in $L^{\infty}$ by setting it to be
equal to zero for negative times is straightforward. The extension of $\sigma
(u) \xi$ follows as before. Therefore, in the same way as for the interior
estimate, the new mass term only adds a term $m^2 \| u \|_{D_0^{T_1}}$ to the
control of $\interleave \tilde{f} \interleave_{\gamma - 1 - \kappa, - 1 -
\kappa, (- \infty, 0) \cup (0, T]}$ in \eqref{boundonextendeddistribution}.
Again, using \eqref{useofschauderwithoutblowups}, we arrive at
\[ \begin{array}{lll}
     \| u_2 \|_{D_0^{T}} & \lesssim & \displaystyle T^{\frac{1 - \kappa}{2}} \Cnoise{} \|
     u \|^{\gamma_1 - 1}_{D_0^{T}}  [U]_{\gamma, 0, [0, T]}\\
     & + & \displaystyle \left[ T^{\frac{1 - \kappa}{2}} \Cnoise{} + T^{1 - \kappa}
     (\Cnoise{}^2 + \Cdumb{}) \right] [U]_{\gamma, 0, [0, T]}\\
     & + & \displaystyle \left[ T^{\frac{1 - \kappa}{2}} m^2 + T^{\frac{1 - \kappa}{2}}
     \Cnoise{} + T^{1 - \kappa} (\Cnoise{}^2 + \Cdumb{}) \right] \| u
     \|_{D_0^{T}}\\
     & + & \displaystyle T^{\frac{(1 - \kappa) \gamma}{2}} \Cnoise{}^{\gamma} + T^{1
     - \kappa} (\Cnoise{}^2 + \Cdumb{}) + T^{\frac{3 - 3 \kappa}{2}}
     (\Cdumb{} \Cnoise{} + \Cnoise{}^3),
   \end{array} \]
and by the same argument as in the proof of Theorem
\ref{theoinfinityschauderpostprocessing}, but now with the choice in
\eqref{choiceofTtoabsorbmass}, which allows for the absorption of the mass
contribution, we get
\[ \| u \|_{D_0^{T}} \leqslant \| u_1 \|_{D_0^{T}} + \| u_2 \|_{D_0^{T}}
   \leqslant \| u_0 \| + \frac{1}{2} \| u \|_{D_0^{T}} + \frac{1}{2} \]
which readily gives
\[ \| u \|_{D_0^{T}} \leqslant 2 (\| u_0 \| \vee 1) . \]

Now we are ready to tackle the Maximum principle step of Section
\ref{sectionmaximumprinciple}. Recall that we want to take $\gamma_1 = 1 +
\kappa + \delta$, for $\delta > 0$, $\gamma_2 = 2 - 2 \kappa$ and
\[ T_1 \sim m^{- \frac{1}{1 - \kappa}} \vee \Cdumb{}^{- \frac{1}{1 -
   \kappa}} \vee \Cnoise{}^{- \frac{2}{1 - \kappa}} C_{\star}^{- \frac{2 (\kappa
   + \delta)}{1 - \kappa}}, \]
which satisfies \eqref{choiceofTtoabsorbmass}. Then, repeating the arguments
in Section \ref{sectionmaximumprinciple}, we may rewrite the regularised
equation \eqref{massiveequation} for every $z \in D_{T_1}^{T_{\star}}$ and
scale $L \leqslant \sqrt{T_1}/2$ as
\[ (\partial_t - \mathLaplace) u_L (z) = a_L (z) + b_L (z) \cdot \nabla u_L
   (z) + c_L (z) - m^2 u_L (z), \]
where $a_L, b_L, c_L$ satisfy \eqref{controlALdominant1}-\eqref{controlonCL}
in Proposition \ref{regularisedequationwithtransport} with $\Cdumb{}
\vee \Cnoise{}^2 C_{\star}^{2 \kappa + 2 \delta}$ replaced by $m^2 \vee
\Cdumb{} \vee \Cnoise{}^2 C_{\star}^{2 \kappa + 2 \delta}$, due to the new
choice of $T_1$ in \eqref{choiceofTtoabsorbmass}. This leads to a new choice
of scale $\tilde{L}$, analogous to \eqref{choiceofL}, which is now equal to
\[ \tilde{L} = C^{- \frac{2}{3 - 2 \kappa}}_{\star} (m^2 \vee \Cdumb{}
   \vee \Cnoise{}^2 C_{\star}^{2 \kappa + 2 \delta})^{- \frac{1}{2 (1 -
   \kappa)}} \]
that again respects $\tilde{L} \leqslant \sqrt{T_1} / 2$. Analogously to
\eqref{equationforv}, we obtain that $v \assign u_{\tilde{L}}$ now solves
\begin{equation}
  \left\{\begin{array}{rll}
    (\partial_t - \Delta) v & = & b \cdot \nabla v - m^2 v + f\\
    v_{| t = T_1 \nobracket} & = & u_{\tilde{L} | t = T_1 \nobracket}
  \end{array}\right., \label{massivetransportequation}
\end{equation}
where $f$ satisfies a control analogous to \eqref{controlonfRHSofv}, given by
\[ \| f \|_{D_{T_1}^{T_{\star}}} \lesssim \left( m^{\frac{1 + \kappa}{1 -
   \kappa}} \Cnoise{} \vee \Cdumb{}^{\frac{1}{1 - \kappa}} \right)
   C^{\betadumb}_{\star} \vee \Cnoise{}^{\frac{2}{1 - \kappa}}
   {C^{\betanoise}_{\star}} , \]
where $\betanoise, \betadumb$ are given by \eqref{betanoiseofdelta} and
\eqref{betas}. Then, for every $t \in [T_1, T_{\star}], x \in
\mathbb{T}^d$, let $(X^{t, x}_s)_{s \in [T_1, t]}$ be the solution to the same
It{\^o}'s SDE as before,
\[ \left\{\begin{array}{l}
     \mathd X^{t, x}_s = b (t - s, X^{t, x}_s) \mathd s + \sqrt{2} \mathd
     B_s\\
     X^{t, x}_{T_1} = x
   \end{array}\right., \]
where $(B_t)_{t \geqslant 0}$ is a Brownian motion under the probability
measure $\mathbb{P}$. Then, denoting by $\mathbb{E}_x$ the expectation w.r.t.
$\mathbb{P}$ conditioned to $X^{t, x}_{T_1} = x$, we get that
$D_{T_1}^{T_{\star}} \ni (t, x) \mapsto v_t (x)$ given by
\[ v_t (x) =\mathbb{E}_x [e^{- m^2 (t - T_1)} v_{T_1} (X^{t, x}_t)] + 
   \mathbb{E}_x \left[ \int_{T_1}^t e^{- m^2 (s - T_1)} f (t - s, X^{t, x}_s)
   \tmop{ds} \right] \]
solves \eqref{massivetransportequation}. Therefore, we obtain
\begin{eqnarray*}
  \| v \|_{D_{T_1}^{T_{\star}}} & = & \sup_{T_1 \leqslant t \leqslant
  T_{\star}} \| v_t \|\\
  & \leqslant & e^{- m^2 (T_{\star} - T_1)} \| v_{T_1} \| + \frac{1}{m^2} (1
  - e^{- m^2 (T_{\star} - T_1)}) \| f \|_{D_{T_1}^{T_{\star}}}\\
  & \leqslant & e^{- m^2 (T_{\star} - T_1)} \| v_{T_1} \| + \frac{1}{m^2} \|
  f \|_{D_{T_1}^{T_{\star}}}\\
  & \leqslant & e^{- m^2 (T_{\star} - T_1)} \| v_{T_1} \|\\
  & + & C (\kappa, d, C_{\sigma}) \left[ \left( m^{\frac{1 + \kappa}{1 -
  \kappa}} \Cnoise{} \vee \Cdumb{}^{\frac{1}{1 - \kappa}} \right)
  C^{\betadumb}_{\star} \vee \Cnoise{}^{\frac{2}{1 - \kappa}}
  {C^{\betanoise}_{\star}}  \right]
\end{eqnarray*}
for a constant $C (\kappa, d, C_{\sigma}) > 0$ that only depends on $\kappa$,
$d$, and $C_{\sigma}$. 
Note that by the same argument as before, we conclude
that
\[ \| u - v \|_{D_{T_1}^{T_{\star}}} \lesssim 1, \]
since the estimates \eqref{ineqhL-h}, \eqref{unormbyUnorm} are the same for
$u$ solution to \eqref{massivetransportequation}. Hence we obtain that
\begin{equation}
  \begin{array}{lll}
    \| u \|_{D_{T_1}^{T_{\star}}} & \leqslant & \| v \|_{D_{T_1}^{T_{\star}}}
    + \| u - v \|_{D_{T_1}^{T_{\star}}}\\
    & \leqslant & e^{- m^2 (T_{\star} - T_1)} {\| u \|_{D_0^{T_1}}} \\
    & + & C (\kappa, d, C_{\sigma}) \left[ \left( m^{\frac{1 + \kappa}{1 -
    \kappa}} \Cnoise{} \vee \Cdumb{}^{\frac{1}{1 - \kappa}} \right)
    C^{\betadumb}_{\star} \vee \Cnoise{}^{\frac{2}{1 - \kappa}}
    {C^{\betanoise}_{\star}}  \right]
  \end{array} \label{boundrepresentationwithkilling}
\end{equation}
since $\| v_{T_1} \| = \| u_{\tilde{L}} (T_1, \cdot) \| \leqslant {\| u
\|_{D_0^{T_1}}} $. 
\refchange{
\begin{remark}
    Note that \eqref{boundrepresentationwithkilling} is an
    improved version of the bound
    \eqref{controlawayfrom0} in
    Theorem \ref{theoremtopostprocessintopolynimial}. In fact,
    $m^2>0$ produces a
    contraction factor 
    $e^{- m^2 (T_{\star} - T_1)}<1$
    in front of the term $\| u \|_{D_0^{T_1}}$
    in \eqref{boundrepresentationwithkilling},
    which eventually leads to 
    \eqref{diferenceequationformassive}-\eqref{yesitgrows}.
\end{remark}}
Putting everything together gives
\begin{eqnarray*}
  \| u \|_{D_0^{T_{\star}}} & \leqslant & \max \left\{ \| u \|_{D_0^{T_1}}, \|
  u \|_{D_{T_1}^{T_{\star}}} \right\}\\
  & \leqslant & \| u \|_{D_0^{T_1}} + C (\kappa, d, C_{\sigma}) \left[ \left(
  m^{\frac{1 + \kappa}{1 - \kappa}} \Cnoise{} \vee \Cdumb{}^{\frac{1}{1 -
  \kappa}} \right) C^{\betadumb}_{\star} \vee \Cnoise{}^{\frac{2}{1 -
  \kappa}} {C^{\betanoise}_{\star}}  \right]\\
  & \leqslant & 2 \| u_0 \| + C (\kappa, d, C_{\sigma}) \left[ \left(
  m^{\frac{1 + \kappa}{1 - \kappa}} \Cnoise{} \vee \Cdumb{}^{\frac{1}{1 -
  \kappa}} \right) C^{\betadumb}_{\star} \vee \Cnoise{}^{\frac{2}{1 -
  \kappa}} {C^{\betanoise}_{\star}}  \right]
\end{eqnarray*}
since $m > 0$, which now leads to a choice
\[ C_{\star} \gtrsim \max \left\{ \| u_0 \|, \left( m^{\frac{1 + \kappa}{1 -
   \kappa}} \Cnoise{} \right)^{\frac{1}{1 - \betadumb}}, 
   \Cdumb{}^{\frac{1}{(1 - \kappa) (1 - \betadumb)}},
   \Cnoise{}^{\frac{2}{(1 - \kappa) (1 - \betanoise)}} \right\} \]
which, analogous to Theorem \ref{maintheorem}, gives
\[ \| u \|_{D_0^1} \lesssim \max \left\{ \| u_0 \|, \left( m^{\frac{1 +
   \kappa}{1 - \kappa}} \Cnoise{} \right)^{\frac{1}{1 - \betadumb}},
   \Cdumb{}^{\frac{1}{(1 - \kappa) (1 - \betadumb)}},
   \Cnoise{}^{\frac{2}{(1 - \kappa) (1 - \betanoise)}} \right\} . \]
The main point of this section is to adapt the proof of Theorem
\ref{theorempolynomialgrowth}. For that, consider the same definitions as
before, for every $n \in \mathbb{N}$, the constants $\Cnoise{,n}$ and
$\Cdumb{,n}$ as in \eqref{assumptiononstochasticobjects},
$\tCnoise{,n + 1}$ and $\tCdumb{,n + 1}$ as in
\eqref{definitionofCtildenoise}, $Y_n \assign \| u \|_{D_{n - 1}^n}$ and
$T_{\star, n}$ as in \eqref{definitionofTstarN}. The control for $Y_1$ is the
one in the previous display, with $\Cdumb{} = \Cdumb{,1}$ and
$\Cnoise{} = \Cnoise{,1}$. Now assume control for $Y_n$ and proceed by induction.
By the same argument, we obtain that by choosing
\[ C_{\star, n + 1} \gtrsim \max \left\{ \| u_{| t = n \nobracket} \|, \left(
   m^{\frac{1 + \kappa}{1 - \kappa}} \Cnoise{,n + 1} \right)^{\frac{1}{1 -
   \betadumb}}, \Cdumb{,n + 1}^{\frac{1}{(1 - \kappa) (1 -
   \betadumb)}}, \Cnoise{,n + 1}^{\frac{2}{(1 - \kappa) (1 -
   \betanoise)}} \right\} \]
we have $T_{\star, n + 1} = n + 1$ and the control
\[ \| u \|_{D_n^{n + 1}} \leqslant C_{\star, n + 1} \lesssim Y_n + \left(
   m^{\frac{1 + \kappa}{1 - \kappa}} \Cnoise{,n + 1} \right)^{\frac{1}{1 -
   \betadumb}} \vee \Cdumb{,n + 1}^{\frac{1}{(1 - \kappa) (1 -
   \betadumb)}} \vee \Cnoise{,n + 1}^{\frac{2}{(1 - \kappa) (1 -
   \betanoise)}} \]
since $\| u_{| t = n \nobracket} \| \leqslant \| u \|_{D_{n - 1}^n} = Y_n$.
Now writing the analog of \eqref{boundrepresentationwithkilling} over the
whole interval $D_n^{n + 1}$, like in \eqref{intermediatepolynomialgrowth}, we
get that
\begin{eqnarray*}
  \| u \|_{D_n^{n + 1}} & \leqslant & e^{- m^2} {\| u \|_{D_{n - 1}^n}} \\
  & + & C (\kappa, d, C_{\sigma}) \left[ \left( m^{\frac{1 + \kappa}{1 -
  \kappa}} \tCnoise{,n + 1} \vee \tCdumb{, n + 1}
  ^{\frac{1}{1 - \kappa}} \right) C^{\betadumb}_{\star} \vee
  \tCnoise{,n + 1}^{\frac{2}{1 - \kappa}} {C^{\betanoise}_{\star}} 
  \right] .
\end{eqnarray*}
Plugging one estimate into the other and performing the same as before in
\eqref{differenceequationpolynomial}, gives rise to
\begin{equation}
  Y_{n + 1} \leqslant a Y_n + r_n {Y_n^{\betanoise}} ,
  \label{diferenceequationformassive}
\end{equation}
where we use the simplified notation
\begin{equation}
  a \assign e^{- m^2} < 1 \quad, \quad r_n \assign C (\kappa, d, C_{\sigma})
  \left( m^{\frac{1 + \kappa}{1 - \kappa}} \tCnoise{,n + 1} \vee
  \tCdumb{,n + 1}^{\frac{1}{1 - \kappa}} \vee \tCnoise{,n + 1}
  ^{\frac{2}{1 - \kappa}} \right)^{\frac{1}{1 - \betanoise}} .
  \label{constantscandr}
\end{equation}
We refer to Appendix \ref{Appendixdifferenceequations} to see how
\eqref{diferenceequationformassive} can be estimated recursively to obtain,
for every $n < t \leqslant n + 1$
\begin{equation}
  \| u \|_{D_0^t} \leqslant \max_{1 \leqslant i \leqslant n + 1} Y_i \leqslant
  \max \left\{ Y_1, \left( \frac{R_{n + 1}}{(1 - a)} \right)^{\frac{1}{1 -
  \beta}} \right\}, \quad \tmop{for} \quad R_n \assign \max_{1 \leqslant i
  \leqslant n} r_i, \label{yesitgrows}
\end{equation}
which gives \eqref{corollarymassive}. The proof of \eqref{growthofmoments} can
also be found in the Appendix \ref{Appendixdifferenceequations}.

\section*{Acknowledgements}
AC gratefully acknowledges support by the EPSRC through the Standard Grant “Multi-Scale Stochastic Dynamics with Fractional Noise” EP/V026100/1 and also the “Mathematics of Random Systems” CDT EP/S023925/1. 
HW was funded by the European Union (ERC, GE4SPDE, 101045082). GF and HW are funded by the Deutsche Forschungsgemeinschaft (DFG, German Research Foundation) under Germany’s Excellence Strategy EXC 2044 -390685587, Mathematics Münster: Dynamics–Geometry–Structure. We thank Salvador Esquivel
for his careful reading of a draft of the manuscript
and for numerous discussions.

\appendix

\section{Reconstruction and Integration} \label{appendixreconstructionandintegration}

\refchange{
The precise version of 
the Reconstruction Theorem {\cite[Thm 3.10]{Hai14}}
we use here is taken from 
{\cite[Thm 2.8]{Moinatandweber2020cpam}},
which was itself based on \cite[Prop. 1]{OSSW18}. See also \cite[Thm 5.1]{caravennaandzambotti2021}.
}
We
write
its proof in our setting with the \refchange{interest of using details of the proof within Section~\ref{sectionmaximumprinciple}.}

In Section \ref{sectionstrategyofproof} we
mentioned the choice of a special test function $\varphi \in \mathfrak F$ for which
formula \eqref{multiscalereconstructionstrategy} holds. We now describe its construction.
For a fixed $\psi \in \mathfrak F$, let, for any $L > 0$,
\begin{equation}
   \varphi^L = \lim_{n \rightarrow \infty} \varphi^{L, n}, \quad \tmop{where}
   \quad \varphi^{L, n} \assign \psi^{\frac{L}{2}} \ast \psi^{\frac{L}{2^2}}
   \ast \cdots \ast \psi^{\frac{L}{2^n}},
   \label{choiceofmollifiersemigroup}
\end{equation}
with the convention that $\varphi^{L, 0} = \delta$. Here $\ast$ denotes the
convolution over $D$. Note that for every $L > 0$
\begin{equation}
  \varphi^L = \psi^{\frac{L}{2}} \ast \varphi^{\frac{L}{2}} = \varphi^{L, 1}
  \ast \varphi^{\frac{L}{2}} . \label{propertyofkernel}
\end{equation}
Moreover, it also holds that
\[ \varphi^{L, m + n} = \varphi^{\frac{L}{2^n}, m} \ast \varphi^{L, n}, \]
which implies that for a function $f : D \rightarrow \mathbb{R}$, we have
\begin{equation}
  \left\langle \left\langle f, \varphi_{\cdot}^{\frac{L}{2^n}, 1}
  \right\rangle, \varphi_{z_1}^{L, n} \right\rangle = \left\langle f, \int
  \varphi^{\frac{L}{2^n}, 1} (\cdot - w) \varphi^{L, n} (w - z_1) \mathd w
  \right\rangle = \langle f, \varphi_{z_1}^{L, n + 1} \rangle 
  \label{onescaletotheother} .
\end{equation}
\begin{theorem}[Reconstruction]
  \label{reconstructiontheorem}Let $\tilde{\theta} > 0$ and $\Theta$ a finite
  subset of $(- \infty, \tilde{\theta}]$. Let $L \in (0, 1)$ and $z \in D$.
  For a given family of distributions $B (z, L) \ni z_1 \mapsto G_{z_1} \in
  \mathcal{D} (\mathbb{T}^d)$, assume that for every $\theta \in \Theta$ there
  exist a constant $C_{\theta} > 0$ and an exponent $\gamma_{\theta} \geqslant
  \tilde{\theta}$, such that for all $\ell \in (0, L)$ and every $z_1, z_2 \in
  B (z, L - \ell)$, it holds that
  \begin{equation}
    | \langle G_{z_1} - G_{z_2}, \varphi_{z_1}^{\ell} \rangle | \leqslant
    \sum_{\theta \in \Theta} C_{\theta} d(z_2,z_1)^{\gamma_{\theta} -
    \theta} \ell^{\theta} . \label{conditionforreconstruction}
  \end{equation}
  As a consequence, there exists a constant $C_{\mathcal{R}} > 0$ depending
  only on $\gamma_{\theta}$, such that
  \begin{equation}
    | \langle \mathcal{R}G - G_z, \varphi_z^L \rangle | \leqslant
    C_{\mathcal{R}} \sum_{\theta \in \Theta} C_{\theta} L^{\gamma_{\theta}},
  \end{equation}
  where the distribution $\mathcal{R}G$ is called the reconstruction of $G$
  and can be obtained as
  \refchange
  {
  \begin{equation}
    \langle \mathcal{R}G, \varphi_z^L \rangle = \lim_{N \rightarrow \infty}
    \left\langle \left\langle G_{\cdot}, \varphi_{\cdot}^{\frac{L}{2^N}}
    \right\rangle, \varphi^{L, N}_z \right\rangle = \lim_{N \rightarrow
    \infty} \int \left\langle G_{\bar{z}}, \varphi_{\bar{z}}^{\frac{L}{2^N}}
    \right\rangle \varphi_z^{L, N} (\bar{z}) \mathd \bar{z} .
    \label{definitionofreconstruction}
  \end{equation}
  In particular, for any $N \in \N$ it holds that
  \begin{equation}
    \begin{array}{ll}
    & \left| \left\langle \left\langle G_{\cdot},
    \varphi_{\cdot}^{\frac{L}{2^N}} \right\rangle, \varphi^{L, N}_z
    \right\rangle - \nobracket \langle G_z \nobracket, \varphi^L_z \rangle
    \right| \\
    & \displaystyle
    \leqslant \sum_{n = 0}^{N - 1} \int \int \left| \left\langle G_{z_2}
    - G_{z_1}, \varphi_{z_2}^{\frac{L}{2^{n + 1}}} \right\rangle \right|
    \psi^{\frac{L}{2^{n + 1}}} (z_2 - z_1) \varphi^{L, n} (z_1 - z) \mathd z_2
    \mathd z_1 .
  \end{array}
  \label{proofofreconstruction}
  \end{equation}
  }
\end{theorem}

\begin{proof}
  The following bound holds uniformly over $z \in D$. For every $N \in
  \mathbb{N}$, use \eqref{onescaletotheother} with $f : D \rightarrow
  \mathbb{R}$ given by $f (z) = \left\langle G_z, \varphi_z^{\frac{L}{2^{n +
  1}}} \right\rangle$, then \eqref{propertyofkernel} and
  \eqref{conditionforreconstruction} to get
  \begin{eqnarray}
    &  & \left| \left\langle \left\langle G_{\cdot},
    \varphi_{\cdot}^{\frac{L}{2^N}} \right\rangle, \varphi^{L, N}_z
    \right\rangle - \nobracket \langle G_z \nobracket, \varphi^L_z \rangle
    \right| \nonumber\\
    & = & \left| \sum_{n = 0}^{N - 1} \left\langle \left\langle \left\langle
    G_{\cdot}, \varphi_{\cdot}^{\frac{L}{2^{n + 1}}} \right\rangle,
    \varphi^{\frac{L}{2^n}, 1}_{\cdot} \right\rangle - \left\langle G_{\cdot},
    \varphi^{\frac{L}{2^n}}_{\cdot} \right\rangle, \varphi^{L, n}_z
    \right\rangle \right| \nonumber\\
    & = & \left| \sum_{n = 0}^{N - 1} \left\langle \left\langle \left\langle
    G_{\cdot}, \varphi_{\cdot}^{\frac{L}{2^{n + 1}}} \right\rangle,
    \varphi^{\frac{L}{2^n}, 1}_{\cdot} \right\rangle - \left\langle G_{\cdot},
    \varphi^{\frac{L}{2^{n + 1}}} \ast \varphi^{\frac{L}{2^n}, 1}_{\cdot}
    \right\rangle, \varphi^{L, n}_z \right\rangle \right| \nonumber\\
    & = & \hspace{-1.7mm} \begin{array}{l}
      \left| \displaystyle \sum_{n = 0}^{N - 1} \int \left( \int \left\langle G_{z_2},
      \varphi_{z_2}^{\frac{L}{2^{n + 1}}} \right\rangle
      \varphi^{\frac{L}{2^n}, 1}_{z_1} (z_2) \mathd z_2 \right. \right.\\
      \left. \left. - \displaystyle \int \left\langle G_{z_1}, \varphi_{z_2}^{\frac{L}{2^{n
      + 1}}} \right\rangle \varphi_{z_1}^{\frac{L}{2^n}, 1} (z_2) \mathd z_2
      \right) \varphi^{L, n}_z (z_1) \mathd z_1 \right|
    \end{array} \nonumber\\
    & \leqslant & \sum_{n = 0}^{N - 1} \int \int \left| \left\langle G_{z_2}
    - G_{z_1}, \varphi_{z_2}^{\frac{L}{2^{n + 1}}} \right\rangle \right|
    \psi^{\frac{L}{2^{n + 1}}} (z_2 - z_1) \varphi^{L, n} (z_1 - z) \mathd z_2
    \mathd z_1  \label{proofofreconstructioninside}
    \\
    & \leqslant & \sum_{n = 0}^{N - 1} \sum_{\theta \in \Theta} C_{\theta}
    \left( \frac{L}{2^{n + 1}} \right)^{\theta} \int \int | z_2 - z_1
    |^{\gamma_{\theta} - \theta} \psi^{\frac{L}{2^{n + 1}}} (z_2 - z_1)
    \varphi^{L, n} (z_1 - z) \mathd z_2 \mathd z_1 \nonumber\\
    & \leqslant & \sum_{n = 0}^{N - 1} \sum_{\theta \in \Theta} C_{\theta}
    \left( \frac{L}{2^{n + 1}} \right)^{\gamma_{\theta}} \int \varphi^{L, n}
    (z_1 - z) \mathd z_1 \nonumber\\
    & \leqslant & C_{\mathcal{R}} \sum_{\theta \in \Theta} C_{\theta}
    L^{\gamma_{\theta}} \nonumber
  \end{eqnarray}
  where the constant $C_{\mathcal{R}} > 0$ only depends on $\gamma_{\theta}$,
  being in particular independent of $N$ and $z \in D$.
  \refchange{
  Expression \eqref{proofofreconstructioninside} gives
  exactly \eqref{proofofreconstruction}.
  }
\end{proof}


\refchange{
The following Schauder estimate is taken from {\cite[Lem.
2.11]{Moinatandweber2020cpam}}, which is a version of \cite[Prop. 2]{OSSW18} that accounts for the blow ups near the boundary
and can be seen as a variant of
\cite[Thm 5.12, Prop. 6.16]{Hai14}. see also the recent work \cite{BCZ23}.

In {\cite[Lem. 2.11]{Moinatandweber2020cpam}},
the authors introduce a blow up near the
space-time boundary to overcome the
lack of boundary conditions
in the space-time periodic setting of \cite[Prop. 2]{OSSW18}.
Since we work on the
torus, the blow up only appears at the time boundary,
so we restated the theorem in our setting for clarity.
}

Recall that for $[a, b]
\subset [0, 1]$, $D_a^b \assign [a, b] \times \mathbb{T}^d$.

\begin{lemma}
  \label{lemmaschauderestimate}Let $1 < \gamma < 2$ and $\Beta$ a finite
  subset of $(- \infty, \gamma]$. For $a, b > 0$ such that $[a, b] \subset [0,
  1]$, let U be a bounded function defined over $D_a^b \times D_a^b$ and such
  that $U_z (z) = 0$ for all $z \in D_a^b$. For any $\tau \in (a, b]$ let
  $d_{\tau} \assign \sqrt{\tau - a}$ and assume that for any $L \leqslant
  d_{\tau}/4$, there exists a constant $M^{(1)}_{D_{\tau}^b, L} > 0$
  such that for all base points $z \in D_{\tau + L}^b$ and scales $\ell
  \leqslant L$, the following holds
  \begin{equation}
    \| (\partial_t - \Delta) (U_z)_{\ell} \|_{B (z, L)} \leqslant
    M^{(1)}_{D_{\tau}^b, L} \sum_{\beta \in \Beta} \ell^{\beta - 2} L^{\gamma
    - \beta} . \label{orderboundcondition}
  \end{equation}
  Furthermore, assume that for $L_1 \leqslant d_{\tau}/2$ and $L_2
  \leqslant d_{\tau}/4$, there exists a constant $M^{(2)}_{D_{\tau}^b,
  L_1, L_2} > 0$ such that for every $z_0 \in D_{\tau + L_1}^b$, $z_1 \in B
  (z_0, L_1)$, and $z_2 \in B (z_1, L_2)$, the following ``three-point
  continuity'' holds:
  \begin{equation}
    | U_{z_0} (z_2) - U_{z_0} (z_1) - U_{z_1} (z_2) | \leqslant
    M^{(2)}_{D_{\tau}^b, L_1, L_2} \sum_{\beta \in \Beta} | z_1 - z_0
    |^{\beta} d(z_1,z_2)^{\gamma - \beta} . \label{3ptcontinuitycondition}
  \end{equation}
  Then, we have that there exists a constant $C_S > 0$ which only depends on
  $\gamma$ and $\Beta$ such that
  \begin{equation}
    [U]_{\gamma, 0, [a, b]} \leqslant C_S \left[ \sup_{\tau \in (a, b]}
    d_{\tau}^{\gamma} \left( M^{(1)}_{D_{\tau}^b, \frac{d_{\tau}}{2}} +
    M^{(2)}_{D_{\tau}^b, \frac{d_{\tau}}{2}, \frac{d_{\tau}}{4}} \right) +
    \sup_{\tau \in (a, b]} \| U \|_{D_{\tau}^b, d_{\tau}} \right],
    \label{outputofschauder}
  \end{equation}
  where $\| U \|_{B, d}$ denotes the $L^{\infty}$ norm of $U$ restricted to
  the set $\{ z, w \in B : d(w,z) \leqslant d \}$.
\end{lemma}

The following lemma is inspired by {\cite[Cor.
2.12]{Moinatandweber2020cpam}} and provides bounds on the generalised gradient
$u_X$. Its proof becomes simpler in our case and we include it here.

\begin{lemma}
  \label{lemmacontrolongradient}Let $1 < \gamma < 2$. Fix $\tau \in (a, b)$,
  where $[a, b] \subset [0, 1] .$ Then, for the function $u_X$ in
  \eqref{generalisedgradient}, we have for any $r \in \left[ 0,
  \frac{d_{\tau}}{2} \right]$
  \begin{equation}
    \| u_X \|_{D_{\tau}^b} \leqslant r^{\gamma - 1} [U]_{\gamma, D_{\tau}^b,
    \frac{d_{\tau}}{2}} + r^{- 1} \| U \|_{D_{\tau}^b, r} .
    \label{boundonlinfinitygradient}
  \end{equation}
  If also \eqref{3ptcontinuitycondition} holds for all $z_0, z_1, z_2 \in
  D_{\tau}^b$, we also have
  \begin{equation}
    [u_X]_{\gamma - 1, D_{\tau}^b} \leqslant [U]_{\gamma, D_{\tau}^b,
    d_{\tau}} + M_{D_{\tau}^b, \frac{d_{\tau}}{4}, \frac{d_{\tau}}{4}}^{(2)} .
    \label{boundongammanormgradient}
  \end{equation}
\end{lemma}

\begin{proof}
  From the definition of $[U]_{\gamma, B}$ in \eqref{finitegammanormofU} and
  \eqref{generalisedgradient}, we get that for every $z, w \in D_{\tau}^b$
  with $d(z,w) < \frac{d_{\tau}}{2}$,
  \[ | u_X (z) \cdot X (w - z) | \leqslant [U]_{\gamma, D_{\tau}^b,
     \frac{d_{\tau}}{2}} d(z,w)^{\gamma} + \| U \|_{D_{\tau}^b, d(z,w)}
     . \]
 Since $D_{\tau}^b = [\tau, b] \times \mathbb{T}^d$, for every $r \in
  \left[ 0, \frac{d_{\tau}}{2} \right]$ and $z \in D_{\tau}^b$, we may find $w
  \in D_{\tau}^b$ such that $d(z,w) = r$ and
  \[ X (w - z) = \frac{u_X (z)}{| u_X (z) |} r, \]
  so that for this $w$ we get
  \begin{equation}
    | u_X (z) | r \leqslant r^{\gamma} [U]_{\gamma, D_{\tau}^b,
    \frac{d_{\tau}}{2}} + \| U \|_{D_{\tau}^b, r},
  \end{equation}
  which gives \eqref{boundonlinfinitygradient}.
  
  Again by \eqref{finitegammanormofU}, we obtain that for every $z_0 \in
  D_{\tau + \frac{d_{\tau}}{4}}^b$, $z_1 \in B \left( z_0, \frac{d_{\tau}}{4}
  \right)$, and $z_2 \in B \left( z_1, \frac{d_{\tau}}{4} \right)$,
  \begin{eqnarray*}
    &  & | U_{z_0} (z_2) - U_{z_0} (z_1) - U_{z_1} (z_2) - (u_X (z_0) - u_X
    (z_1)) \cdot X (z_2 - z_1) |\\
    & \leqslant & [U]_{\gamma, D_{\tau}^b, d_{\tau}} (d(z_0,z_2)^{\gamma}
    + d(z_0,z_1)^{\gamma} + d(z_1,z_2)^{\gamma}) .
  \end{eqnarray*}
  Now again since $D_{\tau}^b = [\tau, b] \times \mathbb{T}^d$, for every
  fixed $z_0 \neq z_1$, we may choose $z_2$ such that
  \[ X (z_2 - z_1) = \frac{u_X (z_1) - u_X (z_0)}{| u_X (z_1) - u_X (z_0) |}
    d(z_0,z_1)  \infixand d(z_1,z_2) = d(z_0,z_1), \]
  so by the three-point continuity \eqref{3ptcontinuitycondition}, we arrive
  at
  \begin{eqnarray*}
    | u_X (z_1) - u_X (z_0) | d(z_0,z_1) & = & | (u_X (z_1) - u_X (z_0))
    \cdot X (z_2 - z_1) |\\
    & \leqslant & \left( [U]_{\gamma, D_{\tau}^b, d_{\tau}} +
    M^{(2)}_{D_{\tau}^b, \frac{d_{\tau}}{4}, \frac{d_{\tau}}{4}} \right) 
    d(z_0,z_1)^{\gamma}
  \end{eqnarray*}
  since $d(z_0,z_2) \leqslant 2 d(z_0,z_1)$. This in turn gives
  \eqref{boundongammanormgradient}.
\end{proof}

\section{Auxiliary calculations}\label{AppendixB}

\subsection{Stochastic objects}

\refchange{For more details on the stochastic
estimates that are discussed below, see
\cite[Sec. 3.3]{dLF24}: Sections 3.3.1 and 3.3.2 provide full proofs in detail for the cases
of a spatial white noise and a
spatially correlated Wiener process, while
Section 3.3.3 expands the discussion below
based on \cite{HS16}.}

\subsubsection{2-d Sine-Gordon}
\label{AppendixstochasticobjectsSG}

We now discuss how the results in \cite{HS16}, where the case $\beta^2 \in (0,16\pi /3)$ 
for \eqref{equationforSG} is treated,
can be used to
obtain the corresponding order bounds for $\Pi \dumb_{a,b}$, $a,b \in \{c,s\}$, in our setting.
The corresponding objects in \cite{HS16} will be denoted with a tilde here.

The difference between our setting and the one in \cite{HS16} is the definition of
$\seqeps{\lollib}_a$, $a \in \{c,s\}$. While they use a truncated heat kernel
$K$ that is compactly supported and 
integrates polynomials to zero, we consider the full heat kernel. This is because we need
$\seqeps{\lollib}_a$ to solve $\heat \seqeps{\lollib}_a = \seqeps{\noiseb}_a$, while in their case
this is only true up to a smooth error. 

Let $\zeta$ be the $2\pi$-periodic in space
space-time white noise 
considered in Corollary \ref{corollaryglobalSG} and
$\seqeps{\zeta}(z) := \lb \zeta, \vphi_z^\veps \rb$, for $\vphi \in \Ftest$, 
be a regularisation of $\zeta$ at scale $\veps$. Recall that $\seqeps{\tilde Z}(z) := K \ast \seqeps{\zeta}(z)$
and the definition of the noise $\seqeps{\noiseb}$ in \eqref{defofnoisesSG2} as
\begin{equation*}
  \seqeps{\noiseb} := \veps^{-\frac{\beta^2}{4\pi}} e^{i \beta \seqeps{\tilde{Z}}} ,
  \quad
  \seqeps{\noiseb}_c := \veps^{-\frac{\beta^2}{4\pi}} \cos(\beta \seqeps{\tilde{Z}}) ,
  \quad
  \seqeps{\noiseb}_s := \veps^{-\frac{\beta^2}{4\pi}} \sin(\beta \seqeps{\tilde{Z}}) .
\end{equation*}
With the same post-processed heat kernel $K$, let $\seqeps{\tilde{\lollib}}_b := K \ast \seqeps{\noiseb}_b$ and
\begin{equation*}
  \seqeps{\tilde{\Pi}_z} \dumb_{a,b} := \left(\seqeps{\tilde{\lollib}}_b - \seqeps{\tilde{\lollib}}_b(z)\right)
  \seqeps{\noiseb}_a
  - \frac{1}{2} \seqeps{C} \delta_{a b} ,
  \quad \text{for} \quad
  a,b \in \{c,s\} ,
\end{equation*}
where 
the constant
$\seqeps{C} := \E[\seqeps{\tilde{\lollib}}(0) \seqeps{\bar \noiseb}(0)] \sim \veps^{-\beta^2/4\pi}$ 
(see \cite[Rem. 4.2]{HS16}).
Here, the point $z=0$ in 
\refchange{$\seqeps{C}$}
is arbitrary since the objects $\seqeps{\noiseb}$ and $\seqeps{\tilde \lollib}$
are stationary. 
This way, \cite[Thm 2.8]{HS16} gives convergence
of $\seqeps{\tilde{\Pi}}$ to a $\tilde{\Pi}$ which satisfies the OBs
in \eqref{orderbounds}-\eqref{orderboundlollipop} with
$\E[\Cnoise{,n}^p], \E[\Cdumb{,n}^p] < \infty$
uniformly in $\veps>0$, for any $\kappa > \beta^2/4\pi - 1$, $n \in \N$ and $p \geqslant 1$, where
\refchange{
\begin{equation*}
  \tCnoise{,n} := 
  \max_{\btau \in \{\noise_a, \lolli_a, \Xnoise_a \; : \; a=c,s\}} [\tilde\Pi;\btau]_{|\btau|,D_{n - 1}^n} ,
  \; \; 
  \tCdumb{,n} := \max_{a,b \in \{c,s\}} [\tilde\Pi;\dumb_{a,b}]_{- 2 \kappa, D_{n-1}^n}
\end{equation*}
}are as in \eqref{assumptiononstochasticobjects} in view of Remark \ref{remarkmultinoise},
\refchange{with $\Pi$ replaced by $\tilde \Pi$}.
For $b \in \{c,s\}$, define $\seqeps{\lollib}_b$ as
\begin{equation*}
  \seqeps{\lollib}_b(z) := \int_{0}^t \int_{\T^2} p_{t-s}(y - x) \seqeps{\noiseb}_b(s,y) \mathd s \mathd y,
  \quad
  z=(t,x) \in D , 
\end{equation*}
where $p_s(y) := (4 \pi s)^{-1} \sum_{k \in \Z^2} e^{-|y - 2\pi k|^2/4s}$ denotes the periodisation of
the standard heat kernel.
So $\seqeps{\lollib}_b$ solves $\heat \seqeps{\lollib}_b = \seqeps{\noiseb}_b$ starting from $0$ at $t=0$.
\refchange{
On the other hand, if we set
$\seqeps{R}_{\noiseb,b} := \heat \seqeps{\tilde{\lollib}}_b - \seqeps{\noiseb}_b$
and
the difference 
$\seqeps{\mathfrak R}_{\noiseb,b} := \seqeps{\tilde{\lollib}}_b - \seqeps{\lollib}_b$,
then it holds that
}
$\heat \seqeps{\mathfrak R}_{\noiseb,b} = \seqeps{R}_{\noiseb,b}$. 
\refchange{Furthermore, it can be shown that} 
the error
\refchange{
$\seqeps{R}_{\noiseb,b}$
converges to $R_{\noiseb,b}$}
in $C^\gamma$ for any $\gamma \geqslant 0$
in the limit $\veps \downarrow 0$. 
\refchange{Similarly, one can show that}
also
$\seqeps{\mathfrak R}_{\noiseb,b} \to \mathfrak R_{\noiseb,b}$
in $C^\gamma$ for any $\gamma \geqslant 0$ as $\veps \downarrow 0$. 
Now, 
let
\begin{equation*}
  \seqeps{\Pi_z} \dumb_{a,b} := \left(\seqeps{\lollib}_b - \seqeps{\lollib}_b(z)\right)
  \seqeps{\noiseb}_a
  - \frac{1}{2} \seqeps{C} \delta_{a b} ,
  \quad \text{for} \quad
  a,b \in \{c,s\} ,
\end{equation*}
for the same constant $\seqeps{C}$.
Therefore, we obtain the 
control on $\Cnoise{,n}, \Cdumb{,n}$ in \eqref{defofconstantsOBforSG}
\refchange{(now w.r.t. $\Pi$ itself) from the control
on $\tCnoise{,n}, \tCdumb{,n}$ defined above}, 
since
\begin{equation*}
   \seqeps{\tilde{\Pi}_z} \dumb_{a,b} - \seqeps{\Pi_z} \dumb_{a,b}
   = \left(\seqeps{\mathfrak R}_{\noiseb,b}
   - \seqeps{\mathfrak R}_{\noiseb,b}(z) \right) \seqeps{\noiseb}_a 
   ,
\end{equation*}
which can be defined uniquely using
Theorem \ref{reconstructiontheorem} (Reconstruction). 
In fact, we have that for any $b \in \{c,s\}$, $1 + \kappa < \gamma < 2$ and $B \subset D$
\begin{equation*}
  [\seqeps{\mathfrak R}_{\noiseb,b}]_{\gamma, B} := 
  \sup_{z \in B} \sup_{w \neq z \in B}
  \frac{| \seqeps{\mathfrak R}_{\noiseb,b}(w) - \seqeps{\mathfrak R}_{\noiseb,b}(z)
  - \nabla \seqeps{\mathfrak R}_{\noiseb,b} (z) \cdot (\Pi_z \X) (w) |}
  {d (z,w)^{\gamma}} < \infty
\end{equation*}
and $[\nabla \seqeps{\mathfrak R}_{\noiseb,b}]_{\gamma-1,B} < \infty$
uniformly in $\veps > 0$,
\refchange{
since $\seqeps{\mathfrak R}_{\noiseb,b}$ is smooth
uniformly in $\veps > 0$
(here $\nabla \seqeps{\mathfrak R}_{\noiseb,b}$
denotes the standard gradient of a $C^1$ function).}
So if we let $\seqeps{\mathcal G}_{a,b}(z) :=
\seqeps{\mathfrak R}_{\noiseb,b}(z) \seqeps{\Pi}_z \noise_a
+ \nabla \seqeps{\mathfrak R}_{\noiseb,b} (z) \cdot \seqeps{\Pi}_z \Xnoise_a$, then
\begin{eqnarray*}
  \seqeps{\mathcal G}_{a,b}(w) - \seqeps{\mathcal G}_{a,b}(z)
  & = & \left(\seqeps{\mathfrak R}_{\noiseb,b}(w) - \seqeps{\mathfrak R}_{\noiseb,b}(z)
  - \nabla \seqeps{\mathfrak R}_{\noiseb,b} (z) \cdot (\Pi_z \X) (w)\right) \seqeps{\Pi}_z \noise_a \\
  & + & \left( \nabla \seqeps{\mathfrak R}_{\noiseb,b} (w)
  - \nabla \seqeps{\mathfrak R}_{\noiseb,b} (z) \right) \seqeps{\Pi}_z \Xnoise_a ,
\end{eqnarray*}
which yields that for for any $a,b \in \{c,s\}$, $L>0$ and $z$ such that $B(z,L) \subset D$
\begin{eqnarray*}
  & & |\lb \seqeps{\Pi}_z \dumb_{a,b} , \vphi_z^L \rb | \\
  & \leqslant & [ \seqeps{\tilde \Pi} ; \dumb_{a,b} ]_{-2\kappa} L^{-2\kappa} 
  + \|\nabla \seqeps{\mathfrak R}_{\noiseb,b} \|_{B(z,L)} [\seqeps{\Pi} ; \Xnoise_a]_{-\kappa} L^{-\kappa}
  \\
  & + & 
  \left( [\seqeps{\mathfrak R}_{\noiseb,b}]_{\gamma, B(z,L)} [\seqeps{\Pi} ; \noise_a]_{-1-\kappa} 
  + [\seqeps{\nabla \mathfrak R}_{\noiseb,b}]_{\gamma - 1, B(z,L)} [\seqeps{\Pi} ; \Xnoise_a]_{-\kappa} \right)
  L^{\gamma - 1 - \kappa}
\end{eqnarray*}
uniformly in $\veps > 0$. Since $\seqeps{\mathfrak R}_{\noiseb,b} = (K - p) \ast \seqeps{\noiseb}_b$, 
where $(K - p)$ is a smooth kernel that agrees with the heat kernel outside a ball of radius $1$ around the
origin and it is zero inside this ball,
all the semi-norms and norms of $\seqeps{\mathfrak R}_{\noiseb,b}$ 
\refchange{can be shown to be}
bounded by the norm
$[\seqeps{\Pi} ; \noise_b]_{-1-\kappa}$, uniformly in $\veps > 0$.

In view of \eqref{decorateddumb}, since $\sin'(x)\cos(x) + \cos'(x)\sin(x) = 0$,
the renormalisation constants in $\seqeps{\Pi}_z \dumb{a,b}$ cancel out in
the renormalised equation
for $\ueps = \seqeps{\Phi} - \seqeps{\tilde{Z}}$, which becomes \eqref{renormalisedequationforSG}.

\subsubsection{Multiplicative SPDEs in the Da Prato-Zabczyk regime}
\label{AppendixstochasticobjectsDZSPDE}

Recall the definition of the noise we are considering in \eqref{cylindricalwienerRDS}
\begin{equation*}
  W = \sum_{k \in \Z^d} e_k c_k B_k ,
  \quad \text{where} \quad 
  c_k := \frac{1}{(2\pi)^\frac{d}{2}(1 + |k|^2)^{\frac{d}{4} - \frac{\delta}{2}}} ,
\end{equation*}
where $e_k(x) = e^{i k \cdot x}$ and $B_k$ are independent (up to $\bar{B}_k = B_{-k}$)
complex-valued Brownian motions.
For any $z=(t,x)\in D$, let $\lollib(z) := B_0(t) + \tilde{\lollib}(z)$, for
\begin{equation}
   \tilde{\lollib}(z) := \sum_{k \in \Z^d, k\neq0} 
   \left(e^{-|k|^2 t} \eta_k + c_k \int_0^t e^{-|k|^2(t-s)} \mathd B_k(s) \right) e_k(x)
   \;, 
   \label{fourierserieslolliRDS}
\end{equation}
where $\{\eta_k\}_{k \in \Z^d, k\neq 0}$ are independent (up to $\bar{\eta_k} = \eta_{-k}$) mean zero complex Gaussians
with variance $\E[\eta_k \eta_{-k}] = c_k^2 / 2|k|^2$,
which make $\tilde \lollib$ stationary. Thus $\lollib$
satisfies \eqref{definitionlollipop} starting from $\lollib_{|t=0} = \tilde{\lollib}$. Again despite $\lollib$ not being stationary in time,
its increments are. 
A bound like $\E[|\lollib(z) - \lollib(w)|^p] \lesssim C_{p,W} d(z,w)^{(1-\kappa)p}$ 
holds for any $\kappa > \delta$.
Therefore, we obtain that $[\Pi ; \lolli]_{1-\kappa,D_{n-1}^n} < \infty$ in \eqref{orderboundlollipop} a.s. for any 
$n \in \N$ and $\kappa>\delta$. 
Denote by $\noiseb^{(\veps)}$ and $\lollib^{(\veps)}$ the regularisation
of the objects in \eqref{cylindricalwienerRDS} and \eqref{fourierserieslolliRDS}
at scale
$\veps > 0$. In space we consider
ultraviolet cut-off and 
in time via replacing $B_k$ by
$B^{(\veps)}_k(t) := \int \rho^{\veps} (t-s) \mathd B_k(s)$,
where $\rho^{\veps} := \veps^{-2} \rho(t \veps^{-2})$
for a smooth, even, compactly supported $\rho$ such that $\int \rho = 1$. This way,
$\Pi_z \dumb$ can be constructed as
\begin{equation*}
  \Pi_z \dumb =
  \lim_{\veps \downarrow 0} 
  \left(\lollib^{(\veps)} - \lollib^{(\veps)}(z)\right)
  \noiseb^{(\veps)} - C^{(\veps)} ,
\end{equation*}
where 
$C^{(\veps)} :=
  \E[\tilde\lollib^{(\veps)}(1,0) \noiseb^{(\veps)}(1,0)]$
is such that
\begin{align*}
  C^{(\veps)} 
  & = \frac{1}{(2\pi)^d} \sum_{0<|k| \leqslant 1/\veps} 
  c_k^2 \int_0^t e^{-|k|^2 s} \rho^{\veps} \ast \rho^{\veps} (s) \mathd s \\
  & \sim \frac{1}{(2\pi)^d} \sum_{0<|k| \leqslant 1/\veps} c_k^2
  \sim \frac{\veps^{-2\delta}}{(2\pi)^d},
\end{align*}
i.e., it diverges as the
\refchange{$L^2$-}trace
of the operator $(1 - \Delta)^{- (d / 2 - \delta)}$ as the spatial regularisation
is removed.
%
For these objects, it can be shown that
the OBs
in \eqref{orderbounds}-\eqref{orderboundlollipop} satisfy
$\E[\Cnoise{,n}^p] < \infty$ , $\E[\Cdumb{,n}^p] < \infty$
  uniformly in $\veps> 0$, for every $p \geqslant 1$, where
 and $\Cnoise{,n}, \Cdumb{,n}$ are as in 
\eqref{assumptiononstochasticobjects} - see \cite[Sec. 4.5]{HP15}, where the more irregular case 
of $\kappa > 1/2$ is treated.

Furthermore, if we consider the solutions $u^{(\veps)}$
to \eqref{mainequationrenormalised} w.r.t. $\noiseb^{(\veps)}$,
$u^{(\veps)}$ converges 
to the It\^{o} solution $u$ 
to \eqref{mainequation} as $\veps \downarrow 0$ (see \cite{HP15}).
Since our estimates
in Theorems \ref{maintheorem} and \ref{theorempolynomialgrowth} are independent of
$C^{(\veps)}$, they pass to the limit $\varepsilon \downarrow 0$. 

\subsection{Difference inequalities in Sections
\ref{sectionmaximumprinciple} and \ref{sectionthemassiveequation}
}\label{Appendixdifferenceequations}

\subsubsection{Pathwise growth}

Here we shall address how to estimate a solution to the difference inequalities
in \eqref{differenceequationpolynomial} and
\eqref{diferenceequationformassive}. Starting with
\eqref{differenceequationpolynomial}, let us rewrite the inequality as
\begin{equation}
  \left\{\begin{array}{lll}
    Y_{n + 1} & \leqslant & Y_n + q_n {Y_n^{\beta}} \\
    Y_1 & = & \| u \|_{D_0^1}
  \end{array}\right., \label{differenceequationpolynomialsolving}
\end{equation}
where the control on $\| u \|_{D_0^1}$ is given by \eqref{mainresultbound} and
\[ q_n \assign C (\kappa, d, C_{\sigma}) \left( \tCdumb{,n + 1}
   ^{\frac{1}{1 - \kappa}} \vee \tCnoise{,n + 1}^{\frac{2}{1 -
   \kappa}} \right)^{\frac{1}{1 - \betanoise}}  \quad \tmop{and} \quad \beta
   \assign \betanoise < 1 . \]
Further, let us make one more definition, for every $n \in \mathbb{N}$, let
\begin{equation}
  Q_n \assign \max_{1 \leqslant i \leqslant n} q_i \leqslant C (\kappa, d,
  C_{\sigma}) \max_{1 \leqslant i \leqslant n + 1} \left\{ \Cdumb{,i}
  ^{\frac{1}{1 - \kappa}}, \Cnoise{,i}^{\frac{2}{1 - \kappa}}
  \right\}^{\frac{1}{1 - \betanoise}} . \label{auxiliarQn}
\end{equation}
With this at hand, let us make an ansatz for
\eqref{differenceequationpolynomialsolving} using the solution to the
corresponding ODE, which reads as
\begin{equation}
  \left\{\begin{array}{lllll}
    y'_t & = & q_t {y_t^{\beta}}  & , & t \in (1, \infty)\\
    y_1 & = & Y_1 & . & 
  \end{array}\right., \label{ODEforpolynomialgrowth}
\end{equation}
where
\begin{equation}
  q_t = \sum_{n = 1}^{\infty} q_n \mathbbm{1}_{\{ n < t \leqslant n + 1 \}} .
\end{equation}
The solution to \eqref{ODEforpolynomialgrowth} then reads as
\[ y_t = \left( y^{1 - \beta}_1 + (1 - \beta) \int_1^t q_s \mathd s
   \right)^{\frac{1}{1 - \beta}} \quad, \quad t \geqslant 1 . \]
Therefore, we make the ansatz that $Y_n$ solution to
\eqref{differenceequationpolynomialsolving} satisfies
\begin{equation}
  Y_n \leqslant (Y^{1 - \beta}_1 + (1 - \beta) Q_{n - 1} (n - 1))^{\frac{1}{1
  - \beta}} \label{theansatzforYn}
\end{equation}
and we prove this by induction using
\eqref{differenceequationpolynomialsolving}. For $n = 1$ we simply get $Y_1
\leqslant Y_1$. To show that
\begin{equation}
  Y_n \leqslant (Y^{1 - \beta}_1 + (1 - \beta) Q_{n - 1} (n - 1))^{\frac{1}{1
  - \beta}} \Longrightarrow Y_{n + 1} \leqslant (Y^{1 - \beta}_1 + (1 - \beta)
  Q_n n)^{\frac{1}{1 - \beta}}, \label{inductionstep}
\end{equation}
we use \eqref{differenceequationpolynomialsolving} to write
\[ Y_{n + 1} \leqslant (Y^{1 - \beta}_1 + (1 - \beta) Q_{n - 1} (n -
   1))^{\frac{1}{1 - \beta}} + q_n (Y^{1 - \beta}_1 + (1 - \beta) Q_{n - 1} (n
   - 1))^{\frac{\beta}{1 - \beta}} . \]
Now note that by the definition of $Q_n$ in \eqref{auxiliarQn},
\[ Q_{n - 1} (n - 1) + q_n \leqslant Q_n n, \]
Letting $\bar{Y} \assign Y^{1 - \beta}_1 + (1 - \beta) Q_{n - 1} (n - 1)$,
\eqref{inductionstep} becomes a consequence of
\[ \bar{Y}^{\frac{1}{1 - \beta}} + q_n \bar{Y}^{\frac{\beta}{1 - \beta}}
   \leqslant (\bar{Y} + (1 - \beta) q_n)^{\frac{1}{1 - \beta}} . \]
The above can be seen from the fact that, for every $s > 0$,
\[ \frac{\mathd}{\mathd s} (\bar{Y} + s)^{\frac{1}{1 - \beta}} = \frac{1}{1 -
   \beta} (\bar{Y} + s)^{\frac{\beta}{1 - \beta}}, \]
so we may write
\[ (\bar{Y} + (1 - \beta) q_n)^{\frac{1}{1 - \beta}} = \bar{Y}^{\frac{1}{1 -
   \beta}} + \int_0^{(1 - \beta) q_n} \frac{1}{1 - \beta} (\bar{Y} +
   s)^{\frac{\beta}{1 - \beta}} \mathd s \geqslant \bar{Y}^{\frac{1}{1 -
   \beta}} + q_n \bar{Y}^{\frac{\beta}{1 - \beta}}, \]
which concludes the proof of \eqref{theansatzforYn}. This together with
\eqref{auxiliarQn} gives \eqref{finalpolynomialgrowth}.

Now we move to \eqref{diferenceequationformassive}, which corresponds to
\begin{equation}
  \left\{\begin{array}{lll}
    Y_{n + 1} & \leqslant & a Y_n + r_n {Y_n^{\beta}} \\
    Y_1 & = & \| u \|_{D_0^1}
  \end{array}\right., \label{polynomialvsexponential}
\end{equation}
where $0 < a < 1$ and $r_n$ are defined in \eqref{constantscandr}, and $\beta
\assign \betanoise < 1$. Now, we make an ansatz in the following way. Assume
that for large $n \in \mathbb{N}$, $Y_n \leqslant Y_{n + 1}$. Since $Y_1 > 0$
we may divide \eqref{polynomialvsexponential} by $Y_{n + 1}$ on both sides to
obtain
\[ 1 \leqslant a \frac{Y_n}{Y_{n + 1}} + r_n \frac{{Y_n^{\beta}} }{Y_{n + 1}}
   \leqslant a + \frac{r_n }{Y_n^{1 - \beta}} \]
which implies that for large $n \in \mathbb{N}$,
\[ Y_n < \left( \frac{r_n}{(1 - a)} \right)^{\frac{1}{1 - \beta}} \]
otherwise we get a contradiction to \eqref{polynomialvsexponential}. As an
immediate consequence, we see that if $r_n$ is bounded, so it is $Y_n$. We
shall prove by induction that \eqref{polynomialvsexponential} implies
\[ Y_n \leqslant \max \left\{ Y_1, \left( \frac{R_n}{(1 - a)}
   \right)^{\frac{1}{1 - \beta}} \right\}, \quad \tmop{for} \quad R_n \assign
   \max_{1 \leqslant i \leqslant n} r_i, \]
which is clearly true for every $n \in \mathbb{N}$ such that $R_n \leqslant (1
- a) Y_1^{1 - \beta}$. Assume then that
\[ R_n > (1 - a) Y_1^{1 - \beta}, \]
and then
\begin{eqnarray*}
  Y_{n + 1} & \leqslant & a Y_n + r_n {Y_n^{\beta}} \\
  & = & \frac{a R_n^{\frac{1}{1 - \beta}}}{(1 - a)^{\frac{1}{1 - \beta}}} +
  \frac{R_n^{\frac{1}{1 - \beta}}}{(1 - a)^{\frac{\beta}{1 - \beta}}} = \left(
  \frac{R_n}{(1 - a)} \right)^{\frac{1}{1 - \beta}} \leqslant \left(
  \frac{R_{n + 1}}{(1 - a)} \right)^{\frac{1}{1 - \beta}} .
\end{eqnarray*}
which gives \eqref{yesitgrows}.

\subsubsection{Growth of moments}

Let us now come back to \eqref{polynomialvsexponential} with assumption
\eqref{assumptiononmoments} on the moments of $\Cnoise{,n}$ and $\Cdumb{,n}$
in force. Letting $Z_n \assign \mathbb{E} {[Y_n^p]^{1 / p}} $, for $p
\geqslant 1$, and using \eqref{polynomialvsexponential}, we get
\[ Z_{n + 1} \leqslant a Z_n +\mathbb{E} [r_n^p Y_n^{p \beta}]^{\frac{1}{p}}
   \leqslant a Z_n +\mathbb{E} \left[ r_n^{\frac{p}{1 - \beta}}
   \right]^{\frac{1 - \beta}{p}} \mathbb{E} [Y_n^p]^{\frac{\beta}{p}} \]
by H{\"o}lder's inequality. Recall 
that $\nu^{- 1} =
(1 - \kappa) (1 - \beta)^2$ and from \eqref{definitionofCtildenoise} and
\eqref{constantscandr} that
\begin{eqnarray*}
  r_n^{\frac{p}{1 - \beta}} & = & C (\kappa, d, C_{\sigma})^{\frac{p}{1 -
  \beta}} \left( m^{\frac{1 + \kappa}{1 - \kappa}} \tCnoise{,n + 1} \vee
  \tCdumb{,n + 1}^{\frac{1}{1 - \kappa}} \vee \tCnoise{,n + 1}
  ^{\frac{2}{1 - \kappa}} \right)^{\frac{p}{(1 - \beta)^2}}\\
  & \leqslant & C (m, \kappa, d, p, C_{\sigma}) (\Cdumb{,n}^{\nu p} +
  \Cdumb{,n + 1}^{\nu p} + \Cnoise{,n}^{2 \nu p} + \Cnoise{,n + 1}^{2 \nu
  p})
\end{eqnarray*}
where $C (m, \kappa, d, p, C_{\sigma}) > 0$ is a constant that only depends on
$m, \kappa, d, p$ and $C_{\sigma}$. Thus, we see that for every $p \geqslant
\nu$ and every $n \in \mathbb{N}$, by \eqref{assumptiononmoments}
\begin{eqnarray*}
  \mathbb{E} \left[ r_n^{\frac{p}{1 - \beta}} \right]^{\frac{1 - \beta}{p}} &
  \lesssim & \mathbb{E} [\Cdumb{,n}^{\nu p}]^{\frac{1 - \beta}{p}}
  +\mathbb{E} [\Cdumb{,n + 1}^{\nu p}]^{\frac{1 - \beta}{p}}
  +\mathbb{E} [\Cnoise{,n}^{2 \nu p}]^{\frac{1 - \beta}{p}} +\mathbb{E}
  [\Cnoise{,n + 1}^{2 \nu p}]^{\frac{1 - \beta}{p}}\\
  & \lesssim & \Mdumb^{\nu (1 - \beta)} + \Mnoise^{2
  \nu (1 - \beta)} \backassign M_p \quad,
\end{eqnarray*}
where $\lesssim$ hides a constant $C (m, \kappa, d, p, C_{\sigma}) > 0$. We
then arrive to a difference inequality for $Z_n$ which reads as
\begin{equation}
  \left\{\begin{array}{lll}
    Z_{n + 1} & \leqslant & a Z_n + M_p {Z_n^{\beta}} \\
    Z_1 & = & \mathbb{E} [\| u \|^p_{D_0^1}]^{\frac{1}{p}}
  \end{array}\right. . \label{differenceequationformoments}
\end{equation}
Since $M_p$ does not depend on $n$, if we consider $\tilde{Z}_n$ as the
solution to \eqref{differenceequationformoments} with an equal sign instead of
the inequality, $\tilde{Z}_n$ converges to a fix point independent of the
initial condition $Z_1$. The limit has to be
\[ \bar{Z} \assign \lim_{n \rightarrow \infty} \tilde{Z}_n = \left(
   \frac{M_p}{1 - a} \right)^{\frac{1}{1 - \beta}} . \]
Indeed, since $\bar{Z}$ is such that $a \bar{Z} + M_p {\bar{Z}^{\beta}}  =
\bar{Z}$, it is easy to see that if $Z_1 < \bar{Z}$, then $\tilde{Z}_n$ is
non-decreasing and $\tilde{Z}_n \leqslant \bar{Z}$ for every $n \in
\mathbb{N}$ and if $Z_1 > \bar{Z}$, then $\tilde{Z}_n$ is non-increasing and
$\tilde{Z}_n \geqslant \bar{Z}$ for every $n \in \mathbb{N}$. Therefore, we
conclude that
\[ Z_n \leqslant C (m, \kappa, d, p, C_{\sigma}) \max \left\{ \mathbb{E} [\| u
   \|^p_{D_0^1}]^{\frac{1}{p}}, \Mnoise^{2 \nu}, \Mdumb^{\nu} \right\} \]
which by \eqref{mainresultbound} gives \eqref{growthofmoments}.

\subsection{Proof of Proposition~\ref{regularisedequationwithtransport}} \label{appendixtomaindecomposition}

Here we discuss why the choice for $\tilde{L}_3$ in \eqref{choiceofL}
corresponding to the leading order term in \eqref{controlALdominant3} actually
controls all the other terms in
\eqref{controlALdominant1}-\eqref{controlALnoise}. 
In order to verify this, we
just need to plug in the choice of $\tilde{L}_3$ into each of the estimates in
\eqref{controlALdominant1}-\eqref{controlALnoise} and verify that the
corresponding upper bound is not worse than the one in \eqref{controlonfRHSofv}.
\refchange{
First, we argue that all the terms with prefactors of $\Cnoise{} L^{1 -
\kappa}$ or $(\Cdumb{} + \Cnoise{}^2) L^{1 - \kappa}$ and satisfy an upper bound no worse than \eqref{controlonfRHSofv}. 
Consider the term
\begin{equation}\label{eq:harmless_term}
 \Cnoise{} L^{1 - \kappa} \Cnoise{} L^{2 - 3 \kappa} C^2_{\star}
    (\Cdumb{} \vee \Cnoise{}^2 C_{\star}^{2 \kappa + 2 \delta})^{1 +
    \frac{1}{2 (1 - \kappa)}}
\end{equation}
appearing
\eqref{controlALdominant3}.
With our choice of $L = \tilde{L}_{3}$, the prefactor  $\Cnoise{} L^{1 - \kappa}$ can be written 
\[ \Cnoise{} \left( C^{- \frac{2}{3 - 2 \kappa}}_{\star} (\Cdumb{} \vee
   \Cnoise{}^2 C_{\star}^{2 \kappa + 2 \delta})^{- \frac{1}{2 (1 - \kappa)}}
   \right)^{1 - \kappa} = \Cnoise{} C^{- \frac{2 (1 - \kappa)}{3 - 2
   \kappa}}_{\star} (\Cdumb{} \vee \Cnoise{}^2 C_{\star}^{2 \kappa + 2
   \delta})^{- \frac{1}{2}} \]
which leads to the bound (recall \eqref{controlofRHSofvbeforeyoung})
\begin{eqnarray*}
  &  & \eqref{eq:harmless_term} \leqslant \Cnoise{} C^{- \frac{2 (1 - \kappa)}{3 - 2 \kappa}}_{\star} (\Cdumb{}
   \vee \Cnoise{}^2 C_{\star}^{2 \kappa + 2 \delta})^{- \frac{1}{2}} \Cnoise{}
  C^{\frac{2 (1 + \kappa)}{3 - 2 \kappa}}_{\star} (\Cdumb{} \vee
  \Cnoise{}^2 C_{\star}^{2 \kappa + 2 \delta})^{\frac{1 + \kappa}{2 (1 -
  \kappa)}}\\
  & = & \Cnoise{}^2 \Cdumb{}^{\frac{\kappa}{1 - \kappa}} C^{\frac{4
  \kappa}{3 - 2 \kappa}}_{\star} \vee \Cnoise{}^{\frac{2}{1 - \kappa}}
  C_{\star}^{\frac{4 \kappa}{3 - 2 \kappa} + \frac{2 \kappa (\kappa +
  \delta)}{1 - \kappa}},
\end{eqnarray*}
which is much smaller than \eqref{controlonfRHSofv}, since by Young's
inequality
\[ \Cnoise{}^2 \Cdumb{}^{\frac{\kappa}{1 - \kappa}} \leqslant \frac{1}{1 -
   \kappa} \Cnoise{}^{\frac{2}{1 - \kappa}} + \frac{1}{\kappa} \Cdumb{}^{\frac{1}{1 -
   \kappa}} . \]
A similar argument applies for all the terms with prefactor $\Cnoise{} L^{1 -
\kappa}$ or $(\Cdumb{} + \Cnoise{}^2) L^{1 - \kappa}$ in
\eqref{controlALdominant1}-\eqref{controlALnoise}.
This is why we may only consider the terms without these prefactors in the
sequel.}
Finally, we recall that it is enough to compare with
\eqref{controlofRHSofvbeforeyoung}, since it implies \eqref{controlonfRHSofv}.
We start with \eqref{controlALdominant1}, which gives a bound
\begin{eqnarray*}
  &  & \Cnoise{} \left( C^{- \frac{2}{3 - 2 \kappa}}_{\star} (\Cdumb{}
  \vee \Cnoise{}^2 C_{\star}^{2 \kappa + 2 \delta})^{- \frac{1}{2 (1 - \kappa)}}
  \right)^{1 - 3 \kappa} C_{\star} (\Cdumb{} \vee \Cnoise{}^2 C_{\star}^{2
  \kappa + 2 \delta})\\
  & = & \Cnoise{} \left( C^{1 - \frac{2 (1 - 3 \kappa)}{3 - 2 \kappa}}_{\star}
  (\Cdumb{} \vee \Cnoise{}^2 C_{\star}^{2 \kappa + 2 \delta})^{1 - \frac{1
  - 3 \kappa}{2 (1 - \kappa)}} \right)\\
  & = & \Cnoise{} C^{\frac{1 + 4 \kappa}{3 - 2 \kappa}}_{\star} (\Cdumb{}
  \vee \Cnoise{}^2 C_{\star}^{2 \kappa + 2 \delta})^{\frac{1 + \kappa}{2 (1 -
  \kappa)}},
\end{eqnarray*}
which is less than \eqref{controlofRHSofvbeforeyoung}, since $1 + 4 \kappa
\leqslant 2 + 2 \kappa \Longleftrightarrow \kappa \leqslant 1 / 2$. We move to
\eqref{controlALdominant2}, which leads to
\begin{eqnarray*}
  &  & \Cnoise{} \left( C^{- \frac{2}{3 - 2 \kappa}}_{\star} (\Cdumb{}
  \vee \Cnoise{}^2 C_{\star}^{2 \kappa + 2 \delta})^{- \frac{1}{2 (1 - \kappa)}}
  \right)^{3 - 5 \kappa} C^2_{\star} (\Cdumb{} \vee \Cnoise{}^2
  C_{\star}^{2 \kappa + 2 \delta})^2\\
  & = & \Cnoise{} \left( C^{2 - \frac{2 (3 - 5 \kappa)}{3 - 2 \kappa}}_{\star}
  (\Cdumb{} \vee \Cnoise{}^2 C_{\star}^{2 \kappa + 2 \delta})^{2 - \frac{3
  - 5 \kappa}{2 (1 - \kappa)}} \right)\\
  & = & \Cnoise{} C^{\frac{6 \kappa}{3 - 2 \kappa}}_{\star} (\Cdumb{}
  \vee \Cnoise{}^2 C_{\star}^{2 \kappa + 2 \delta})^{\frac{1 + \kappa}{2 (1 -
  \kappa)}}
\end{eqnarray*}
which is again bounded by \eqref{controlofRHSofvbeforeyoung} since $6 \kappa
\leqslant 2 + 2 \kappa \Longleftrightarrow \kappa \leqslant 1 / 2$. We jump to
\eqref{controlALmixed1}, which gives
\begin{eqnarray*}
  &  & \Cnoise{}^2 \left( C^{- \frac{2}{3 - 2 \kappa}}_{\star} (\Cdumb{}
  \vee \Cnoise{}^2 C_{\star}^{2 \kappa + 2 \delta})^{- \frac{1}{2 (1 - \kappa)}}
  \right)^{1 - 2 \kappa} C_{\star} (\Cdumb{} \vee \Cnoise{}^2 C_{\star}^{2
  \kappa + 2 \delta})^{\frac{1}{2 (1 - \kappa)}}\\
  & = & \Cnoise{}^2 \left( C^{1 - \frac{2 (1 - 2 \kappa)}{3 - 2
  \kappa}}_{\star} (\Cdumb{} \vee \Cnoise{}^2 C_{\star}^{2 \kappa + 2
  \delta})^{\frac{1}{2 (1 - \kappa)} - \frac{1 - 2 \kappa}{2 (1 - \kappa)}}
  \right)\\
  & = & \Cnoise{}^2 C^{\frac{1 + 2 \kappa}{3 - 2 \kappa}}_{\star} (\Cdumb{}
   \vee \Cnoise{}^2 C_{\star}^{2 \kappa + 2 \delta})^{\frac{\kappa}{1 -
  \kappa}}
\end{eqnarray*}
which is less than \eqref{controlofRHSofvbeforeyoung} since
$\Cnoise{} (\Cdumb{} \vee \Cnoise{}^2 C_{\star}^{2 \kappa + 2
   \delta})^{\frac{\kappa}{1 - \kappa}} \leqslant (\Cdumb{} \vee \Cnoise{}^2
   C_{\star}^{2 \kappa + 2 \delta})^{\frac{1 + \kappa}{2 - 2 \kappa}}$
is equivalent to $\Cnoise{} \leqslant (\Cdumb{} \vee \Cnoise{}^2
C_{\star}^{2 \kappa + 2 \delta})^{1 / 2}$. Next is \eqref{controlALmixed2},
giving
\begin{eqnarray*}
  &  &
    \Cnoise{}^{2 + \frac{1 - 2 \kappa}{2 (1 - \kappa)}} \left( C^{- \frac{2}{3 -
    2 \kappa}}_{\star} (\Cdumb{} \vee \Cnoise{}^2 C_{\star}^{2 \kappa + 2
    \delta})^{- \frac{1}{2 (1 - \kappa)}} \right)^{\frac{3}{2} - 3 \kappa} \\
    & & \times C^{\frac{1}{2 (1 - \kappa)}}_{\star} (\Cdumb{} \vee \Cnoise{}^2
    C_{\star}^{2 \kappa + 2 \delta})^{\frac{1}{2 (1 - \kappa)}} \\
  & = & \Cnoise{}^{2 + \frac{1 - 2 \kappa}{2 (1 - \kappa)}} \left(
  C^{\frac{1}{2 (1 - \kappa)} - \frac{3 - 6 k}{3 - 2 \kappa}}_{\star}
  (\Cdumb{} \vee \Cnoise{}^2 C_{\star}^{2 \kappa + 2 \delta})^{\frac{1}{2
  (1 - \kappa)} - \frac{3 - 6 \kappa}{4 (1 - \kappa)}} \right) \\
  & = & \Cnoise{}^{2 + \frac{1 - 2 \kappa}{2 (1 - \kappa)}} C^{\frac{1}{2 (1 -
  \kappa)} - \frac{3 - 6 k}{3 - 2 \kappa}}_{\star} (\Cdumb{} \vee
  \Cnoise{}^2 C_{\star}^{2 \kappa + 2 \delta})^{\frac{6 \kappa - 1}{4 (1 -
  \kappa)}}
\end{eqnarray*}
which is also less than \eqref{controlofRHSofvbeforeyoung} since $\kappa
\leqslant 1 / 2$ clearly implies
\[ \frac{1}{2 (1 - \kappa)} - \frac{3 - 6 k}{3 - 2 \kappa} \leqslant \frac{1 +
   \kappa}{2 (1 - \kappa)} \quad \Longleftrightarrow \quad 0 \leqslant
   \frac{\kappa}{2 (1 - \kappa)} + \frac{3 - 6 k}{3 - 2 \kappa} \]
and
\begin{eqnarray*}
  \Cnoise{}^{1 + \frac{1 - 2 \kappa}{2 (1 - \kappa)}} (\Cdumb{} \vee
  \Cnoise{}^2 C_{\star}^{2 \kappa + 2 \delta})^{\frac{6 \kappa - 1}{4 (1 -
  \kappa)}} & \leqslant & (\Cdumb{} \vee \Cnoise{}^2 C_{\star}^{2 \kappa +
  2 \delta})^{\frac{1 + \kappa}{2 (1 - \kappa)}}\\
  & \Updownarrow & \\
  \Cnoise{}^{1 + \frac{1 - 2 \kappa}{2 (1 - \kappa)}} & \leqslant & (\Cdumb{}
   \vee \Cnoise{}^2 C_{\star}^{2 \kappa + 2 \delta})^{\frac{3 - 4 \kappa}{4
  (1 - \kappa)}}\\
  & \Updownarrow & \\
  1 + \frac{1 - 2 \kappa}{2 (1 - \kappa)} & \leqslant & \frac{3 - 4 \kappa}{2
  (1 - \kappa)} .
\end{eqnarray*}
Moving to \eqref{controlALmixed4}, we find
\begin{eqnarray*}
  &  &
    \Cnoise{}^{1 + \frac{1 - 2 \kappa}{2 (1 - \kappa)}} \left( C^{- \frac{2}{3 -
    2 \kappa}}_{\star} (\Cdumb{} \vee \Cnoise{}^2 C_{\star}^{2 \kappa + 2
    \delta})^{- \frac{1}{2 (1 - \kappa)}} \right)^{\frac{5}{2} - 4 \kappa} \\
    & & \times C^{1 + \frac{1}{2 (1 - \kappa)}}_{\star} (\Cdumb{} \vee
    \Cnoise{}^2 C_{\star}^{2 \kappa + 2 \delta})^{1 + \frac{1}{2 (1 - \kappa)}} \\
  & = & \Cnoise{}^{1 + \frac{1 - 2 \kappa}{2 (1 - \kappa)}} \left( C^{\frac{3 -
  2 \kappa}{2 (1 - \kappa)} - \frac{5 - 8 \kappa}{3 - 2 \kappa}}_{\star}
  (\Cdumb{} \vee \Cnoise{}^2 C_{\star}^{2 \kappa + 2 \delta})^{\frac{3 - 2
  \kappa}{2 (1 - \kappa)} - \frac{5 - 8 \kappa}{4 (1 - \kappa)}} \right)\\
  & = & \Cnoise{}^{1 + \frac{1 - 2 \kappa}{2 (1 - \kappa)}} C^{\frac{3 - 2
  \kappa}{2 (1 - \kappa)} - \frac{5 - 8 \kappa}{3 - 2 \kappa}}_{\star}
  (\Cdumb{} \vee \Cnoise{}^2 C_{\star}^{2 \kappa + 2 \delta})^{\frac{1 + 4
  \kappa}{4 (1 - \kappa)}},
\end{eqnarray*}
which is dominated by \eqref{controlofRHSofvbeforeyoung} since $\kappa
\leqslant 1 / 3$ implies
\[ \frac{3 - 2 \kappa}{2 (1 - \kappa)} - \frac{5 - 8 \kappa}{3 - 2 \kappa}
   \leqslant \frac{1 + \kappa}{2 (1 - \kappa)} \quad \Longleftrightarrow \quad
   \frac{2 - 3 \kappa}{2 (1 - \kappa)} \leqslant 1 \leqslant \frac{5 - 8
   \kappa}{3 - 2 \kappa} \]
and
\begin{eqnarray*}
  \Cnoise{}^{\frac{1 - 2 \kappa}{2 (1 - \kappa)}} (\Cdumb{} \vee \Cnoise{}^2
  C_{\star}^{2 \kappa + 2 \delta})^{\frac{1 + 4 \kappa}{4 (1 - \kappa)}} &
  \leqslant & (\Cdumb{} \vee \Cnoise{}^2 C_{\star}^{2 \kappa + 2
  \delta})^{\frac{1 + \kappa}{2 (1 - \kappa)}}\\
  & \Updownarrow & \\
  \Cnoise{}^{\frac{1 - 2 \kappa}{2 (1 - \kappa)}} & \leqslant & (\Cdumb{}
  \vee \Cnoise{}^2 C_{\star}^{2 \kappa + 2 \delta})^{\frac{1 - 2 \kappa}{4 (1 -
  \kappa)}} .
\end{eqnarray*}
Following there is \eqref{controlALmixed5}, that create two terms. The first
one is
\begin{eqnarray*}
  &  & \Cnoise{}^2 \left( C^{- \frac{2}{3 - 2 \kappa}}_{\star} (\Cdumb{}
  \vee \Cnoise{}^2 C_{\star}^{2 \kappa + 2 \delta})^{- \frac{1}{2 (1 - \kappa)}}
  \right)^{2 - 2 \kappa} C^2_{\star} (\Cdumb{} \vee \Cnoise{}^2
  C_{\star}^{2 \kappa + 2 \delta})^{\frac{1}{1 - \kappa}}\\
  & = & \Cnoise{}^2 C^{2 - \frac{4 - 4 \kappa}{3 - 2 \kappa}}_{\star}
  (\Cdumb{} \vee \Cnoise{}^2 C_{\star}^{2 \kappa + 2 \delta})^{\frac{1}{1
  - \kappa} - 1}
   =  \Cnoise{}^2 C^{\frac{2}{3 - 2 \kappa}}_{\star} (\Cdumb{} \vee
  \Cnoise{}^2 C_{\star}^{2 \kappa + 2 \delta})^{\frac{\kappa}{1 - \kappa}},
\end{eqnarray*}
again less than \eqref{controlofRHSofvbeforeyoung}, where the term
$\Cnoise{} (\Cdumb{} \vee \Cnoise{}^2 C_{\star}^{2 \kappa + 2
\delta})^{\frac{\kappa}{1 - \kappa}}$ already appeared above in the treatment
of \eqref{controlALmixed1}. The second term coming from
\eqref{controlALmixed5} is
\begin{eqnarray*}
  &  &
    \Cnoise{}^{2 + \frac{1 - 2 \kappa}{2 (1 - \kappa)}} \left( C^{- \frac{2}{3 -
    2 \kappa}}_{\star} (\Cdumb{} \vee \Cnoise{}^2 C_{\star}^{2 \kappa + 2
    \delta})^{- \frac{1}{2 (1 - \kappa)}} \right)^{\frac{5}{2} - 3 \kappa} \\
    & & \times C^{1 + \frac{1}{2 (1 - \kappa)}}_{\star} (\Cdumb{} \vee
    \Cnoise{}^2 C_{\star}^{2 \kappa + 2 \delta})^{\frac{1}{1 - \kappa}} \\
  & = & \Cnoise{}^{2 + \frac{1 - 2 \kappa}{2 (1 - \kappa)}} \left( C^{\frac{3 -
  2 \kappa}{2 (1 - \kappa)} - \frac{5 - 6 \kappa}{3 - 2 \kappa}}_{\star}
  (\Cdumb{} \vee \Cnoise{}^2 C_{\star}^{2 \kappa + 2 \delta})^{\frac{1}{1
  - \kappa} - \frac{5 - 6 \kappa}{4 (1 - \kappa)}} \right)\\
  & = & \Cnoise{}^{2 + \frac{1 - 2 \kappa}{2 (1 - \kappa)}} C^{\frac{3 - 2
  \kappa}{2 (1 - \kappa)} - \frac{5 - 6 \kappa}{3 - 2 \kappa}}_{\star}
  (\Cdumb{} \vee \Cnoise{}^2 C_{\star}^{2 \kappa + 2 \delta})^{\frac{6
  \kappa - 1}{4 (1 - \kappa)}},
\end{eqnarray*}
also bounded by \eqref{controlofRHSofvbeforeyoung}, since again $\kappa
\leqslant 1 / 2$ implies
\[ \frac{3 - 2 \kappa}{2 (1 - \kappa)} - \frac{5 - 6 \kappa}{3 - 2 \kappa}
   \leqslant \frac{1 + \kappa}{2 (1 - \kappa)} \quad \Longleftrightarrow \quad
   \frac{2 - 3 \kappa}{2 (1 - \kappa)} \leqslant 1 \leqslant \frac{5 - 6
   \kappa}{3 - 2 \kappa} \]
and the term $\Cnoise{}^{2 + \frac{1 - 2 \kappa}{2 - 2 \kappa}} (\Cdumb{}
\vee \Cnoise{}^2 C_{\star}^{2 \kappa + 2 \delta})^{\frac{6 \kappa - 1}{4 - 4\kappa}}$
 has already appeared in the treatment of \eqref{controlALmixed2}.
Looking to \eqref{controlALmixed3}, we see that
\begin{eqnarray*}
  &  &
    \Cnoise{}^{2 + \frac{1 - 2 \kappa}{1 - \kappa}} \left( C^{- \frac{2}{3 - 2
    \kappa}}_{\star} (\Cdumb{} \vee \Cnoise{}^2 C_{\star}^{2 \kappa + 2
    \delta})^{- \frac{1}{2 (1 - \kappa)}} \right)^{3 - 4 \kappa} \\
    & & \times C^{\frac{1}{1 - \kappa}}_{\star} (\Cdumb{} \vee \Cnoise{}^2
    C_{\star}^{2 \kappa + 2 \delta})^{\frac{1}{1 - \kappa}} \\
  & = & \Cnoise{}^{2 + \frac{1 - 2 \kappa}{1 - \kappa}} \left( C^{\frac{1}{1 -
  \kappa} - \frac{6 - 8 \kappa}{3 - 2 \kappa}}_{\star} (\Cdumb{} \vee
  \Cnoise{}^2 C_{\star}^{2 \kappa + 2 \delta})^{\frac{1}{1 - \kappa} - \frac{3 -
  4 \kappa}{2 (1 - \kappa)}} \right)\\
  & = & \Cnoise{}^{2 + \frac{1 - 2 \kappa}{1 - \kappa}} C^{\frac{1}{1 - \kappa}
  - \frac{6 - 8 \kappa}{3 - 2 \kappa}}_{\star} (\Cdumb{} \vee \Cnoise{}^2
  C_{\star}^{2 \kappa + 2 \delta})^{\frac{4 \kappa - 1}{2 (1 - \kappa)}},
\end{eqnarray*}
which is once again bounded by \eqref{controlofRHSofvbeforeyoung}, since
\[ \frac{1}{1 - \kappa} - \frac{6 - 8 \kappa}{3 - 2 \kappa} \leqslant \frac{1
   + \kappa}{2 (1 - \kappa)} \quad \Longleftrightarrow \quad \frac{1}{2}
   \leqslant \frac{6 - 8 \kappa}{3 - 2 \kappa} \]
and
\begin{eqnarray*}
  \Cnoise{}^{1 + \frac{1 - 2 \kappa}{1 - \kappa}} (\Cdumb{} \vee \Cnoise{}^2
  C_{\star}^{2 \kappa + 2 \delta})^{\frac{4 \kappa - 1}{2 (1 - \kappa)}} &
  \leqslant & (\Cdumb{} \vee \Cnoise{}^2 C_{\star}^{2 \kappa + 2
  \delta})^{\frac{1 + \kappa}{2 (1 - \kappa)}}\\
  & \Updownarrow & \\
  \Cnoise{}^{\frac{2 - 3 \kappa}{1 - \kappa}} & \leqslant & (\Cdumb{} \vee
  \Cnoise{}^2 C_{\star}^{2 \kappa + 2 \delta})^{\frac{2 - 3 \kappa}{2 (1 -
  \kappa)}} .
\end{eqnarray*}
Finally, we see that \eqref{controlALnoise} is bounded by
\begin{eqnarray*}
  &  & (\Cnoise{}^2 + \Cdumb{}) \Cnoise{} \left( C^{- \frac{2}{3 - 2
  \kappa}}_{\star} (\Cdumb{} \vee \Cnoise{}^2 C_{\star}^{2 \kappa + 2
  \delta})^{- \frac{1}{2 (1 - \kappa)}} \right)^{1 - 3 \kappa}\\
  & \lesssim & (\Cdumb{} \vee \Cnoise{}^2) \Cnoise{} \left( C^{- \frac{2 -
  6 \kappa}{3 - 2 \kappa}}_{\star} (\Cdumb{} \vee \Cnoise{}^2 C_{\star}^{2
  \kappa + 2 \delta})^{- \frac{1 - 3 \kappa}{2 (1 - \kappa)}} \right)\\
  & \leqslant & \Cnoise{} C^{- \frac{2 - 6 \kappa}{3 - 2 \kappa}}_{\star}
  (\Cdumb{} \vee \Cnoise{}^2 C_{\star}^{2 \kappa + 2 \delta})^{\frac{1 +
  \kappa}{2 (1 - \kappa)}}
\end{eqnarray*}
which is bounded by \eqref{controlofRHSofvbeforeyoung} since $C_{\star}
\geqslant 1$. This finishes the argument.

\subsection{Proofs of Lemma~\ref{splittinglemma2} and Lemma~\ref{splittinglemma3}.}\label{proofofdecompositionlemmas}
\begin{proof*}{Proof of Lemma \ref{splittinglemma2}.}
  We start by noting that
  \begin{eqnarray*}
    A_{z_1}^{\dumb} (z_2) & = & \int_0^1 \frac{\mathd}{\mathd \lambda}
    \sigma' \sigma (u (z_2) \lambda + (1 - \lambda) u (z_1)) \mathd \lambda\\
    & = & (u (z_2) - u (z_1)) \int_0^1 (\sigma' \sigma)' (u (z_2) \lambda +
    (1 - \lambda) u (z_1)) \mathd \lambda\\
    & = : & I_1 ((\sigma' \sigma)', z_1, z_2) (u (z_2) - u (z_1)) .
  \end{eqnarray*}
  Then, we write $A_{z_1}^{\dumb} (z_2) = A_{z_1}^{\dumb, 1} (z_2) +
  \tilde{B}_{z_1}^{\dumb, 1} (z ; z_2) + A_{z_1}^{\dumb, 2} (z ;
  z_2) + A_{z_1}^{\dumb, 3} (z_2)$, where
  \begin{eqnarray}
    A_{z_1}^{\dumb, 1} (z_2) & = & I_1 ((\sigma' \sigma)', z_1, z_2)
    (U_{z_1} (z_2) - u_X (z_1) \cdot (\Pi_{z_1} \X) (z_2))  \label{reductionA2}\\
    \tilde{B}_{z_1}^{\dumb, 1} (z ; z_2) & = & I_1 ((\sigma' \sigma)',
    z_1, z_2) (\Pi_{z_1} \X) (z_2) \cdot u_X (z)  \label{reductionB1}\\
    A_{z_1}^{\dumb, 2} (z ; z_2) & = & I_1 ((\sigma' \sigma)', z_1, z_2)
    (\Pi_{z_1} \X) (z_2) \cdot \left( u_X \left( {z_1}  \right) - u_X (z )
    \right) \nonumber\\
    A_{z_1}^{\dumb, 3} (z_2) & = & I_1 ((\sigma' \sigma)', z_1, z_2)
    \sigma (u (z_1)) (\Pi_{z_1} \lolli) (z_2) . \nonumber
  \end{eqnarray}
  Now $a_L^{\dumb, 1} (z) = \Lambda_L [A^{\dumb, 1} \Pi \dumb
  ] (z)$ exists and satisfies \eqref{controlALdominant1}, since by
  \eqref{comparisonofnorms}
  \begin{eqnarray*}
    | a_L^{\dumb, 1} (z) | & \leqslant & \| (\sigma' \sigma)' \|
    [U]_{\gamma_2, B (z, L)} [\Pi ; \dumb]_{- 2 \kappa} L^{\gamma_2 - 2
    \kappa}\\
    & \lesssim & \Cdumb{} L^{2 - 4 \kappa} C_{\star} (\Cdumb{}
    \vee \Cnoise{}^2 C_{\star}^{2 \kappa + 2 \delta}) .
  \end{eqnarray*}
  $a_L^{\dumb, 2} (z) = \Lambda_L [A^{\dumb, 2} \Pi \dumb]
  (z)$ also exists and satisfies \eqref{controlALdominant1},
  \eqref{controlALmixed1}, \eqref{controlALmixed2} and \eqref{controlALnoise},
  since by \eqref{reductionofgammaofgradient}
  \begin{eqnarray*}
    &  & | a_L^{\dumb, 2} (z) |\\
    & \leqslant & \| (\sigma' \sigma)' \| [\Pi ; \dumb]_{- 2 \kappa}
    [u_X]_{\gamma_2 - 1, B (z, L)} L^{\gamma_2 - 2 \kappa}\\
    & \lesssim & (\Cdumb{} L^{2 - 4 \kappa} + \Cnoise{} \Cdumb{}
    L^{3 - 5 \kappa}) C_{\star} (\Cdumb{} \vee \Cnoise{}^2 C_{\star}^{2
    \kappa + 2 \delta}) + \Cnoise{}^2 \Cdumb{} L^{2 - 4 \kappa}\\
    & + & \left( \Cnoise{} \Cdumb{} L^{2 - 3 \kappa} C_{\star} +
    \Cnoise{}^{1 + \frac{1 - 2 \kappa}{2 (1 - \kappa)}} \Cdumb{}
    L^{\frac{5}{2} - 4 \kappa} C^{\frac{1}{2 (1 - \kappa)}}_{\star} \right)
    (\Cdumb{} \vee \Cnoise{}^2 C_{\star}^{2 \kappa + 2
    \delta})^{\frac{1}{2 (1 - \kappa)}} .
  \end{eqnarray*}
  The same is true for $a_L^{\dumb, 3} (z) = \Lambda_L [A^{\dumb, 3}
  \Pi \dumb] (z)$, which satisfies \eqref{controlALnoise}, since
  \[ | a_L^{\dumb, 3} (z) | \leqslant \| (\sigma' \sigma)' \| \| \sigma
     \| [\Pi ; \lolli]_{1 - \kappa} [\Pi ; \dumb]_{- 2 \kappa} L^{1 - 3 \kappa}
     \lesssim \Cdumb{} \Cnoise{} L^{1 - 3 \kappa} . \]
  In addition $b^{\dumb, 1}_L (z) \cdot u_X (z ) = \Lambda_L
  [\tilde{B}^{\dumb, 1} \Pi \dumb] (z)$ exists and satisfies
  \eqref{controlonBL}, since
  \[ | b^{\dumb, 1}_L (z) | \leqslant \| (\sigma' \sigma)' \| [\Pi ; \dumb]_{- 2 \kappa} 
  L^{1 - 2 \kappa} \lesssim \Cdumb{} L^{1 - 2 \kappa} . \]
  To conclude, set $a_L^{\dumb, 4} \assign b^{\dumb, 1}_L \cdot (u_X
  - \nabla u_L)$, which satisfies \eqref{controlALdominant1} and
  \eqref{controlALnoise}, since again by \eqref{formulaforgradient},
  \eqref{controlonremaindergradient} and \eqref{comparisonofnorms}
  \begin{eqnarray*}
    | a_L^{\dumb, 4} (z) | & \leqslant & | b^{\dumb, 1}_L (z) | (\|
    \sigma \| [\Pi ; \lolli]_{1 - \kappa} L^{- \kappa} + [U]_{\gamma_2, B (z, L)}
    L^{\gamma_2 - 1})\\
    & \lesssim & \Cdumb{} L^{1 - 2 \kappa} (\Cnoise{} L^{- \kappa} +
    C_{\star} (\Cdumb{} \vee \Cnoise{}^2 C_{\star}^{2 \kappa + 2 \delta})
    L^{1 - 2 \kappa})\\
    & = & \Cdumb{} L^{2 - 4 \kappa} C_{\star} (\Cdumb{} \vee
    \Cnoise{}^2 C_{\star}^{2 \kappa + 2 \delta}) + \Cnoise{} \Cdumb{} L^{1 -
    3 \kappa} .
  \end{eqnarray*}
  To conclude, set
  \[ a_L^{\dumb} (z) \assign \sum_{i = 1}^4 a_L^{\dumb, i} (z)
     \infixand b_L^{\dumb} (z) \assign b^{\dumb, 1}_L (z) . \]
\end{proof*}

\vspace{2mm}
\begin{proof*}{Proof of Lemma \ref{splittinglemma3}.}
  Recall that
  \[ A_{z_1}^{\Xnoise} (z_2) = \sigma' (u (z_2)) (u_X (z_2) - u_X (z_1)) +
     (\sigma' (u (z_2)) - \sigma' (u (z_1))) u_X (z_1), \]
  and we further expand $\sigma' (u (z_2)) - \sigma' (u (z_1))$ in the second
  term above as
  \begin{eqnarray*}
    \sigma' (u (z_2)) - \sigma' (u (z_1)) & = & \int_0^1 \frac{\mathd}{\mathd
    \lambda} \sigma' (\lambda u (z_2) + (1 - \lambda) u (z_1)) \mathd
    \lambda\\
    & = & (u (z_2) - u (z_1)) \int_0^1 \sigma'' (\lambda u (z_2) + (1 -
    \lambda) u (z_1)) \mathd \lambda\\
    & \backassign & I_1 (\sigma'', z_1, z_2) (u (z_2) - u (z_1)) .
  \end{eqnarray*}
  Now, we write $A_{z_1}^{\Xnoise} (z_2)$ as the sum of the following terms
  \begin{eqnarray*}
    A^{\Xnoise, 1}_{z_1} (z_2) & = & \sigma' (u (z_2)) (u_X (z_2) - u_X (z_1))\\
    \tilde{B}^{\Xnoise, 1}_{z_1} (z ; z_2) & = & I_1 (\sigma'', z_1, z_2)
    (U_{z_1} (z_2) - u_X (z_1) \cdot (\Pi_{z_1} \X) (z_2)) u_X (z)\\
    A^{\Xnoise, 2}_{z_1} (z ; z_2) & = & I_1 (\sigma'', z_1, z_2) (U_{z_1} (z_2)
    - u_X (z_1) \cdot (\Pi_{z_1} \X) (z_2)) (u_X (z_1) - u_X (z))\\
    \tilde{B}^{\Xnoise, 2}_{z_1} (z ; z_2) & = & I_1 (\sigma'', z_1, z_2) [u_X
    (z_1) \cdot (\Pi_{z_1} \X) (z_2)] u_X (z)\\
    A^{\Xnoise, 3}_{z_1} (z ; z_2) & = & I_1 (\sigma'', z_1, z_2) [u_X (z_1)
    \cdot (\Pi_{z_1} \X) (z_2)] (u_X (z_1) - u_X (z))\\
    \tilde{B}^{\Xnoise, 3}_{z_1} (z ; z_2) & = & I_1 (\sigma'', z_1, z_2)
    [\sigma (u (z_1)) (\Pi_{z_1} \lolli) (z_2)] u_X (z)\\
    A^{\Xnoise, 4}_{z_1} (z ; z_2) & = & I_1 (\sigma'', z_1, z_2) [\sigma (u
    (z_1)) (\Pi_{z_1} \lolli) (z_2)] (u_X (z_1) - u_X (z))
  \end{eqnarray*}
  and proceed to control each of them as done before. The term $a_L^{\Xnoise, 1}
  (z) = \Lambda_L [A^{\Xnoise, 1} \Pi \Xnoise] (z)$ exists and satisfies
  \eqref{controlALdominant1}, \eqref{controlALmixed1}, \eqref{controlALmixed2}
  and \eqref{controlALnoise}, since by \eqref{reductionofgammaofgradient}
  \begin{eqnarray*}
    &  & | a_L^{\Xnoise, 1} (z) |\\
    & \leqslant & \| \sigma' \| [u_X]_{\gamma_2 - 1, B (z, L)} [\Pi ; \Xnoise]_{-
    \kappa} L^{\gamma_2 - 1 - \kappa}\\
    & \lesssim & (\Cnoise{} L^{1 - 3 \kappa} + \Cnoise{}^2 L^{2 - 4 \kappa})
    C_{\star} (\Cdumb{} \vee \Cnoise{}^2 C_{\star}^{2 \kappa + 2 \delta})
    + \Cnoise{}^3 L^{1 - 3 \kappa}\\
    & + & \left( \Cnoise{}^2 L^{1 - 2 \kappa} C_{\star} + \Cnoise{}^{2 + \frac{1
    - 2 \kappa}{2 (1 - \kappa)}} L^{\frac{3}{2} - 3 \kappa} C^{\frac{1}{2 (1 -
    \kappa)}}_{\star} \right) (\Cdumb{} \vee \Cnoise{}^2 C_{\star}^{2
    \kappa + 2 \delta})^{\frac{1}{2 (1 - \kappa)}} ;
  \end{eqnarray*}
  $a_L^{\Xnoise, 2} (z) = \Lambda_L [A^{\Xnoise, 2} \Pi \Xnoise] (z)$ exists and
  satisfies \eqref{controlALdominant1}, \eqref{controlALdominant2},
  \eqref{controlALdominant3} and \eqref{controlALmixed4}, since by
  \eqref{comparisonofnorms} and \eqref{reductionofgammaofgradient}
  \begin{eqnarray*}
    &  & | a_L^{\Xnoise, 2} (z) |\\
    & \leqslant & \| \sigma'' \| [U]_{\gamma_2, B (z, L)} [u_X]_{\gamma_2, B
    (z, L)} [\Pi ; \Xnoise]_{- \kappa} L^{2 \gamma_2 - 1 - \kappa}\\
    & \lesssim & (\Cnoise{} L^{3 - 5 \kappa} + \Cnoise{}^2 L^{4 - 6 \kappa})
    C^2_{\star} (\Cdumb{} \vee \Cnoise{}^2 C_{\star}^{2 \kappa + 2
    \delta})^2\\
    & + & \Cnoise{}^3 L^{3 - 5 \kappa} C_{\star} (\Cdumb{} \vee \Cnoise{}^2
    C_{\star}^{2 \kappa + 2 \delta})\\
    & + & \left( \Cnoise{}^2 L^{3 - 4 \kappa} C_{\star} + \Cnoise{}^{2 + \frac{1
    - 2 \kappa}{2 (1 - \kappa)}} L^{\frac{7}{2} - 5 \kappa} C^{\frac{1}{2 (1 -
    \kappa)}}_{\star} \right) C_{\star} (\Cdumb{} \vee \Cnoise{}^2
    C_{\star}^{2 \kappa + 2 \delta})^{1 + \frac{1}{2 (1 - \kappa)}} ;
  \end{eqnarray*}
  $a_L^{\Xnoise, 3} (z) = \Lambda_L [A^{\Xnoise, 3} \Pi \Xnoise] (z)$ exists and
  satisfies \eqref{controlALdominant3}, \eqref{controlALmixed1},
  \eqref{controlALmixed2}, \eqref{controlALmixed4}, \eqref{controlALmixed5}
  and \eqref{controlALmixed3}, since by \eqref{reductionofLinfinityofgradient}
  and \eqref{reductionofgammaofgradient} (same treatment as for $A^{\noise, 4}_L
  (z)$ in Lemma \ref{splittinglemma1}.)
  \begin{eqnarray*}
    &  & | a_L^{\Xnoise, 3} (z) |\\
    & \leqslant & \| \sigma'' \| \| u_X \|_{B (z, L)} [u_X]_{\gamma_2 - 1, B
    (z, L)} [\Pi ; \Xnoise]_{- \kappa} L^{\gamma_2 - \kappa}\\
    & \lesssim &
      (1 + \Cnoise{} L^{1 - \kappa}) \left( \Cnoise{} L^{2 - 3 \kappa} C^2_{\star}
      + \Cnoise{}^{1 + \frac{1 - 2 \kappa}{2 (1 - \kappa)}} L^{\frac{5}{2} - 4
      \kappa} C^{1 + \frac{1}{2 (1 - \kappa)}}_{\star} \right) \\
      & & \times (\Cdumb{} \vee \Cnoise{}^2 C_{\star}^{2 \kappa + 2
      \delta})^{1 + \frac{1}{2 (1 - \kappa)}} \\
    & + & \left( \Cnoise{}^3 L^{2 - 3 \kappa} C_{\star} + \Cnoise{}^{3 + \frac{1
    - 2 \kappa}{2 (1 - \kappa)}} L^{\frac{5}{2} - 4 \kappa} C^{\frac{1}{2 (1 -
    \kappa)}}_{\star} \right) (\Cdumb{} \vee \Cnoise{}^2 C_{\star}^{2
    \kappa + 2 \delta})^{\frac{1}{2 (1 - \kappa)}}\\
    & + & \left( \Cnoise{}^2 L^{2 - 2 \kappa} C^2_{\star} + \Cnoise{}^{2 +
    \frac{1 - 2 \kappa}{2 (1 - \kappa)}} L^{\frac{5}{2} - 3 \kappa} C^{1 +
    \frac{1}{2 (1 - \kappa)}}_{\star} \right) (\Cdumb{} \vee \Cnoise{}^2
    C_{\star}^{2 \kappa + 2 \delta})^{\frac{1}{1 - \kappa}}\\
    & + &
      \left( C^{2 + \frac{1 - 2 \kappa}{2 (1 - \kappa)}}_{\xi} L^{\frac{5}{2}
      - 3 \kappa} C^{1 + \frac{1}{2 (1 - \kappa)}}_{\star} + \Cnoise{}^{2 +
      \frac{1 - 2 \kappa}{1 - \kappa}} L^{3 - 4 \kappa} C^{\frac{1}{1 -
      \kappa}}_{\star} \right) \\
      & & \times (\Cdumb{} \vee \Cnoise{}^2 C_{\star}^{2 \kappa + 2
      \delta})^{\frac{1}{1 - \kappa}} ;
  \end{eqnarray*}
  $a_L^{\Xnoise, 4} (z) = \Lambda_L [A^{\Xnoise, 4} \Pi \Xnoise] (z)$ exists and
  satisfies \eqref{controlALdominant1}, \eqref{controlALmixed1},
  \eqref{controlALmixed2} and \eqref{controlALnoise}, since by
  \eqref{reductionofgammaofgradient}
  \begin{eqnarray*}
    &  & | a_L^{\Xnoise, 4} (z) |\\
    & \leqslant & \| \sigma'' \| \| \sigma \| [\Pi ; \lolli]_{1 - \kappa}
    [u_X]_{\gamma_2 - 1, B (z, L)} [\Pi ; \Xnoise]_{- \kappa} L^{\gamma_2 - 2
    \kappa}\\
    & \lesssim & (\Cnoise{}^2 L^{2 - 4 \kappa} + \Cnoise{}^3 L^{3 - 5 \kappa})
    C_{\star} (\Cdumb{} \vee \Cnoise{}^2 C_{\star}^{2 \kappa + 2 \delta})
    + \Cnoise{}^4 L^{2 - 4 \kappa}\\
    & + & \left( \Cnoise{}^3 L^{2 - 3 \kappa} C_{\star} + \Cnoise{}^{3 + \frac{1
    - 2 \kappa}{2 (1 - \kappa)}} L^{\frac{5}{2} - 4 \kappa} C^{\frac{1}{2 (1 -
    \kappa)}}_{\star} \right) (\Cdumb{} \vee \Cnoise{}^2 C_{\star}^{2
    \kappa + 2 \delta})^{\frac{1}{2 (1 - \kappa)}} ;
  \end{eqnarray*}
  $b^{\Xnoise, 1}_L (z) \cdot u_X (z ) = \Lambda_L [\tilde{B}^{\Xnoise, 1}
  \Pi \Xnoise] (z)$ exists and satisfies \eqref{controlonBL3}, since by
  \eqref{comparisonofnorms}
  \[ | b^{\Xnoise, 1}_L (z) | \leqslant \| \sigma'' \| [U]_{\gamma_2, B (z, L)}
     [\Pi ; \Xnoise]_{- \kappa} L^{\gamma_2 - \kappa} \lesssim \Cnoise{} L^{2 - 3
     \kappa} C_{\star} (\Cdumb{} \vee \Cnoise{}^2 C_{\star}^{2 \kappa + 2
     \delta}), \]
  and setting $a_L^{\Xnoise, 5} \assign b^{\Xnoise, 1}_L \cdot (u_X - \nabla
  u_L)$, it satisfies \eqref{controlALdominant1} and
  \eqref{controlALdominant2}, since by \eqref{formulaforgradient},
  \eqref{controlonremaindergradient} and \eqref{comparisonofnorms}
  \begin{eqnarray*}
    &  & | a_L^{\Xnoise, 5} (z) |\\
    & \leqslant & | b^{\Xnoise, 1}_L (z) | (\| \sigma \| [\Pi ; \lolli]_{1 - \kappa}
    L^{- \kappa} + [U]_{\gamma_2, B (z, L)} L^{\gamma_2 - 1})\\
    & \lesssim & \Cnoise{} L^{3 - 5 \kappa} C^2_{\star} (\Cdumb{} \vee
    \Cnoise{}^2 C_{\star}^{2 \kappa + 2 \delta})^2 + \Cnoise{}^2 L^{2 - 4 \kappa}
    C_{\star} (\Cdumb{} \vee \Cnoise{}^2 C_{\star}^{2 \kappa + 2 \delta})
    ;
  \end{eqnarray*}
  $b^{\Xnoise, 2}_L (z) \cdot u_X (z ) = \Lambda_L [\tilde{B}^{\Xnoise, 2} \Pi \Xnoise
  ] (z)$ exists and satisfies \eqref{controlonBL2}, since by
  \eqref{reductionofLinfinityofgradient}
  \begin{eqnarray*}
    &  & | b^{\Xnoise, 2}_L (z) |\\
    & \leqslant & \| \sigma'' \| \| u_X \|_{B (z, L)} [\Pi ; \Xnoise]_{- \kappa} L^{1
    - \kappa}\\
    & \lesssim & \left( \Cnoise{} L^{1 - \kappa} C_{\star} + \Cnoise{}^{1 +
    \frac{1 - 2 \kappa}{2 (1 - \kappa)}} L^{\frac{3}{2} - 2 \kappa}
    C^{\frac{1}{2 (1 - \kappa)}}_{\star} \right) (\Cdumb{} \vee
    \Cnoise{}^2 C_{\star}^{2 \kappa + 2 \delta})^{\frac{1}{2 (1 - \kappa)}},
  \end{eqnarray*}
  and setting $a_L^{\Xnoise, 6} \assign b^{\Xnoise, 2}_L \cdot (u_X - \nabla
  u_L)$, it satisfies \eqref{controlALdominant3}, \eqref{controlALmixed1},
  \eqref{controlALmixed2} and \eqref{controlALmixed4} since by
  \eqref{formulaforgradient}, \eqref{controlonremaindergradient} and
  \eqref{comparisonofnorms}
  \begin{eqnarray*}
    &  & | a_L^{\Xnoise, 6} (z) |\\
    & \leqslant & | b^{\Xnoise, 2}_L (z) | (\| \sigma \| [\Pi ; \lolli]_{1 - \kappa}
    L^{- \kappa} + [U]_{\gamma_2, B (z, L)} L^{\gamma_2 - 1})\\
    & \lesssim & \left( \Cnoise{} L^{2 - 3 \kappa} C^2_{\star} + \Cnoise{}^{1 +
    \frac{1 - 2 \kappa}{2 (1 - \kappa)}} L^{\frac{5}{2} - 4 \kappa} C^{1 +
    \frac{1}{2 (1 - \kappa)}}_{\star} \right) (\Cdumb{} \vee \Cnoise{}^2
    C_{\star}^{2 \kappa + 2 \delta})^{1 + \frac{1}{2 (1 - \kappa)}}\\
    & + & \left( \Cnoise{}^2 L^{1 - 2 \kappa} C_{\star} + \Cnoise{}^{2 + \frac{1
    - 2 \kappa}{2 (1 - \kappa)}} L^{\frac{3}{2} - 3 \kappa} C^{\frac{1}{2 (1 -
    \kappa)}}_{\star} \right) (\Cdumb{} \vee \Cnoise{}^2 C_{\star}^{2
    \kappa + 2 \delta})^{\frac{1}{2 (1 - \kappa)}} ;
  \end{eqnarray*}
  $b^{\Xnoise, 3}_L (z) \cdot u_X (z ) = \Lambda_L [\tilde{B}^{\Xnoise, 3} \Pi \Xnoise
  ] (z)$ exists and satisfies \eqref{controlonBL}, since
  \[ | b^{\Xnoise, 3}_L (z) | \leqslant \| \sigma'' \| \| \sigma \| [\Pi ; \lolli]_{1
     - \kappa} [\Pi ; \Xnoise]_{- \kappa} L^{1 - 2 \kappa} \lesssim \Cnoise{}^2 L^{1 - 2
     \kappa}, \]
  and setting $a_L^{\Xnoise, 7} \assign b^{\Xnoise, 3}_L \cdot (u_X - \nabla
  u_L)$, it satisfies \eqref{controlALdominant1} and \eqref{controlALnoise}
  since by \eqref{formulaforgradient}, \eqref{controlonremaindergradient} and
  \eqref{comparisonofnorms}
  \begin{eqnarray*}
    | a_L^{\Xnoise, 7} (z) | & \leqslant & | b^{\Xnoise, 3}_L (z) | (\| \sigma \|
    [\Pi ; \lolli]_{1 - \kappa} L^{- \kappa} + [U]_{\gamma_2, B (z, L)} L^{\gamma_2
    - 1})\\
    & \lesssim & \Cnoise{}^2 L^{2 - 4 \kappa} C_{\star} (\Cdumb{} \vee
    \Cnoise{}^2 C_{\star}^{2 \kappa + 2 \delta}) + \Cnoise{}^3 L^{1 - 3 \kappa} ;
  \end{eqnarray*}
  we conclude the proof by setting
  \[ A^{\Xnoise}_L (z) \assign \sum^7_{i = 1} a_L^{\Xnoise, i} (z)  \infixand b^{\Xnoise
     }_L (z) \assign \sum_{i = 1}^3 b^{\Xnoise, i}_L (z) . \]
  
\end{proof*}

\section{Symbolic Index}
\label{app:notation}

Below we list various constants, norms and other objects used in this article along 
with their meanings and references to their definitions in the main text.

\begin{center}
\begin{longtable}{p{.12\textwidth}p{.65\textwidth}p{.12\textwidth}}
\toprule
Object & Meaning & Ref. \\
\midrule
\endhead
\bottomrule
\endfoot
$\kappa$ & $\xi \in C^{-1-\kappa}$, equation is singular for $\kappa > 0$ & p.~\pageref{kappaintro}\\
$d(z,\bar{z})$ & Parabolic distance between space-time points $z$, $\bar{z}$ & p.~\pageref{parabolicdistance}\\
$B(z,L)$ & Parabolic ball in the past, centered at $z$ of radius $L$ & p.~\pageref{balllookingtopast}\\
 $\varphi_z^L$ & Space-time test function centered at $z$ with scale $L$ & p.~\pageref{defofmollifierkernel}\\
 $[\Pi;\btau]_{|\btau|, B}$ & ``order bound'' or ``OB'' - regularity seminorm for $\btau$, localised to region $B$ & p.~\pageref{orderbounds},~\pageref{orderboundlollipop}\\
$D_a^b$ & space-time points with time $t \in [a,b]$, with $D = D_{0}^{\infty}$ & p.~\pageref{spacetimedomain}\\
$\|\cdot\|$ & supremum norm over $D$, $D \times D$, $\R$, or $\R^{d}$ & p.~\pageref{supnorm0},\pageref{supnorm2}\\
$\|\cdot\|_{B}$ & supremum norm over $B \subset \R^{d+1}$ or $B \times B \subset \R^{d+1}$ & p.~\pageref{supnorm1},\pageref{supnorm2}\\
$C_{\sigma}$ & Bound on non-linearity $\sigma$ and $\sigma'$, $\sigma''$ & p.~\pageref{supnorm1}\\
$\Cnoise{,n}$ & Order bound on objects linear in the noise over $[n-1,n] \times \T^{d}$, $\Cnoise{} \assign \Cnoise{,1}$ & p.~\pageref{assumptiononstochasticobjects}\\
$\Cdumb{,n}$ & Order bound on objects quadratic in the noise over $[n-1,n] \times \T^{d}$, $\Cdumb{} \assign \Cdumb{,1}$ & p.~\pageref{assumptiononstochasticobjects}\\
$\si(u) \diamond \xi$ & renormalised product between $\si(u)$ and $\xi$ & p.~\pageref{mainequationrenormalised}\\
$u_{X}$ & ``Modelledness'' gradient of $u$ & p.~\pageref{page_generalised_grad1},\pageref{page_generalised_grad2}\\
$u_{L}$  & Solution locally averaged at scale $L > 0$ &p.~\pageref{firstdefinitionofuL}\\
$U_{z}(w)$ & ``Taylor''/modelled remainder of solution $u$ with base point $z$ and evaluated at $w$ &p.~\pageref{defofU}\\
$\sigma (U)_z (w)$ & Same as above for $\sigma(u)$ &p.~\pageref{defsigmaofU}\\
$[U]_{\gamma,B}$ & Modelledness regularity seminorm of $U$, with exponent $\gamma \in (1,2)$ and localised to $B$  &p.~\pageref{finitegammanormofU}\\
$C_{\star}$, $T_{\star}$ &Respectively, bootstrap threshold and bootstrap time &p.~\pageref{definitionstoppingtime}\\
$[u]_{\alpha,B}$ & Classical H\"{o}lder seminorm with exponent $\alpha \in (0,1)$ and localised to $B$  
&p.~\pageref{finitegammanormofU}\\
$C^{\sigma,\tau}_{z,L}$ & H\"{o}lder seminorm over $B(z,L)$ of coefficient of $\tau$ in modelled expansion of $\sigma(u)$  
&p.~\pageref{page_modelled_coeff_bounds}\\
$G_{z}(w)$  & Local modelled expansion of $\sigma(u)\xi$ with base point $z$ and evaluated at $w$ &p.~\pageref{localapproxG} \\
$[U]_{\gamma, \eta, [a, b]}$  & Modelledness seminorm on $[a,b] \times \T^{d}$, with regularity exponent $\gamma$ and blow up exponent $\eta$ near $t=a$ &p.~\pageref{normofUblowup}\\
$d_{\tau}$ & $d_{\tau} = \sqrt{\tau - a}$ when working at time $\tau \in [a,b]$ on $D_{a}^{b}$ &p.~\pageref{dist_to_bdry}\\
$E_z^L$ & Inexplicit term in difference of $\nabla u_{L}(z)$ and $u_{X}(z)$ &p.~\pageref{page_gradientdifference} \\
$\mathcal{R}_{L}$ & Reconstruction error (at scale $L$) of $\sigma(u)\xi$ and $G$ &p.~\pageref{page_reconstructionterms}\\
$\tilde{B}_L(z)$ & Term with $u_{X}$ in $\langle G_z,\varphi_{z}^{L} \rangle$ &p.~\pageref{page_reconstructionterms}\\
 $\Xi_L(z)$ & Terms with $\Pi_z \dumb$ and $\Pi_z \noise$ in $\langle G_z,\varphi_{z}^{L} \rangle$ &p.~\pageref{page_reconstructionterms}\\
\end{longtable}
\end{center}

\bibliographystyle{alpha}
\bibliography{documents.bib}

\end{document}